\documentclass[11pt,british]{article}
\usepackage{declare-style-article}
\usepackage{declare-maths-article}
\usepackage{declare-text-article}
\usepackage{declare-theorems-article}
\hypersetup{
	pdfauthor 
		= {Mohamad Alameddine, Olivier Marchal, Nikita Nikolaev, Nicolas Orantin},
	pdftitle 
		= {Resurgence of the Deformed Painleve I Equation},
	pdfsubject
		= {},
	pdfcreator
		= {},
	pdfproducer
		= {},
	pdfkeywords
		= {},
}

\DeclareSymbolFont{bbold}{U}{bbold}{m}{n}
\DeclareSymbolFontAlphabet{\mathbbold}{bbold}

\usepackage{appendix}
\usepackage[font={footnotesize}]{caption}
\usepackage[utf8]{inputenc}
\fancypagestyle{frontpage}{

\cfoot{}
\lfoot{}
}

\usepackage{titletoc}
\usepackage{subcaption}
\usepackage{changepage}

\newcommand{\inv}{{\scriptscriptstyle{\text{--}1}}}

\newcommand{\sfSigma}{\sf{\Sigma}}

\newcommand{\tinybrac}[1]{{\scriptscriptstyle(#1)}}

\makeatletter
\renewcommand\@dotsep{200}
\makeatother

\usetikzlibrary{shapes.misc}
\tikzset{cross/.style={cross out, draw=black, ultra thick, minimum size=2*(#1-\pgflinewidth), inner sep=1pt, outer sep=1pt}}
\usetikzlibrary{decorations.markings}
\tikzset{>=stealth}
\tikzset{->-/.style={decoration={
	markings,
	mark=at position #1 with {\arrow{>}}},
	postaction={decorate}}}
\tikzset{-<-/.style={decoration={
	markings,
	mark=at position #1 with {\arrow{<}}},
	postaction={decorate}}}

\usetikzlibrary{decorations.pathmorphing}
\usetikzlibrary{hobby}

\begin{document}

\newgeometry{top=1.5cm, bottom=0.5cm, left=2.5cm, right=2.5cm}


\title{Resurgence of Tritronquées Solutions \\ of the Deformed Painlevé I Equation}

\author{Mohamad Alameddine, Olivier Marchal, Nikita Nikolaev, and Nicolas Orantin}

\date{3 April 2025}

\maketitle
\thispagestyle{frontpage}

\begin{abstract}
\noindent
We prove that the formal $\hbar$-power series solution of the deformed Painlevé I equation is resurgent, which means it is generically Borel summable and its Borel transform admits endless analytic continuation.
In particular, we find that the Borel transform defines a global multivalued holomorphic function on a singular algebraic surface isomorphic to the Fermat quintic surface $x^5 + y^5 + z^5 = 0$ modulo an involution.
This surface is an algebraic fibration over the complex plane of the differential equation with generic fibre a smooth quintic curve.
Each fibre is equipped with a fivefold covering map over another complex plane (the Borel plane) with ten ramification points (the Borel singularities) spread equally over two branch points giving two opposite Stokes rays.
\end{abstract}

{\footnotesize

\textbf{2020 MSC:} 
\MSCSubjectCode{34A26} (primary),
\MSCSubjectCode{34M40},
\MSCSubjectCode{34M55},
\MSCSubjectCode{34M04},
\MSCSubjectCode{34M25}
}


{\begin{spacing}{0.9}
\small
\setcounter{tocdepth}{2}
\tableofcontents
\end{spacing}
}

\newpage
\restoregeometry

\newpage
\setcounter{section}{-1}
\counterwithout{equation}{section}
\section{Introduction}

In this paper, we study the \dfn{deformed Painlevé I equation}; i.e., the second-order nonlinear ordinary differential equation
\begin{eqntag}
\label{250204124541}
	\hbar^2 \ddot{q} = 6q^2 + t
\end{eqntag}
in the complex $t$-plane depending on a complex parameter $\hbar \in \CC$.
When $\hbar = 1$, this is nothing but the usual Painlevé I equation.
In fact, for all nonzero $\hbar$, the deformed Painlevé I equation can be transformed into the usual Painlevé I equation by a simple rescaling change of variables which also trades the $\hbar \to 0$ limit for the $t \to \infty$ limit.
However, the deformed Painlevé I equation has an even richer structure because, unlike the usual Painlevé I equation, it is a \textit{family} of ordinary differential equations indexed by $\hbar$ with the notable feature that the nature of the problem is substantially different, and simpler, for $\hbar = 0$ when it becomes the quadratic equation
\begin{eqntag}
\label{250312055354}
	6q_0^2 + t = 0
\fullstop
\end{eqntag}
Our interest is to understand the behaviour of solutions in $\hbar$ and, in particular, to reconstruct \textit{exact} solutions of \eqref{250204124541} (meaning holomorphic in both $t$ and $\hbar$ in some domain) starting from a solution of the much simpler problem \eqref{250312055354}.
In other words, we search for holomorphic solutions $q = q (t, \hbar)$ with prescribed asymptotic behaviour as $\hbar \to 0$, casting this problem within the framework of so-called \textit{exact perturbation theory}.

Upon selecting a root of \eqref{250312055354}, a straightforward calculation shows that the deformed Painlevé I equation can be solved uniquely in $\hbar$-power series with locally holomorphic coefficients,
\begin{eqntag}
\label{250205071510}
	\hat{q} (t,\hbar) = \sum_{n = 0}^\infty q_n (t) \hbar^n
\fullstop
\end{eqntag}
It turns out that only the even-order coefficients $q_{2n}$ are nonzero and, remarkably, they can be determined completely explicitly: $q_{2n} = a_{2n} t^{(1-5n)/2}$ for $a_{2n} \in \CC$; see \autoref{250204162846}.
However, $\hat{q}$ is strictly a \textit{formal solution} because the constants $a_{2n}$ exhibit factorial growth in $n$, so the power series \eqref{250205071510} is divergent and hence does not define a holomorphic function of $\hbar$.
Standard asymptotic existence methods allow us to lift the formal solution $\hat{q}$ to true holomorphic solutions $q$ which are asymptotic to $\hat{q}$ as $\hbar \to 0$ in sectors of the $\hbar$-plane, but such lifts are highly non-unique and typically non-constructive.
The aim of exact perturbation theory is to construct such asymptotic lifts in a canonical and explicit way using the \textit{Borel resummation}.
This lifts the divergent series $\hat{q}$ in a uniquely prescribed way, depending only on a phase $\alpha \in \RR / 2\pi\ZZ$, to a holomorphic function, defined in a sector of the $\hbar$-plane of opening angle $\pi$ bisected by $\alpha$, given by a Laplace integral:
\begin{eqntag}
\label{250312063806}
	q_\alpha (t,\hbar)
		\coleq s_\alpha [\, \hat{q} \, ] (t,\hbar)
		= q_0 (t) + \int_{e^{i\alpha} \RR_+} e^{-\xi/\hbar} \hat{\omega} (t,\xi) \d{\xi}
\fullstop
\end{eqntag}
The integrand $\hat{\omega}$, called the \textit{Borel transform} of $\hat{q}$, is the convergent power series in $\xi$ obtained explicitly from $\hat{q}$ by the formula
\begin{eqntag}
\label{250312070005}
	\hat{\omega} (t,\xi)
		= \sum_{n = 0}^\infty \frac{q_{n+1} (t)}{n!} \xi^{n}
\fullstop
\end{eqntag}
Whether or not the Laplace integral in \eqref{250312063806} is well-defined is a question about the geometry of the analytic continuation of $\hat{\omega}$ in the $\xi$-plane, sometimes called the \textit{Borel plane}.
Obstructions to this analytic continuation are the $t$-dependent singularities of $\hat{\omega}$ in the $\xi$-plane, called \textit{Borel singularities}.
The directions $\alpha$ in which the integration contour in \eqref{250312063806} encounters a Borel singularity for a given $t$ are called \textit{Stokes rays}.
As $\alpha$ varies without crossing a Stokes ray, the different Borel resummations $q_\alpha$ assemble into a single holomorphic function defined in a larger sector of the $\hbar$-plane.
But as $\alpha$ crosses a Stokes ray, the Borel resummation of $\hat{q}$ exhibits a discontinuous jump, offering a beautiful example of the Stokes phenomenon.

\textit{Resurgence} is the study of the global geometric structure of singularities of the Borel transform.
Its aim is to identify an exceptionally well-behaved class of divergent series called \dfn{resurgent series}.
Roughly speaking, these are characterised by being generically \textit{Borel summable} (i.e., such that the Laplace integral in \eqref{250312063806} is well-defined for almost every $t$ and almost every $\alpha$) and such that the associated Stokes phenomenon across Stokes rays can be described in terms of the Borel resummation of new divergent series extracted from the singularities of the Borel transform.
In particular, the Borel transform must admit an \textit{endless analytic continuation}; i.e., it can be analytically continued along all paths in the Borel $\xi$-plane that avoids the Borel singularities.

\paragraph{Main results.}
In this paper, we give a complete description of the geometry of the analytic continuation of the Borel transform $\hat{\omega}$.
Our main theorem can be formulated as follows.

\enlargethispage{10pt}

{\bfseries\autoref{250304131355}.}
{\itshape The formal solution $\hat{q} (t,\hbar)$ of the deformed Painlevé I equation is resurgent.}

Resurgent is a loaded concept which for the deformed Painlevé I we describe in detail in \autoref{250313153618}.
More specifically, we prove that the Borel transform $\hat{\omega} (t,\xi)$ of $\hat{q} (t,\hbar)$ defined in \eqref{250312070005} is convergent (\autoref{250312151012}), and describe the full geometry of the resurgent structure of $\hat{q}$ which can be summarised as follows.

{\bfseries\autoref{250210171316} and \autoref{250313152819} (endless analytic continuation; summary).}
{\itshape The Borel transform $\hat{\omega} (t,\xi)$ admits endless analytic continuation of exponential type in all directions for every nonzero $t$.
More specifically:
\begin{enumerate}
\item There is a complex algebraic surface $M$ (the \dfn{Borel space}) and an algebraic curve $S \subset M$ (the locus of \dfn{Borel singularities}) such that the Borel transform $\hat{\omega}$ naturally extends to a global multivalued holomorphic function $\omega$ on the complement $M^\ast \coleq M \smallsetminus S$ contained in the smooth locus of $M$.
Namely, $\omega$ is a global holomorphic function with exponential bounds at infinity on a two-dimensional holomorphic manifold $\tilde{M}$ equipped with a holomorphic surjective submersion $\nu : \tilde{M} \to M^\ast$.
\item The surface $M$ is isomorphic to the quotient of the Fermat quintic surface $x^5 + y^5 + z^5 = 0$ in $\CC^3$ by the involution automorphism $\sigma : (x,y,z) \mapsto - (y,x,z)$.
It has only one singular point located at the origin, and the locus of Borel singularities $S$ corresponds to the intersections of the Fermat quintic with the planes $x = 0$ and $y = 0$; see \autoref{250312110535}.
\item There is an algebraic surjective submersion $\rm{s} : M \to \CC$ (the \dfn{source map}) to the Riemann surface of the square-root of $t$ defined by the quadratic equation \eqref{250312055354}, such that every nonzero fibre $M_{\tau \neq 0} \coleq \rm{s}^\inv (\tau)$ (the \dfn{Borel surface}) is a smooth algebraic curve, isomorphic to the quotient of a plane quintic curve by an involution, and the intersection $S_\tau \coleq S \cap M_\tau$ consists of exactly ten points.
Moreover, the restriction of the submersion $\nu : \tilde{M} \to M^\ast$ is the universal covering map $\nu : \tilde{M}_\tau = \widetilde{M_\tau \smallsetminus S_\tau} \to M_\tau \smallsetminus S_\tau$.
\item There is an algebraic surjective map $\Z : M \to \CC$ (the \dfn{central charge}) to the \dfn{Borel $\xi$-plane}, which is a submersion away from $S$, and whose restriction $\Z_\tau : M_\tau \to \CC$ to any nonzero fibre $M_{\tau \neq 0}$ is a fivefold covering map ramified at each of the ten Borel singularities in $S_\tau$, each with ramification order $5$, and distributed equally over two opposite branch points in the $\xi$-plane (the \dfn{Borel singular values}); see \autoref{250304114549}.
\end{enumerate}
}

The significance of understanding the global geometry of the Borel transform $\hat{\omega}$ is the ability to draw conclusions regarding the Borel summability of the formal solution $\hat{q}$.
We find (\autoref{250223133010}) that for every nonzero point in the $t$-plane, the formal power series $\hat{q} (t, \hbar)$ is Borel summable in all but two directions $\alpha_\pm$, called \dfn{Stokes directions}, corresponding to the two Borel singular values $\xi_\pm$; see \autoref{250313163646}.
As $t$ varies, the Stokes rays $\alpha_\pm$ rotate at $5/4$ the speed, which can be used to trace out maximal domains in the $t$-plane where Borel resummation in a fixed direction $\alpha$ is well-defined.

{\bfseries\autoref{250210133432} and \autoref{250210115249} (Borel summability in Stokes sectors; summary).}
{\itshape
For any phase $\alpha$, the $t$-plane is divided into five pairwise overlapping \dfn{Stokes regions} $U_k$ with a distinguished choice of the square-root branch (see \autoref{250224185222}).
For each Stokes region, the corresponding branch of the formal solution $\hat{q} (t, \hbar)$ of the deformed Painlevé I equation is Borel summable in the direction $\alpha$ locally uniformly for all $t \in U_k$.
Thus, the Borel resummation $q_\alpha (t,\hbar) = s_\alpha [\, \hat{q} \,] (t,\hbar)$ given by \eqref{250312063806} determines a holomorphic solution of the deformed Painlevé I equation well-defined for all $t \in U_k$ and all sufficiently small $\hbar$ in the halfplane bisected by $\alpha$.
Consequently, for every $\alpha$ there are in total five such special holomorphic solutions $q_\alpha$ (which we call the \dfn{deformed tritronquée solutions}) that are uniquely specified in sectors of the $t$-plane by their asymptotic expansion as $\hbar \to 0$ in the direction $\alpha$.
}

Finally, thanks to the detailed geometric description of the resurgent structure, we are able to calculate relevant quantities such as the Stokes jumps.

{\bfseries\autoref{250314214443} (Stokes phenomenon; summary).}
{\itshape
If $\alpha$ is a Stokes ray for a fixed nonzero $t_0$ and $\xi_0 \in \CC$ is the corresponding Borel singular value, then the formal power series $\hat{q} (t_0, \hbar)$ is laterally Borel summable in the direction $\alpha$, and the Stokes jump across the Stokes ray $\alpha$ (i.e., the difference be the left and right lateral resummations $q_\alpha^\textup{L}$ and $q_\alpha^\textup{R}$) is given by the formula
\begin{eqn}
	\Delta_{\alpha} \hat{q} (t_0, \hbar)
	\coleq q_\alpha^\textup{L} (t_0, \hbar) - q_\alpha^\textup{R} (t_0, \hbar)
	= e^{-\xi_0/\hbar} \Laplace_\alpha \big[ \Delta_{\xi_0} \omega \big] (t_0, \hbar)
\fullstop{,}
\end{eqn}
where $\Delta_{\xi_0} \omega (t_0, \xi) \coleq \omega (t_0, \xi_0 + \xi^\textup{L}) - \omega (t_0, \xi_0 + \xi^\textup{R})$ is the variation of $\omega$ at $\xi_0$; i.e., the difference between its values on two consecutive sheets of the universal cover of the punctured neighbourhood of $\xi_0$, so that $\xi^\textup{R} = e^{2 \pi i } \xi^\textup{L}$.
}

\begin{figure}[ht]
\centering
\begin{subfigure}{0.45\textwidth}
\centering
\includegraphics[width=\textwidth]{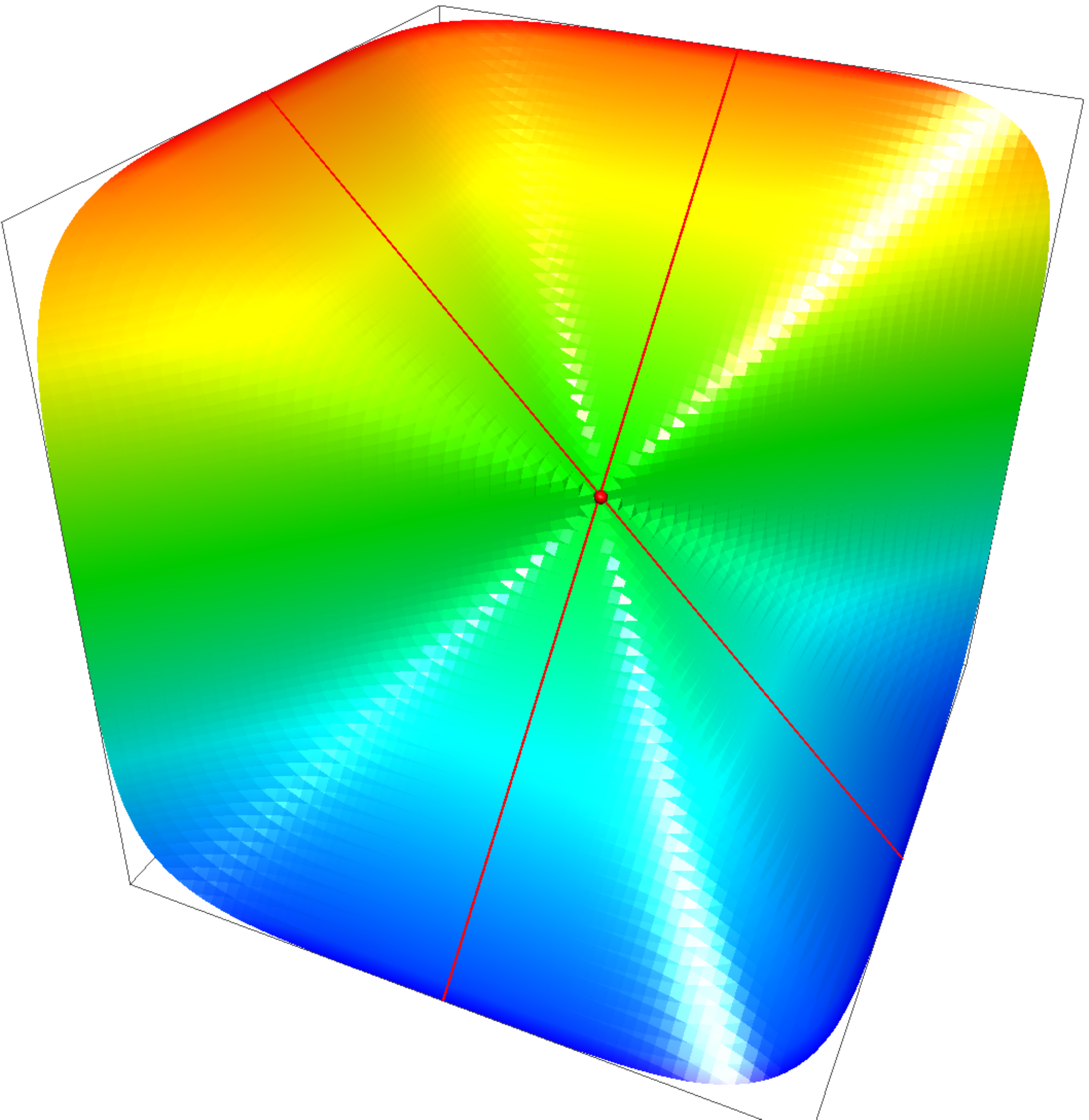}
\caption{The real slice of the Fermat quintic surface in $\CC^3$ given by the equation $x^5 + y^5 + z^5 = 0$.}
\label{250312110535}

\begin{tikzpicture}
\draw [black, ultra thick] (0,0) circle (1.5);
\fill [red] (15:1.5) circle (2pt) node [right] {$\alpha_+$};
\fill [red] (195:1.5) circle (2pt) node [left] {$\alpha_-$};
\draw [red, ->] (0,0) -- (15:1.4);
\draw [red, ->] (0,0) -- (195:1.4);
\fill [black] (0,0) circle (1pt);
\node at (105:1.2) {$A_1$};
\node at (-75:1.2) {$A_2$};
\end{tikzpicture}
\caption{The resurgent Stokes diagram at any nonzero point $t$.
There are only Stokes rays $\alpha_+, \alpha_-$.}
\label{250313163646}
\end{subfigure}
\begin{subfigure}{0.45\textwidth}
\begin{tikzpicture}[scale=0.8, every node/.style={scale=0.8}]
\begin{scope}[yshift=-6cm]
\draw [fill = grey, dashed, fill opacity=0.8] (0,1) -- (7,1) -- (5,-1) -- (-2,-1) -- cycle;
\draw [red, decorate, decoration={zigzag, segment length=4, amplitude=1}] (1,0) -- (-1,0);
\draw [red, decorate, decoration={zigzag, segment length=4, amplitude=1}] (4,0) -- (6,0);
\node [cross, red] at (1,0) {};
\node [cross, red] at (4,0) {};
\node at (6.5,-0.5) {$M_\tau$};
\end{scope}
\begin{scope}[yshift=-4.5cm]
\draw [fill = grey, dashed, fill opacity=0.8] (0,1) -- (7,1) -- (5,-1) -- (-2,-1) -- cycle;
\draw [red, decorate, decoration={zigzag, segment length=4, amplitude=1}] (1,0) -- (-1,0);
\draw [red, decorate, decoration={zigzag, segment length=4, amplitude=1}] (4,0) -- (6,0);
\node [cross, red] at (1,0) {};
\node [cross, red] at (4,0) {};
\end{scope}
\begin{scope}[yshift=-3cm]
\draw [fill = grey, dashed, fill opacity=0.8] (0,1) -- (7,1) -- (5,-1) -- (-2,-1) -- cycle;
\draw [red, decorate, decoration={zigzag, segment length=4, amplitude=1}] (1,0) -- (-1,0);
\draw [red, decorate, decoration={zigzag, segment length=4, amplitude=1}] (4,0) -- (6,0);
\node [cross, red] at (1,0) {};
\node [cross, red] at (4,0) {};
\end{scope}
\begin{scope}[yshift=-1.5cm]
\draw [fill = grey, dashed, fill opacity=0.8] (0,1) -- (7,1) -- (5,-1) -- (-2,-1) -- cycle;
\draw [red, decorate, decoration={zigzag, segment length=4, amplitude=1}] (1,0) -- (-1,0);
\draw [red, decorate, decoration={zigzag, segment length=4, amplitude=1}] (4,0) -- (6,0);
\node [cross, red] at (1,0) {};
\node [cross, red] at (4,0) {};
\end{scope}
\begin{scope}
\draw [fill = grey, dashed, fill opacity=0.8] (0,1) -- (7,1) -- (5,-1) -- (-2,-1) -- cycle;
\draw [red, decorate, decoration={zigzag, segment length=4, amplitude=1}] (1,0) -- (-1,0);
\draw [red, decorate, decoration={zigzag, segment length=4, amplitude=1}] (4,0) -- (6,0);
\node [cross, red] at (1,0) {};
\node [cross, red] at (4,0) {};
\end{scope}
\begin{scope}[yshift=-9.5cm]
\draw [fill = grey, dashed, fill opacity=0.8] (0,1) -- (7,1) -- (5,-1) -- (-2,-1) -- cycle;
\node [cross, red] at (1,0) {};
\node [cross, red] at (4,0) {};
\node [red] at (1,0) [below] {$\xi_-$};
\node [red] at (4,0) [below] {$\xi_+$};
\fill [black] (2.5,0) circle (1pt) node [below] {$0$};
\node at (6.5,-0.5) {$\CC_\xi$};
\draw [->] (2.5,2.25) -- (2.5,0.5);
\node at (2.5,1.5) [right] {$\Z_\tau$};
\end{scope}
\end{tikzpicture}
\caption{The Borel surface $M_{\tau}$ for any nonzero $\tau$.
The central charge $\Z_\tau : M_\tau \to \CC$ is a fivefold covering map with ten ramification points $S_\tau \subset M_\tau$ and two branch points $\xi_+, \xi_-$.}
\label{250304114549}
\end{subfigure}
\caption{}
\label{250313142325}
\end{figure}

\paragraph{Context.}
The literature on the subject of the Painlevé I equation is vast, appearing in many areas of mathematics such as analysis of PDEs, integrable systems, isomonodromic deformations, and enumerative geometry, but also in mathematical physics including string theory and random matrices.
There have been several attempts to describe the geometry of resurgence for the Painlev\'{e} I equation.
For example, the link with isomonodromic systems provides a natural framework to study the connection formulas and the direct monodromy problem \cite{Kapaev1993ConnectionFF, MR1404282, MR1746249, Iwaki-P1, Lisovyy_2017}.
More direct approaches focusing on the non-autonomous Hamiltonian formulation of the Painlev\'{e} I equation and the possible asymptotics at large $t$ of the solutions lead to the classification of tritronqu\'{e}es ($0$-parameter), tronqu\'{e}es ($1$-parameter), or general ($2$-parameter) solutions of the Painlev\'{e} I equation \cite{MR1854431,Garoufalidis:2010ya,MR3346717,delabaerebook} with several conjectures regarding instanton solutions and wall-crossing formulas \cite{Baldino:2022aqm,Aniceto:2011nu}.
The main difficulty in these approaches is to find a way to upgrade formal solutions given in terms of formal power series or formal transseries to analytic objects via resummation in a rigorous way.
Almost without exception, this requires a detailed description of the underlying geometry that controls the resummation method which has only been initiated in a few papers \cite{CostinBook,Sauzin,KAMIMOTO2022108533,Kontsevich-Soibelman,Fantini-Rellapaper,FantiniFenyes}.

More recently, a fully geometric approach to resurgence was developed by one of the authors in \cite{240622121512}, which, to the best of our knowledge, is the first instance that suggests the idea of engaging the full power of holomorphic Lie groupoids such as the fundamental groupoid of a Riemann surface.
The present paper may be regarded as a demonstration of the power of the tools and techniques proposed in \cite{240622121512}.
Having said that, in order to prove \autoref{250304131355}, it was necessary to overcome several new challenges that did not previously arise.
Our methods are by no means limited to the Painlevé I equation: we chose it as the focus of this paper because the Painlevé I equation is remarkable in that it provides a rare instance in mathematics where the analysis of the problem is rather nontrivial and interesting yet completely explicit.

\paragraph{Relation to the usual Painlevé I equation.}
Setting $\hbar = 1$ in the deformed Painlevé I equation yields the usual Painlevé I equation.
Alternatively, for nonzero $\hbar$, the deformed Painlevé I equation can be related to the usual Painlevé I equation by rescaling the variables as follows.
Choose a branch cut in the $\hbar$-plane emanating from the origin, and introduce new variables $x$ and $y(x)$ by
\begin{eqntag}
\label{250205090153}
	t = \hbar^{4/5} x
\qtext{and}
	q (t, \hbar) = \hbar^{2/5} y (x)
\fullstop
\end{eqntag}
Then $q$ is a solution of the deformed Painlevé I equation \eqref{250204124541} if and only if $y$ is a solution of the usual Painlevé I equation in the $x$-plane:
\begin{eqntag}
\label{250205090931}
	y'' = 6 y^2 + x
\fullstop
\end{eqntag}

The rescaling \eqref{250205090153} also relates the formal solution \eqref{250205071510} to the asymptotic solution $\hat{y} (x)$ as $x \to \infty$ which has a similar square-root ambiguity.
There is a special collection of maximally regular solutions of the usual Painlevé I equation which are asymptotic to $\hat{y} (x)$ as $x \to \infty$ in a sector of opening angle $4\pi/5$.
They are often called \textit{tritronquées solutions} \cite{MR1854431} or \textit{$0$-parameter solutions}.
For this reason, we refer to solutions $q (t,\hbar)$ of the deformed Painlevé I equation obtained by the Borel resummation \eqref{250312063806} as the \dfn{(deformed) tritronquées solutions} as well.
Whether or not the present method can be extended to tronquées ($1$-parameter) or general ($2$-parameters) solutions of the Painlevé I equation, whose natural starting point are rather formal $\hbar$-transseries remains an open question.

\paragraph{Organisation.}
The article is organised as follows.
In \autoref{250315160940}, we construct the formal solution $\hat{q}$ and prove that its coefficients exhibit at most factorial growth.
In \autoref{250313191046}, we state our main results about the resurgence of $\hat{q}$, its Borel summability properties, and the associated Stokes phenomenon.
The rest of the article is devoted to the proof of these assertions.
First, in \autoref{250218190702}, we shift the perspective to work with the associated first-order Hamiltonian system.
We make some convenient changes of variables that help us write down an Initial Value Problem for the Borel transform.
In \autoref{250315131538}, motivated by the goal of constructing a global solution, we introduce the basic geometric tools needed to give a global reformulation of this Initial Value Problem using the fundamental groupoid.
Then \autoref{250304152932} is the geometric heart of the article where we construct and describe the geometric spaces where our Initial Value Problem can be solved globally.
Finally, in \autoref{250304161635}, we construct the global solution using a Contraction Mapping Principle.

\paragraph{Acknowledgements.}
\sloppy{Olivier Marchal was supported by the fundamental junior IUF grant G752IUFMAR.
Nikita Nikolaev was supported by the European Union's Horizon 2020 Research and Innovation Programme under the Marie Skłodowska-Curie Grant Agreement No.\,101026083 (\href{https://cordis.europa.eu/project/id/101026083}{\texttt{AbQuantumSpec}}), as well as the Leverhulme Trust through the Leverhulme Research Project grant \textit{Extended Riemann-Hilbert Correspondence, Quantum Curves and Mirror Symmetry}.}

\counterwithin{equation}{section}
\section{The Formal Power Series Solution}
\label{250315160940}

We begin by searching for formal power series solutions of the deformed Painlevé I equation; i.e., formal power series in $\hbar$ with locally holomorphic (i.e., holomorphic but possibly multivalued) coefficients that satisfy \eqref{250204124541} in the formal sense of differentiating order-by-order in $\hbar$.
Thus, we introduce the following formal power series ansatz for $q$ as well as the corresponding momentum $p = \hbar \del_t q$:
\begin{eqntag}
\label{250205091941}
	\hat{q} (t,\hbar) \coleq \sum_{n=0}^\infty q_n (t) \hbar^n
\qqtext{and}
	\hat{p} (t,\hbar) \coleq \sum_{n=0}^\infty p_n (t) \hbar^n
\fullstop
\end{eqntag}
In this section, by way of derivation that culminates in \autoref{250204162846}, we show that such solutions exist and in fact there is only one such solution up to a square-root ambiguity.

\subsection{Existence and Uniqueness}

Let us insert the power series ansatz \eqref{250205091941} into the deformed Painlevé I equation and compare the coefficients of like powers of $\hbar$.
We find that the leading-order coefficient $q_0$ must be a root of the quadratic polynomial equation
\begin{eqntag}
\label{250205092551}
	6q_0^2 = t
\fullstop
\end{eqntag}
Furthermore, once a root $q_0$ is chosen, then all the higher-order coefficients of $\hat{q}$ can be recursively determined from $q_0$.
Meanwhile, the coefficients of $\hat{p}$ are determined from those of $\hat{q}$ thanks to the relation $\hat{p} = \hbar  \del_t \hat{q}$.
In this manner, we find the following:
\begin{eqntag}
\label{250205092716}
\begin{aligned}
	q_{2n-1} = 0
&\qqtext{and}
	q_{2n} = \frac{1}{12 q_0} \left( \ddot{q}_{2n-2} - 6 \sum_{i + j = n}^{i,j \neq 0} q_{2i} q_{2j} \right)
\qquad
	\forall n \geq 1
\fullstop{,}
\\
	p_{2n} = 0
&\qqtext{and}
	p_{2n+1} = \dot{q}_{2n}
\qquad
	\forall n \geq 0
\fullstop
\end{aligned}
\end{eqntag}
Here and everywhere, the empty sum is understood to mean $0$.
More explicitly, the first few coefficients of $\hat{q}$ are
\begin{eqn}
	q_2 = \frac{\ddot{q}_0}{12 q_0}
~,\quad
	q_4 = \frac{\ddot{q}_2 - 6 q_0^2}{12 q_0}
~,\quad
	q_6 = \frac{\ddot{q}_4 - 12 q_2 q_4}{12 q_0}
~,\quad
	q_8 = \frac{\ddot{q}_4 - 12 q_2 q_6 - 6 q_4^2}{12 q_0}
~,\quad
\ldots
\end{eqn}
Furthermore, using the relation \eqref{250205092551}, we can derive an explicit expression for all the derivatives of $q_0$ in terms of $q_0$ itself:
\begin{eqn}
	\dot{q}_0 = - \frac{1}{12q_0}
~,\qquad
	\ddot{q}_0 = - \frac{1}{144 q_0^3}
~,\qquad
	\dddot{q}_0 = - \frac{3}{1728 q_0^5}
~,\qquad
	\ddddot{q}_0 = - \frac{15}{20736 q_0^7}
~,\qquad
\ldots
\fullstop{,}
\end{eqn}
and more generally (with empty product understood to mean $1$)
\begin{eqntag}
\label{250205093815}
    q_0^\tinybrac{k} 
    	= -\frac{(2k-3)!}{12^k 2^{k-2}(k-2)! q_0^{2k-1}}
    	= -\frac{1}{12^k q_0^{2k-1}} \prod_{i=1}^{k-1} (2i-1)
\qquad
	\forall k \geq 1
\fullstop
\end{eqntag}
As a consequence, we can deduce by induction the explicit dependence on $t$ of all the coefficients of $\hat{q}$.
Namely, let us introduce the constant
\begin{eqntag}
\label{250206193643}
	\kappa \coleq i / \sqrt{6}
\fullstop{,}
\end{eqntag}
as well as the following sequences of complex numbers $a^\pm_{2n}$ and $b^\pm_{2n+1}$ for $n \geq 0$:
\begin{eqntag}
\label{250205150639}
	a_{2n}^\pm \coleq - (\pm 1)^{n-1} \frac{c_{n}}{3^n 2^{5n-1} \kappa^{n-1}}
\qqtext{and}
	b_{2n+1}^\pm \coleq - \tfrac{1}{2} (5n-6) a_{2n}^\pm
\fullstop{,}
\end{eqntag}
where $c_0 \coleq -1/2$ and $(c_{n})_{n \geq 1}$ is a positive integer sequence defined by the following recursion:
\begin{eqntag}
\label{250208130316}
	c_{n} \coleq 2 (5n-6) (5n-4) c_{n-1} + \sum_{i+j = n}^{i,j\neq 0} c_{i} c_{j}
\qqquad
	\forall n \geq 1
\fullstop
\end{eqntag}
Explicitly, the first few terms (for $n = 0,1,\ldots,5$) of the sequences \eqref{250205150639} are:
\begin{eqntag}
\begin{aligned}
	a^\pm_{2n} ~&:~ \pm \tfrac{i}{\sqrt{6}},\quad
		- \tfrac{1}{48},\quad
		\pm \tfrac{49 i \sqrt{6}}{4608},\quad
		+\tfrac{1225}{9216},\quad
		\mp \tfrac{4\,412\,401 i \sqrt{6}}{7\,077\,888},\quad
		- \tfrac{73\,560\,025}{2\,359\,296},\quad
		\ldots
\fullstop{;}
\\
	b^\pm_{2n+1} ~&:~ \pm \tfrac{3i}{\sqrt{6}},\quad
		- \tfrac{1}{96},\quad
		\mp \tfrac{49 i \sqrt{6}}{2304},\quad
		- \tfrac{1225}{2048},\quad
		\pm \tfrac{4\,412\,401 i \sqrt{6}}{1\,179\,648},\quad
		- \tfrac{1\,397\,640\,475}{4\,718\,592},\quad
		\ldots
\fullstop
\end{aligned}
\end{eqntag}

We summarise the content of the above discussion in the following proposition.

\begin{proposition}[Formal Existence and Uniqueness Theorem]
\label{250204162846}
The deformed Painlevé I equation \eqref{250204124541} has a unique formal power series solution $\hat{q} (t,\hbar)$ of the form \eqref{250205091941} whose coefficients are holomorphic but multivalued functions on the punctured $t$-plane $\CC \smallsetminus \set{0}$.
Namely, for any branch of the square root $\sqrt{t}$, the two branches of $\hat{q} (t,\hbar)$, as well as the two branches of the corresponding momentum $\hat{p} (t,\hbar) = \hbar \del_t \hat{q} (t,\hbar)$, can be expressed as follows:
\begin{eqntag}
\label{231217212302}
    \hat{q}_\pm (t, \hbar)
        = \sqrt{t} \sum_{n=0}^{\infty} a_{2n}^\pm t^{-5n/2} \hbar^{2n}
\qtext{and}
	\hat{p}_\pm (t, \hbar)
        = \frac{1}{\sqrt{t}} \sum_{n=0}^{\infty} b_{2n+1}^\pm t^{-5n/2} \hbar^{2n + 1}
\fullstop{,}
\end{eqntag}
where the coefficients $a_{2n}^\pm, b_{2n+1}^\pm \in \CC$ are defined by the recurrence relations \eqref{250205150639}.
\end{proposition}

\begin{definition}[formal solutions]
\label{250205091635}
We will refer to the formal power series solution of the deformed Painlevé I equation defined by \autoref{250204162846} as the \dfn{formal power series solution} or \dfn{formal $0$-parameter solution}, or simply \dfn{formal solution} from now on.
\end{definition}

The multivalued nature of the formal solution $\hat{q}$ is the choice of square-root branch that can be easily eliminated by passing to the double cover
\begin{eqntag}
\label{250313115427}
	\CC_\tau \to \CC_t
\qtext{given by the relation}
	\tau^2 = \kappa^2 t = -t/6
\fullstop
\end{eqntag}
Thus, the multivalued formal power series $\hat{q}$ and $\hat{p}$ can be regarded as formal power series with holomorphic coefficients on the punctured $\tau$-plane, where in fact all coefficients extend to the whole $\tau$-plane as meromorphic functions with a pole at the origin:
\begin{eqntag}
\label{250206201929}
	\hat{q} (\tau, \hbar) \coleq \sum_{n=0}^\infty \tilde{a}_{2n}^\pm \tau^{-5n+1} \hbar^{2n}
\qtext{and}
	\hat{p} (\tau, \hbar) \coleq \sum_{n=0}^\infty \tilde{b}_{2n+1}^\pm \tau^{-5n-1} \hbar^{2n+1}
\end{eqntag}
where $\tilde{a}_{2n}^\pm \coleq a_{2n}^\pm \kappa^{5n-1}$ and $\tilde{b}_{2n+1}^\pm \coleq b_{2n+1}^\pm \kappa^{5n+1}$.
In symbols, $\hat{q}, \hat{p} \in \cal{M} (\CC_\tau) \bbrac{\hbar}$, where $\cal{M} (\CC_\tau)$ is the space of meromorphic functions on $\CC_\tau$.
Notice that the order of the pole is unbounded as the order $n$ of the coefficients increases, so it is not possible to factorise out a common denominator (a telltale sign of breakdown of uniform asymptotics in the parameter $\hbar$).

\begin{remark}
Under the rescaling change of variables \eqref{250205090153}, the formal power series solution $\hat{q} (t, \hbar)$ is directly related to the unique multivalued asymptotic solution of the usual Painlevé I equation \eqref{250205090931} at large $x$.
Indeed, for any branches of $\sqrt{t}$ and $\sqrt[5]{\hbar}$, let $\sqrt{x}$ be the square-root branch such that $\sqrt{t} = \hbar^{2/5} \sqrt{x}$.
Then, upon using the change of variables $x = \hbar^{-2/5} t$, the formal Puiseux series
\begin{eqntag}
	\hat{y}_\pm (x) 
		\coleq \hbar^{-2/5} \hat{q}_\pm (t,\hbar)
		= \sqrt{x} \sum_{n = 0}^\infty a^\pm_{2n} x^{-5n/2}
		\in \CC \ppbracbig{x^{1/2}}
\end{eqntag}
is a formal solution of the usual Painlevé I equation \eqref{250205090931}.
\end{remark}

\subsection{Factorial Type}
\label{250313152915}

In this subsection, we demonstrate that the formal solution $\hat{q}$ is divergent for all $t$ by showing that its coefficients exhibit factorial growth.
First, let us recall some relevant notions.

\begin{definition}
\label{250312144829}
A formal power series
\begin{eqn}
	\hat{f} (\hbar) = \sum_{n=0}^\infty a_n \hbar^n \in \CC \bbrac{\hbar}
\end{eqn}
is said to be of \dfn{factorial type} (or is a \textit{Gevrey-1 series}) if its coefficients $a_n$ grow at most factorially with $n$; i.e., there are constants $\C,\M > 0$ such that 
\begin{eqn}
	|a_n| \leq \C \M^n n!
\qqquad
	\forall n \geq 0
\fullstop
\end{eqn}
Equivalently, we could demand the existence of a constant $\M > 0$ such that $|a_n| \leq \M^{n+1} n!$ for all $n \geq 0$.
Note that if $|a_0| \leq 1$ then we can always arrange $\C$ to be $1$.
\end{definition}

\begin{proposition}[factorial type]
\label{250224202049}
The formal power series solution $\hat{q} (t, \hbar)$ of the deformed Painlevé I equation \eqref{250204124541} is of factorial type, locally uniformly for all nonzero $t$.
More precisely, let $\hat{q} (t, \hbar)$ be either branch of the multivalued formal solution of the deformed Painlevé I equation \eqref{250204124541}.
Then every nonzero point $t_0 \in \CC_t$ has a neighbourhood $U_0$ such that there exist real constants $\C, \M > 0$ that provide the following uniform bound on $U_0$ for every $n$:
\begin{eqntag}
\label{250224202445}
	\big| q_n (t) \big| \leq \C \M^n n!
\qquad
	\forall t \in U_0
\fullstop
\end{eqntag}
\end{proposition}

In order to prove this Proposition, we need to make a closer examination of the coefficients of $\hat{q}$ using \eqref{250205150639}.
In particular, the terms of the integer sequence $(c_n)_{n \geq 1}$ grow rather fast: the first few terms are
\begin{eqn}
	1
,~
	49							
,~
	9\,800						
,~
	4\,412\,401					
,~
	3\,530\,881\,200			
,~
	4\,414\,129\,955\,298		
,~
	7\,945\,866\,428\,953\,600
,~
	\ldots
\fullstop
\end{eqn}
We can give a more accurate appraisal of its growth as $n \to \infty$.

\begin{lemma}
\label{250208122514}
The integer sequence $(c_n)_{n \geq 1}$ grows like $(2n-1)!$; more precisely, there is a constant $\M > 0$ such that $c_n \leq \M^n (2n-1)!$ for all $n \geq 1$.
\end{lemma}

\begin{proof}
We demonstrate this in two steps.
First, we will inductively construct a sequence $(\M_n)_{n \geq 1}$ of positive real numbers such that $c_n \leq \M_n (2n-1)!$.
We will then show that there exists a constant $\M > 0$ such that $\M_n \leq \C^n$ for all $n$.

We can take $\M_1 \coleq 1$ because $c_1 = 1$.
Now, for any $n \geq 2$, let us assume that $\M_1, \ldots, \M_{n-1}$ have been constructed such that $c_i \leq \M_i (2i-1)!$ for all $i = 1, \ldots, n-1$.
Then using \eqref{250208130316}, we can estimate $c_n$ as follows:
\begin{eqns}
	c_n 
	&\leq 2 (5n-6)(5n-4) \M_{n-1} (2n-3)! + \sum_{i+j = n}^{i,j \neq 0} \M_i \M_j (2i-1)! (2j-1)!
\\	&\leq 13 (2n-2) (2n-1) \M_{n-1} (2n-3)! + \sum_{i+j = n}^{i,j \neq 0} \M_i \M_j (2n-2)!
\\	&\leq \left( 13 \M_{n-1} + \sum_{i+j = n}^{i,j \neq 0} \M_i \M_j \right) (2n-1)!
\fullstop
\end{eqns}
Here, in the second line, we used the general inequality $i! j! \leq (i+j+1)!$ and also simplified the bound by using $25/2 \leq 13$.
Thus, we define
\begin{eqntag}
\label{250208130847}
	\M_n \coleq 13 \M_{n-1} + \sum_{i+j = n}^{i,j \neq 0} \M_i \M_j
\fullstop
\end{eqntag}
In fact, this formula is valid for all $n \geq 1$ if we set $\M_0 \coleq 1/13$.
The first few terms of this sequence are
\begin{eqn}
	1,
\quad	14,
\quad	210,
\quad	3{,}346,
\quad	56{,}070,
\quad	978{,}838,
\quad	17{,}657{,}850,
\quad	327{,}020{,}330,
\quad\ldots
\fullstop
\end{eqn}

Now we show that there exists a constant $\M > 0$ such that $\M_n \leq \M^n$ for all $n$.
To this end, consider the following formal power series in an abstract variable $z$:
\begin{eqn}
	\hat{f} (z) \coleq \sum_{n=1}^\infty \M_n z^n \in \CC \bbrac{z}
\fullstop
\end{eqn}
Note that $\hat{f} (0) = 0$.
We will show that $\hat{f} (z)$ also has a non-zero radius of convergence.
The key observation is that $\hat{f} (z)$ satisfies the following algebraic equation:
\begin{eqn}
	\hat{f} (z) = z + 13 z \hat{f}(z) + \hat{f} (z)^2
\fullstop
\end{eqn}
This can be easily verified by expanding both sides in powers of $z$ and using the defining identity \eqref{250208130847} for the coefficients $\M_n$.
Now, examine the following function of two complex variables $(z,u)$:
\begin{eqn}
	\F (z,u) \coleq -u + z + 13 z u + u^2
\fullstop
\end{eqn}
It has the following properties:
\begin{eqn}
	\F (0,0) = 0
\qtext{and}
	\frac{\del \F}{\del u} \bigg|_{(z,u) = (0,0)} = -1 \neq 0
\fullstop
\end{eqn}
By the Holomorphic Implicit Function Theorem, there exists a unique holomorphic function $f(z)$ near $z = 0$ such that $f(0) = 0$ and $\F \big( z, f(z) \big) = 0$.
It follows that $\hat{f} (z)$ must be its necessarily convergent Taylor series expansion at $z = 0$, and so its coefficients $(\M_n)_{n \geq 1}$ must grow at most exponentially; i.e., there is a constant $\M > 0$ such that $\M_n \leq \M^n$ for all $n \geq 1$.
\end{proof}

\begin{proof}[Proof of \autoref{250224202049}.]
Recall that all the odd-order coefficients of $\hat{q}$ are zero, so we just need to prove that every $t_0$ has a neighbourhood $U_0$ such that there are constants $\C, \M > 0$ providing the bounds $|q_{2n} (t)| \leq \C \M^{2n} (2n)!$ for all $n \geq 1$ and all $t \in U_0$.
The proof exploits the recursive formula \eqref{250205150639} for the coefficients of $\hat{q}$, and in particular the growth estimate on the terms of the sequence $(c_n)$ established in \autoref{250208122514}.

Fix any nonzero $t_0$, and two radii $r,\R > 0$ such that $r < |t_0| < \R$.
Then, for all $t$ in the annulus $U_0 \coleq \set{ r < |t| < \R }$, we have $|q_0 (t)| = |\kappa| |t|^{1/2} \leq \sqrt{\R / 6}$.
Moreover, by \autoref{250208122514}, there is a constant $\tilde{\M} > 0$ such that for all $n \geq 1$,
\begin{eqn}
	\big| q_{2n} (t) \big|
		= |t|^{1/2} |a^\pm_{2n}| |t|^{-5n/2}
		\leq \frac{\sqrt{6^{n-1} \R} }{3^n 2^{5n-1} \sqrt{r^{5n}}} \tilde{\M}^n (2n-1)!
		\leq \C \M^{2n} (2n)!
\end{eqn}
for some $\C,\M > 0$.
\end{proof}

\section{Resurgence of the Deformed Painlevé I Equation}
\label{250313191046}

This section contains the main results of this paper.
We describe the full resurgent structure of the formal power series solution $\hat{q} (t, \hbar)$ of the deformed Painlevé I equation constructed in \autoref{250204162846}.
Our Main Theorem is stated in \autoref{250313153618}.
Then in \autoref{250304162404} we state a range of consequences of resurgence about the Borel summability of $\hat{q}$.
There we also deduce an Existence and Uniqueness Theorem of special solutions of the deformed Painlevé I equation which we term the \dfn{deformed tritronquée solutions} (\autoref{250304154226}).
Finally, in \autoref{250313153813} we describe the Stokes phenomenon associated with the Borel resummation of $\hat{q}$.

\subsection{The Resurgent Structure}
\label{250313153618}

The main result of this paper can be formulated concisely as follows.

\begin{theorem}
\label{250304131355}
The formal solution $\hat{q}$ of the deformed Painlevé I equation is resurgent.
\end{theorem}

This Theorem packs a lot of information, so we break it down into three parts, presented as separate propositions that clarify and elaborate on the defining aspects of resurgence of the formal solution $\hat{q}$.
First, the convergence of the Borel transform of $\hat{q}$ is specified in \autoref{250312151012} in \autoref{250304162005}.
Second, the property of endless analytic continuation of the Borel transform is described in \autoref{250210171316} in \autoref{250304151341}.
Third, the property of exponential type of the Borel transform is expressed in \autoref{250313152819} in \autoref{250313152814}.
Thus, to prove \autoref{250304131355}, we need to prove Propositions \ref{250312151012}, \ref{250210171316}, and \ref{250313152819}, which will be done in \autoref{250304161635}.

\subsubsection{Convergence of the Borel transform.}
\label{250304162005}
First, we describe the convergence of the Borel transform $\hat{\omega}$.
Recall that the \dfn{Borel transform} of a formal power series in $\hbar$ is defined as the power series in an auxiliary variable $\xi$ (sometimes called the \dfn{Borel variable}) defined as follows:
\begin{eqntag}
\label{250208115459}
	\hat{\Borel} : \hat{f} (\hbar) = \sum_{n=0}^\infty a_n \hbar^n
	~~\mapstoo~~
	\hat{\Borel} \big[ \, \hat{f} \, \big] (\xi)
		\coleq \sum_{n=0}^\infty \tfrac{1}{n!} a_{n+1} \xi^n
\fullstop
\end{eqntag}
It is clear that a power series is of \hyperref[250312144829]{factorial type} if and only if its Borel transform is convergent.
In \autoref{250313152915}, we showed that the formal power series solution $\hat{q} (t, \hbar)$ of the deformed Painlevé I equation is of factorial type.
Therefore, we have the following result that follows immediately from \autoref{250224202049}.

\begin{proposition}[Convergence of the Borel transform]
\label{250312151012}
The Borel transform
\begin{eqntag}
\label{250207201408}
	\hat{\omega} (t,\xi) 
		\coleq \Borel \big[ \, \hat{q} \, \big] (t, \xi)
		= \sum_{n = 0}^\infty \frac{q_{n+1} (t)}{(n+1)!} \xi^n
\end{eqntag}
of the formal solution $\hat{q} (t, \hbar)$ of the deformed Painlevé I equation is a convergent power series in $\xi$, locally uniformly for all nonzero $t$.
\end{proposition}

More concretely, this Proposition asserts that every nonzero point $t_0 \in \CC_t$ has a neighbourhood $U_0 \subset \CC_t$ such that any branch of $\hat{\omega} (t,\xi)$ in $U_0$ converges uniformly for all $t \in U_0$.
In symbols, this branch defines an element $\hat{\omega} \in \cal{O} (U_0) \cbrac{\xi}$.
In more practical terms, this means that there is a neighbourhood $\Xi_0 \subset \CC_\xi$ of the origin such that the chosen branch $\hat{\omega} (t, \xi)$ is a holomorphic function on $U_0 \times \Xi_0$.
More globally, the multivalued nature of $\hat{\omega} (t, \xi)$ can be resolved by passing to the square-root cover $\CC_\tau \to \CC_t$ given by the relation $\tau^2 = -t/6$, as we did in \eqref{250313115427}.
Thus, we can treat the Borel transform \eqref{250207201408} as a univalued function of $\tau$.
Abusing notation, we may write $\hat{\omega} (\tau, \xi) \coleq \hat{\omega} (t,\xi)$ where $\tau^2 = -t/6$, and so we obtain the following consequence of \autoref{250312151012}.

\begin{corollary}
\label{250313115511}
There is an open neighbourhood $W \subset \CC_\tau^\ast \times \CC_\xi$ of the subset $\set{\xi = 0}$ such that the Borel transform $\hat{\omega} (\tau,\hbar)$ of the formal solution $\hat{q} (\tau, \hbar)$ of the deformed Painlevé I equation defines a holomorphic function on $W$; i.e., $\hat{\omega} \in \cal{O} (W)$.
\end{corollary}

\subsubsection{Endless analytic continuation.}
\label{250304151341}
Second, we describe the property of endless analytic continuation of the Borel transform $\hat{\omega}$.
The next proposition presents the full resurgent-geometric structure.
In essence, it says that for every nonzero $t_0$, the convergent $\xi$-power series $\hat{\omega} (t_0, \xi)$ admits maximal analytic continuation to the universal cover of a Riemann surface punctured in ten points which is a fivefold covering of the Borel $\xi$-plane.
Moreover, these analytic continuations for different $t_0$ patch together to a global holomorphic function $\omega$ on a two-dimensional holomorphic manifold.

\begin{proposition}[endless analytic continuation]
\label{250210171316}
The Borel transform $\hat{\omega}$ admits endless analytic continuation $\omega$ to a two-dimensional holomorphic manifold.
More specifically:
\begin{enumerate}
\item There is a complex algebraic surface $M$ (called the \dfn{Borel space}) and an algebraic curve $S \subset M$ (called the locus of \dfn{Borel singularities}) such that the Borel transform $\hat{\omega}$ naturally extends to a global multivalued holomorphic function $\omega$ on the complement $M^\ast \coleq M \smallsetminus S$ contained in the smooth locus of $M$.
Namely, there is a two-dimensional holomorphic manifold $\tilde{M}$ (called the \dfn{Borel covering space}) with a holomorphic surjective submersion $\nu : \tilde{M} \to M^\ast$ such that $\omega$ is a global holomorphic function on $\tilde{M}$.
\item The surface $M$ is isomorphic to the quotient of the Fermat quintic surface in $\CC^3$ by the involution automorphism $\sigma : (u,v,w) \mapsto - (v,u,w)$:
\begin{eqntag}
\label{250313154124}
	M \cong \set{ (u,v,w) \in \CC^3 ~\big|~ u^5 + v^5 + w^5 = 0 } / \sigma
\fullstop
\end{eqntag}
This surface $M$ has only one singular point which corresponds to the origin in $\CC^3$, and $S$ is isomorphic to the quotient of the union of intersections of the Fermat quintic with the planes $u = 0$ and $v = 0$.
See \autoref{250312110535}.
\newpage
\item There is an algebraic surjective submersion $\rm{s} : M \to \CC_\tau$ (called the \dfn{source map}), such that every nonzero fibre $M_{\tau \neq 0} \coleq \rm{s}^\inv (\tau)$ (called the \dfn{Borel surface} at $\tau$) is a smooth algebraic curve isomorphic to the quotient of a plane quintic curve by the involution $\sigma$.
The intersection $S_\tau \coleq S \cap M_\tau$ consists of exactly ten points, and the submersion $\nu : \tilde{M} \to M^\ast$ restricts to the fibre $\tilde{M}_\tau \coleq (\rm{s} \circ \nu)^\inv (\tau)$ to define the universal covering map
\begin{eqntag}
\label{250313154127}
	\nu : \tilde{M}_\tau = \widetilde{M_\tau \smallsetminus S_\tau} \to M_\tau \smallsetminus S_\tau
\fullstop
\end{eqntag}
\item There is an algebraic surjective map $\Z : M \to \CC_\xi$ (called the \dfn{central charge}) to the Borel plane, which is a submersion away from $S$, and whose restriction $\Z_\tau : M_\tau \to \CC_\xi$ to any nonzero Borel surface $M_{\tau \neq 0}$ is a fivefold ramified covering map with ramification locus $S_\tau$ consisting of all ten Borel singularities in $M_\tau$, each with ramification order $5$, and distributed equally over two branch points $\xi_\pm = \xi_\pm (\tau) \in \CC_\xi$ (called the \dfn{Borel singular values}) satisfying $\xi_+ + \xi_- = 0$.
See \autoref{250304114549}.
\item For each nonzero $\tau$, the pair $(M_\tau, \Z_\tau)$ is an endless Riemann surface of algebraic type in the sense of \cite[Definition 1.6 and 1.7]{240622121512}, and $(\tilde{M}_\tau, \Z_\tau)$ is an endless Riemann surface of log-algebraic type.
\end{enumerate}
\end{proposition}

We will construct these geometric structures in \autoref{250304152932} and prove \autoref{250210171316} in \autoref{250304161635}.
Specifically, it will follow as a consequence of \autoref{250314125007}.
We may sometimes refer to the holomorphic function $\omega \in \cal{O} (\tilde{M})$ as the \dfn{global Borel transform} if we find it necessary to distinguish it from the \textit{a priori} only locally defined holomorphic function $\hat{\omega}$.
The property of endless analytic continuation of $\hat{\omega}$ may be stated in the following more elementary terms if we sacrifice the much more refined geometric information.

\begin{corollary}
\label{250304131641}
Fix any $t_0 \in \CC^\ast_t$, select a square-root branch $\sqrt{t}$ near $t_0$, and let $\hat{q} (t,\hbar)$ be the corresponding branch of the formal power series solution of the deformed Painlevé I equation.
Put 
\begin{eqntag}
\label{250314214209}
	\xi_\pm \coleq \pm \tfrac{1}{30} e^{\pi i/4} 24^{5/4} t_0^{5/4}
\fullstop
\end{eqntag}
Then the Borel transform $\hat{\omega} (t_0, \xi) \in \CC \cbrac{\xi}$ of $\hat{q} (t_0, \hbar)$ extends to a holomorphic function $\omega (t_0, \xi)$ on the universal cover of the twice punctured $\xi$-plane $\CC_\xi \smallsetminus \set{ \xi_+, \xi_-}$; in symbols, $\omega (t_0, \xi) \in \cal{O} \big( \widetilde{ \CC_\xi \smallsetminus \set{ \xi_\pm} } \big)$.
Moreover, since the universal cover of a twice-punctured complex plane is (noncanonically) isomorphic to the upper halfplane $\HH$, we can (noncanonically) view $\omega$ as a holomorphic vector-valued function on $\HH$; i.e., $\omega \in \cal{O} (\HH)$.
\end{corollary}

\subsubsection{Exponential type.}
\label{250313152814}
The third and final ingredient of resurgence is the exponential bounds in $\xi$ at infinity on the global Borel transform $\omega$.

\begin{proposition}[exponential type]
\label{250313152819}
The Borel transform $\hat{\omega} (t,\xi)$ has exponential type in every direction at infinity in $\xi$, locally uniformly for all nonzero $t$.
More precisely, the global Borel transform $\omega \in \cal{O} (\tilde{M})$ has exponential type in every direction at infinity on $\tilde{M}$, locally uniformly for all nonzero $\tau$.
That is to say, for every nonzero $\tau$ and any direction $\alpha$ at infinity in $\tilde{M}_\tau$, there are real constants $\C, \K > 0$ and a sectorial neighbourhood $\sfSigma_\tau \subset \tilde{M}_\tau$ whose opening contains $\alpha$ such that
\begin{eqntag}
\label{250313151019}
	\big| \omega (\bm{\gamma}) \big| \leq \C e^{\K |\Z (\bm{\gamma})|}
\qqquad
	\forall \bm{\gamma} \in \sfSigma_{\tau} \subset \tilde{M}_\tau
\fullstop{,}
\end{eqntag}
where $\Z (\bm{\gamma}) \coleq \Z (\nu (\bm{\gamma}))$.
Furthermore, every nonzero $\tau_0 \in \CC_\tau$ has a neighbourhood $U_0 \subset \CC_\tau$ with the property that $\C, \K$ can be chosen uniformly for all $\tau \in U_0$, provided that $\alpha$ and $\sfSigma_\tau$ are chosen continuously for all $\tau \in U_0$.
Moreover, $\K$ can be taken arbitrarily small provided that $\tau_0$ is sufficiently large.
\end{proposition}

We will prove this Proposition in \autoref{250304161635}.
Specifically, it will follow as a consequence of \autoref{250314125137}.
The reason for denoting points of $\tilde{M}$ by the symbol $\bm{\gamma}$ will become clear in \autoref{250304152932}.
The property of exponential type of $\omega$ may be stated in the following more elementary terms using the point of view on $\omega$ as a multivalued function on the universal cover of the twice-punctured $\xi$-plane.

\begin{corollary}
\label{250224200423}
Assume the hypotheses of \autoref{250304131641}.
Then for any choice of branch of the multivalued holomorphic function $\omega (t_0, \xi) \in \cal{O} (\widetilde{\CC_\xi \smallsetminus \set{\xi_\pm}})$, there are constants $\C, \K > 0$ such that, for all $\xi$ sufficiently large,
\begin{eqntag}
\label{250313151023}
	\big| \omega (t_0, \xi) \big| \leq \C e^{\K |\xi|}
\fullstop
\end{eqntag}
Furthermore, every nonzero $t_0 \in \CC_t$ has a neighbourhood $U_0 \subset \CC_\tau$ such that $\C, \K$ can be chosen uniformly for all $t \in U_0$, provided that the branch of $\omega (t, \xi)$ is chosen continuously for all $t \in U_0$.
Moreover, $\K$ can be taken arbitrarily small provided that $t_0$ is sufficiently large.
\end{corollary}

\subsection{Borel Summability and Stokes Sectors}
\label{250304162404}

We now state several corollaries that articulate the Borel summability properties of the formal power series solution of the deformed Painlevé I equation.
Whilst these corollaries share a common theme and may appear rather similar in content, they apply to different scenarios and are neither equivalent nor interchangeable.
We believe that presenting a clear statement for each scenario separately enhances both clarity and applicability, making the ostensible repetitiveness a price worth paying.

\subsubsection{Arcs, sectors, and Stokes directions.}
The output of Borel resummation is a holomorphic function defined in sectors of the $\hbar$-plane.
To facilitate the description of Borel summability, let us introduce some notation and terminology regarding sectors in the $\hbar$-plane.

\enlargethispage{10pt}
\paragraph{Arcs and rays.}
Elements of $\RR / 2\pi\ZZ \cong \SS^1$ are called \dfn{phases}, \dfn{directions}, or \dfn{rays}.
We will use the symbol ``$\equiv$'' to mean equality of real numbers mod $2\pi$.
An \dfn{arc} $A \subset \RR / 2\pi\ZZ$ is a nonempty open interval; i.e., a nonempty, connected, and simply connected open subset.
We write arcs using the interval notation $(\alpha_1, \alpha_2)$ for a pair of rays $\alpha_1, \alpha_2 \in \SS^1$ respecting the standard orientation of the circle.
In this case, we say that $\alpha_1, \alpha_2$ are the \dfn{bounding rays} of $A$.
An arc $A = (\alpha_1, \alpha_2)$ has a well-defined length called \dfn{opening angle}: $|A| \coleq \alpha_2 - \alpha_1 \leq 2\pi$.
When $|A| = 2\pi$,~ $A$ is the complement of a single point in $\RR / 2\pi\ZZ$.
We denote the arc bisected by $\alpha$ with opening angle $\beta$ by 
\begin{eqn}
	\sfop{Arc}_\beta (\alpha) \coleq (\alpha - \beta/2, \alpha + \beta/2)
\fullstop
\end{eqn}
Similarly, given an arc $A = (\alpha_1, \alpha_2)$, we define
\begin{eqn}
	\sfop{Arc}_\beta (A) \coleq (\alpha_1 - \beta/2, \alpha_2 + \beta/2) = \Cup_{\alpha \in A} \sfop{Arc}_\beta (\alpha)
\fullstop
\end{eqn}
See \autoref{250227143304}.
This definition assumes that $|A| + \beta \leq 2\pi$, though it can be easily generalised by passing to the universal cover of $\SS^1$.

\paragraph{Sectors.}
The \dfn{real-oriented blowup} of the $\hbar$-plane $\CC_\hbar$ \textit{at the origin} is the bordered Riemann surface $\CC^\ast_\hbar \sqcup \SS^1$, often denoted by $[\CC_\hbar : 0]$, together with the holomorphic surjective map $[\CC_\hbar : 0] \to \CC_\hbar$, called the \dfn{blowdown map}, which sends the boundary circle $\SS^1$ to the point $0$ and is the identity map on the complement $[\CC_\hbar : 0] \smallsetminus \SS^1$.
Standard polar coordinates $(r,\theta)$ give a parameterisation of $[\CC_\hbar : 0]$ with respect to which the blowdown map has the familiar formula $(r, \theta) \mapsto \hbar = re^{i\theta}$.

A \dfn{sectorial domain} (or simply a \dfn{sector}) at the origin in the $\hbar$-plane $\CC_\hbar$ is any simply connected domain $U \subset \CC_\hbar^\ast$ such that there is a (necessarily unique) simply connected open subset $\tilde{U}$ in the real-oriented blowup $[\CC_\hbar : 0] = \CC^\ast_\hbar \sqcup \SS^1$ which intersects the boundary circle $\SS^1 = \RR / 2\pi \ZZ$ in an open arc $A \subset \SS^1$.
See \autoref{250227142106}.
This open arc $A$ is called the \dfn{opening} of $U$, and its length $|A|$ is called the \dfn{opening angle} of $U$.
In particular, a \dfn{straight sector} bisected by $\alpha \in \SS^1$ of opening angle $\beta$ and radius $r$ is the open set
\begin{eqn}
	\sfop{Sect}_\beta (\alpha; r) \coleq \set{ \hbar \in \CC^\ast ~\Big|~ \arg (\hbar) \in \sfop{Arc}_\beta (\alpha) \qtext{and} |\hbar| < r}
\fullstop
\end{eqn}
Similarly, we can define a straight sector of opening angle $\beta$, radius $r$, and \textit{bisected by an arc} $A \subset \SS^1$ is the open set
\begin{eqn}
	\sfop{Sect}_\beta (A; r) \coleq \set{ \hbar \in \CC^\ast ~\Big|~ \arg (\hbar) \in \sfop{Arc}_\beta (A) \qtext{and} |\hbar| < r}
\fullstop
\end{eqn}
Again, this definition assumes that $|A| + \beta \leq 2\pi$, though it can be easily generalised by passing to the universal cover of $\SS^1$.

\begin{definition}[Stokes directions]
\label{250224133114}
For any $t_0 \in \CC_t^\ast$, let $\theta_0 \coleq \arg (t_0)$.
We define the \dfn{Stokes directions} at $t_0$ to be the directions
\begin{eqn}
	\alpha_+ \coleq \tfrac{5}{4}(\theta_0 + \pi)
\qtext{and}
	\alpha_- \coleq \alpha_+ + \pi \equiv \tfrac{5}{4}(\theta_0 + \pi) + \pi
\fullstop
\end{eqn}
See \autoref{250221154526}.
All other directions are called \dfn{regular directions}, and we put
\begin{eqn}
	A_1 \coleq (\alpha_+, \alpha_-)
\qtext{and}
	A_2 \coleq (\alpha_-, \alpha_+)
\fullstop
\end{eqn}
\end{definition}

\begin{figure}[t]
\begin{adjustwidth}{-1cm}{1cm}
\centering
\begin{subfigure}{0.33\textwidth}
\begin{tikzpicture}
\begin{scope}[scale=1.2]
\fill [white, draw = black, ultra thick] (0,0) circle (35pt);
\draw [dashed, blue] (0,0) -- (130:45pt);
\draw [dashed, blue] (0,0) -- (-110:45pt);
\draw [dashed, orange] (0,0) -- (-20:35pt);
\draw [dashed, orange] (0,0) -- (40:35pt);
\draw [ultra thick, blue] (-110:35pt) arc (-110:130:35pt);
\draw [ultra thick, orange] (-20:33pt) arc (-20:40:33pt);
\draw [thin, blue, <->] (-110:38pt) arc (-110:130:38pt);
\draw [thin, black, <->] (40:32pt) arc (40:130:32pt);
\draw [thin, black, <->] (-20:32pt) arc (-20:-110:32pt);
\node at (130:43pt) [above, blue] {\scriptsize$\alpha_2 + \beta/2$};
\node at (-110:41pt) [below, blue] {\scriptsize$\alpha_1 - \beta/2$};
\node at (40:45pt) [orange] {$\alpha_2$};
\node at (-20:45pt) [orange] {$\alpha_1$};
\node at (85:25pt) {\scriptsize$\beta/2$};
\node at (-65:25pt) {\scriptsize$\beta/2$};
\node at (10:32pt) [left, orange] {$A$};
\node at (10:38pt) [right, blue] {$\sfop{Arc}_\beta (A)$};
\node at (210:27pt) {$\SS^1$};
\end{scope}
\end{tikzpicture}
\caption{Arcs on $\SS^1$.}
\label{250227143304}
\end{subfigure}
\hfill
\begin{subfigure}{0.66\textwidth}
\begin{tikzpicture}[scale=0.5]
\begin{scope}
\fill [grey] (-3,-5) rectangle (7,5);
\draw [blue, densely dashed, fill = blue, fill opacity = 0.3] (110:50pt) to [out=110,in=180] (1.5,4) arc (90:-100:4) to [out=165,in=260] (-100:50pt);
\fill [white, draw = red, ultra thick] (0,0) circle (50pt);
\draw [ultra thick, blue] (-100:50pt) arc (-100:110:50pt);
\node at (-1.5,4) {$[\CC_\hbar : 0]$};
\node at (110:45pt) [below] {$\alpha_2$};
\node at (-100:50pt) [above] {$\alpha_1$};
\node at (0:45pt) [above left, blue] {$A$};
\node at (25:130pt) [blue] {$\tilde{U}$};
\node at (180:20pt) [red] {$\SS^1$};
\node at (8.5,0) {$\too$};
\end{scope}
\begin{scope}[xshift=13cm]
\fill [grey] (-3,-5) rectangle (7,5);
\draw [blue, densely dashed, fill = blue, fill opacity = 0.3] (110:3pt) to [out=110,in=180] (1.5,4) arc (90:-100:4) to [out=165,in=260] (-100:3pt);
\fill [red, draw = red, ultra thick] (0,0) circle (3pt) node [left, red] {$0$};
\node at (-2,4) {$\CC_\hbar$};
\node at (25:130pt) [blue] {$U$};
\end{scope}
\end{tikzpicture}
\caption{The real-oriented blowup $[\CC_\hbar : 0] \to \CC_\hbar$ and a sectorial neighbourhood $U$ of the origin with opening $A = (\alpha_1, \alpha_2)$.}
\label{250227142106}
\end{subfigure}
\caption{}
\label{}
\end{adjustwidth}
\end{figure}

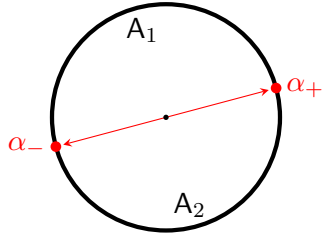
\begin{figure}[t]
\centering
\begin{tikzpicture}
\draw [black, ultra thick] (0,0) circle (1.5);
\fill [red] (15:1.5) circle (2pt) node [right] {$\alpha_+$};
\fill [red] (195:1.5) circle (2pt) node [left] {$\alpha_-$};
\draw [red, ->] (0,0) -- (15:1.4);
\draw [red, ->] (0,0) -- (195:1.4);
\fill [black] (0,0) circle (1pt);
\node at (105:1.2) {$A_1$};
\node at (-75:1.2) {$A_2$};
\end{tikzpicture}
\caption{The resurgent Stokes diagram at a point $t_0 \in \CC^\ast_t$ or $z_0 \in \CC^\ast_z$ such that $z_0^4 = -24 t_0$.
There are two Stokes rays at antipodal directions $\alpha_+$ and $\alpha_-$.
If we vary $t_0$ in the punctured $t$-plane by rotating anti-clockwise around the origin, these rays remain antipodal but rotate clockwise.
A rotation by $4 \pi / 5$ in the $t$-plane swaps the Stokes rays $\alpha_+, \alpha_-$.}
\label{250221154526}
\end{figure}

\subsubsection{Pointwise Borel summability.}
First, we spell out the case of Borel summability for a fixed nonzero point in the $t$-plane.
We start by considering a single regular direction.

\begin{proposition}[Pointwise Borel Summability in a Single Direction]
\label{250223133010}
Fix any $t_0 \in \CC_t^\ast$ and select a branch of $\hat{q} (t, \hbar)$ near $t_0$.
Then the formal power series $\hat{q} (t_0, \hbar) \in \CC \bbrac{\hbar}$ is stably Borel summable in every regular direction $\alpha$ at $t_0$.
Thus, the Borel resummation 
\begin{eqn}
	q_\alpha (t_0, \hbar) 
		\coleq s_\alpha \big[ \, \hat{q} \, \big] (t_0, \hbar)
\end{eqn}
in the direction $\alpha$ defines a holomorphic function on a sectorial neighbourhood $S \subset \CC_\hbar$ of the origin with opening $\sfop{Arc}_\pi (\alpha)$ which is asymptotic to $\hat{q} (t_0, \hbar)$ of uniform factorial type:
\begin{eqntag}
\label{250224152925}
	q_\alpha (t_0, \hbar) \simeq \hat{q} (t_0, \hbar)
\qquad
	\text{as $\hbar \to 0$ unif. along $\sfop{Arc}_\pi (\alpha)$}
\fullstop
\end{eqntag}
In fact, $q_\alpha (t_0, \hbar)$ is the unique holomorphic function on $S$ with this property.
Furthermore, the sectorial neighbourhood $S \subset \CC_\hbar$ can be chosen to be the straight sector $S = \sfop{Sect}_\pi (\alpha; r)$ of some radius $r > 0$ that can be taken arbitrarily large provided that $t_0$ is sufficiently large.
\end{proposition}

This Proposition will follow as a consequence of \autoref{250224130127}.
To be clear about the order of the quantifiers in the final assertion, it says that for every $r > 0$, there is a point $t_0$ with $\arg (t_0) = \theta_0$ such that $q_\alpha (t_0, \hbar)$ defines a holomorphic function on $S = \sfop{Sect}_\pi (\alpha; r)$.
Let us also recall explicitly that the Borel resummation in the direction $\alpha$ is written in terms of the Laplace transform in the direction $\alpha$ of the Borel transform:
\begin{eqn}
	q_\alpha (t_0, \hbar) 
		= s_\alpha \big[ \, \hat{q} \, \big] (t_0, \hbar)
		= q_0 (t_0) + \Laplace_\alpha \big[\, \omega \,] (t_0, \hbar)
		= q_0 (t_0) + \int_{e^{i\alpha}\RR_+} e^{-\xi/\hbar} \omega (t_0, \xi) \d{\xi}
\fullstop
\end{eqn}

As we vary the ray $\alpha$ through an arc $A$ consisting of regular directions only, the different Borel resummations $q_\alpha (t_0, \hbar)$ can be patched together into a single holomorphic function $q_A (t_0, \hbar)$ on a larger sector in the $\hbar$-plane.
This leads to the following description of Borel summability along an arc of directions.

\begin{corollary}[Pointwise Borel Summability in an Arc of Directions]
\label{250224153234}
Fix any $t_0 \in \CC_t^\ast$, select a branch of $\hat{q} (t, \hbar)$ near $t_0$, and choose an arc $A \subset \SS^1$ of regular directions at $t_0$.
The Borel resummations $q_\alpha (t_0, \hbar)$ of $\hat{q} (t_0, \hbar)$ for $\alpha \in A$ assemble into a single holomorphic function $q_A (t_0, \hbar)$ defined on a sectorial neighbourhood $S \subset \CC_\hbar$ of the origin with opening $\sfop{Arc}_\pi (A)$ which is asymptotic to $\hat{q} (t_0, \hbar)$ of factorial type:
\begin{eqntag}
\label{250223133816}
	q_A (t_0, \hbar) \simeq \hat{q} (t_0, \hbar)
\qquad
	\text{as $\hbar \to 0$ along $\sfop{Arc}_\pi (A)$}
\fullstop
\end{eqntag}
In fact, $q_A (t_0, \hbar)$ is the unique holomorphic function on $S$ with this property.
Furthermore, if $A$ is not bounded by a Stokes ray at $t_0$, then the sectorial neighbourhood $S \subset \CC_\hbar$ can be chosen to be the straight sector $S = \sfop{Sect}_\pi (A; r)$ of some radius $r > 0$ that can be taken arbitrarily large provided that $t_0$ is sufficiently large.
\end{corollary}

Thus, for example, if we fix a real and positive $t_0$, then $\theta_0 \equiv 0$ and so $\alpha_+ = 5\pi/4$ and $\alpha_- = \pi/4$.
Then, for any two $\alpha$ and $\alpha'$ in $(\pi/4, 5\pi/4)$, the Borel resummations $q_{\alpha} (t_0, \hbar) \in \cal{O} (S)$ and $q_{\alpha'} (t_0, \hbar) \in \cal{O} (S')$ coincide on the double intersection $S \cap S'$, and therefore they patch together to define the same holomorphic function on $S \cup S'$.
This is true for all $\alpha \in (\pi/4, 5\pi/4)$, and so we end up with a holomorphic function defined on a sectorial neighbourhood $S$ whose opening is the entire circle $\SS^1$ except the direction $3\pi/4$.

Note that there is no claim regarding the radial size of the sectorial neighbourhood $S$ when $A$ is bounded by a Stokes ray.
In fact, in this case it is not possible to take $S$ to be a straight sector.
Instead, we can take $S$ to be a countable union of straight sectors of decreasing radius.
Namely, if we fix an exhaustive sequence $(A_n)$ of proper subarcs $A_n \Subset A$, then we get a sequence of straight sectors $\sfop{Sect}_\pi (A_n; r_n)$ of some radius $r_n > 0$, so we can take $S$ to be the union of all $\sfop{Sect}_\pi (A_n; r_n)$.
If $A$ is bounded by a Stokes ray, then the sequence $(r_n)$ necessarily goes to $0$ as $n \to \infty$.

\subsubsection{Locally uniform Borel summability.}
Next, we state the Borel summability property when $t$ is allowed to vary in a small local neighbourhood of a fixed nonzero point in the $t$-plane.
Again, we start by considering summability in a single regular direction.

\begin{proposition}[Locally Uniform Borel Summability in a Single Direction]
\label{250211094021}
Fix any $t_0 \in \CC_t^\ast$, select a branch of $\hat{q} (t, \hbar)$ near $t_0$, and choose a regular direction $\alpha$ at $t_0$.
Then there is a neighbourhood $U \subset \CC_t^\ast$ around $t_0$ such that $\hat{q} (t,\hbar)$ is stably Borel summable in the direction $\alpha$ uniformly for all $t \in U$.
Thus, there is a sectorial neighbourhood $S \subset \CC_\hbar$ of the origin with opening $\sfop{Arc}_\pi (\alpha)$ such that the Borel resummation $q_\alpha (t, \hbar) = s_\alpha [ \, \hat{q} \, ] (t,\hbar)$ in the direction $\alpha$ defines a holomorphic function on the domain $U \times S$ which is uniformly asymptotic to $\hat{q} (t, \hbar)$ of uniform factorial type:
\begin{eqntag}
\label{250224155430}
	q_\alpha (t, \hbar) \simeq \hat{q} (t, \hbar)
\qquad
	\text{as $\hbar \to 0$ unif. along $\sfop{Arc}_\pi (\alpha)$,}
\end{eqntag}
uniformly for all $t \in U$.
In fact, $q_\alpha$ is the unique holomorphic function on $U \times S$ with this property.
Furthermore, the sectorial neighbourhood $S \subset \CC_\hbar$ can be chosen to be the straight sector $S = \sfop{Sect}_\pi (\alpha; r)$ of some radius $r > 0$ that can be taken arbitrarily large provided that $t_0$ is sufficiently large and $U$ is sufficiently small.
\end{proposition}

This Proposition will follow as a consequence of \autoref{250224161315}.

\begin{corollary}[Locally Uniform Borel Summability in an Arc of Directions]
\label{250224155618}
Fix any $t_0 \in \CC_t^\ast$, select a branch of $\hat{q} (t, \hbar)$ near $t_0$, and choose an arc $A \subset \SS^1$ of regular directions at $t_0$ which is not bounded by a Stokes ray at $t_0$.
Then there is a neighbourhood $U \subset \CC_t^\ast$ around $t_0$ such that $\hat{q} (t,\hbar)$ is Borel summable in every direction $\alpha \in A$ uniformly for all $t \in U$.
Thus, there is a sectorial neighbourhood $S \subset \CC_\hbar$ of the origin with opening $\sfop{Arc}_\pi (A)$ such that the Borel resummations $q_\alpha (t,\hbar)$ of $\hat{q} (t,\hbar)$ for $\alpha \in A$ assemble into a single holomorphic function $q_A (t,\hbar)$ defined on the domain $U \times S$ which is uniformly asymptotic to $\hat{q} (t_0, \hbar)$ of factorial type:
\begin{eqntag}
\label{250224160811}
	q_A (t, \hbar) \simeq \hat{q} (t, \hbar)
\qquad
	\text{as $\hbar \to 0$ along $\sfop{Arc}_\pi (A)$,}
\end{eqntag}
uniformly for all $t \in U$.
In fact, $q_A$ is the unique holomorphic function on $U \times S$ with this property.
Furthermore, the sectorial neighbourhood $S \subset \CC_\hbar$ can be chosen to be the straight sector $S = \sfop{Sect}_\pi (A; r)$ of some radius $r > 0$ that can be taken arbitrarily large provided that $t_0$ is sufficiently large and $U$ is sufficiently small.
\end{corollary}

\subsubsection{Borel summability in Stokes sectors.}
Now we describe the Borel summability properties in the large.
For this purpose, we fix a direction for resummation and describe the distinguished regions in the $t$-plane and the $\tau$-plane where the Borel resummation in this direction is well-defined.

\begin{definition}[Stokes lines and sectors]
\label{250224171850}
Fix any phase $\alpha \in \SS^1$.
A \dfn{Stokes line} in the $\tau$-plane is any infinite straight ray $e^{i \theta} \RR_+ \subset \CC_\tau$ with phase $\theta$ satisfying $5 \theta + 2 \alpha \equiv 0$.
The union of all Stokes lines is called the \dfn{Stokes graph}.
Any connected component $V \subset \CC_\tau$ of the complement of the Stokes graph is called a \dfn{Stokes sector}.
\end{definition}

Note that Stokes lines and Stokes sectors depend on the chosen phase $\alpha$, and that the equation $5 \theta + 2 \alpha \equiv 0$ has exactly five distinct solutions in $\SS^1$ distributed evenly around the circle.
So for the deformed Painlevé I equation, there are a total of five Stokes lines and therefore five Stokes sectors each of which is an infinite sector in the $\tau$-plane with opening angle $2\pi/5$.

If we choose the representative of $\alpha$ in $[0, 2\pi)$, we can describe the Stokes lines and sectors more explicitly; see \autoref{250224175327}.
The five numbers $\theta_k \coleq \tfrac{2}{5} (k\pi - \alpha) \in \RR$, for $k = 0, \ldots, 4$, determine the five distinct solutions of $5 \theta + 2 \alpha \equiv 0$.
So the five Stokes lines are $e^{i \theta_k} \RR_+$ and the five Stokes sectors are $V_k \coleq \set{ \theta_k < \arg (\tau) < \theta_{k+1} } \subset \CC_\tau^\ast$ where $k+1$ is understood mod $5$.
It is also helpful to be aware of their projections $e^{i\theta'_k} \RR_+$ and $U_k$ to the $t$-plane (which are also traditionally called \textit{Stokes lines} and \textit{Stokes sectors}), where $\theta'_k = 2\theta_k + \pi$; see \autoref{250224180406}.

\begin{figure}[t]
\begin{adjustwidth}{-1cm}{1cm}
\centering
\begin{subfigure}[t]{0.28\linewidth}
\centering
\begin{tikzpicture}
\clip (-2,-2) rectangle (2,2);
\fill [grey] (-2,-2) rectangle (2,2);
\draw [thick, orange] (0,0) -- (0:200pt);
\draw [thick, orange] (0,0) -- (72:200pt);
\draw [thick, orange] (0,0) -- (144:200pt);
\draw [thick, orange] (0,0) -- (216:200pt);
\draw [thick, orange] (0,0) -- (288:200pt);
\node at (36:35pt) {$V_{0}$};
\node at (108:35pt) {$V_{1}$};
\node at (180:35pt) {$V_{2}$};
\node at (252:35pt) {$V_{3}$};
\node at (324:35pt) {$V_{4}$};
\node [orange] at (-8:50pt) {$\theta_0$};
\node [orange] at (62:55pt) {$\theta_1$};
\node [orange] at (136:60pt) {$\theta_2$};
\node [orange] at (208:55pt) {$\theta_3$};
\node [orange] at (280:50pt) {$\theta_4$};
\node at (45:65pt) {$\CC_{\tau}$};
\end{tikzpicture}
\caption{}
\label{250224175327}
\end{subfigure}
\hfill
\begin{subfigure}[t]{0.68\linewidth}
\centering
\begin{tikzpicture}
\begin{scope}
\clip (-2,-2) rectangle (2,2);
\fill [grey] (-2,-2) rectangle (2,2);
\draw [thick, orange] (0,0) -- (180:200pt);
\draw [thick, orange] (0,0) -- (324:200pt);
\draw [thick, orange] (0,0) -- (108:200pt);
\node at (252:35pt) {$U_0$};
\node at (36:35pt) {$U_1$};
\node at (135:65pt) {$\CC_{t}$};
\node [orange] at (180:45pt) [below] {$\theta'_0$};
\node [orange] at (324:55pt) [below] {$\theta'_1$};
\node [orange] at (108:50pt) [right] {$\theta'_2$};
\end{scope}
\begin{scope}[xshift=4.25cm]
\clip (-2,-2) rectangle (2,2);
\fill [grey] (-2,-2) rectangle (2,2);
\draw [thick, orange] (0,0) -- (108:200pt);
\draw [thick, orange] (0,0) -- (252:200pt);
\draw [thick, orange] (0,0) -- (36:200pt);
\node at (180:35pt) {$U_2$};
\node at (324:35pt) {$U_3$};
\node at (135:65pt) {$\CC_{t}$};
\node [orange] at (108:45pt) [left] {$\theta'_2$};
\node [orange] at (252:50pt) [left] {$\theta'_3$};
\node [orange] at (36:45pt) [below right] {$\theta'_4$};
\end{scope}
\begin{scope}[xshift=8.5cm]
\clip (-2,-2) rectangle (2,2);
\fill [grey] (-2,-2) rectangle (2,2);
\draw [thick, orange] (0,0) -- (36:200pt);
\draw [thick, orange] (0,0) -- (180:200pt);
\draw [thick, orange] (0,0) -- (324:200pt);
\node at (108:35pt) {$U_4$};
\node at (252:35pt) {$U_0$};
\node at (135:65pt) {$\CC_{t}$};
\node [orange] at (38:50pt) [above] {$\theta'_4$};
\node [orange] at (180:45pt) [below] {$\theta'_0$};
\node [orange] at (324:55pt) [below] {$\theta'_1$};
\end{scope}
\end{tikzpicture}
\caption{}
\label{250224180406}
\end{subfigure}
\end{adjustwidth}
\caption{Stokes sectors $V_k$ and Stokes lines $e^{i\theta_k} \RR_+$ in the $\tau$-plane and their projections $U_k$ and $e^{i \theta'_k} \RR_+$ to the $t$-plane for $\alpha = 0$.
In this case, $\theta_k = 0, \tfrac{2\pi}{5}, \tfrac{4\pi}{5}, \tfrac{6\pi}{5}, \tfrac{8\pi}{5}$ and $\theta'_k = \pi, \tfrac{9\pi}{5}, \tfrac{3\pi}{5}, \tfrac{7\pi}{5}, \tfrac{2\pi}{5}, \tfrac{\pi}{5}$.}
\label{250224185222}
\end{figure}
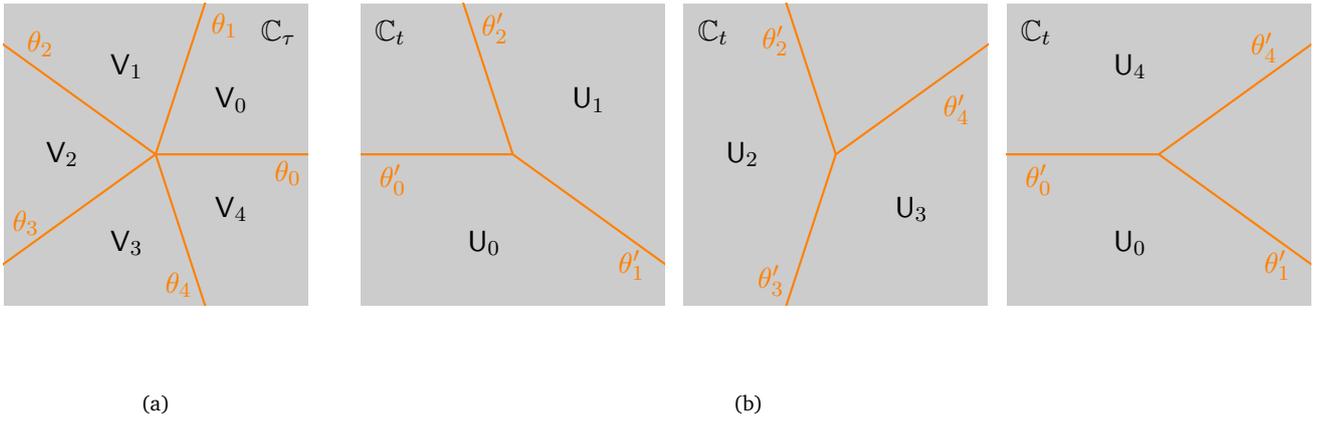

Observe that the choice of a Stokes sector $V \subset \CC_\tau$ determines a distinguished choice of square-root branch $\sqrt{t}$ in the projected Stokes sector $U \subset \CC_t$.
Consequently, it determines a distinguished choice of branch $\hat{q} (t,\hbar)$ in $U$ of the multivalued formal solution.
We call it the \textit{distinguished branch} of $\hat{q}$ associated with the given Stokes sector.

\begin{proposition}[Borel Summability in Stokes Sectors]
\label{250210133432}
Fix any phase $\alpha \in \SS^1$ and select a Stokes sector $V \subset \CC_\tau$.
Let $U \subset \CC_t$ be its projection to the $t$-plane, and let $\hat{q} (t,\hbar)$ be the distinguished branch on $U$ of the multivalued formal solution.
Then $\hat{q} (t,\hbar)$ is stably Borel summable in the direction $\alpha$, locally uniformly for all $t \in U$.
Thus, the Borel resummation $\hat{q}_\alpha (t,\hbar) = s_\alpha [\, \hat{q} \,] (t,\hbar)$ defines a holomorphic function on a domain $\UU \subset U \times \CC_\hbar$ with the following property: every $t_0 \in U$ has a neighbourhood $U_0 \subset U$ such that there is a sectorial neighbourhood $S_0 \subset \CC_\hbar$ of the origin with opening $\sfop{Arc}_\pi (\alpha)$ satisfying $U_0 \times S_0 \subset \UU$.
Furthermore, $q_\alpha$ is locally uniformly asymptotic to $\hat{q}$ of uniform factorial type:
\begin{eqntag}
	q_\alpha (t, \hbar) \simeq \hat{q} (t,\hbar)
\qquad
	\text{as $\hbar \to 0$ unif. along $\sfop{Arc}_\pi (\alpha)$,}
\end{eqntag}
locally uniformly for all $t \in U$.
In fact, $q_\alpha$ is the unique holomorphic function on $\UU$ with this property.
Moreover, $U$ is a maximal domain of Borel summability of $\hat{q}$ in the direction of $\alpha$; i.e., there does not exist another domain $U'$ properly containing $U$ such that $\hat{q} (t,\hbar)$ is stably Borel summable in the direction $\alpha$, locally uniformly for all $t \in U$.
\end{proposition}

This Proposition will follow as a consequence of \autoref{250305151549}.

\subsubsection{Existence and uniqueness of deformed tritronquée solutions.}
We conclude this subsection with an existence and uniqueness result for special solutions of the deformed Painlevé I equation, which follows immediately from \autoref{250210133432}.

\begin{corollary}[Exact Existence and Uniqueness Theorem]
\label{250210115249}
Fix any phase $\alpha \in \SS^1$ and select a Stokes sector $V \subset \CC_\tau$.
Let $U \subset \CC_t$ be its projection to the $t$-plane, and let $\hat{q} (t,\hbar)$ be the distinguished branch on $U$ of the multivalued formal solution.
Then the deformed Painlevé I equation has a unique holomorphic solution $q = q (t,\hbar)$ defined on a domain $\UU \subset U \times \CC_\hbar$ which is locally uniformly asymptotic to $\hat{q}$ of uniform factorial type:
\begin{eqntag}
\label{241025133201}
	q (t, \hbar) \simeq \hat{q} (t, \hbar)
\qquad
	\text{as $\hbar \to 0$ unif. along $\sfop{Arc}_\pi (\alpha)$,}
\end{eqntag}
locally uniformly for all $t \in U$.
\end{corollary}

\begin{definition}[Deformed Tritronquée Solutions]
\label{250304154226}
We refer to the solutions of the deformed Painlevé I equation defined in \autoref*{250210115249} as the \dfn{deformed tritronquée solutions} or the \dfn{deformed $0$-parameter solutions}.
\end{definition}

Observe that there are exactly five distinct deformed tritronquée solutions for every fixed phase $\alpha$, one of each of the five Stokes sectors in the $t$-plane with phase $\alpha$.

\subsection{The Stokes Phenomenon}
\label{250313153813}
In this subsection, we describe the failure of Borel summability in a Stokes direction and the associated Stokes phenomenon.

\begin{proposition}[Lateral Borel Summability]
\label{250225101025}
Fix any $t_0 \in \CC_t^\ast$, select a branch of $\hat{q} (t, \hbar)$ near $t_0$, and suppose $\alpha$ is a Stokes direction at $t_0$.
Then the formal power series $\hat{q} (t_0, \hbar) \in \CC \bbrac{\hbar}$ is laterally Borel summable in the direction $\alpha$.
Thus, the left and right lateral Borel resummations 
\begin{eqntag}
\label{250315125248}
	q_\alpha^\textup{L/R} (t_0, \hbar) 
		\coleq s_\alpha^\textup{L/R} \big[ \, \hat{q} \, \big] (t_0, \hbar)
		= q_0 + \Laplace_\alpha^\textup{L/R} \big[ \, \omega \, \big] (t_0, \hbar)
\end{eqntag}
in the direction $\alpha$ define two holomorphic functions on a sectorial neighbourhood $S \subset \CC_\hbar$ of the origin with opening $\sfop{Arc}_\pi (\alpha)$, each of which is asymptotic to $\hat{q} (t_0, \hbar)$ of factorial type:
\begin{eqntag}
\label{250312152300}
	q_\alpha^\textup{L/R} (t_0, \hbar) \simeq \hat{q} (t_0, \hbar)
\qquad
	\text{as $\hbar \to 0$ along $\sfop{Arc}_\pi (\alpha)$}
\fullstop
\end{eqntag}
\end{proposition}

Recall that the left and right lateral Laplace transforms in the direction $\alpha$ are given by
\begin{eqntag}
\label{250315125628}
	\Laplace_\alpha^\textup{L/R} \big[ \, \omega \, \big] (t_0, \hbar)
		\coleq \int_{e^{i \alpha} \RR_+^\textup{L/R}} e^{-\xi/\hbar} \omega (t_0, \xi) \d{\xi}
\fullstop{,}
\end{eqntag}
where the contours $e^{i \alpha} \RR_+^\textup{L}$ and $e^{i \alpha} \RR_+^\textup{R}$ are the infinite straight rays in the direction $\alpha$ which avoid $\xi_0$ on the left and on the right, respectively.
Let us also stress the absence of the qualifier ``uniform'' in \eqref{250312152300}, for otherwise there could only be one such holomorphic function.
This would mean that the left and right lateral Borel resummations of $\hat{q}$ actually agree in the direction $\alpha$, which is not typically the case.
This discrepancy is an example of the Stokes phenomenon, which we now describe more precisely.

\begin{definition}[Stokes jump]
\label{250314213817}
Fix any $t_0 \in \CC_t^\ast$, select a branch of $\hat{q} (t, \hbar)$ near $t_0$, and suppose $\alpha$ is a Stokes direction at $t_0$.
Let $q_\alpha^\textup{L/R} (t_0, \hbar) \in \cal{O} (S)$ be the two lateral Borel resummations of $\hat{q} (z_0, \hbar)$ in the direction $\alpha$ from \autoref{250225101025}.
The \dfn{Stokes jump} across $\alpha$ from right to left is the difference between the left and right lateral Borel resummations of $\hat{q}$ in the direction $\alpha$:
\begin{eqntag}
\label{250314214012}
	\Delta_{\alpha} \hat{q} (t_0, \hbar) 
		\coleq q_\alpha^\textup{L} (t_0, \hbar) - q_\alpha^\textup{R} (t_0, \hbar)
\fullstop
\end{eqntag}
\end{definition}

The Stokes jump is clearly a holomorphic function on $S$; i.e., $\Delta_{\alpha} \hat{q} (t_0, \hbar) \in \cal{O} (S)$.
Furthermore, since both $q_\alpha^\textup{L}$ and $q_\alpha^\textup{R}$ admit the same asymptotic expansion $\hat{q}$ as $\hbar \to 0$ in $S$, the Stokes jump is asymptotic to $0$ of factorial type:
\begin{eqntag}
\label{250314214016}
	\Delta_{\alpha} \hat{q} (t_0, \hbar) \simeq 0
\qquad
	\text{as $\hbar \to 0$ along $\sfop{Arc}_\pi (\alpha)$}
\fullstop
\end{eqntag}
Again, let us point out the absence of the qualifier ``uniformly'' in \eqref{250314214016}, for otherwise it would force the Stokes jump to be identically zero.

\begin{definition}[variation]
\label{250314214037}
Fix any $t_0 \in \CC_t^\ast$, select a branch of $\hat{q} (t, \hbar)$ near $t_0$, and suppose $\alpha$ is a Stokes ray at $t_0$.
Let $q_\alpha^\textup{L/R} (t_0, \hbar) \in \cal{O} (S)$ be the two lateral Borel resummations of $\hat{q} (z_0, \hbar)$ in the direction $\alpha$ from \autoref{250225101025}.
Let $\xi_0 \in \CC_\xi$ the Borel singular value corresponding to the visible singularity in the direction $\alpha$ at $t_0$; namely, $\xi_0$ is given by \eqref{250314214209} for some choice of sign.
The \dfn{variation} of the multivalued function $\omega (t_0, \xi)$ \textit{at $\xi_0$} is the difference between its values on two consecutive sheets:
\begin{eqntag}
\label{250314214433}
	\Delta_{\xi_0} \omega (t_0, \xi) 
		\coleq \omega (t_0, \xi_0 + \xi^\textup{L}) - \omega (t_0, \xi_0 + \xi^\textup{R})
\fullstop
\end{eqntag}
where $\xi^\textup{L}$ and $\xi^\textup{R}$ are two preimages of $\xi$ on the universal cover of the punctured neighbourhood of $\xi_0$ related to each other by the anti-clockwise primitive generator of the deck transformations; i.e., $\xi^\textup{R} = e^{2 \pi i } \xi^\textup{L}$.
\end{definition}

The variation $\Delta_{\xi_0} \omega (t_0, \xi)$ may be regarded as a sectorial germ at the origin in the Borel $\xi$-plane.
But since both $\omega (t_0, \xi_0 + \xi^\textup{L})$ and $\omega (t_0, \xi_0 + \xi^\textup{R})$ admit analytic continuation in the direction $\alpha$ with exponential type at infinity, we obtain the following characterisation of the Stokes jump in terms of the Laplace transform of the Borel transform's variation at the visible singularity.

\begin{proposition}[Stokes Phenomenon]
\label{250314214443}
Fix any $t_0 \in \CC_t^\ast$, select a branch of $\hat{q} (t, \hbar)$ near $t_0$, and suppose $\alpha$ is a Stokes ray at $t_0$.
Let $q_\alpha^\textup{L/R} (t_0, \hbar) \in \cal{O} (S)$ be the two lateral Borel resummations of $\hat{q} (z_0, \hbar)$ in the direction $\alpha$ from \autoref{250225101025}.
Let $\xi_0 \in \CC_\xi$ the Borel singular value corresponding to the visible singularity in the direction $\alpha$ at $t_0$; namely, $\xi_0$ is given by \eqref{250314214209} for some choice of sign.
Then for all $\hbar \in S$, the Stokes jump across $\alpha$ is
\begin{eqntag}
\label{250315130035}
	\Delta_{\alpha} \hat{q} (t_0, \hbar)
	= e^{-\xi_0/\hbar} \Laplace_\alpha \big[ \Delta_{\xi_0} \omega \big] (t_0, \hbar)
\fullstop{,}
\end{eqntag}
where $\Delta_{\xi_0} \omega (t_0, \xi)$ is the variation of $\omega$ at $\xi_0$ given by \eqref{250314214433}.
\end{proposition}

\begin{proof}
This is a matter of a simple calculation.
Expanding \eqref{250314214012} according to the definition \eqref{250315125248} and especially \eqref{250315125628}, we have:
\begin{eqn}
	\Delta_{\alpha} \hat{q} (t_0, \hbar) 
		= q_\alpha^\textup{L} (t_0, \hbar) - q_\alpha^\textup{R} (t_0, \hbar)
		= \Laplace_\alpha^\textup{L} \big[ \, \omega \, \big] (t_0, \hbar) 
			- \Laplace_\alpha^\textup{R} \big[ \, \omega \, \big] (t_0, \hbar)
		= \int_{\mathscr{C}} e^{-\xi/\hbar} \omega (t_0, \xi) \d{\xi}
\end{eqn}
where the integration contour $\mathscr{C}$ homotopic to the keyhole contour around $\xi_0$ obtained by concatenating the inverse of the contour $e^{i \alpha} \RR_+^\textup{R}$ with $e^{i \alpha} \RR_+^\textup{L}$.
Let $\tau_0$ be root of $\tau^2 = -t/6$ corresponding to the chosen branch of $\hat{q}$ near $t_0$.
Then the two arms of this keyhole contour lift via the central charge $\Z_{\tau_0} : M_{\tau_0} \to \CC_\xi$ to infinite paths on two consecutive sheets of the Borel surface emanating from the visible Borel singularity above $\xi_0$.
Therefore, the above contour integral may be written like so:
\begin{eqn}
	\Delta_{\alpha} \hat{q} (t_0, \hbar)
	= 
	\int_{\xi_0}^{e^{i \alpha} \cdot \infty} e^{-\xi/\hbar} \Big( \omega (t_0, \xi^\textup{L}) - \omega (t_0, \xi^\textup{R}) \Big) \d{\xi}
\fullstop{,}
\end{eqn}
where $\omega (t_0, \xi^\textup{L})$ and $\omega (t_0, \xi^\textup{R})$ denote the values at $(t_0, \xi)$ of the two relevant branches of the multivalued function $\omega (t_0, \xi)$.
Shifting the integration variable $\xi \to \xi_0 + \xi$, the factor $e^{-\xi_0/\hbar}$ appears and the integral becomes the Laplace transform of the variation \eqref{250314214433}:
\begin{eqn}
	e^{-\xi_0/\hbar} \int_{0}^{e^{i \alpha} \cdot \infty} e^{-\xi/\hbar} 
		\Big( \omega (t_0, \xi_0 + \xi^\textup{L}) - \omega (t_0, \xi_0 + \xi^\textup{R}) \Big)
		\d{\xi}
	= e^{-\xi_0/\hbar} \Laplace_\alpha \big[ \Delta_{\xi_0} \omega \big] (t_0, \hbar)	
\fullstop
\end{eqn}
So we obtain the desired identity \eqref{250315130035}.
\end{proof}

\section{The Borel-Laplace Transformation}
\label{250218190702}

In this section, we derive the equation for the Borel transform $\hat{\omega} = \Borel [\, \hat{q} \,]$ of the formal solution $\hat{q}$ of the deformed Painlevé I equation.

\paragraph{}\removespace
First, let us recall some elementary properties of the Borel-Laplace transformation.
The Borel transform of a monomial $\hbar^n$ with $n \geq 1$ is given by $\Borel [\hbar^n] = \xi^{n-1} / (n-1)!$, whilst $\Borel [c] = 0$ whenever $c$ is independent of $\hbar$.
The Borel transform preserves $t$-derivatives and converts the standard product of function into the convolution product.
Namely, if $\hat{f}, \hat{g}$ are any two formal power series in $\hbar$ with $t$-dependent coefficients, then $\Borel [\, \del_t \hat{f} \,] = \del_t \Borel [\, \hat{f} \,]$ and $\Borel [\, \hat{f} \cdot \hat{g} \,] = \Borel [\, \hat{f} \,] \ast \Borel [\, \hat{g} \,] + f_0 \Borel [\, \hat{g} \,] + g_0 \Borel [\, \hat{f} \,]$ where $f_0 (t) \coleq \hat{f} (t,0)$ and $g_0 (t) \coleq \hat{g} (t,0)$.
From these rules, one can deduce the following two identities:
\begin{eqn}
	\Borel [\, \hbar \hat{f} \,] = 1 \ast \Borel [\, \hat{f} \,] = \int_0^\xi \Borel [\, \hat{f} \,] \d{\xi'}
\qqtext{and}
	\Borel [\, \hbar^\inv \hat{f} \,] = \del_\xi \Borel [\, \hat{f} \,]
\fullstop
\end{eqn}

\paragraph{}\removespace
The most naive approach to obtaining an equation for the Borel transform $\hat{\omega}$ is to apply the Borel-Laplace transformation directly to the deformed Painlevé I equation \eqref{250204124541}.
From the standard rules recalled above it we find that $\Borel [\, \del^2_t \hat{q} \,] = \del^2_t \hat{\omega}$ and $\Borel [\, \hat{q}^2 \,] = \hat{\omega} \ast \hat{\omega} + 2 q_0 \hat{\omega}$
Consequently, the Borel transform $\hat{\omega}$ satisfies the following integro-differential equation in the unknown variable $\omega$:
\begin{eqntag}
\label{250315103819}
	\xi \ast \del^2_t \omega + \xi q_0 = 6 \omega \ast \omega + 2q_0 \omega
\fullstop
\end{eqntag}
This equation can be slightly improved if we make a change of variables $q \to \Q$ given by $q = q_0 + \hbar \Q$ and rewrite \eqref{250315103819} instead as an equation for the Borel transform of the power series $\hat{\Q} \coleq \hbar^\inv (\hat{q} - q_0)$.
Since $6q_0^2 + t = 0$, the deformed Painlevé I equation \eqref{250204124541} gets transformed under this change of variables into the equation $\hbar^2 \ddot{q}_0 + \hbar^3 \ddot{\Q} = 12 q_0 \hbar \Q + 6 \hbar^2 \Q^2$.
If we divide through by $\hbar^3$ and rearrange, we find
\begin{eqn}
	\ddot{\Q} - 12 q_0 \hbar^\inv \Q = \hbar^\inv (6\Q^2 - \ddot{q}_0)
\fullstop
\end{eqn}
Consequently, the Borel transform $\hat{\!\mathit{\Omega}} \coleq \Borel [\, \hat{\Q} \,]$ satisfies the following non-linear partial integro-differential equation in the unknown variable $\mathit{\Omega}$:
\begin{eqntag}
\label{250311185553}
	\big( \del^2_t - 12 q_0 \del^2_\xi \big) \mathit{\Omega} 
		= 6 \del_\xi (\mathit{\Omega} \ast \mathit{\Omega} )
\fullstop
\end{eqntag}
Note that the Borel transforms $\hat{\omega}$ and $\hat{\!\mathit{\Omega}}$ are related by the identity
\begin{eqntag}
\label{250315105907}
	\hat{\omega} = 1 \ast \hat{\!\mathit{\Omega}}
\qqtext{i.e.}
	\hat{\omega} (t,\xi) = \int_0^\xi \hat{\!\mathit{\Omega}} (t,s) \d{s}
\fullstop
\end{eqntag}
In this paper, we will construct a global holomorphic solution of \eqref{250315103819} or equivalently of \eqref{250311185553}.
However, instead of working with these equations directly, we take a different approach by passing to the associated first-order Hamiltonian system for the deformed Painlevé I equation.
We further conveniently simplify this Hamiltonian system in \autoref{250208144823} before applying the Borel-Laplace transformation in \autoref{250221201337}.

\subsection{The Associated System}
\label{250208144823}

Let us motivate some of the constructions in this subsection.
Recall that, when analysing a nonlinear differential system $\dot{\mathbf{x}} = \F (\mathbf{x})$, a common and powerful technique in dynamical systems is to consider its linearisation around a stationary solution $\mathbf{x}_0$ (which means $\F (\mathbf{x}_0) = 0$).
This amounts to replacing the nonlinear matrix-valued function $\F (\mathbf{x})$ on the righthand side with the linear matrix-valued function $\J_0 \mathbf{x}$ given as the multiplication by the Jacobian $\J_0$ of $\F$ at the stationary solution $\mathbf{x}_0$.
Then the Hartman-Grobman Theorem asserts that, locally around $\mathbf{x}_0$, the nonlinear differential system $\dot{\mathbf{x}} = \F (\mathbf{x})$ behaves exactly like its linearisation $\dot{\mathbf{x}} = \J_0 \mathbf{x}$.
Meanwhile, the behaviour of the linearisation is completely governed by the eigenvalues of the Jacobian $\J_0$.
Peculiarly, our analysis of the deformed Painlevé I equation \eqref{250204124541} has a similar spirit.

\subsubsection{The classical Jacobian.}
Instead of the deformed Painlevé I equation \eqref{250204124541}, we work with the equivalent Hamiltonian system:
\begin{eqntag}
\label{241030114918}
	\hbar \del_t \mtx{ q \\ p }
	= \F (q,p,t,\hbar) = \mtx{ p \\ 6q^2 + t }
\fullstop
\end{eqntag}
The Jacobian of this system is by definition the matrix
\begin{eqn}
	\J (q,p)
		\coleq \frac{\del \F}{\del (q,p)}
		= \mtx{ \frac{\del \F_1}{\del q} & \frac{\del \F_1}{\del p}
			\\ \frac{\del \F_2}{\del q} & \frac{\del \F_2}{\del p}}
	 	= \mtx{ 0 & 1 \\ 12 q & 0}
\fullstop
\end{eqn}
If $\hat{q}$ is any formal solution lent to us by \autoref{250204162846}, then we may evaluate the Jacobian $\J$ at the leading-order solution $(q_0, p_0) = (q_0, 0)$, which results in the holomorphic matrix
\begin{eqn}
	\J_0 (t) \coleq \J \big|_{(q,p) = ( q_0 (t), 0 )} = \mtx{ 0 & 1 \\ 12 q_0 (t) & 0 }
\fullstop
\end{eqn}
We refer to it as the \dfn{classical Jacobian} of Painlevé I.
The classical Jacobian is a multivalued matrix function on the $t$-plane which can be viewed on the $\tau$-plane as a well-defined entire matrix-valued function that we denote by the same letter:
\begin{eqntag}
\label{250206213044}
	\J_0 (\tau) \coleq \mtx{ 0 & 1 \\ 12 \tau & 0 }
\fullstop
\end{eqntag}
Away from the origin, $\J_0$ has two nonvanishing holomorphic, but multivalued, eigenvalues $\pm \sqrt{12 \tau}$.
To remove the square-root ambiguity, we pass to another double cover
\begin{eqntag}
\label{250206214139}
	\pi : \CC_z \to \CC_\tau
\qtext{given by the relation}
	z^2 = 12 \tau
\fullstop
\end{eqntag}
Note that $\CC_z$ is also a fourfold cover of the $t$-plane:
\begin{eqntag}
\label{250208173123}
	\CC_z \to \CC_\tau
\qtext{given by the relation}
	z^4 = - 24 t
\fullstop
\end{eqntag}

\begin{definition}[{cf. \cite[Definition 3]{MY220915172017}}]
\label{250208105949}
We refer to the ramification point at $z = 0$ as the \dfn{transition point}, and to the branch points at $\tau = 0$ or $t = 0$ as the \dfn{turning points}.
\end{definition}

The turning point $t = 0$ or $\tau = 0$ is the only point where the classical Jacobian $\J_0$ is not invertible.
The two multivalued eigenvalues of $\J_0$ are just the two branches of the same globally defined holomorphic function $z$ on this double cover $\CC_z$.

An obvious advantage of the fourfold cover $\CC_z$ is that by regarding the leading-order Jacobian $\J_0$ as defined on the $z$-plane,
\begin{eqntag}
	\J_0 (z) \coleq \mtx{ 0 & 1 \\ z^2 & 0 }
\fullstop{,}
\end{eqntag}
it is globally diagonalisable away from the transition point.
Take the holomorphic matrix
\begin{eqnstag}
\label{250208170723}
	\G (z) \coleq \mtx{ 1 & 1 \\ -z & z}
\qtext{with}
	\G (z)^\inv = \frac{1}{2z} \mtx{ z & -1 \\ z & + 1}
\qtext{so}
	\G^\inv \J_0 \G = \mtx{ -z & 0 \\ 0 & + z }
\fullstop
\end{eqnstag}
In the next subsection, we will use $\G$ as a gauge transformation, so let us record here the resulting gauge term:
\begin{eqn}
	\G^\inv \del_z \G = \frac{1}{2z} \mtx{ +1 & -1 \\ -1 & +1}
\fullstop	
\end{eqn}

\subsubsection{A preliminary transformation.}
\label{250219104611}
Now we make a change of variables
\begin{eqntag}
\label{250215121808}
	(q,p) \leftrightarrow (\Q,\P)
\qqtext{given by}
	q = q_0 + \hbar \Q
\qtext{and}
	p = \hbar \P
\fullstop
\end{eqntag}
Recalling that $p_0 = 0$ and $p_1 = \del_t q_0$, the Hamiltonian system \eqref{241030114918} gets transformed into
\begin{eqntag}
\label{241029141541}
	\hbar \del_t
	\mtx{ \Q \\ \P }
	=
	\mtx{ \P - p_1 \\ 12 q_0 \Q + 6 \hbar \Q^2 }
	=
	\S (t,\hbar; \Q, \P)
	\coleq
	\mtx{ 0 & 1 \\ 12q_0 & 0 }
	\mtx{ \Q \\ \P }
	+
	\mtx{ - p_1 \\ 6 \hbar \Q^2 }
\fullstop
\end{eqntag}
Notice that $\S (t,0;0,0) = 0$ and that
\begin{eqntag}
	\frac{\del \S (t,0; \Q, \P)}{\del (\Q,\P)} \bigg|_{(\Q,\P) = (0,0)}
		= \J_0 (t)
		= \mtx{ 0 & 1 \\ 12q_0 (t) & 0 }
\fullstop{,}
\end{eqntag}
where $\J_0$ is the classical Jacobian.
As explained in \autoref{250208144823}, $\J_0$ admits a global diagonalisation $\G^\inv \J_0 \G = \diag (-z, + z)$ when regarded on the fourfold cover $\CC_z \to \CC_t$ given by $z^4 = - 24 t$.
We therefore make two more changes of variables: first, we do the change
\begin{eqntag}
\label{250215121526}
	(t, \Q, \P) \leftrightarrow (z, \tilde{\Q}, \tilde{\P})
\qtext{given by}
	\tilde{\Q} (z,\hbar) \coleq \Q (t,\hbar)
\qtext{and} \tilde{\P} (z, \hbar) \coleq \P (t, \hbar)
\fullstop{,}
\end{eqntag}
in the sense that $\tilde{\Q} (z,\hbar) = \Q \big( -z^4 / 24, \hbar)$, etc; and then we do the change
\begin{eqntag}
\label{250215121531}
	(z, \tilde{\Q}, \tilde{\P}) \leftrightarrow (z, f_-, f_+)
	\qtext{given by}
	\mtx{ \tilde{\Q} \\ \tilde{\P} }
		= \G \mtx{ f_- \\ f_+ }
		= \mtx{ f_+ + f_- \\ z (f_+ - f_-) }
\fullstop
\end{eqntag}
Under these changes, --- and after defining $\tilde{p}_1 (z) \coleq p_1 (t)$, using the relation $\del_t = - 6 z^{-3} \del_z$, and multiplying by $\G^\inv$ on the left --- the differential system \eqref{241029141541} is transformed into the following system:
\begin{eqn}
	- \frac{6\hbar}{z^3} \del_z \mtx{ f_- \\ f_+}
	- \mtx{-z & 0 \\ 0 & +z} \mtx{f_- \\ f_+}
	=	- \frac{\tilde{p}_1}{2} \mtx{1 \\ 1}
		- \mtx{-1 \\ +1} \left( \frac{3 \hbar}{z^4}  (f_+ - f_-)
		- \frac{3\hbar}{z} (f_+ - f_-)^2 \right)
\fullstop
\end{eqn}
Then we multiply on the left by the diagonal matrix $\diag (-z, +z)^\inv$ and swap the rows to obtain the following differential system:
\begin{eqntag}
\label{250208174617}
\begin{cases}
\displaystyle
	- 6\hbar z^{-4} \del_z f_- - f_-
		= 3\hbar z^{-2} (f_+ - f_-)^2 - 3\hbar z^{-5} (f_+ - f_-) - \tfrac{1}{2} z^{-1} \tilde{p}_1
\fullstop{,}
\\[5pt]\displaystyle
	+ 6\hbar z^{-4} \del_z f_+ - f_+
		= 3\hbar z^{-2} (f_+ - f_-)^2 - 3\hbar z^{-5} (f_+ - f_-) + \tfrac{1}{2} z^{-1} \tilde{p}_1
\fullstop
\end{cases}
\end{eqntag}
Finally, let us rewrite this differential system in a more standardised form.
Introduce the coefficients $a (z) \coleq +3z^{-2}$ and $b (z) \coleq -3z^{-5}$ as well as the meromorphic vector field $\V \coleq - 6 z^{-4} \del_z = z^{-1} \del_t$.
Then the differential system \eqref{250208174617} becomes
\begin{eqntag}
\label{250216174108}
\begin{cases}
\displaystyle
	+ \hbar \V f_+ - f_+
		= \hbar a (z) (f_+ - f_-)^2 + \hbar b (z) (f_+ - f_-) - \tfrac{1}{2z} \tilde{p}_1
\fullstop{,}
\\[5pt]\displaystyle
	- \hbar \V f_- - f_-
		= \hbar a (z) (f_+ - f_-)^2 + \hbar b (z) (f_+ - f_-) + \tfrac{1}{2z} \tilde{p}_1
\fullstop
\end{cases}
\end{eqntag}
We summarise these transformations as follows.

\begin{lemma}
\label{250215120832}
Upon a choice of a root $z$ of $z^4 = -24t$, the deformed Painlevé I Hamiltonian system \eqref{241030114918} is equivalent to the differential system \eqref{250216174108} via the sequence of transformations
\begin{eqn}
	(t; q, p) \leftrightarrow (t; \Q, \P) \leftrightarrow (z; \tilde{\Q}, \tilde{\P}) \leftrightarrow (z; f_+, f_-)
\end{eqn}
given by \eqref{250215121808}, \eqref{250215121526}, and \eqref{250215121531}.
\end{lemma}

These transformations can be summarised more explicitly as to the following relations:
\begin{eqn}
	\mtx{ q \\ p }
	= \mtx{ q_0 + \hbar \big( f_+ + f_- \big)
		\\ q_0 + z \hbar \big( f_+ - f_- \big)}
\qtext{and}
	\mtx{ f_+ \\ f_- }
	= \frac{1}{2z\hbar} \mtx{ z (q - q_0) + p \\ z (q - q_0) - p }
\fullstop
\end{eqn}

By \autoref{250204162846}, $(\hat{q}, \hat{p})$ is the unique formal $\hbar$-power series solution of the Hamiltonian system \eqref{241030114918} defined on the punctured $\tau$-plane.
By \autoref{250215120832}, the differential system \eqref{250216174108} has a unique formal $\hbar$-power series solution $(\hat{f}_+, \hat{f}_-)$,
\begin{eqn}
	\hat{f}_\pm (z, \hbar) \coleq \sum_{n=0}^\infty f_n^\pm (z) \hbar^n
\fullstop{,}
\end{eqn}
whose coefficients are explicitly expressible in terms of those of $\hat{q}$ and $\hat{p}$ as follows:
\begin{eqntag}
\label{250216160341}
\begin{aligned}
	f_0^\pm (z) &\coleq \pm \tfrac{1}{2z} p_1 (t) = \pm 3 z^{-3}
\fullstop{,}
\\
	f_n^\pm (z) &\coleq \tfrac{1}{2z} \big( z q_{n+1} (t) \pm p_{n+1} (t) \big)
\qquad
	\forall\, n \geq 1
\fullstop
\end{aligned}
\end{eqntag}
In particular, let us record here the relationship between the formal power series $\hat{q}$ and $\hat{f}$.

\begin{lemma}
\label{250224110231}
The formal solution $\hat{q} (t, \hbar)$ of the deformed Painlevé I equation \eqref{250204124541} can be expressed in terms of the formal solution $\hat{f} (z, \hbar)$ of the differential system \eqref{250216174108} as follows:
\begin{eqntag}
\label{250218092253}
	\hat{q} (t,\hbar) = q_0 (t,\hbar) + \hbar \big( \hat{f}_+ (z,\hbar) + \hat{f}_- (z,\hbar) \big)
\qtext{where}
	z^4 = -24 t
\fullstop
\end{eqntag}
Accordingly, the power series $\hat{f}_+, \hat{f}_-$ have the following symmetry under negation $z \mapsto -z$:
\begin{eqn}
	\hat{f}_\pm (-z,\hbar) = \hat{f}_\mp (z, \hbar)
\fullstop
\end{eqn}
\end{lemma}

\subsection{The Borel-Laplace Transformation}
\label{250221201337}

Next, we apply the \textit{Borel transform} to the differential system \eqref{250216174108}.
Using standard properties of the Borel transform, we find that it converts \eqref{250216174108} into the following system of equations:
\begin{eqntag}
\label{250216175521}
\begin{cases}
\displaystyle
	(+ \V - \del_\xi) \phi_+
		= a (z) (\phi_+ - \phi_-)^{\ast 2} + b (z) (\phi_+ - \phi_-)
\fullstop{,}
\\[5pt]\displaystyle
	(- \V - \del_\xi) \phi_-
		= a (z) (\phi_+ - \phi_-)^{\ast 2} + b (z) (\phi_+ - \phi_-)
\fullstop
\end{cases}
\end{eqntag}
This is a system of \textit{integro}-differential equations because of the presence of the convolution product in the variable $\xi$ such as
\begin{eqn}
	(\phi_+ \ast \phi_-) (z,\xi) \coleq \int_0^\xi \phi_+ (z,s) \phi_- (z, \xi-s) \d{s}
\fullstop
\end{eqn}
The new unknown variables $(\phi_+, \phi_-)$ are related to $(f_+, f_-)$ via the Borel-Laplace transformation with some pre-selected phase $\alpha \in \RR / 2\pi\ZZ$:
\begin{eqntag}
\label{250208181527}
	\phi_\pm = \Borel_\alpha \big[ f_\pm \big]
\qtext{and}
	f_\pm = f_0^\pm + \Laplace_\alpha \big[ \phi_\pm \big]
\fullstop
\end{eqntag}
For this relationship to make sense, solutions $\phi_\pm$ and $f_\pm$ must respectively admit well-defined Laplace and Borel transforms.
In effect, this means that the Borel-Laplace transformation \eqref{250208181527} is an equivalence not between the systems but between suitable boundary value problems.
To be more precise, we summarise this as follows.

\begin{lemma}
\label{250215122306}
Let $\alpha \in \RR / 2\pi\ZZ$ be any phase.
Consider the differential system \eqref{250216174108} with the following asymptotic boundary conditions:
\begin{eqntag}
\label{250215123234}
	f_\pm (z,\hbar) \simeq \hat{f}_\pm (z,\hbar)
\quad
	\text{as $\hbar \to 0$ unif. along $\sfop{Arc} (\alpha)$}
\fullstop{,}
\end{eqntag}
locally uniformly in $z$.
Then this asymptotic boundary value problem is equivalent to finding solutions $\hat{\phi} = (\hat{\phi}_+, \hat{\phi}_-)$ of the Initial Value Problem \eqref{250216175521} which are Laplace-transformable in the direction $\alpha$.
Thus, such solutions of \eqref{250216175521} are related to solutions of the asymptotic boundary value problem \eqref{250216174108} and \eqref{250215123234} via the Borel-Laplace transformation \eqref{250208181527}.
\end{lemma}

Under the Borel-Laplace transformation of \autoref{250215122306}, the unique $\hbar$-formal power series solution $\hat{f} = (\hat{f}_+, \hat{f}_-)$ of the differential system \eqref{250216174108} gets transformed into a unique $\xi$-power series solution $\hat{\phi} = (\hat{\phi}_+, \hat{\phi}_-)$ defined as the formal Borel transform
\begin{eqntag}
\label{250216160456}
	\hat{\phi}_\pm (z,\xi) = \sum_{n=0}^\infty \phi^\pm_n (z) \xi^n
	\coleq \Borel \big[ \, \hat{f}_\pm \, \big] (z, \xi)
	= \sum_{n=0}^\infty \tfrac{1}{n!} f^\pm_{n+1} (z) \xi^n
\fullstop
\end{eqntag}
Explicitly, using \eqref{250216160341},
\begin{eqn}
	\phi^\pm_n (z) = \tfrac{1}{2z n!} \big( z q_{n+2} (t) \pm p_{n+2} (t) \big)
\fullstop
\end{eqn}
In particular, using \eqref{250218092253}, we record here the relationship between the Borel transforms $\hat{\omega}$ and $\hat{\phi}$.

\begin{lemma}
\label{250224110855}
The Borel transform $\hat{\omega} (t,\xi)$ of the formal power series solution $\hat{q} (t,\hbar)$ of the deformed Painlevé I equation can be expressed in terms of the Borel transform $\hat{\phi} (z, \xi)$ of the formal power series solution $\hat{f} (z, \hbar)$ of the differential system \eqref{250216174108} as follows:
\begin{eqntag}
\label{250223140730}
	\hat{\omega} (t,\xi) = f_0^+ (z) + f_0^- (z) + \int_0^\xi \big( \hat{\phi}_+ (z,s) + \hat{\phi}_- (z,s) \big) \d{s}
\qtext{where}
	z^4 = -24 t
\fullstop
\end{eqntag}
Accordingly, the components $\hat{\phi}_+, \hat{\phi}_-$ as well as $f_0^+, f_0^-$ have the following symmetry under the negation map $z \mapsto -z$:
\begin{eqn}
	\hat{\phi}_\pm (-z,\xi) = \hat{\phi}_\mp (z, \xi)
\qtext{and}
	f_0^\pm (-z) = f_0^\mp (z)
\fullstop
\end{eqn}
\end{lemma}

It follows that the vector-valued power series $\hat{f} = (\hat{f}_+, \hat{f}_-)$ is also of factorial type locally uniformly for all nonzero $z$, and therefore its Borel transform $\hat{\phi} = (\hat{\phi}_+, \hat{\phi}_-)$ is a convergent power series in $\xi$, locally uniformly for all nonzero $z$.
We conclude that $\hat{\phi}$ is the unique $\xi$-power series solution --- and hence the unique local holomorphic solution near $\xi = 0$ --- of the integro-differential system \eqref{250216175521} satisfying the initial condition
\begin{eqntag}
\label{250220083244}
	\hat{\phi}_\pm (z, 0) 
		= \phi^\pm_0 (z) 
		= f^\pm_1 (z)
		= \tfrac{1}{2} q_2 (t)
		= \tfrac{1}{2} q_2 \big( -z^4 / 24 \big)
		= -6 z^{-8}
		\eqcol c(z)
\fullstop
\end{eqntag}

We summarise the content of this section in the following proposition.

\begin{proposition}
\label{250221201747}
Consider the following Initial Value Problem on $\CC_z \times \CC_\xi$:
\begin{eqntag}
\label{250306084618}
\begin{cases}
\displaystyle
	(+ \V - \del_\xi) \phi_+
		= \Phi
\fullstop{,}
\\[5pt]\displaystyle
	(- \V - \del_\xi) \phi_-
		= \Phi
\fullstop{,}
\end{cases}
\qqtext{such that}
	\phi_\pm (z,0) = c(z)
\fullstop{,}
\end{eqntag}
where $\V = -6z^{-4} \del_z$, $c(z) = -6z^{-8}$, and $\Phi$ is the following expression in $\phi_+, \phi_-$ with coefficients $a(z) = 3z^{-2}$, $b(z) = -3z^{-5}$:
\begin{eqn}
	\Phi = \Phi (z, \xi; \phi_+, \phi_-)
	\coleq a (\phi_+ - \phi_-)^{\ast 2} + b (\phi_+ - \phi_-)
\fullstop
\end{eqn}
This Initial Value Problem has a unique local holomorphic solution $\hat{\phi} = (\hat{\phi}_+, \hat{\phi}_-)$ defined on an open neighbourhood $\Xi \subset \CC^\ast_z \times \CC_\xi$ of $\xi = 0$.
\end{proposition}

Our main goal now is to construct the analytic continuation of $\hat{\phi}$ in the variable $\xi$.
We will do so by building a global solution $\phi$ of the Initial Value Problem \eqref{250306084618} whose Taylor expansion at $\xi = 0$ is $\hat{\phi}$.

\section{Global Solutions to Initial Value Problems}
\label{250315131538}

In this section, we describe --- in as concrete terms as possible --- the basic geometric tools needed for the construction of global solutions to the Initial Value Problem \eqref{250306084618}.

\subsection{Fundamental Groupoids}

\subsubsection{Review of fundamental groupoids.}
Recall that the \dfn{fundamental groupoid} $\Pi_1 (X)$ of a Riemann surface $X$ is the space of all continuous paths $\gamma : [0,1] \to X$ considered up to homotopy with fixed endpoints.
It is a two-dimensional holomorphic manifold equipped with two holomorphic surjective submersions
\begin{eqn}
	\rm{s}, \rm{t} : \Pi_1 (X) \to X
\end{eqn}
called the \dfn{source} and \dfn{target maps}, which respectively extract the source $\rm{s} (\gamma) \coleq \gamma (0)$ and the target $\rm{t} (\gamma) \coleq \gamma (1)$.
For any point $x \in X$, the fibres $\rm{s}^\inv (x)$ and $\rm{t}^\inv (x)$ are simply connected Riemann surfaces called the \dfn{source} and \dfn{target fibres} of $x$.
There is a closed embedding $X \inj \Pi_1 (X)$, called the \dfn{identity bisection}, given by viewing each point $x \in X$ as the constant path $1_x \in \Pi_1 (X)$; we denote by $1_X \subset \Pi_1 (X)$ the submanifold of constant paths.
It is called a \textit{bi}section because $X \inj \Pi_1 (X)$ is a section of both the source and the target map.
The fundamental groupoid has a \dfn{groupoid composition law} $\circ$ given by the usual concatenation of paths, which we shall always write from right to left; i.e., ``$\gamma_2 \circ \gamma_1$'' means first $\gamma_1$ then $\gamma_2$.
Not all paths are composable, though, but only those that have suitably matching source and target.
The space of such \dfn{composable paths} is the fibre product
\begin{eqn}
	\Pi_1 (X) \underset{\rm{s},\rm{t}}{\times} \Pi_1 (X)
		\coleq \set{ (\gamma_2, \gamma_1) \in \Pi_1 (X) \times \Pi_1 (X) 
			~\Big|~ \rm{s} (\gamma_2) = \rm{t} (\gamma_1) }
\fullstop
\end{eqn}
The groupoid law then becomes a holomorphic map
\begin{eqn}
	\Pi_1 (X) \underset{\rm{s},\rm{t}}{\times} \Pi_1 (X) \too \Pi_1 (X)
\qtext{sending}
	(\gamma_2, \gamma_1) \mapstoo \gamma_2 \circ \gamma_1
\fullstop
\end{eqn}
Altogether, these properties mean $\Pi_1 (X)$ has the structure of a \dfn{holomorphic Lie groupoid}, or \dfn{groupoid} for short.
This fact is often denoted by $\Pi_1 (X) \rightrightarrows X$.

\emph{Note.}
From now on, whenever we say ``path'' we usually implicitly mean ``path considered up to homotopy with fixed endpoints''.
This additional detail will be clear from the context, but if we ever need to emphasise the contrary, we may say ``concrete path''.
To be precise, by a \dfn{parameterised path} we mean any smooth map $\gamma : I \to X$ from a closed interval $I = [a,b] \subset \RR$ (or more generally an oriented complex straight line segment $I = [a,b] \subset \CC$).
A \textit{reparameterisation} of $\gamma : I \to X$ is any other parameterised path $\gamma' : I' \to X$ such that there is an orientation-preserving diffeomorphism $g : I \iso I'$ such that $\gamma' = \gamma \circ g$.
Then a \dfn{concrete path} is a parameterised path considered up to reparameterisation, and a \dfn{path} is a concrete path considered up to homotopy with fixed endpoints.

A key feature of the fundamental groupoid is that the identity bisection $1_X$ intersects all source and all target fibres  transversely.
Consequently, the tangent space $T_x X$ at any $x \in X$ is canonically isomorphic to the tangent space of the source fibre $\rm{s}^\inv (x)$ at the point $1_x \in 1_X$ and also to the tangent space of the target fibre $\rm{t}^\inv (x)$ at $1_x$.
Denoting by $\rm{s}_\ast$ and $\rm{t}_\ast$ the derivatives of $\rm{s}$ and $\rm{t}$, the latter two spaces are canonically isomorphic to $\ker (\rm{s}_\ast |_{1_x})$ and $\ker (\rm{t}_\ast |_{1_x})$, respectively.
Altogether, these vector space isomorphisms arrange themselves into the following natural isomorphisms of holomorphic line bundles on $X$:
\begin{eqntag}
\label{250219092627}
	\cal{ker} (\rm{t}_\ast) |_X \cong \cal{T}_X \cong \cal{ker} (\rm{s}_\ast) |_X
\fullstop
\end{eqntag}
Here, e.g., $\cal{ker} (\rm{t}_\ast) |_X$ is the holomorphic line bundle on $X$ whose fibre at any $x \in X$ is the vector space $\ker (\rm{t}_\ast |_{1_x})$.
It is the restriction of the holomorphic line bundle $\cal{ker} (\rm{t}_\ast)$ on $\Pi_1 (X)$ whose fibre at any $\gamma \in \Pi_1 (X)$ is the vector subspace $\ker (\rm{t}_\ast |_\gamma) \subset T_\gamma \Pi_1 (X)$ of the tangent space to $\Pi_1 (X)$ at $\gamma$ consisting of vectors tangent to the target fibre $\rm{t}^\inv (\gamma)$.

For any $x \in X$ and any $\gamma \in \rm{s}^\inv (x)$, the left multiplication by $\gamma$ yields the isomorphism $\ker (\rm{t}_\ast |_{1_x}) \cong \ker (\rm{t}_\ast |_{\gamma})$.
Similarly, for any $\gamma \in \rm{t}^\inv (x)$, the right multiplication by $\gamma$ yields the isomorphism $\ker (\rm{s}_\ast |_{1_x}) \cong \ker (\rm{s}_\ast |_{\gamma})$.
As a result, we have the following natural isomorphisms of holomorphic line bundles on $\Pi_1 (X)$:
\begin{eqntag}
\label{250219092834}
	\cal{ker} (\rm{t}_\ast) \cong \rm{s}^\ast \big( \cal{ker} (\rm{t}_\ast) |_X \big)
\qtext{and}
	\cal{ker} (\rm{s}_\ast) \cong \rm{t}^\ast \big( \cal{ker} (\rm{s}_\ast) |_X \big)
\fullstop
\end{eqntag}
Combining \eqref{250219092627} and \eqref{250219092834}, we conclude that every tangent vector field $\V \in \cal{T}_X$ has two canonical lifts to vector fields $\V^\rm{s}, \V^\rm{t} \in \cal{T}_{\Pi_1 (X)}$ on the groupoid $\Pi_1 (X)$, one of which is tangent to the target foliation and the other to the source foliation.
More precisely, $\V^\rm{s}$ is the unique vector field (called the \dfn{source-lift} of $\V$) such that $\V^\rm{s} \in \cal{ker} (\rm{t}_\ast)$ and $\rm{s}_\ast \V^\rm{s} = \V$; similarly, $\V^\rm{t}$ is the unique vector field (called the \dfn{target-lift} of $\V$) such that $\V^\rm{t} \in \cal{ker} (\rm{s}_\ast)$ and $\rm{t}_\ast \V^\rm{t} = \V$.

\subsubsection{The fundamental groupoid of the $z$-plane.}
Consider the fundamental groupoid $\Pi_1 (\CC_z)$ of the $z$-plane.
This groupoid is very simple to describe completely explicitly.
Since $\CC_z$ is simply connected, any path $\gamma$ with source $z_0$ and target $z_1 \in \CC_z$ is homotopic to the straight line interval $[z_0, z_1]$.
This sets up a holomorphic isomorphism
\begin{eqntag}
\label{250205161722}
	\Pi_1 (\CC_z) \iso \CC_u \times \CC_z
\qtext{by sending}
	\gamma \simeq [z_0, z_1] \mapsto (u,z) = (z_1 - z_0, z_0)
\fullstop
\end{eqntag}

In this trivialisation, the source, target, and groupoid law maps of $\Pi_1 (\CC_z)$ become
\begin{eqn}
	\rm{s} (u,z) = z,
\qquad
	\rm{t} (u,z) = z + u,
\qquad
	(u_2, z_2) \circ (u_1, z_1) = (u_1 + u_2, z_1)
\fullstop{,}
\end{eqn}
provided that $z_2 = z_1 + u_1$.
In fact, we can see that $\Pi_1 (\CC_z)$ is an algebraic manifold and all the groupoid structure maps are algebraic maps.
The source and target fibres of any point $z_0 \in \CC_z$ are respectively the subsets 
\begin{eqn}
	\rm{s}^\inv (z_0) = \set{ (u,z) ~\big|~ z = z_0}
\qtext{and}
	\rm{t}^\inv (z_0) = \set{ (u,z) ~\big|~ z + u = z_0}
\fullstop
\end{eqn}
Both of these are simply connected and isomorphic to $\CC$.
The identity bisection $1_{\CC_z}$ is the subset defined by $u = 0$.
Notice that the target map restricts to each source fibre to give an isomorphism $\rm{t} : \rm{s}^\inv (z_0) \iso \CC_z$.
Similarly, $\rm{s} : \rm{t}^\inv (z_0) \iso \CC_z$.
See \autoref{250226193156}.

The tangent space to any point in $\CC_z$ is spanned by the coordinate tangent vector field $\del_z$, and the tangent space to any point in $\Pi_1 (\CC_z)$ is spanned in the trivialisation by the tangent vectors $\del_u, \del_z$.
With respect to these bases, the derivatives $\rm{s}_\ast, \rm{t}_\ast$ of the source and target maps are respectively the $2 \!\times\! 1$ matrices $\big[ 0 ~~ 1 \big]$ and $\big[ 1 ~~ 1 \big]$.
Thus, the line bundles $\cal{ker} (\rm{t}_\ast)$ and $\cal{ker} (\rm{s}_\ast)$ are respectively spanned by the vector fields $\del_z - \del_u$ and $\del_u$.
If $\V = v(z) \del_z$ is a vector field on $\CC_z$, its source- and target-lifts expressed in the basis $\del_u, \del_z$ are respectively the vector fields $\V^\rm{s} = v(z) (\del_z - \del_u)$ and $\V^\rm{t} = v(z) \del_u$.

\begin{definition}[{cf. \cite[Definition 2.6]{240622121512}}]
\label{250208170930}
A path $\gamma \in \Pi_1 (\CC_z)$ is \dfn{critical} if it terminates at the transition point $z = 0$.
The subset of all critical paths and the subset of critical paths with source $z_0 \in \CC_z$ are denoted respectively by
\begin{eqntag}
	G \coleq \rm{t}^\inv (0) \subset \Pi_1 (\CC_z)
\qqtext{and}
	G_{z_0} \coleq \rm{t}^\inv (0) \cap \rm{s}^\inv (z_0)
\fullstop
\end{eqntag}
\end{definition}

Observe that $G_{z_0}$ contains just one element, the straight line segment from $z_0$ to $0$.
In the trivialisation $\Pi_1 (\CC_z) \cong \CC_u \times \CC_z$, these subsets are
\begin{eqn}
	G \cong \set{ (u,z) ~\big|~ z + u = 0 } \subset \CC_u \times \CC_z
\qqtext{and}
	G_{z_0} \cong \set{ (-z_0, z_0) } \subset \CC_u \times \CC_z
\fullstop
\end{eqn}

\begin{figure}[t]
\begin{adjustwidth}{-1cm}{-0.5cm}
\centering
\begin{subfigure}[t]{0.45\textwidth}
\centering
\begin{tikzpicture}
\begin{scope}
\fill [grey] (-3,-3) rectangle (3,3);
\node at (0,3) [above] {$\Pi_1 (\CC_z) \cong \CC_u \times \CC_z$};
\draw [ultra thick] (3,0) -- (-3,0) node [left] {$\CC_z$};
\draw [->] (-3.25,-3.25) -- (-2.5,-3.25) node [midway, below] {$z$};
\draw [->] (-3.25,-3.25) -- (-3.25,-2.5) node [midway, left] {$u$};
\foreach \i in {1,...,7}
{
	\draw [thin] ({-3 + 0.75*\i},-3) -- ({-3 + 0.75*\i},3);
}
\begin{scope}
\clip (-3,-3) rectangle (3,3);
\foreach \i in {1,...,7}
{
	\draw [thin] ({-6+0.75*\i},3) -- ({0+0.75*\i},-3);
}
\end{scope}
\end{scope}
\end{tikzpicture}
\caption{Real slice of $\Pi_1 (\CC_z)$ in the trivialisation.
Vertical lines are the source fibres, slanted lines are the target fibres.}
\label{250226193156}
\end{subfigure}
\hfill
\begin{subfigure}[t]{0.55\textwidth}
\centering
\begin{tikzpicture}
\begin{scope}
\fill [grey] (-4,-3) rectangle (4,3);
\node at (0,3) [above] {$\Pi_1 (\CC^\ast_z) \cong \CC_u \times \CC_z^\ast$};
\draw [ultra thick] (4,0) -- (-4,0) node [left] {$\CC^\ast_z$};
\draw [->] (-4.25,-3.25) -- (-3.5,-3.25) node [midway, below] {$z$};
\draw [->] (-4.25,-3.25) -- (-4.25,-2.5) node [midway, left] {$u$};
\foreach \i in {0,...,4}
{
	\draw [thin] ({-3.75 + 0.75*\i},-3) -- ({-3.75 + 0.75*\i},3);
}
\foreach \i in {6,...,10}
{
	\draw [thin] ({-3.75 + 0.75*\i},-3) -- ({-3.75 + 0.75*\i},3);
}
\draw [ultra thick, white] (0,-3) -- (0,3);
\begin{scope}
\clip (-4,-3) rectangle (4,3);
\foreach \i in {-5,...,-1}
{
\draw[smooth, domain=-3:3] 
    plot ({(0.75*\i)*e^(-\x)},\x);
}
\foreach \i in {1,...,5}
{
\draw[smooth, domain=-3:3] 
    plot ({(0.75*\i)*e^(-\x)},\x);
}
\end{scope}
\end{scope}
\end{tikzpicture}
\caption{Real slice of $\Pi_1 (\CC_z^\ast)$ in the trivialisation.
Vertical lines are the source fibres, curved lines are the target fibres.
Notice that the subset $\CC_u \times \set{0}$ is missing.}
\label{250226211137}
\end{subfigure}
\end{adjustwidth}
\caption{Fundamental groupoids of $\CC_z$ and $\CC^\ast_z$.}
\label{250226213320}
\end{figure}
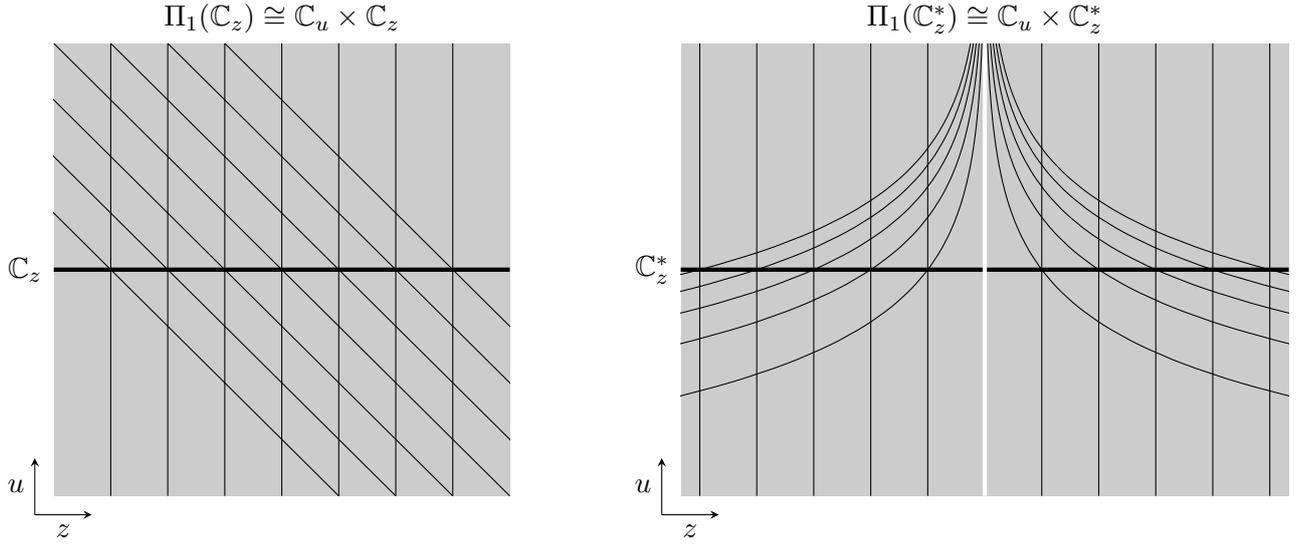

\subsubsection{The fundamental groupoid of the punctured $z$-plane.}
Now we puncture $\CC_z$ at the origin, $\CC^\ast_z \coleq \CC_z \smallsetminus \set{0}$, and consider the fundamental groupoid $\Pi_1 (\CC^\ast_z)$ whose source and target maps we are also denoted by $\rm{s}, \rm{t} : \Pi_1 (\CC^\ast_z) \to \CC^\ast_z$.
This groupoid can likewise be described rather explicitly.
The key observations are: first, that for any two points $z_0, z_1 \in \CC^\ast_z$, there is a complex number $u \in \CC$, unique up to addition of an integer multiple of $2\pi i$, such that $z_1 = e^u z_0$; and second, that any two paths from $z_0$ to $z_1$ are homotopic if and only if they wind the same number of times around the origin.
So if we denote by $\gamma_{u,z}$ the concrete path in $\CC_z^\ast$ from $z$ to $e^{u} z$ defined by $\gamma_{u,z} (r) \coleq e^{r u} z$ for $r \in [0,1]$ (see \autoref{240614104619}), then the groupoid $\Pi_1 (\CC^\ast_z)$ can be trivialised as follows:
\begin{eqntag}
\label{250205162515}
	\Pi_1 (\CC^\ast_z) \iso \CC_u \times \CC^\ast_z
\qtext{sending}
	\gamma_{u,z} \mapsto (u,z)
\fullstop
\end{eqntag}
To see that this map indeed defines an isomorphism, notice that although $e^{u}z_0$ and $e^{u + 2\pi i n} z_0$ define the same target point $z_1$, the parametrised curves $\gamma_{u,z}$ and $\gamma_{u + 2\pi i n, z}$ are certainly not the same and not homotopic (relative to their endpoints).
In fact, we can recognise that the integer $n$ is (essentially) the winding number around the origin; see \autoref{240614104619}.

In this trivialisation, the source, target, and groupoid law maps of $\Pi_1 (\CC^\ast_z)$ become
\begin{eqn}
	\rm{s} (u,z) = z,
\qquad
	\rm{t} (u,z) = e^u z,
\qquad
	(u_2, z_2) \circ (u_1, z_1) = (u_1 + u_2, z_1)
\fullstop{,}
\end{eqn}
provided $z_2 = e^{u_1} z_1$.
Observe that, unlike the groupoid $\Pi_1 (\CC_z)$, the groupoid $\Pi_1 (\CC^\ast_z)$ is not algebraic because its target map $\rm{t}$ is transcendental.
See \autoref{250226211137}.

\begin{figure}[t]
\begin{adjustwidth}{-5cm}{-5cm}
\centering
\begin{subfigure}{0.3\textwidth}
    \includegraphics[trim={10cm 5cm 10cm 5cm}, clip, width=\textwidth]{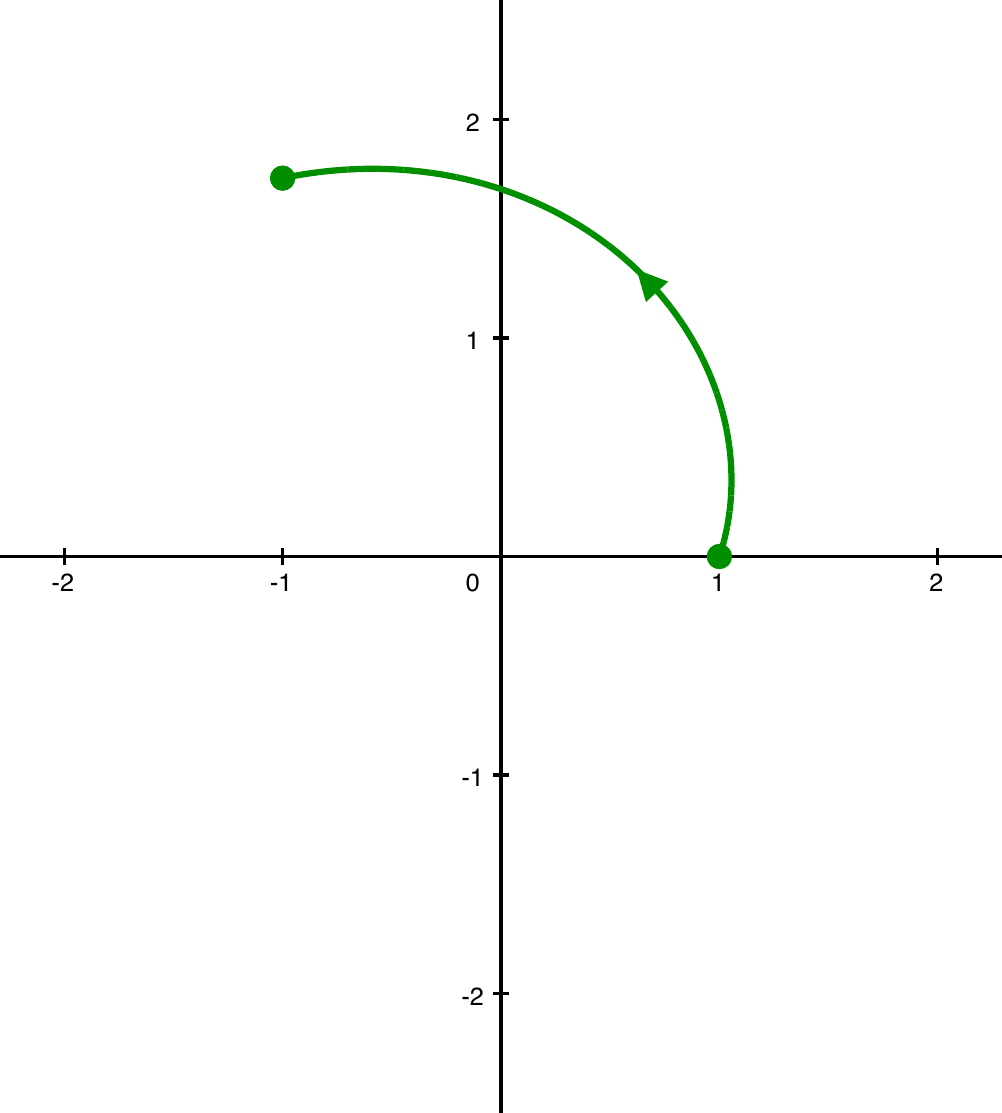}
	\caption{$n=0$}
	\label{240614101351}
\end{subfigure}
\begin{subfigure}{0.3\textwidth}
    \includegraphics[trim={10cm 5cm 10cm 5cm}, clip, width=\textwidth]{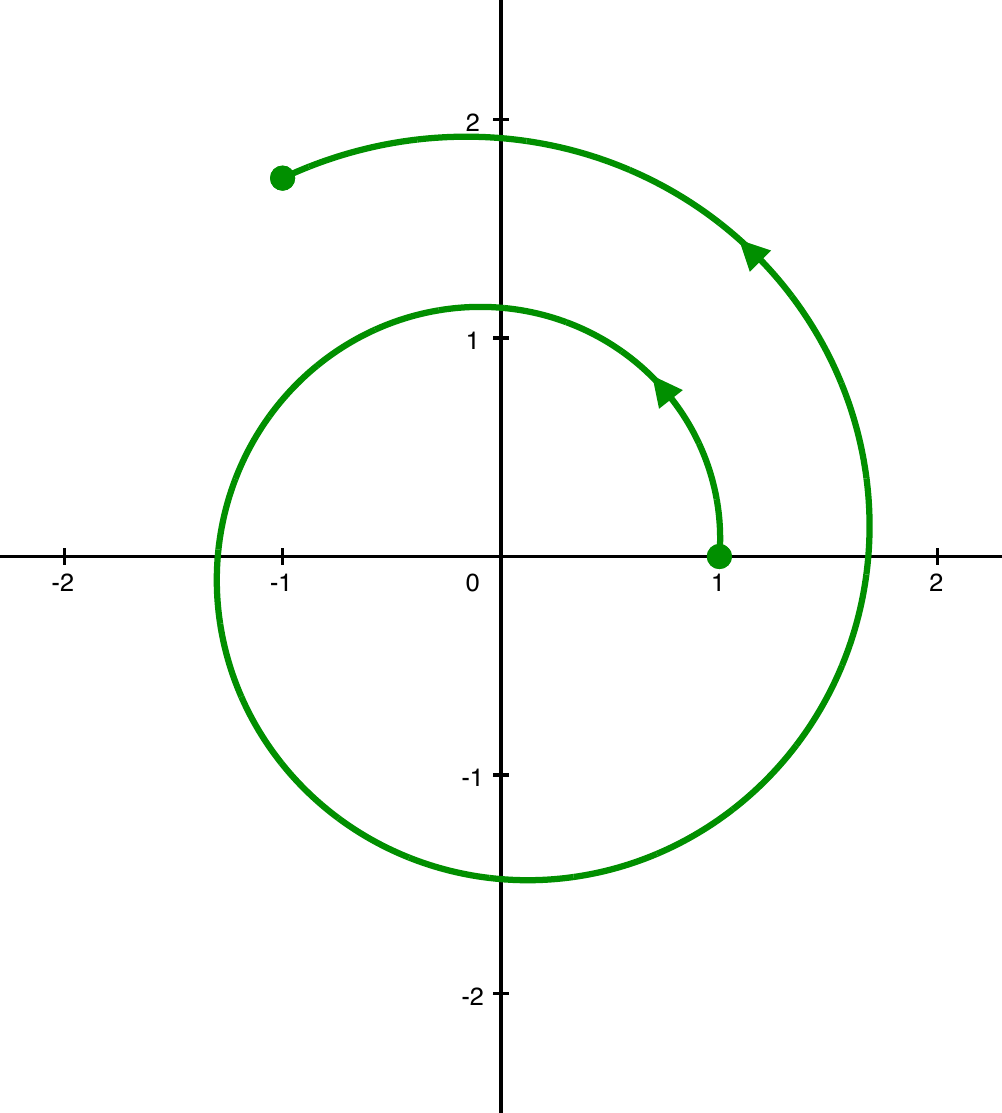}
	\caption{$n=1$}
	\label{240614101352}
\end{subfigure}
\begin{subfigure}{0.3\textwidth}
    \includegraphics[trim={10cm 5cm 10cm 5cm}, clip, width=\textwidth]{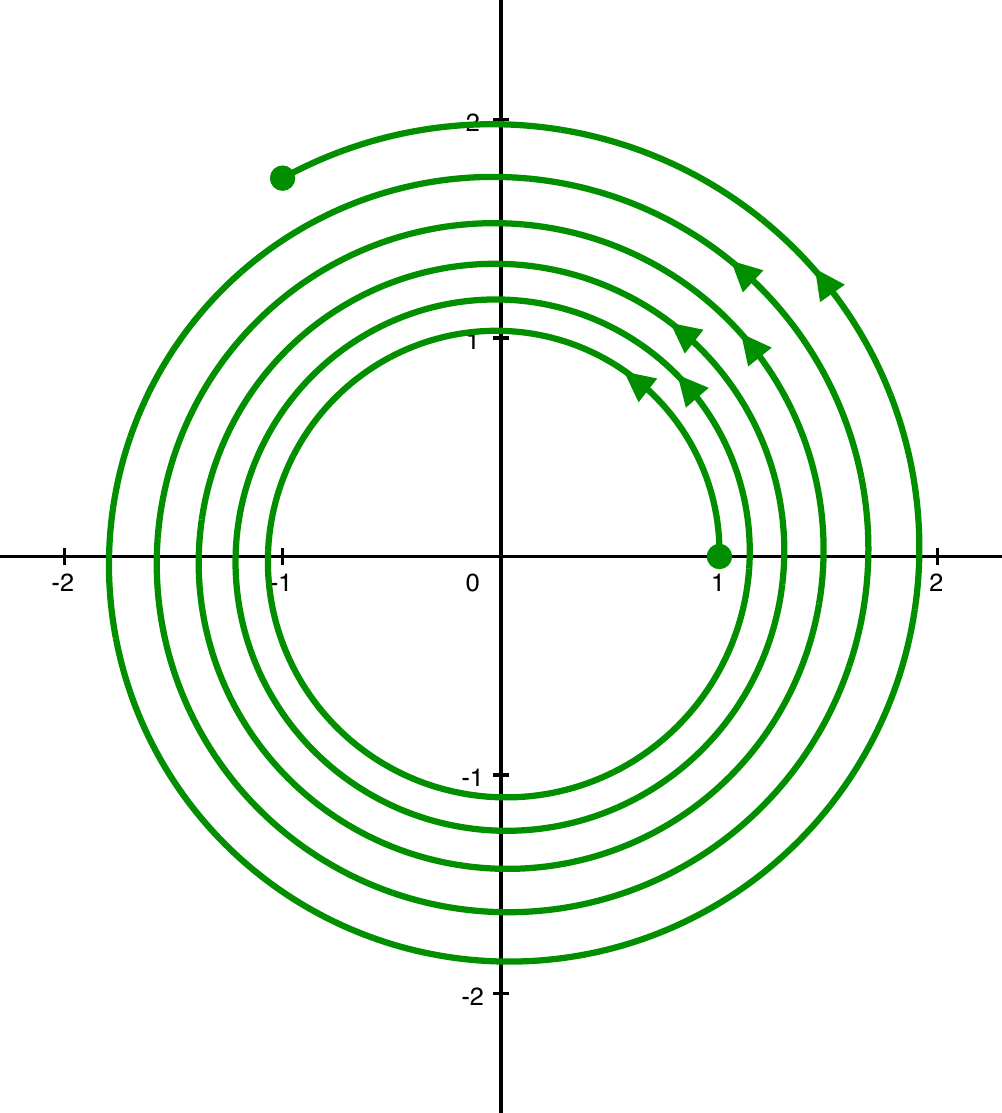}
	\caption{$n=5$}
	\label{240614101353}
\end{subfigure}
\begin{subfigure}{0.3\textwidth}
    \includegraphics[trim={10cm 5cm 10cm 5cm}, clip, width=\textwidth]{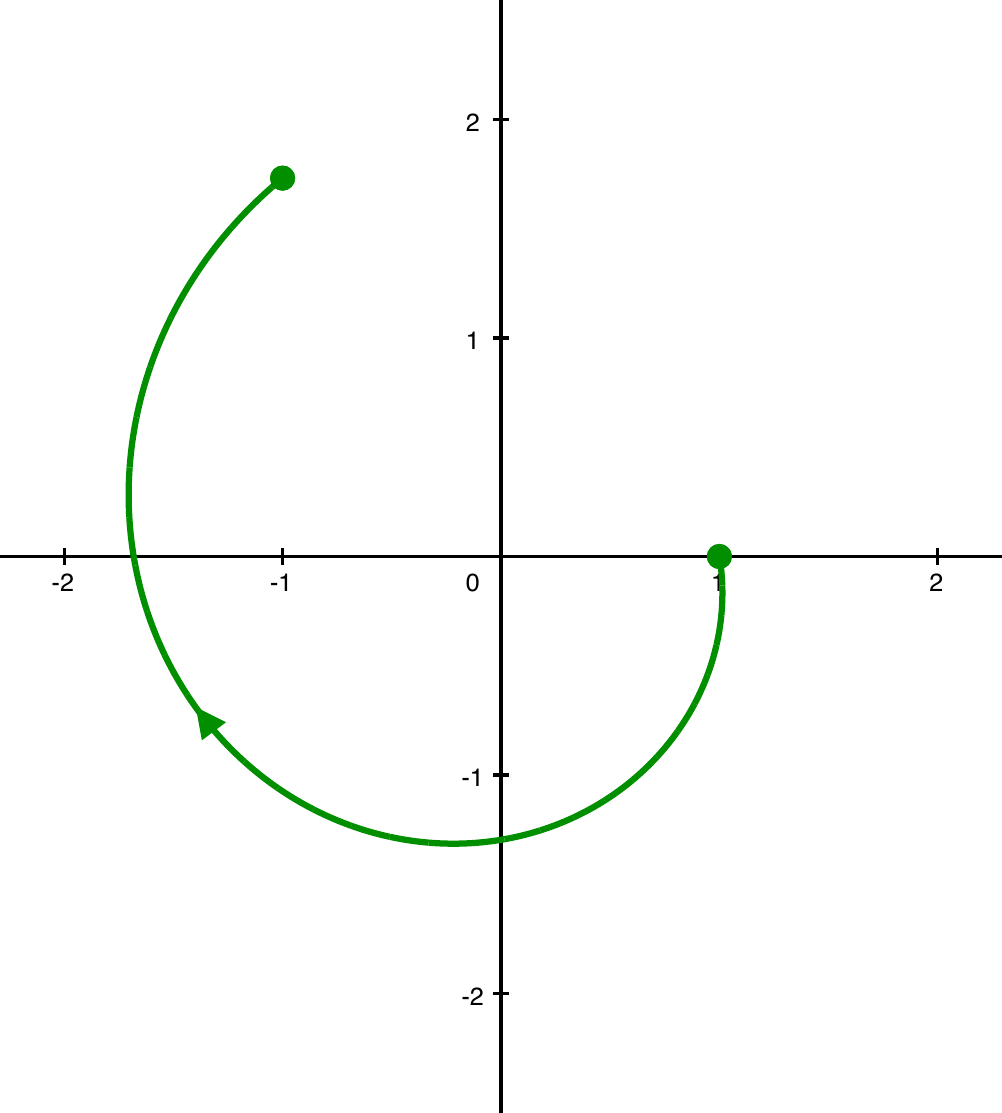}
	\caption{$n=-1$}
	\label{240614101354}
\end{subfigure}
\end{adjustwidth}
\caption{Plots of the parameterised curve $\gamma_{u,z} (r) = e^{ru} z$ with $r \in [0,1]$ where $u = \log(2) + 2\pi i /3 + 2 \pi i n$ for the displayed value of $n$, representing distinct homotopy classes of paths with source at $z = 1$ and target at $z = 2e^{2\pi i/3}$.
Observe that the integer $n$ is nothing but the winding number of each curve around the origin.}
\label{240614104619}
\end{figure}

The source and target fibres in $\Pi_1 (\CC^\ast_z)$ of any point $z_0 \in \CC^\ast_z$ are respectively the subsets 
\begin{eqn}
	\rm{s}^\inv (z_0) \cong \set{ (u,z) ~\big|~ z = z_0}
\qtext{and}
	\rm{t}^\inv (z_0) \cong \set{ (u,z) ~\big|~ e^{u} z = z_0}
\fullstop
\end{eqn}
Once again, both of these subsets are simply connected and isomorphic to $\CC$.
The identity bisection $1_{\CC^\ast_z}$ is also given by $u = 0$.
However, the restriction of the target map to each source fibre is no longer an isomorphism.
Instead, it is the universal covering map of $\CC^\ast_z$ based at $z_0$:
\begin{eqn}
	\rm{t} : \rm{s}^\inv (z_0) \cong \CC_u \to \CC_z^\ast
\qtext{given by}
	u \mapsto z = e^u z_0
\fullstop
\end{eqn}
Similarly, the source map $\rm{s}$ restricts to each target fibre $\rm{t}^\inv (z_0)$ determining the universal covering map $\rm{s} : u \mapsto z = e^{-u} z_0$ of $\CC^\ast_z$ based at $z_0$.
More precisely, $\rm{t}^\inv (z_0)$ is canonically isomorphic to the standard universal cover based at $z_0$ by sending each path $\gamma \in \rm{t}^\inv (z_0)$ to its inverse $\gamma^\inv \in \rm{s}^\inv (z_0)$; in coordinates, this operation is the negation map $u \mapsto -u$.

Like for $\Pi_1 (\CC_z)$, the tangent space to any point $\Pi_1 (\CC^\ast_z)$ is also spanned by the tangent vectors $\del_u, \del_z$ in this trivialisation, but the expressions for the derivatives $\rm{s}_\ast, \rm{t}_\ast$ in this basis are more interesting: they are $\big[ 0 ~~ 1 \big]$ and $\big[ e^u z ~~~ e^u \big] = e^u \big[ z ~~ 1 \big]$, respectively.
Thus, the line bundles $\cal{ker} (\rm{t}_\ast)$ and $\cal{ker} (\rm{s}_\ast)$ are respectively spanned by the vector fields $z \del_z - \del_u$ and $\del_u$.
If $\V = v(z) \del_z$ is a vector field on $\CC^\ast_z$, its source- and target-lifts expressed in the basis $\del_u, \del_z$ are respectively the vector fields $\V^\rm{s} = z^\inv v(z) (z\del_z - \del_u)$ and $\V^\rm{t} = e^{-u} z^\inv v(z) \del_u$.

\emph{Note.}
Whenever we need to explicitly distinguish the source and target maps of $\Pi_1 (\CC^\ast_z)$ from those of $\Pi_1 (\CC_z)$ (as we do in \autoref{250307105659}), we will denote them by $\tilde{\rm{s}}, \tilde{\rm{t}} : \Pi_1 (\CC^\ast_z) \to \CC^\ast_z$.
Otherwise, the intended meaning of the symbols $\rm{s}$ and $\rm{t}$ will either be clear from the context, said in words, or the discussion applies equally to both situations.

\subsubsection{The source-fibrewise universal cover.}
\label{250307105659}
The inclusion $\CC^\ast_z \inj \CC_z$ induces a holomorphic map $\nu : \Pi_1 (\CC^\ast_z) \to \Pi_1 (\CC_z)$ which is neither injective nor surjective.
Surjectivity fails because $\Pi_1 (\CC^\ast_z)$ contains no paths that begin or end at the origin, and injectivity fails because a loop around the origin is contractible in $\CC_z$ but not in $\CC^\ast_z$.
However, if we \textit{restrict} the groupoid $\Pi_1 (\CC_z)$ to $\CC^\ast_z$ then we do get a holomorphic surjective map
\begin{eqntag}
\label{250209092625}
	\nu : \Pi_1 (\CC^\ast_z) \too \Pi_1 (\CC_z) \big|_{\CC^\ast_z} 
		\coleq \rm{s}^\inv (\CC^\ast_z) \cap \rm{t}^\inv (\CC^\ast_z)
\fullstop
\end{eqntag}
See \autoref{250403114158}.
The restriction of $\Pi_1 (\CC_z)$ to $\CC^\ast_z$ is a holomorphic groupoid $\Pi_1 (\CC_z) |_{\CC^\ast_z} \rightrightarrows \CC^\ast_z$ in its own right: it is the dense open subgroupoid of $\Pi_1 (\CC_z)$ obtained by removing the source and target fibres of the origin in the $z$-plane, $\rm{s}^\inv (0)$ and $\rm{t}^\inv (0)$.
The restricted groupoid's source fibre $\rm{s}^\inv (z_0) |_{\CC^\ast_z}$ of any nonzero $z_0 \in \CC_z$ is therefore equal to the unrestricted groupoid's source fibre $\rm{s}^\inv (z_0)$ punctured along its intersection with the target fibre $\rm{t}^\inv (0)$, which is the one-element subset $\rm{s}^\inv (z_0) \cap \rm{t}^\inv (0) = G_{z_0}$ consisting of the unique critical path in $\CC_z$ starting at $z_0$:
\begin{eqntag}
\label{250209104616}
	\rm{s}^\inv (z_0) |_{\CC^\ast_z} = \rm{s}^\inv (z_0) \smallsetminus G_{z_0}
\fullstop
\end{eqntag}
Let us use the notation $\tilde{\rm{s}}, \tilde{\rm{t}}$ for the source and target maps of $\Pi_1 (\CC^\ast_z)$ to distinguish them from those of $\Pi_1 (\CC_z)$.
Then the restriction of $\nu$ to each source fibre $\tilde{\rm{s}}^\inv (z_0)$ is a simply connected covering map, hence it is nothing but the universal covering map of the restricted groupoid's source fibre $\rm{s}^\inv (z_0) |_{\CC^\ast_z}$ based at the point $1_{z_0}$, or equivalently of the unrestricted groupoid's punctured source fibre based at the point $1_{z_0}$:
\begin{eqntag}
\label{250209092822}
	\tilde{\rm{s}}^\inv (z_0)
	= \widetilde{ \rm{s}^\inv (z_0) \smallsetminus G_{z_0} }
	\xrightarrow{~~\nu~~}
	\rm{s}^\inv (z_0) \smallsetminus G_{z_0}
	= \rm{s}^\inv (z_0) \big|_{\CC^\ast_z}
\fullstop
\end{eqntag}
Consequently, $\nu$ is an isomorphism from an open neighbourhood of $1_{z_0} \in \Pi_1 (\CC^\ast_z)$ to an open neighbourhood of $1_{z_0} \in \Pi_1 (\CC_z) |_{\CC^\ast_z}$.
We summarise these observations as follows.

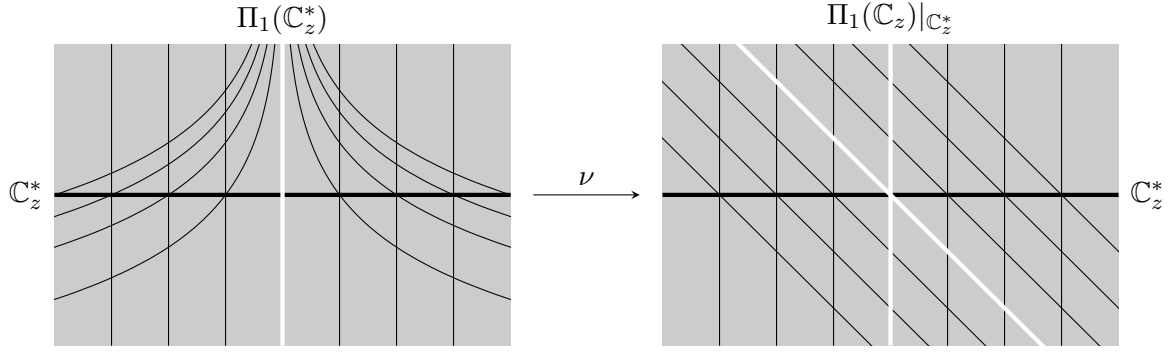
\begin{figure}[t]
\centering
\begin{tikzpicture}
\begin{scope}
\fill [grey] (-3,-2) rectangle (3,2);
\node at (0,2) [above] {$\Pi_1 (\CC^\ast_z)$};
\foreach \i in {2,...,4}
{
	\draw [thin] ({-3.75 + 0.75*\i},-2) -- ({-3.75 + 0.75*\i},2);
}
\foreach \i in {6,...,8}
{
	\draw [thin] ({-3.75 + 0.75*\i},-2) -- ({-3.75 + 0.75*\i},2);
}
\begin{scope}
\clip (-3,-2) rectangle (3,2);
\foreach \i in {-4,...,-1}
{
\draw[smooth, domain=-2:2] 
    plot ({(0.75*\i)*e^(-\x)},\x);
}
\foreach \i in {1,...,4}
{
\draw[smooth, domain=-2:2] 
    plot ({(0.75*\i)*e^(-\x)},\x);
}
\end{scope}
\draw [ultra thick] (3,0) -- (-3,0) node [left] {$\CC^\ast_z$};
\draw [ultra thick, white] (0,-2) -- (0,2);
\draw [->] (3.3,0) -- node [above] {$\nu$} (4.7,0);
\end{scope}

\begin{scope}[xshift=8cm]
\fill [grey] (-3,-2) rectangle (3,2);
\node at (0,2) [above] {$\Pi_1 (\CC_z) |_{\CC^\ast_z}$};
\draw [ultra thick] (-3,0) -- (3,0) node [right] {$\CC^\ast_z$};

\foreach \i in {1,...,7}
{
	\draw [thin] ({-3 + 0.75*\i},-2) -- ({-3 + 0.75*\i},2);
}
\begin{scope}
\clip (-3,-2) rectangle (3,2);
\foreach \i in {1,...,7}
{
	\draw [thin] ({-6+0.75*\i},3) -- ({0+0.75*\i},-3);
}
\end{scope}
\draw [ultra thick, white] (2.1,-2.1) -- (-2.1,2.1);
\draw [ultra thick, white] (0,-2) -- (0,2);
\end{scope}
\end{tikzpicture}
\caption{The source-fibrewise universal cover $\nu$.}
\label{250403114158}
\end{figure}

\begin{lemma}
\label{250212200132}
The map $\nu$ is the source-fibrewise universal covering map with basepoints along the identity bisection $1_{\CC^\ast_z}$, and it restricts to a holomorphic isomorphism in a neighbourhood of $1_{\CC^\ast_z}$.
That is, there is an open neighbourhood $\Omega \subset \Pi_1 (\CC_z) |_{\CC^\ast_z}$ of $1_{\CC^\ast_z}$ and an open neighbourhood $\tilde{\Omega} \subset \Pi_1 (\CC^\ast_z)$ of $1_{\CC^\ast_z}$ such that $\nu$ restricts to an isomorphism $\nu : \tilde{\Omega} \iso \Omega$.
\end{lemma}

Explicitly with respect to our trivialisations, the map $\nu : \Pi_1 (\CC^\ast_z) \to \Pi_1 (\CC_z)$ is
\begin{eqn}
	\nu : \CC_u \times \CC_z^\ast \to \CC_u \times \CC_z
	\qtext{given by}
		(u,z) \mapsto (e^u z - z, z)
\fullstop{,}
\end{eqn}
and the restriction of the groupoid $\Pi_1 (\CC_z)$ to $\CC^\ast_z$ is the dense open subset
\begin{eqn}
	\Pi_1 (\CC_z) \big|_{\CC^\ast_z} \cong \set{ (u,z) ~\big|~ z \neq 0 \qtext{and} z + u \neq 0 }
\fullstop
\end{eqn}
The restricted groupoid's source and target fibres of any nonzero $z_0$ are the subsets 
\begin{eqn}
	\rm{s}^\inv (z_0) |_{\CC^\ast_z} \cong \set{ (u,z_0) ~\big|~ z_0 + u \neq 0 }
\qtext{and}
	\rm{t}^\inv (z_0) |_{\CC^\ast_z} \cong \set{(u,z) ~\big|~ z + u = z_0 ~~\&~~ z \neq 0 }	
\end{eqn}
both of which are isomorphic to $\CC^\ast$.
Meanwhile, restricting to the source fibre $\tilde{\rm{s}}^\inv (z_0)$,
\begin{eqn}
	\nu : \tilde{\rm{s}}^\inv (z_0) \to \rm{s}^\inv (z_0) \big|_{\CC^\ast_z}
\qtext{is equivalent to}
	(u, z_0) \mapsto (e^u z_0 - z_0, z_0)
\fullstop
\end{eqn}
Notice that, as expected, the point $(-z_0,z_0)$ is not in the image of this map because $z_0 \neq 0$, whilst the distinguished point $(0,z_0) \in \tilde{\rm{s}}^\inv (z_0)$ representing the constant path $1_{z_0} \in \Pi_1 (\CC^\ast_z)$ gets sent to the base point $(0,z_0) \in \rm{s}^\inv (z_0)$ representing the constant path $1_{z_0} \in \Pi_1 (\CC_z) |_{\CC^\ast_z}$.

\subsubsection{Canonical retractions.}
\label{250306154548}
Pick any path $\gamma \in \Pi_1 (\CC^\ast_z)$ with source $z_0$ and target $z_1$.
Then $\gamma$ represents a point in the source fibre $\rm{s}^\inv (z_0)$ which is the universal cover of $\CC^\ast_z$ based at $z_0$.
Such a universal cover is a simply connected space with a distinguished point represented by the constant path $1_{z_0} \in \Pi_1 (\CC^\ast_z)$.
Therefore, up to homotopy with fixed endpoints, there is a unique path $\tilde{\gamma}^\rm{s}$ on the space $\rm{s}^\inv (z_0)$ which starts at the point $1_{z_0}$ and ends at the point $\gamma$.
Similarly, $\gamma$ represents a point in the target fibre $\rm{t}^\inv (z_1)$, which is likewise the universal cover of $\CC^\ast_z$ based at $z_1$, so there is a unique path $\tilde{\gamma}^\rm{t}$ on the space $\rm{t}^\inv (z_1)$ which starts at the point $\gamma$ and ends at the point $1_{z_1}$.

\begin{definition}[source/target-retractions]
\label{250225130210}
For any $\gamma \in \Pi_1 (\CC^\ast_z)$ with source $z_0$ and target $z_1$, we refer to the unique element $\tilde{\gamma}^\rm{s} \in \Pi_1 \big( \rm{s}^\inv (z_0) \big)$ with source $1_{z_0}$ and target $\gamma$ as the \dfn{source-retraction} of $\gamma$.
Similarly, we refer to the unique element $\tilde{\gamma}^\rm{t} \in \Pi_1 \big( \rm{t}^\inv (z_1) \big)$ with source $\gamma$ and target $1_{z_1}$ as the \dfn{target-retraction} of $\gamma$.
\end{definition}

Choosing a concrete path $[0,1] \to \rm{s}^\inv (z_0)$ representing $\tilde{\gamma}^\rm{s}$ is the same as choosing a one-parameter family of paths on $\CC^\ast_z$ interpolating between the constant path $1_{z_0}$ and the given path $\gamma$.
In other words, $\tilde{\gamma}^\rm{s}$ is represented by a homotopy through paths with source $z_0$ between the constant path $1_{z_0}$ and the given path $\gamma$; i.e., a continuous family $(\gamma_s, s \in [0,1])$ where each $\gamma_s$ is an element of $\Pi_1 (\CC^\ast_z)$ with source $z_0$, such that $\gamma_0 = 1_{z_0}$ and $\gamma_1 = \gamma$; see \autoref{250301112229}.
Similarly, choosing a concrete path $[0,1] \to \rm{t}^\inv (z_1)$ representing $\tilde{\gamma}^\rm{t}$ is the same as choosing a homotopy $(\bar{\gamma}_s, s \in [0,1])$ through paths with target $z_1$ between the constant path $1_{z_1}$ and the given path $\gamma$; see \autoref{250301110312}.

A homotopy $(\gamma_s)$ representing the source-retraction $\tilde{\gamma}^\rm{s}$ determines a concrete representative of the path $\gamma$ itself by assembling the target points of each $\gamma_s$; i.e., $\gamma : [0,1] \to \CC^\ast_z$ is given by $\gamma (s) = \rm{t} (\gamma_s)$.
Similarly, a homotopy $(\bar{\gamma}_s)$ representing the target-retraction $\tilde{\gamma}^\rm{t}$ determines a concrete representative of the path $\gamma$ by assembling the source points of each $\bar{\gamma}_s$; i.e., $\gamma : [0,1] \to \CC^\ast_z$ is given by $\gamma (s) = \rm{s} (\bar{\gamma}_s)$.

Conversely, given a concrete path $\gamma : [0,1] \to \CC^\ast_z$, it determines a homotopy $(\gamma_s)$ representing the source-retraction $\tilde{\gamma}^\rm{s}$ by assembling the truncations $\gamma_s \coleq \gamma |_{[0,s]} : [0,s] \subset [0,1] \to \CC^\ast_z$.
It also determines a homotopy $(\bar{\gamma}_s)$ representing the target-retraction $\tilde{\gamma}^\rm{t}$ by assembling the truncations $\bar{\gamma}_s \coleq \gamma |_{[1-s,1]} : [1-s,1] \subset [0,1] \to \CC^\ast_z$.

\begin{definition}[source/target-truncations]
\label{250306125556}
Given a concrete path ${\gamma : [0,1] \to \CC^\ast_z}$, we define its \dfn{source-truncation} and \dfn{target-truncation} at time $s \in [0,1]$ to be the restriction of $\gamma$ to the subintervals $[0,s]$ and $[1-s,1]$, respectively, which we denote as follows:
\begin{eqn}
	\gamma_s \coleq \gamma |_{[0,s]} : [0,s] \to \CC^\ast_z
\qqtext{and}
	\bar{\gamma}_s \coleq \gamma |_{[1-s,1]} : [1-s,1] \to \CC^\ast_z
\fullstop
\end{eqn}
\end{definition}

\begin{figure}[t]
\centering
\begin{subfigure}{0.45\textwidth}
\centering
\begin{tikzpicture}
\begin{scope}[xshift=-4.5cm]
\fill [darkgreen] (0,0) circle (2pt) node [below] {$z_0$};
\fill [darkgreen] (0,3) circle (2pt) node [above] {$z_1$};
\draw [darkgreen, thick, dashed] (0,0) to [out=150, in=225] (0.5,1) to [out=45, in=-60] (-0.5,2) to [out=120, in=250] (0,3);
\node [darkgreen] at (0,0) [left, xshift=-5, yshift=3] {$1_{z_0}$};
\end{scope}
\begin{scope}[xshift=-3cm]
\fill [darkgreen] (0,0) circle (2pt) node [below] {$z_0$};
\fill [darkgreen] (0,3) circle (2pt) node [above] {$z_1$};
\fill [darkgreen] (0.5,1) circle (2pt) node [below right, scale=0.75] {$z_s$};
\draw [darkgreen, thick, ->-=0.5] (0,0) to [out=150, in=225] (0.5,1);
\draw [darkgreen, thick, dashed] (0.5,1) to [out=45, in=-60] (-0.5,2) to [out=120, in=250] (0,3);
\node [darkgreen] at (0.9,1.5) {$\gamma_s$};
\end{scope}
\begin{scope}[xshift=-1.5cm]
\fill [darkgreen] (0,0) circle (2pt) node [below] {$z_0$};
\fill [darkgreen] (0,3) circle (2pt) node [above] {$z_1$};
\fill [darkgreen] (-0.5,2) circle (2pt) node [above right, scale=0.75] {$z_s$};
\draw [darkgreen, thick] (0,0) to [out=150, in=225] (0.5,1);
\draw [darkgreen, thick, ->-=0.01] (0.5,1) to [out=45, in=-60] (-0.5,2);
\draw [darkgreen, thick, dashed] (-0.5,2) to [out=120, in=250] (0,3);
\end{scope}
\begin{scope}
\fill [darkgreen] (0,0) circle (2pt) node [below] {$z_0$};
\fill [darkgreen] (0,3) circle (2pt) node [above] {$z_1$};
\draw [darkgreen, thick] (0,0) to [out=150, in=225] (0.5,1);
\draw [darkgreen, thick, ->-=0.5] (0.5,1) to [out=45, in=-60] (-0.5,2);
\draw [darkgreen, thick] (-0.5,2) to [out=120, in=250] (0,3);
\node [darkgreen] at (0,1.5) [above right] {$\gamma$};
\end{scope}
\end{tikzpicture}
\caption{Source-retraction of $\gamma$: a homotopy $(\gamma_s)$ through paths with source $z_0$ between the constant path $1_{z_0}$ and the given path $\gamma$.
Each $\gamma_s : [0,s] \to \CC^\ast_z$ is the source-truncation of $\gamma : [0,1] \to \CC^\ast_z$ starting at $z_0$ and terminating at $z_s$.
}
\label{250301112229}
\end{subfigure}
\hfill
\begin{subfigure}{0.45\textwidth}
\centering
\begin{tikzpicture}
\begin{scope}
\fill [darkgreen] (0,0) circle (2pt) node [below] {$z_0$};
\fill [darkgreen] (0,3) circle (2pt) node [above] {$z_1$};
\draw [darkgreen, thick] (0,0) to [out=150, in=225] (0.5,1);
\draw [darkgreen, thick, ->-=0.5] (0.5,1) to [out=45, in=-60] (-0.5,2);
\draw [darkgreen, thick] (-0.5,2) to [out=120, in=250] (0,3);
\node [darkgreen] at (0,1.5) [above right] {$\gamma$};
\end{scope}
\begin{scope}[xshift=-1.5cm]
\fill [darkgreen] (0,0) circle (2pt) node [below] {$z_0$};
\fill [darkgreen] (0,3) circle (2pt) node [above] {$z_1$};
\fill [darkgreen] (0.5,1) circle (2pt) node [below right, scale=0.75] {$z_{1-s}$};
\draw [darkgreen, thick, dashed] (0,0) to [out=150, in=225] (0.5,1);
\draw [darkgreen, thick, ->-=0.99] (0.5,1) to [out=45, in=-60] (-0.5,2);
\draw [darkgreen, thick] (-0.5,2) to [out=120, in=250] (0,3);
\node [darkgreen] at (0.9,1.5) {$\bar{\gamma}_s$};
\end{scope}
\begin{scope}[xshift=-3cm]
\fill [darkgreen] (0,0) circle (2pt) node [below] {$z_0$};
\fill [darkgreen] (0,3) circle (2pt) node [above] {$z_1$};
\fill [darkgreen] (-0.5,2) circle (2pt) node [xshift=1, right, scale=0.75] {$z_{1-s}$};
\draw [darkgreen, thick, dashed] (0,0) to [out=150, in=225] (0.5,1) to [out=45, in=-60] (-0.5,2);
\draw [darkgreen, thick, ->-=0.5] (-0.5,2) to [out=120, in=250] (0,3);
\end{scope}
\begin{scope}[xshift=-4.5cm]
\fill [darkgreen] (0,0) circle (2pt) node [below] {$z_0$};
\fill [darkgreen] (0,3) circle (2pt) node [above] {$z_1$};
\draw [darkgreen, thick, dashed] (0,0) to [out=150, in=225] (0.5,1) to [out=45, in=-60] (-0.5,2) to [out=120, in=250] (0,3);
\node [darkgreen] at (0,3) [left, xshift=-5] {$1_{z_1}$};
\end{scope}
\end{tikzpicture}
\caption{Target-retraction of $\gamma$: a homotopy $(\bar{\gamma}_s)$ through paths with target $z_1$ between the constant path $1_{z_1}$ and the given path $\gamma$.
Each $\bar{\gamma}_s : [0,s] \to \CC^\ast_z$ is the target-truncation of $\gamma : [0,1] \to \CC^\ast_z$ starting at $z_{1-s}$ and terminating at $z_1$.
}
\label{250301110312}
\end{subfigure}
\caption{Canonical retractions.}
\label{250301112427}
\end{figure}
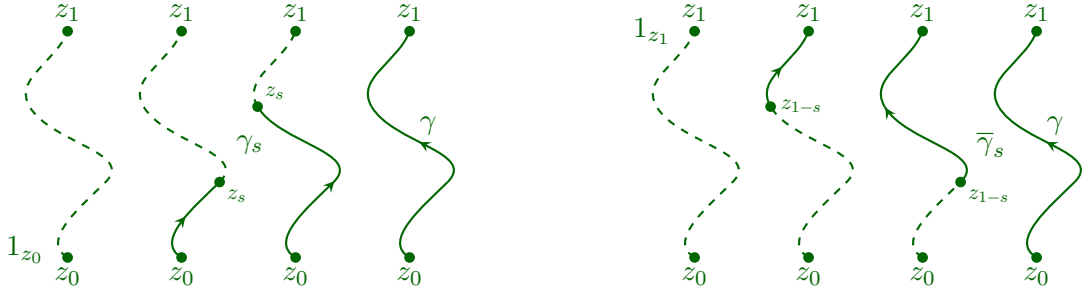

\subsection{The Central Charge}
\label{250315133506}

The $z$-plane carries a natural holomorphic differential form
\begin{eqn}
	\lambda \coleq z \d{t} = - \tfrac{1}{6} z^4 \d{z}
\end{eqn}
arising from its structure as the fourfold branched cover $z^4 = - 24 t$ of the $t$-plane.
Its significance to us is that it is dual to the meromorphic vector field $\V \coleq z^{-1} \del_t = - 6 z^{-4} \del_z$ appearing on the lefthand side of \eqref{250306084618} in the sense that
\begin{eqn}
	\lambda (\V) = 1
\fullstop
\end{eqn}
Integrating $\lambda$ along paths in the $z$-plane defines a global holomorphic function on the fundamental groupoid which will play a central role in our constructions.

\begin{definition}[central charge]
\label{250225170308}
For $X = \CC_z$ or $X = \CC^\ast_z$, the \dfn{central charge} on $\Pi_1 (X)$ associated with the holomorphic differential form $\lambda$ is the holomorphic function
\begin{eqn}
	\Z : \Pi_1 (X) \to \CC
\qtext{given by}
	\Z : \gamma \mapsto \xi = \Z (\gamma) = \int_\gamma \lambda
\fullstop
\end{eqn}
\end{definition}

The central charge is a \textit{groupoid $1$-cocycle} on $\Pi_1 (X)$ which means it respects the groupoid law: i.e., $\Z (\gamma_2 \circ \gamma_1) = \Z (\gamma_2) + \Z (\gamma_1)$ for all pairs of composable paths $\gamma_1, \gamma_2 \in \Pi_1 (X)$, and also $\Z (\gamma^\inv) = - \Z (\gamma)$ for all $\gamma \in \Pi_1 (X)$.
The two central charges $\Z : \Pi_1 (\CC_z) \to \CC$ and $\Z : \Pi_1 (\CC^\ast_z) \to \CC$ are intertwined by the map $\nu : \Pi_1 (\CC^\ast_z) \to \Pi_1 (\CC_z)$ in the sense that we have the following commutative diagram:
\begin{eqn}
\begin{tikzcd}[row sep = tiny, column sep = large]
		\Pi_1 (\CC^\ast_z)
			\ar[dd, "\nu"']
			\ar[dr, "\Z"]
\\
&		\CC	\fullstop
\\		\Pi_1 (\CC_z)
			\ar[ur, "\Z"']
\end{tikzcd}
\end{eqn}

Explicitly in our trivialisations, if $\gamma \in \Pi_1 (\CC_z)$ is any path from $z$ to $z + u$, then its central charge is the complex number
\begin{eqntag}
\label{250226060217}
	\xi = \Z (\gamma) 
		= \int_\gamma \lambda
		= \int_{z}^{z + u} \big( -\tfrac{1}{6} s^4 \big) \d{s}
		= \tfrac{1}{30} \big( z^5 - (z + u)^5 \big)
\fullstop
\end{eqntag}
Similarly, if $\gamma \in \Pi_1 (\CC^\ast_z)$ is any path from $z$ to $e^u z$, then its central charge is the complex number
\begin{eqntag}
\label{250226061238}
	\xi = \Z (\gamma)
		= \int_\gamma \lambda
		= \int_{z}^{e^u z} \big( -\tfrac{1}{6} s^4 \big) \d{s}
		= \tfrac{1}{30} ( 1 - e^{5u} ) z^5
\fullstop
\end{eqntag}

\begin{lemma}
\label{250214121905}
The central charge $\Z : \Pi_1 (\CC_z) \to \CC$ is a holomorphic submersion everywhere except at the point $1_0 \in \Pi (\CC_z)$; i.e., the constant path at the origin in the $z$-plane.
By comparison, the central charge $\Z : \Pi_1 (\CC^\ast_z) \to \CC$ is a holomorphic submersion everywhere.
\end{lemma}

\begin{proof}
The central charge $\Z : \Pi_1 (\CC_z) \to \CC$ fails to be a submersion at $\gamma \in \Pi_1 (\CC_z)$ if and only if its derivative $\d \Z_\gamma : T_\gamma \Pi_1 (\CC_z) \to T_{\Z (\gamma)} \CC$ is not surjective.
This happens if and only if $\lambda$ has a zero at both the source $\rm{s} (\gamma)$ and the target $\rm{t} (\gamma)$ because 
\begin{eqntag}
\label{250307121013}
	\d \Z_\gamma
		= (\rm{t}^\ast \lambda - \rm{s}^\ast \lambda)_\gamma 
		= \lambda_{\rm{t} (\gamma)} - \lambda_{\rm{s} (\gamma)}
\fullstop
\end{eqntag}
But any such $\gamma$ is contractible and therefore homotopic to the constant path $1_0$.
The submersiveness of $\Z : \Pi_1 (\CC^\ast_z) \to \CC$ follows by the same argument.
\end{proof}

Explicitly using formula \eqref{250226060217} for the central charge $\Z : \Pi_1 (\CC_z) \to \CC$, the $1 \!\times\! 2$ matrix of $\d{\Z}_\gamma$ with respect to the basis $(\del_u, \del_z)$ of $T_\gamma \Pi_1 (\CC_z)$ and $\del_\xi$ of $T_{\Z (\gamma)} \CC$ is 
\begin{eqntag}
\label{250307121002}
	\big[ \d \Z_\gamma \big]_{\del_\xi}^{(\del_u, \del_z)}
		= \Big[ \tfrac{\del{\Z}}{\del{u}} ~~~ \tfrac{\del{\Z}}{\del{z}} \Big] 
		= \Big[ - \tfrac{1}{6} (z + u)^4 ~~~~ \tfrac{1}{6} (z^4 - (z+u)^4) \Big]
\fullstop
\end{eqntag}
Clearly, this linear map fails to be surjective if and only if $(u,z) = (0,0)$.
Similarly using formula \eqref{250226061238} for the central charge $\Z : \Pi_1 (\CC^\ast_z) \to \CC$, the $1 \!\times\! 2$ matrix of $\d{\Z}_\gamma$ is 
\begin{eqntag}
\label{250307121006}
	\big[ \d \Z_\gamma \big]_{\del_\xi}^{(\del_u, \del_z)}
		= \Big[ \tfrac{\del{\Z}}{\del{u}} ~~~ \tfrac{\del{\Z}}{\del{z}} \Big] 
		= \Big[ - \tfrac{1}{6} e^{5u} z^5 ~~~~ \tfrac{1}{6} (1 - e^{5u}) z^5 \Big]
\fullstop
\end{eqntag}
This linear map is always surjective because $z$ is nonzero.

\begin{lemma}
\label{250214122849}
The restriction of the central charge $\Z : \Pi_1 (\CC_z) \to \CC$ to the source fibre $\rm{s}^\inv (z_0) \subset \Pi_1 (\CC_z)$ of any $z_0 \in \CC_z$ determines a fivefold ramified covering map
\begin{eqntag}
\label{250209113514}
	\Z_{z_0} \coleq \Z \big|_{\rm{s}^\inv (z_0)} : \rm{s}^\inv (z_0) \to \CC_\xi
\fullstop
\end{eqntag}
It has just one ramification point of order $5$ located at the point represented by the critical path $\gamma_0$ with source at $z_0$; i.e., at the only point of the critical locus $G_{z_0} = \set{\gamma_0} \subset \rm{s}^\inv (z_0)$.
The corresponding branch point is located at this critical path's central charge:
\begin{eqn}
	\xi_0 \coleq \Z (\gamma_0) = \tfrac{1}{30} z_0^5
\fullstop
\end{eqn}
Similarly, the restriction of $\Z : \Pi_1 (\CC^\ast_z) \to \CC$ to the source fibre $\tilde{\rm{s}}^\inv (z_0) \subset \Pi_1 (\CC^\ast_z)$ of any $z_0 \in \CC^\ast_z$ is the universal covering map based at the origin $0$ of the once-punctured complex $\xi$-plane $\CC_\xi \smallsetminus \set{\xi_0}$:
\begin{eqntag}
\label{250214123417}
	\Z_{z_0} \coleq \Z \big|_{\tilde{\rm{s}}^\inv (z_0)} : 
		\tilde{\rm{s}}^\inv (z_0) \cong \widetilde{\CC_\xi \smallsetminus \set{\xi_0}}
		\to \CC_\xi \smallsetminus \set{\xi_0}
\fullstop
\end{eqntag}
\end{lemma}

\begin{proof}
The map $\Z_{z_0}$ fails to be a submersion at $\gamma \in \rm{s}^\inv (z_0)$ if and only if the differential $\d (\Z_{z_0})_{\gamma} : T_\gamma \big( \rm{s}^\inv (z_0)\big) \to T_{\Z(\gamma)} \CC$ is not surjective.
Since $\d{\Z}_\gamma = (\rm{t}^\ast \lambda - \rm{s}^\ast \lambda)_{\gamma}$ and $\rm{s} (\gamma) = z_0$ is fixed, we find that 
\begin{eqn}
	\d{(\Z_{z_0})}_\gamma = \rm{t}^\ast \lambda_\gamma
\fullstop
\end{eqn}
Consequently, this differential is not surjective if and only if $\lambda$ has a zero at the target point $\rm{t} (\gamma)$, which happens if and only if $\gamma$ is critical, so $\gamma = \gamma_0 \in G_{z_0} \cong {(-z_0, z_0)}$.
Formula \eqref{250226060217} for $\Z : \Pi_1 (\CC_z) \to \CC$ gives an explicit demonstration of the fact that the restriction $\Z_{z_0}$ is a fivefold ramified covering map of the $\xi$-plane with only one ramification point located at $(u,z) = (-z_0, z_0)$ and only one branch point located at $\xi_0 = \Z (-z_0, z_0) = \tfrac{1}{30} z_0^5$.

The same argument shows that $\Z_{z_0} : \tilde{\rm{s}}^\inv (z_0) \to \CC$ is a covering map which therefore must be a universal covering map since $\tilde{\rm{s}}^\inv (z_0)$ is simply connected.
In particular, we can see from the explicit expression \eqref{250226061238} that the restriction $\Z_{z_0}$ is an infinitely-sheeted covering map of the $\xi$-plane with a single logarithmic branch point located at $\xi_0$.
\end{proof}

\subsubsection{The anchor map.}

Combining the groupoid source map with the central charge yields a holomorphic surjection that will play a major role in our constructions.

\begin{definition}[anchor map]
\label{250307123217}
For $X = \CC_z$ or $X = \CC^\ast_z$, the \dfn{anchor map} on $\Pi_1 (X)$ associated with the holomorphic differential form $\lambda$ is the holomorphic surjective map
\begin{eqn}
	\varrho \coleq (\rm{s}, \Z) : \Pi_1 (X) \to X \times \CC_\xi
\fullstop
\end{eqn}
\end{definition}

Like the two central charges, the two anchor maps $\varrho : \Pi_1 (\CC_z) \to \CC_z \times \CC_\xi$ and $\varrho : \Pi_1 (\CC^\ast_z) \to \CC^\ast_z \times \CC_\xi$ are intertwined by the map $\nu : \Pi_1 (\CC^\ast_z) \to \Pi_1 (\CC_z)$ in the sense that they fit into the commutative diagram
\begin{eqn}
\begin{tikzcd}[row sep = tiny, column sep = large]
		\Pi_1 (\CC^\ast_z)
			\ar[dd, "\nu"']
			\ar[dr, "\varrho"]
\\
&		\CC_z \times \CC_\xi \fullstop
\\		\Pi_1 (\CC_z)
			\ar[ur, "\varrho"']
\end{tikzcd}
\end{eqn}

\begin{lemma}
\label{250210081550}
The anchor map $\varrho$ is an isomorphism near the punctured identity bisection $\CC^\ast_z$.
More precisely, the anchor map $\varrho : \Pi_1 (\CC^\ast_z) \to \CC^\ast_z \times \CC_\xi$ as well as the restricted anchor map $\varrho : \Pi_1 (\CC_z) |_{\CC^\ast_z} \to \CC^\ast_z \times \CC_\xi$ define biholomorphisms near the identity bisection $1_{\CC^\ast_z}$.
That is to say, there is an open neighbourhood $\Xi \subset \CC^\ast_z \times \CC_\xi$ of $\xi = 0$ as well as open neighbourhoods $\Omega \subset \Pi_1 (\CC_z) |_{\CC^\ast_z}$ and $\tilde{\Omega} \subset \Pi_1 (\CC^\ast)$ of $1_{\CC^\ast_z}$ (see \autoref{250303172643}) that fit into the following commutative diagram where every arrow is a biholomorphism:
\begin{eqn}
\begin{tikzcd}[row sep = tiny, column sep = large]
		\tilde{\Omega}
			\ar[dd, "\nu"']
			\ar[dr, "\varrho"]
\\
&		\Xi \fullstop
\\		\Omega
			\ar[ur, "\varrho"']
\end{tikzcd}
\end{eqn}
\end{lemma}

\begin{figure}[t]
\centering
\begin{tikzpicture}
\begin{scope}
\fill [grey] (-3,-2) rectangle (3,2);
\node at (0,2) [above] {$\Pi_1 (\CC^\ast_z)$};
\foreach \i in {2,...,4}
{
	\draw [thin] ({-3.75 + 0.75*\i},-2) -- ({-3.75 + 0.75*\i},2);
}
\foreach \i in {6,...,8}
{
	\draw [thin] ({-3.75 + 0.75*\i},-2) -- ({-3.75 + 0.75*\i},2);
}
\begin{scope}
\clip (-3,-2) rectangle (3,2);
\foreach \i in {-4,...,-1}
{
\draw[smooth, domain=-2:2] 
    plot ({(0.75*\i)*e^(-\x)},\x);
}
\foreach \i in {1,...,4}
{
\draw[smooth, domain=-2:2] 
    plot ({(0.75*\i)*e^(-\x)},\x);
}
\end{scope}
\draw [ultra thick] (3,0) -- (-3,0) node [left] {$\CC^\ast_z$};
\begin{scope}
\clip (-3,-2) rectangle (3,2);
\draw[thick, fill=blue, fill opacity = 0.3, blue, dashed] (0,0) to[closed, curve through =
	{ (1.25,1.75) (1.75,1.5) (3,1) (4,0) (3,-2) (2,-1.5) }] (0.4,-1);
\draw[thick, fill=blue, fill opacity = 0.3, blue, dashed] (0,0) to[closed, curve through =
	{ (-1.25,1.75) (-1.75,1.5) (-3,1) (-4,0) (-3,-2) (-2,-1.5) }] (-0.4,-1);
\end{scope}
\draw [ultra thick, white] (0,-2) -- (0,2);
\node [blue] at (2.7,0.75) {$\tilde{\Omega}$};
\draw [->] (0,-2.2) -- node [right] {$\nu$} (0,-2.8);
\draw [->] (3.25,-0.5) -- node [above right] {$\varrho$} (4.75,-1.5);
\end{scope}

\begin{scope}[yshift=-5cm]
\fill [grey] (-3,-2) rectangle (3,2);
\node at (0,-2) [below] {$\Pi_1 (\CC_z) |_{\CC^\ast_z}$};
\draw [ultra thick] (3,0) -- (-3,0) node [left] {$\CC^\ast_z$};

\foreach \i in {1,...,7}
{
	\draw [thin] ({-3 + 0.75*\i},-2) -- ({-3 + 0.75*\i},2);
}
\begin{scope}
\clip (-3,-2) rectangle (3,2);
\foreach \i in {1,...,7}
{
	\draw [thin] ({-6+0.75*\i},3) -- ({0+0.75*\i},-3);
}
\end{scope}
\begin{scope}
\clip (-3,-2) rectangle (3,2);
\draw[thick, fill=blue, fill opacity = 0.3, blue, dashed] (0,0) to[curve through = 
	{ (0.3,1) (1,1.5) (3,1.5) (4,0) (3,-1) (1,-0.6)}] (0,0);
\draw[thick, fill=blue, fill opacity = 0.3, blue, dashed] (0,0) to[curve through =
	{ (-0.3,-1) (-1,-1.5) (-3,-1.5) (-4,0) (-3,1) (-1,0.6)}] (0,0);
\end{scope}
\draw [ultra thick, white] (2.1,-2.1) -- (-2.1,2.1);
\draw [ultra thick, white] (0,-2) -- (0,2);
\node [blue] at (2.7,1) {$\Omega$};
\draw [->] (3.25,0.5) -- node [below right] {$\varrho$} (4.75,1.5);
\end{scope}

\begin{scope}[yshift=-2.5cm, xshift=8cm]
\fill [grey] (-3,-2) rectangle (3,2);
\node at (0,2) [above] {$\CC^\ast_z \times \CC_\xi$};
\draw [ultra thick] (-3,0) -- (3,0) node [right] {$\CC^\ast_z$};
\begin{scope}
\clip (-3,-2) rectangle (3,2);
\draw[thick, fill=blue, fill opacity = 0.3, blue, dashed] (0,0) to[closed, curve through =
	{ (0.5,1) (1,1) (3,1) (4,0) (3,-1) (1,-1) }] (0.5,-1);
\draw[thick, fill=blue, fill opacity = 0.3, blue, dashed] (0,0) to[closed, curve through =
	{ (-0.5,1) (-1,1) (-3,1) (-4,0) (-3,-1) (-1,-1) }] (-0.5,-1);
\end{scope}
\draw [ultra thick, white] (0,-2) -- (0,2);
\node [blue] at (2.5,0.5) {$\Xi$};
\end{scope}

\end{tikzpicture}
\caption{Isomorphisms of \autoref{250210081550}.}
\label{250303172643}
\end{figure}

\begin{proof}
It is enough to prove the claim for $\varrho : \Pi_1 (\CC_z) |_{\CC^\ast_z} \to \CC^\ast_z \times \CC_\xi$ as the rest follows from \autoref{250212200132}.
Choose an arbitrary point $z_0 \in \CC^\ast_z$, and regard it both as the point $1_{z_0} \in 1_{\CC^\ast_z}$ on the identity bisection of $\Pi_1 (\CC_z) \big|_{\CC^\ast_z}$ and as the point $(z_0,0) \in \CC^\ast_z \times \CC_\xi$.
It is enough to show that there is are open neighbourhoods $W_1 \subset \Pi_1 (\CC_z) \big|_{\CC^\ast_z}$ of $1_{z_0}$ and $W_2 \subset \CC^\ast_z \times \CC_\xi$ of $(z_0, 0)$ such that the anchor map $\varrho$ restricts to give a biholomorphism $W_1 \iso W_2$.

Since the differential $\lambda$ is nonvanishing in a neighbourhood of $z_0$, there is a simply connected neighbourhood $U \subset \CC^\ast_z$ of $z_0$ such that the map $\Z_{z_0} : z \mapsto \xi = \int_{z_0}^z \lambda$ defines a biholomorphism from $U$ to the disc $V \subset \CC_\xi$ centred at the origin of some radius $r > 0$.
Note in particular that the central charge of the straight line segment $[0,z_0]$ must be strictly greater than $r$ in absolute value.

Fix a strictly smaller concentric disc $V_0 \subset V$ of radius $r_0 < r$, and let $U_0 \subset U$ be the preimage of $V_0$ under $\Z_0$.
Finally, let $\DD \subset \CC$ be a disc centred at the origin of radius $r - r_0$.
Notice that for any point $\xi_0 \in V_0$ and any $\xi \in \DD$, the point $\xi_0 + \xi$ is still contained in $V$.
Now we define the neighbourhoods $W_1$ and $W_2$ as follows:
\begin{eqn}
	W_1 \coleq \varrho^\inv (U_0 \times \DD) \big|^\textup{c} \subset \Pi_1 (\CC_z) \big|_{\CC^\ast_z}
\qqtext{and}
	W_2 \coleq U_0 \times \DD \subset \CC_z^\ast \times \CC_\xi
\fullstop{,}
\end{eqn}
where `` $ \dummy |^\textup{c}$ '' means the connected component containing the constant path $1_{z_0}$.

Let us make a clarifying comment regarding the open set $W_1$.
The preimage $\varrho^\inv (U_0 \times \DD)$ in $\Pi_1 (\CC_z) |_{\CC^\ast_z}$ is an open subset of consisting of all homotopy classes of paths $\gamma$ on $\CC_z^\ast$ that start in $U_0$ and whose central charge has the property that $\Z (\gamma) \in \DD$; i.e., $|\Z (\gamma)| < r - r_0$.
This implies in particular that the target point of $\gamma$ is necessarily contained in $U$, but it does not mean that the path $\gamma$ is contained in $U$ (more precisely, it does not mean that $\gamma$ is homotopic to a path that is entirely contained in $U$).
Indeed, this subset clearly contains the constant path $1_{z_0}$, but it also contains the homotopy class of any nontrivial loop $\ell$ in $\CC^\ast_z$ based at $z_0$.
However, such a loop $\ell$ does not lie in the same connected component as the constant path $1_{z_0}$ because there is no way to define a one-parameter family of paths interpolating between $1_{z_0}$ and $\ell$ whose central charge remains in $\DD$.
Indeed, the central charge of the straight line segment $[0,z_0]$ is strictly bigger than $r > r - r_0$ in absolute value.
Taking the connected component eliminates such nontrivial loops and ensures that any such path $\gamma$ is (homotopic to a path) entirely contained in $U$.

We claim that the anchor map defines an isomorphism $\varrho : W_1 \iso W_2$.
Suppose $\gamma, \gamma' \in W_1$ are two paths on $\CC^\ast_z$ with source at a point $z_1 \in U_0$ and such that $\xi = \Z (\gamma) = \Z (\gamma') \in \DD$.
By the construction of $W_1$, they are (homotopic to paths) entirely contained in $U$.
But since $U$ is simply connected, their homotopy class in $U$ is completely determined by their endpoints.
If we put $\xi_0 \coleq \Z_0 (z_1)$, then $\xi_0 + \xi \in V$, so $\gamma$ and $\gamma'$ terminate at the same point in $U$, hence define the same element in $W_1$.
Conversely, if $(z_1, \xi) \in U_0 \times \DD$, then there is a unique point $z_2 \in U$ such that $\Z_{z_0} (z_2) = \Z_{z_0} (z_1) + \xi$, so the path $\gamma$ in $U$ going from $z_1$ to $z_2$ is such that $\varrho (\gamma) = (z_1, \xi)$.
\end{proof}

\subsubsection{Source- and target-lifts.}
\label{250307191800}
The central charge relates the canonical source- and target-lifts of $\V = -6z^{-4} \del_z$ with the standard coordinate vector field on $\CC_\xi$ as we now explain.

Let $\V^\rm{s}, \V^\rm{t}$ denote the source- and target-lifts of $\V$ to the fundamental groupoid of the $z$-plane and also the source- and target-lifts to fundamental groupoid of the punctured $z$-plane.
Thus, they are the unique holomorphic vector fields on $\Pi_1 (\CC_z) |_{\CC^\ast_z}$ and respectively on $\Pi_1 (\CC^\ast_z)$ which satisfy the following identities:
\begin{eqntag}
\label{250226130638}
	\rm{s}_\ast \V^\rm{s} = \V,
\quad
	\rm{t}_\ast \V^\rm{s} = 0,
\qqtext{and}
	\rm{s}_\ast \V^\rm{t} = 0,
\quad
	\rm{t}_\ast \V^\rm{t} = \V
\fullstop
\end{eqntag}

Explicitly in the trivialisation $\Pi_1 (\CC_z) \cong \CC_u \times \CC_z$, the source- and target-lifts of $\V$ to the groupoid $\Pi_1 (\CC_z) |_{\CC^\ast_z}$ are
\begin{eqn}
	\V^\rm{s} = -6z^{-4} (\del_z - \del_u)
\qtext{and}
	\V^\rm{t} = -6z^{-4} \del_u
\fullstop
\end{eqn}
Notice that $\V^\rm{s}, \V^\rm{t}$ readily extend to the whole groupoid $\Pi_1 (\CC_z)$ as meromorphic vector fields.
Similarly, in the trivialisation $\Pi_1 (\CC^\ast_z) \cong \CC_u \times \CC^\ast_z$, the source- and target-lifts of $\V$ to the groupoid $\Pi_1 (\CC^\ast_z)$ are
\begin{eqn}
	\V^\rm{s} = -6z^{-5} (z\del_z - \del_u)
\qtext{and}
	\V^\rm{t} = -6e^{-u} z^{-5} \del_u
\fullstop
\end{eqn}

\begin{lemma}
\label{250222194723}
The source- and target-lifts of $\V$ satisfy the following identities:
\begin{eqntag}
\label{250226132732}
\begin{aligned}
	\Z_\ast \V^\rm{s} = - \del_\xi
&\qqtext{and}
	\Z_\ast \V^\rm{t} = + \del_\xi
\fullstop{;}
\\
	\varrho_\ast \V^\rm{s}
		= \V - \del_\xi
&\qqtext{and}
	\varrho_\ast \V^\rm{t}
		= + \del_\xi
\fullstop
\end{aligned}
\end{eqntag}
\end{lemma}

\begin{proof}
Since $\varrho = (\rm{s}, \Z)$, it is enough to demonstrate the first row in \eqref{250226132732}.
Recall that the differential of $\Z$ can be expressed as $\d \Z = \rm{t}^\ast \lambda - \rm{s}^\ast \lambda$.
But $\rm{t}^\ast \lambda (\V^\rm{s}) = \lambda ( \rm{t}_\ast \V^\rm{s}) = 0$ and $\rm{s}^\ast \lambda (\V^\rm{s}) = \lambda (\rm{s}_\ast \V^\rm{s}) = \lambda (\V) = 1$ thanks to \eqref{250226130638}.
\end{proof}

These identities can be demonstrated explicitly in our trivialisations.
For the central charge $\Z : \Pi_1 (\CC_z) \to \CC$, the $1 \!\times\! 2$ matrix of $\d{\Z}_\gamma$ for any $\gamma \in \Pi_1 (\CC_z)$ in the bases $(\del_u, \del_z)$ and $\del_\xi$ is given by \eqref{250307121002}, so if we write $\V^\rm{s}$ and $\V^\rm{t}$ as column vectors in the basis $(\del_u, \del_z)$, then the two identities in the first row of \eqref{250226132732} are the result of simple matrix multiplication
\begin{eqns}
	\Big[ - \tfrac{1}{6} (z + u)^4 ~~~~ \tfrac{1}{6} (z^4 - (z+u)^4) \Big] \mtx{ 6z^{-4} \\ -6z^{-4} }
		&= -1
\fullstop{,}
\\
	\Big[ - \tfrac{1}{6} (z + u)^4 ~~~~ \tfrac{1}{6} (z^4 - (z+u)^4) \Big] \mtx{ -6z^{-4} \\ 0 }
		&= +1
\fullstop
\end{eqns}
Similarly, for the central charge $\Z : \Pi_1 (\CC^\ast_z) \to \CC$, the $1 \!\times\! 2$ matrix of $\d{\Z}_\gamma$ for any $\gamma \in \Pi_1 (\CC^\ast_z)$ is given by \eqref{250307121006}, so the two identities in the first row of \eqref{250226132732} become
\begin{eqns}
	\Big[ - \tfrac{1}{6} e^{5u} z^5 ~~~~ \tfrac{1}{6} (1 - e^{5u}) z^5 \Big] \mtx{ 6z^{-5} \\ -6z^{-4} } 
		&= -1
\fullstop{,}
\\
	\Big[ - \tfrac{1}{6} e^{5u} z^5 ~~~~ \tfrac{1}{6} (1 - e^{5u}) z^5 \Big] \mtx{ -6z^{-5} \\ 0 } 
		&= +1
\fullstop
\end{eqns}

\subsubsection{Local vector field flow.}
The central charge appears naturally when considering the complex flow of the vector field $\V$ as we explain next.

Recall that a \textit{complex flow} on $\CC^\ast_z$ is a holomorphic map $\varphi : \Xi \subset \CC^\ast_z \times \CC_\xi \to \CC_z$ from an open subset $\Xi$ containing $\CC^\ast_z \times \set{0}$ (called the \textit{flow domain}) with the property that $\varphi (z,0) = z$ and $\varphi (z, \xi_1 + \xi_2) = \varphi \big( \varphi (z, \xi_1), \xi_2 \big)$ for all nonzero $z$ and all $\xi_1, \xi_2 \in \CC$ such that $(z, \xi_1), (z, \xi_2), (z, \xi_1 + \xi_2) \in \Xi$.
It is called a \textit{global flow} if $\Xi = \CC^\ast_z \times \CC_\xi$.
The flow $\varphi$ is \textit{generated by} the vector field $\V$ if $\V$ is the velocity field of $\varphi$ in the usual sense of satisfying the identity $\del_\xi \varphi (z,\xi) \big|_{\xi = 0} = \V_z$ for all nonzero $z$.
The \textit{inverse flow} of $\V$ is by definition the flow of $-\V$.

The flow $\varphi (z_0, \xi)$ of any nonzero $z_0 \in \CC_z$ for sufficiently small time $\xi \in \CC$ along $\V$ is the unique point $z_\xi \in \CC^\ast_z$ satisfying $\int_{z_0}^{z_\xi} \lambda = \xi$.
That is to say, $z_\xi$ is such that the central charge of the straight line segment $[z_0, z_\xi]$ is $\xi$; i.e., $\Z \big( [z_0, z_\xi] \big) = \xi$.
Explicitly, $z_\xi = (z_0^5 - 30 \xi)^{1/5}$ where the quintic-root branch is such that $(z_0^5)^{1/5} = z_0$.
This quintic-root branch is of course well-defined only in a neighbourhood of $z_0$.
We can summarise these relationships in the following elementary Lemma.

\begin{lemma}
\label{250226181410}
Fix any nonzero $z_0 \in \CC_z$ and let $\varphi_0 (\xi) \coleq \varphi (z_0, \xi)$.
Then there is a neighbourhood $U_0 \subset \CC^\ast_z$ of $z_0$, neighbourhoods $\tilde{\Omega}^\rm{s}_0 \subset \tilde{\rm{s}}^\inv (z_0)$ and $\tilde{\Omega}^\rm{t}_0 \subset \tilde{\rm{t}}^\inv (z_0)$, as well as $\Omega^\rm{s}_0 \subset \rm{s}^\inv (z_0)$ and $\Omega^\rm{t}_0 \subset \rm{t}^\inv (z_0)$, all containing the constant path $1_{z_0}$, and a neighbourhood $\Xi_0 \subset \CC_\xi$ of the origin (see \autoref{250303172643}) such that we have the following commutative diagrams where every arrow is a biholomorphism:
\begin{eqn}
\begin{tikzcd}[row sep = small]
	\tilde{\Omega}^\rm{s}_0
		\ar[dd, "\nu"']
		\ar[dr, "\Z"]
		\ar[drr, bend left, "\rm{t}"]
\\
&	\Xi_0
		\ar[r, "\varphi_0"]
&	U_0
\\
	\Omega^\rm{s}_0
		\ar[ur, "\Z"']
		\ar[urr, bend right, "\rm{t}"']
\end{tikzcd}
\qqtext{and}
\begin{tikzcd}[row sep = small]
	\tilde{\Omega}^\rm{t}_0
		\ar[dd, "\nu"']
		\ar[dr, "-\Z"]
		\ar[drr, bend left, "\rm{s}"]
\\
&	\Xi_0
		\ar[r, "\varphi_0"]
&	U_0
\fullstop
\\
	\Omega^\rm{t}_0
		\ar[ur, "-\Z"']
		\ar[urr, bend right, "\rm{s}"']
\end{tikzcd}
\end{eqn}
\end{lemma}

\subsubsection{Convolution product.}
The central charge also allows us to make sense of a convolution product $\ast$ on the groupoids which is the natural lift of the standard convolution product on the $\xi$-plane.
Recall that the latter is defined by
\begin{eqntag}
\label{250307073427}
	(f \ast g) (z,\xi) \coleq \int_0^\xi f(z,u) g(z,\xi-u) \d{u}
\end{eqntag}
for any holomorphic functions $f(z,\xi)$ and $g(z,\xi)$ that are well-defined along the integration domain which is the complex interval $[0,\xi]$.
For $\xi$ sufficiently small (namely, such that $(z,\xi) \in \Xi$), we may use the central charge $\Z$ to interpret the interval $[0,\xi]$ as a parameterised curve $\gamma : [0,\xi] \to \CC_z$ with source at $z$ and central charge $\Z(\gamma) = \xi$.
Then using the anchor map $\varrho$, the points $(z,u)$ and $(z,\xi-u)$ in the integrand may be interpreted as the source-truncations of $\gamma$ at times $u$ and $\xi-u$, respectively.
That is to say, $(z,u) = \varrho (\gamma_u)$ and $(z,\xi-u) = \varrho (\gamma_{\xi - u})$, where
\begin{eqn}
	\gamma_u \coleq \gamma |_{[0,u]} : [0,u] \subset [0,\xi] \to \CC_z
\qtext{and}
	\gamma_{\xi - u} \coleq \gamma |_{[0,\xi - u]} : [0,\xi - u] \subset [0,\xi] \to \CC_z
\fullstop
\end{eqn}
Next, let us notice that the differential $\d{u}$ is the value of the differential $\d{\xi}$ at the point $u$ which is the target point of the interval $[0,u]$.
Meanwhile, the differential of the restriction of $\Z$ to the source fibre $\rm{s}^\inv (z)$ is $\d{(\Z_z)} = \Z^\ast_z \d{\xi} = \rm{t}^\ast \lambda$, so $\Z^\ast_z \d{u}$ is the value of $\rm{t}^\ast \lambda$ at the target point of the source-truncation $\gamma_u$
\begin{eqn}
	\Z^\ast_z \d{u} = \Z^\ast_z \d{\xi}_u = \rm{t}^\ast \lambda_{\gamma_u}
\fullstop
\end{eqn}
Consequently, if we put $\tilde{f} \coleq \varrho^\ast f$ and $\tilde{g} \coleq \varrho^\ast g$, then formula \eqref{250307073427} may be written as follows:
\begin{eqntag}
\label{250307081855}
	(\tilde{f} \ast \tilde{g}) (\gamma) \coleq \int_0^{\xi} \tilde{f} (\gamma_u) \tilde{g} (\gamma_{\xi - u}) \D{u}
\qtext{where}
	\D{u} \coleq \rm{t}^\ast \lambda_{\gamma_u}
\fullstop
\end{eqntag}
To upgrade this formula to the whole groupoid, we may reparameterise $\gamma : [0,\xi] \to \CC_z$ by its \textit{arc-length} in the following sense.

\begin{definition}[arc-length]
\label{250217105208}
We define the \dfn{arc-length} (or simply \textit{length}) of a concrete path $\gamma : I \to \CC_z$ to be the positive real number
\begin{eqn}
	|\gamma| \coleq \int_\gamma |\lambda|
\fullstop
\end{eqn}
Any such $\gamma$ has a natural \textit{arc-length parameterisation} $\gamma : [0, \L] \to \CC_z$ where $\L \coleq |\gamma|$ in the usual sense that the restriction of $\gamma$ to any subinterval $[a,b] \subset [0,\L]$ has arc-length $b-a$.
\end{definition}

\begin{definition}[arc-length truncations]
For any concrete path $\gamma : I \to \CC_z$ with arc-length $|\gamma| = \L$, we define the 	\dfn{arc-length source-truncation} of $\gamma$ and the \dfn{arc-length target-truncation} of $\gamma$ respectively as the following parameterised curves:
\begin{eqn}
	\gamma_s \coleq \gamma |_{[0,s]} : [0,s] \subset [0, \L] \to \CC_z
\qtext{and}
	\bar{\gamma}_s \coleq \gamma |_{[\L-s,\L]} : [\L-s,\L] \subset [0, \L] \to \CC_z
\fullstop
\end{eqn}
\end{definition}

Thus, $\gamma_s$ is the source portion of $\gamma$ of length $|\gamma_s| = s$, whilst $\bar{\gamma}_s$ is the target portion of $\gamma$ of length $|\bar{\gamma}_s| = s$.
Note that the arc-length truncations determine two one-parameter families of paths $(\gamma_s)_{s \in [0,\L]}$ and $(\bar{\gamma}_s)_{s \in [0,\L]}$ which interpolate between the constant paths $\gamma_0 = 1_{z_0}$, $\bar{\gamma}_0 = 1_{z_1}$ and the given path $\gamma_\L = \bar{\gamma}_\L = \gamma$, where $z_0$ and $z_1$ are the source and target of $\gamma$.
In other words, they give concrete representatives of the canonical retractions $\tilde{\gamma}^\rm{s}$ and $\tilde{\gamma}^\rm{t}$ of $\gamma$; see \autoref{250301112427}.

\begin{definition}[convolution product]
\label{250306170911}
Let $X = \CC_z$ or $X = \CC^\ast_z$.
For any pair of holomorphic functions $f,g$ on the groupoid $\Pi_1 (X)$, their \dfn{convolution product} is the holomorphic function $f \ast g$ on $\Pi_1 (X)$ given by the following formula: for all $\gamma \in X$,
\begin{eqntag}
\label{250306175130}
	(f \ast g) (\gamma) \coleq \int_0^{\L} f (\gamma_s) g (\gamma_{\L-s}) \D{s}
\qtext{where}
	\D{s} \coleq \rm{t}^\ast \lambda_{\gamma_s}
\fullstop{,}
\end{eqntag}
where $\gamma_s$ is the arc-length source-truncation of any representative of $\gamma$ of length $|\gamma|=\L$.
\end{definition}

That \eqref{250306175130} only depends on the homotopy class $\gamma$ and not on the concrete representative of $\gamma$ is a consequence of the fact that $f$, $g$, and $\lambda$ are holomorphic.
The following lemma explains the sense in which the convolution product that we have defined on the groupoid is the canonical lift of the standard convolution product from the Borel $\xi$-plane.

\begin{lemma}
\label{250306181554}
The pullback by the anchor $\varrho$ distributes over the convolution product $\ast$.
That is, for any two holomorphic functions $f = f(z,\xi)$ and $g = g(z,\xi)$, we have the following identity of holomorphic functions on the groupoid $\Pi_1 (X)$ where $X$ is $\CC_z$ or $\CC^\ast_z$:
\begin{eqntag}
\label{250307141604}
	\varrho^\ast (f \ast g) = (\varrho^\ast f) \ast (\varrho^\ast g)
\fullstop
\end{eqntag}
\end{lemma}

\subsubsection{The Laplace transform.}
Similarly, the central charge allows us to interpret the Laplace transform as an integral on infinitely long paths on the groupoid.
Recall that the Laplace transform in a direction $\alpha$ is defined by
\begin{eqntag}
\label{250310161541}
	\Laplace [\, f \,] (z,\hbar) \coleq \int_{e^{i\alpha} \RR_+} e^{-\xi/\hbar} f(z,\xi) \d{\xi}
\end{eqntag}
for any holomorphic function $f (z,\xi)$ defined in a neighbourhood of the integration contour $e^{i\alpha} \RR_+$.
For this integral to converge, at least in a sectorial domain of $\hbar=0$, $f$ must satisfy exponential bounds for large $\xi \in e^{i\alpha} \RR_+$ of the form $|f(z,\xi)| \leq e^{\K |\xi|}$ for some $\K > 0$ independent of $\xi$.

Fix any nonzero $z_0 \in \CC_z$ and suppose the infinite ray $e^{i \alpha} \RR_+$ does not pass through the branch point of $\xi_0 = \tfrac{1}{30} z_0^5$ of the central charge $\Z_{z_0} : B_{z_0} \to \CC_\xi$.
Then $e^{i \alpha} \RR_+$ has a unique lift to an infinite path $\ell$ on $\Pi_1 (\CC^\ast_z)$ starting at $1_{z_0}$.
If we put $\tilde{f} = \varrho^\ast f$, then we can interpret the Laplace transform \eqref{250310161541} as follows:
\begin{eqntag}
\label{250310162642}
	\Laplace [\, \tilde{f} \,] (z, \hbar)
		\coleq \int_{\ell} e^{-\Z (\ell_s)/\hbar} \tilde{f} (\ell_s) \D{s}
\qtext{where}
	\D{s} = \rm{t}^\ast \lambda_{\ell_s}
\fullstop
\end{eqntag}

\subsection{Initial Value Problems on Groupoids}

In this subsection, we make a partial attempt at reformulating the Initial Value Problem \eqref{250306084618} in global geometric terms.
The correct reformulation appears in \autoref{250221200950} and is logically independent of this subsection.
The discussion is this subsection has been included for purely pedagogical reasons, in the hope that it will clarify and help motivate the geometric constructions of the next section.

\subsubsection{First attempt.}
\label{250227123556}
As a first attempt at finding global solutions of the Initial Value Problem \eqref{250306084618}, let us consider an Initial Value Problem (IVP) of the following much simpler form:
\begin{eqntag}
\label{250227103350}
	\V f = \F
\qtext{such that}
	f (z_0) = a
\fullstop{,}
\end{eqntag}
where $(z_0, a) \in \CC^\ast_z \times \CC$ is any point and $\F = \F \big(z, f(z) \big)$ for some holomorphic map $\F : \CC^\ast_z \times \CC \to \CC$.
Recall that a common strategy to solve such problems is to express the vector field $\V$ in a canonical form that admits straightforward integration.
That is, we look for a local holomorphic change of coordinates $z \mapsto \xi$ with respect to which $\V$ becomes the coordinate vector field $\del_\xi$.
A suitable such change of coordinates is provided by the flow $\varphi$ of $\V$: making use of the chosen basepoint $z_0$, we define
\begin{eqn}
	\varphi_0 : \xi \mapsto z = z_\xi \coleq \varphi (z_0, \xi)
\fullstop
\end{eqn}
By \autoref{250226181410}, there are neighbourhoods $U_0 \subset \CC^\ast_z$ of $z_0$ and $\Xi_0 \subset \CC_\xi$ of the origin where $\varphi_0$ defines a biholomorphism and $\V$ becomes simply the coordinate vector field $\del_\xi$:
\begin{eqn}
	(\varphi_0)_\ast \del_\xi = \V
\fullstop
\end{eqn}
Then, if we write $\tilde{f} = \varphi_0^\ast f$ and $\tilde{\F} = \varphi_0^\ast \F$, the IVP \eqref{250227103350} becomes simply $\del_\xi \tilde{f} = \tilde{\F}$ with $\tilde{f} (0) = a$, which is equivalent to the integral equation $\tilde{f} (\xi) = a + \int_{0}^\xi \tilde{\F} \big( u, \tilde{f} (u) \big) \d{u}$.
This integral equation is well-defined for all sufficiently small $\xi$ and it can now be solved, subject to usual mild assumptions on $\F$, using a standard fixed-point argument to build a solution $f$ in a neighbourhood of $z_0$.
Undoing the coordinate transformation, we find the integral equation in the original variables, well-posed for all $z \in U_0$:
\begin{eqntag}
\label{250227185153}
	f(z) = a + \int_0^\xi \F \big( z_u, f(z_u) \big) \d{u}
\fullstop{,}
\end{eqntag}
where $\xi$ and $z_u$ are uniquely defined by the identities $\xi = \int_{z_0}^z \lambda$ and $u = \int_{z_0}^{z_u} \lambda$.

A shortcoming of this approach is that the domain of definition of a solution obtained in this manner is limited by the size of the neighbourhood $U_0$ of $z_0$ which is ultimately limited by the size of the flow domain $\Xi \subset \CC^\ast_z \times \CC_\xi$ of $\varphi$.
Thus, using this strategy to find global solutions requires the flow to be global.
But global flows are rare, especially in holomorphic geometry.
In particular, our vector field $\V$ does not admit a global flow because of the pole at $z = 0$.
As a result, obtaining \textit{global} solutions of differential equations of the form \eqref{250227103350} poses additional challenges.

This challenge can be addressed by passing to the fundamental groupoid $\Pi_1 (\CC^\ast_z) \rightrightarrows \CC^\ast_z$ and interpreting the given IVP using the canonical lifts of the vector field $\V$ as follows.

\begin{lemma}
\label{250225133855}
Let $U_0 \subset \CC^\ast_z$, $\Xi_0 \subset \CC_\xi$, and $\tilde{\Omega}^\rm{s}_0 \subset \rm{s}^\inv (z_0) \subset \Pi_1 (\CC^\ast_z)$ be as described in \autoref{250226181410}.
Let $\tilde{\V} = \V^\rm{t}$ be the target-lift of $\V$ to $\Pi_1 (\CC^\ast_z)$.
Set $\tilde{f} = \rm{t}^\ast f$ and $\tilde{\F} = \rm{t}^\ast \F$, and consider the following Initial Value Problem on $\rm{s}^\inv (z_0)$:
\begin{eqntag}
\label{250227120026}
	\tilde{\V} \tilde{f} = \tilde{\F}
\qtext{such that}
	\tilde{f} (1_{z_0}) = a
\fullstop
\end{eqntag}
Then the Initial Value Problem \eqref{250227103350} is equivalent to the Initial Value Problem \eqref{250227120026} over the isomorphism $\rm{t} : \tilde{\Omega}^\rm{s}_0 \iso U_0$.
\end{lemma}

By being equivalent \textit{over} the isomorphism $\rm{t}$ we mean simply that $f = (\rm{t}^\inv)^\ast \tilde{f}$ is a solution of \eqref{250227120026} on $U_0$ if and only if $\tilde{f} = \rm{t}^\ast f$ is a solution of \eqref{250227120026} on $\tilde{\Omega}_0^\rm{s}$.

The IVP \eqref{250227120026} makes sense not just over $\tilde{\Omega}_0^\rm{s}$ but over the entire source fibre $\rm{s}^\inv (z_0) \subset \Pi_1 (\CC^\ast_z)$, which is the universal cover of $\CC^\ast_z$ based at $z_0$.
Therefore, a global solution of \eqref{250227120026} can be viewed as a global but multivalued solution of \eqref{250227103350}.
An advantage of this point of view is that the IVP \eqref{250227120026} can be written equivalently as an integral equation \textit{globally} on the source fibre $\rm{s}^\inv (z_0)$.
This can be achieved through a series of simple observations as follows.

First, inspecting the integral equation \eqref{250227185153}, we notice that if we interpret $z$ as the target point of a parameterised path $\gamma : [0,\xi] \to \CC^\ast_z$ with source $z_0$, then the point $z_u$ in the integrand can be interpreted as the target point of the source-truncation $\gamma_u \coleq \gamma |_{[0,u]}$.
Meanwhile, $\xi = \Z (\gamma)$ and $u = \Z (\gamma_u)$, so \eqref{250227185153} can be written for all $\gamma \in \tilde{\Omega}^\rm{s}_0$ as follows:
\begin{eqntag}
\label{250228084937}
	\tilde{f} (\gamma) = a + \int_0^{\xi} \tilde{\F} \big( \gamma_u, \tilde{f} (\gamma_u) \big) \d{u}
\fullstop
\end{eqntag}
From this point of view, the one-form $\d{u}$ in the integrand is nothing but the value of $\d{\xi}$ at $u$ which is the target point of the line segment $[0,u]$.
But since $\Z^\ast \d{\xi} = \rm{t}^\ast \lambda - \rm{s}^\ast \lambda$, we can write $\d{u} = \d{\xi}_u = \rm{t}^\ast \lambda_{\gamma_u}$.
Consequently, integral equation \eqref{250228084937} may be written like so:
\begin{eqntag}
\label{250307154206}
	\tilde{f} (\gamma) = a + \int_0^{\xi} \tilde{\F} \big( \gamma_u, \tilde{f} (\gamma_u) \big) \D{u}
\qtext{where}
	\D{u} = \rm{t}^\ast \lambda_{\gamma_u}
\fullstop
\end{eqntag}
Now, the interval $[0,\xi]$ parameterising the path $\gamma$ is actually a straight line segment in the complex plane, but we can change the parameterisation to the real interval $[0, |\xi|]$ by interpreting $|\xi|$ as the \hyperref[250217105208]{arc-length} of $\gamma$ with respect to $\lambda$ because $|\xi| = \int_\gamma |\lambda|$.
In that case, $\gamma_u = \gamma_s$ whenever $|u| = s$, and so the integral equation \eqref{250228084937} may be written as follows:
\begin{eqntag}
\label{250228085002}
	\tilde{f} (\gamma) = a + \int_0^{|\xi|} \tilde{\F} \big( \gamma_s, \tilde{f} (\gamma_s) \big)
		\D{s}
\qtext{where}
	\D{s} = \rm{t}^\ast \lambda_{\gamma_s}
\fullstop
\end{eqntag}
This equation is well-posed for all $\gamma \in \tilde{\Omega}^\rm{s}_0$ and any arc-length source-truncation $\gamma_s$ of $\gamma$.
But now it is clear how to extend the integral equation from $\tilde{\Omega}^\rm{s}_0$ to the entire source fibre $\rm{s}^\inv (z_0)$.

\begin{lemma}
\label{250225135551}
The Initial Value Problem \eqref{250227120026} over $\rm{s}^\inv (z_0)$ is equivalent to the following integral equation:
\begin{eqntag}
\label{250219124250}
	\tilde{f} (\gamma) = a + \int_0^{|\gamma|} \tilde{\F} \big( \gamma_s, \tilde{f} (\gamma_s) \big) \D{s}
\qtext{with}
	\D{s} = \rm{t}^\ast \lambda_{\gamma_s}
\fullstop{,}
\end{eqntag}
where $\gamma_s$ is any arc-length source-truncation of $\gamma \in \rm{s}^\inv (z_0)$.
\end{lemma}

If it exists, a solution $\tilde{f}$ of this integral equation does not depend on the choice of concrete representative $\gamma$ and defines a global holomorphic function on the entire source fibre $\rm{s}^\inv (z_0)$.
Since the latter is the universal covering space of $\CC^\ast_z$ based at $z_0$, the function $\tilde{f}$ descends via the target map $\rm{t} : \rm{s}^\inv (z_0) \to \CC^\ast_z$ to a global but multivalued holomorphic function on $\CC^\ast_z$ which satisfies the Initial Value Problem \eqref{250227103350}.

\subsubsection{Second attempt.}
\label{250305163823}
As a second attempt at finding global solutions of the Initial Value Problem \eqref{250306084618}, let us now consider an Initial Value Problem of the following form that bears a closer resemblance to \eqref{250306084618}:
\begin{eqntag}
\label{250227123632}
	(\V - \del_\xi) f = \F
\qtext{such that}
	f (z,0) = a(z)
\fullstop{,}
\end{eqntag}
where $a (z)$ is a given holomorphic function on $\CC^\ast_z$ and $\F = \F \big( z, \xi; f (z,\xi) \big)$ for some holomorphic map $\F : \CC^\ast_z \times \CC_\xi \times \CC \to \CC$.
Recall again that the usual strategy is to put the vector field $\V - \del_\xi$ in canonical form that can be integrated.
Thus, we search for a holomorphic change of coordinates $(z,\xi) \mapsto (w,u)$ with respect to which $\V - \del_\xi$ becomes the coordinate vector field $-\del_u$ (which we have negated for convenience, with the benefit of hindsight), and such that the locus $u = 0$ corresponds to the submanifold where the initial condition is posed.

A suitable such change of coordinates is provided by the flow $\Phi$ of $-(\V - \del_\xi)$ (or, in other words, the inverse flow of $\V - \del_\xi$).
It can be expressed in terms of the flow $\varphi$ of $\V$ as $\Phi (z, \xi, u) = \big( \varphi (z, -u), \xi + u \big)$.
In our case, the submanifold of initial conditions is given by $\xi = 0$, so we define a local coordinate change by
\begin{eqn}
	\Phi_0 : (w, u) \mapsto (z,\xi) = \Phi (w,0,u) = \big( \varphi (w,-u), u \big)
\fullstop
\end{eqn}
Under this local coordinate change,
\begin{eqn}
	( \Phi_0 )_\ast (-\del_u) = \V - \del_\xi
\fullstop
\end{eqn}

\begin{lemma}
\label{250228121628}
There is a neighbourhood $\Xi \subset \CC^\ast_z \times \CC_\xi$ of $\xi = 0$, neighbourhoods $\tilde{\Omega} \subset \Pi_1 (\CC^\ast_z)$ and $\Omega \subset \Pi_1 (\CC_z) |_{\CC^\ast_z}$ of the identity bisection $1_{\CC^\ast_z}$, and a neighbourhood $W \subset \CC_w^\ast \times \CC_u$ of $u = 0$ such that we have the following commutative diagram where every arrow is a biholomorphism:
\begin{eqn}
\begin{tikzcd}[row sep = small]
	\tilde{\Omega}
		\ar[dd, "\nu"']
		\ar[dr, "\varrho"]
		\ar[drr, bend left, "{(\rm{t}, \Z)}"]
\\
&	\Xi
&	W
	\ar[l, "\Phi_0"]
\fullstop
\\
	\Omega
		\ar[ur, "\varrho"']
		\ar[urr, bend right, "{(\rm{t}, \Z)}"']
\end{tikzcd}
\end{eqn}
\end{lemma}

Then, if we write $\tilde{f} = \Phi_0^\ast f$, $\tilde{\F} = \Phi_0^\ast \F$, and $\tilde{a} = \Phi_0^\ast a$, the IVP \eqref{250227123632} becomes simply $\del_u \tilde{f} = \tilde{\F}$ with initial condition $\tilde{f} (w, 0) = \tilde{a} (w)$.
This IVP is equivalent to the integral equation $\tilde{f} (w,u) = \tilde{a} (w) - \int_0^u \tilde{\F} \big(w,s; \tilde{f} (w,s) \big) \d{s}$, which is well-defined for all $w$ and all sufficiently small $u$; more precisely, for all $(w,u) \in W$.
It can now be solved, subject to usual mild assumptions on $\F$, using a standard fixed-point argument to build a local solution $f$.
Undoing the coordinate transformation, we find the integral equation in the original variables:
\begin{eqntag}
\label{250228123756}
	f (z, \xi) = a (z) - \int_0^\xi \F \big( z_{\xi - s}, s; f ( z_{\xi - s}, s) \big) \d{s}
\qquad\forall (z,\xi) \in \Xi
\fullstop{,}
\end{eqntag}
where $z_{\xi - s} \coleq \varphi( z, \xi-s )$.
In order to give a global interpretation of this integral equation, we pass to the fundamental groupoid $\Pi_1 (\CC^\ast_z) \rightrightarrows \CC^\ast_z$.

\begin{lemma}
\label{250228131924}
Suppose $\F : \CC^\ast_z \times \CC_\xi \times \CC \to \CC$ is a holomorphic map.
Let $\tilde{\Omega}, \Xi, W$ be as in \autoref{250228121628}, and set $\tilde{f} = \varrho^\ast f$, $\tilde{\F} = \varrho^\ast \F$, and $\tilde{a} = \rm{s}^\ast a$.
Let $\tilde{\V} = \tilde{\V}^\rm{s}$ be the source-lift of $\V$ to the fundamental groupoid $\Pi_1 (\CC^\ast_z)$, so that $\varrho_\ast \tilde{\V} = \V - \del_\xi$.
Then the Initial Value Problem \eqref{250227123632} is equivalent over $\Xi$ to the Initial Value Problem over $\tilde{\Omega}$ given by:
\begin{eqntag}
\label{250228132252}
	\tilde{\V} \tilde{f} = \tilde{\F}
\qtext{such that}
	\tilde{f} \big|_{1_{\CC^\ast_z}} = \tilde{a}
\fullstop
\end{eqntag}
\end{lemma}

This IVP makes sense not just over $\tilde{\Omega}$ but over the whole fundamental groupoid $\Pi_1 (\CC^\ast_z)$, and the integral equation \eqref{250228123756} can be extended to $\Pi_1 (\CC^\ast_z)$ as follows.

First, we notice that if we interpret $z$ in \eqref{250228123756} as the source of a curve $\gamma : [0,\xi] \to \CC^\ast_z$ with central charge $\Z (\gamma) = \xi$, then the point $z_{\xi - s}$ in the integrand can be interpreted as the source of the truncated curve $\gamma_{s} \coleq \gamma |_{[\xi-s,\xi]}$ defined on the subinterval $[\xi-s,\xi] \subset [0,\xi]$ with central charge $\Z (\gamma_s) = s$.
Consequently, \eqref{250228123756} can be written on $\tilde{\Omega}$ as follows:
\begin{eqntag}
\label{250228134348}
	\tilde{f} (\gamma) = \tilde{a} - \int_0^{\Z (\gamma)} \tilde{\F} \big( \gamma_s; \tilde{f} (\gamma_s) \big) \d{s}
\qqquad
	\forall \gamma \in \tilde{\Omega} \subset \Pi_1 (\CC^\ast_z)
\fullstop
\end{eqntag}
Finally, upon reparameterising $\gamma$ using its arc-length, we come to the following conclusion.

\begin{lemma}
\label{250228134234}
The Initial Value Problem \eqref{250228132252} is equivalent to the following integral equation, well-defined for all $\gamma \in \Pi_1 (\CC^\ast_z)$:
\begin{eqntag}
\label{250228134355}
	\tilde{f} (\gamma) = \tilde{a} - \int_0^{|\gamma|} \tilde{\F} \big( \gamma_s; \tilde{f} (\gamma_s) \big) \D{s}
\qqquad
	\forall \gamma \in \Pi_1 (\CC^\ast_z)
\fullstop{,}
\end{eqntag}
where $\D{s} \coleq \rm{s}^\ast \lambda_{\gamma_s}$ and $\gamma_s$ is any arc-length target-truncation of $\gamma$.
\end{lemma}

\subsubsection{Third attempt.}
\label{250315132104}
Let us now consider an Initial Value Problem of the following form that readily specialises to our Initial Value Problem \eqref{250306084618}:
\begin{eqntag}
\label{250303130505}
\begin{cases}
	(+\V - \del_\xi) f_+ = \F_+
\\	(-\V - \del_\xi) f_- = \F_-
\end{cases}
\qqtext{such that}
	f_\pm (z,0) = a_\pm (z)
\fullstop{,}
\end{eqntag}
where $a_\pm (z)$ is a given holomorphic function on $\CC^\ast_z$ and $\F_\pm = \F_\pm \big(z, \xi; f_+ (z,\xi), f_- (z,\xi) \big)$ is some given holomorphic expression in $f_\pm$ which depends holomorphically on $(z,\xi) \in \CC^\ast_z \times \CC_\xi$.
The standard approach to this problem is to construct a holomorphic coordinate transformation $(z,\xi) \mapsto (u_+, u_-)$ that puts both vector fields $\pm \V - \del_\xi$ in canonical form $\del_{u_\pm}$ simultaneously.

We follow a slightly different approach.
Instead of searching for a single transformation that expresses $\pm \V - \del_\xi$ as coordinate vector fields, we search for a suitable transformation for each vector field individually.
Thus, for each choice of $\pm$, we look for a holomorphic change of coordinates $(z,\xi) \mapsto (w, u)$ with respect to which the vector field $\pm \V - \del_\xi$ becomes the (negated) coordinate vector field $- \del_u$, and such that the locus $u = 0$ in both cases corresponds to the submanifold of initial conditions.

A suitable such change of coordinates is provided by the inverse flow $\Phi_\pm$ of $\pm \V - \del_\xi$, which can be expressed in terms of the flow $\varphi_\pm$ of $\pm \V$ as follows:
\begin{eqn}
	\Phi_\pm (z,\xi,u) = \big( \varphi_\pm (z,-u), \xi + u \big)
\fullstop
\end{eqn}
In our case, the submanifold of initial conditions is given by $\xi = 0$, so we define two local coordinate changes by
\begin{eqn}
	\Phi^\pm_0 : (w, u) \mapsto (z, \xi) 
		= \Phi_\pm (w,0,u) 
		= \big( \varphi_\pm (w,-u), u \big)
		= (w^\pm_{-u}, u)
\fullstop
\end{eqn}
Here, $w^\pm_{-u}$ is the unique point satisfying the identity $-u = \pm \int_{w}^{w^\pm_{-u}} \lambda$.
Under this local coordinate change, we have
\begin{eqn}
	(\Phi^\pm_0)_\ast (-\del_u) = \pm \V - \del_\xi
\fullstop
\end{eqn}

In words, given a point $(w,u)$, applying the local biholomorphism $\Phi^\pm_0$ amounts to flowing the point $w$ for time $-u$ along the vector field $\pm \V$.
Conversely, given a point $(z, \xi)$, applying the inverse biholomorphism $(\Phi^\pm_0)^\inv$ amounts to flowing the point $z$ for time $+\xi$ along the vector field $\pm \V$.
See \autoref{250303130219}.

The compositions $\Phi_0^+ \circ (\Phi_0^-)^\inv$ and $\Phi_0^- \circ (\Phi_0^+)^\inv$ are nontrivial local automorphisms in the $(z,\xi)$-space for all nonzero $z$ and sufficiently small $\xi$.
In words, given a point $(z, \xi)$, applying the composition $\Phi_0^+ \circ (\Phi_0^-)^\inv$ amounts to first flowing $z$ to a point $w$ for time $\xi$ along the vector field $-\V$, and then flowing $w$ to a point $z'$ for time $-\xi$ along the vector field $+\V$.
The latter operation is equivalent to flowing $w$ for time $+\xi$ along the vector field $-\V$.
So, altogether, we find that applying the composition $\Phi_0^+ \circ (\Phi_0^-)^\inv$ amounts to flowing $z$ to a point $z'$ for time $2\xi$ along the vector field $-\V$.
Similarly, applying the composition $\Phi_0^- \circ (\Phi_0^+)^\inv$ amounts to flowing $z$ to a point $z''$ for time $2\xi$ along the vector field $+\V$.

\begin{figure}[t]
\centering
\begin{tikzpicture}
\begin{scope}
\fill [grey] (-3,-2) rectangle (3,2);
\node at (0,2) [above] {$\CC^\ast_z \times \CC_\xi$};
\draw [ultra thick] (-3,0) -- (3,0) node [right] {$\CC^\ast_z$};
\begin{scope}
\clip (-3,-2) rectangle (3,2);
\draw[thick, fill=blue, fill opacity = 0.3, blue, dashed] (0,0) to[closed, curve through =
	{ (0.5,1) (1,1) (3,1) (4,0) (3,-1) (1,-1) }] (0.5,-1);
\draw[thick, fill=blue, fill opacity = 0.3, blue, dashed] (0,0) to[closed, curve through =
	{ (-0.5,1) (-1,1) (-3,1) (-4,0) (-3,-1) (-1,-1) }] (-0.5,-1);
\end{scope}
\draw [ultra thick, white] (0,-2) -- (0,2);
\filldraw [darkgreen] (1,0.5) circle (2pt) node [right] {$(z,\xi)$};
\node [blue] at (-2.5,0.5) {$\Xi$};
\end{scope}

\begin{scope}[yshift=-5cm, xshift=-4cm]
\fill [grey] (-2,-2) rectangle (2,2);
\node at (-2,2) [below right] {$\CC_z$};
\draw [dashed] (1.5,-1.25) to [out=170,in=270] (-1,0) to [out=90,in=190] (1.5,1.25);
\filldraw [darkgreen] (-1,0) circle (2pt) node [left] {$z$};
	\begin{scope}
	\clip (-1.5,-0.5) rectangle (1,1.5);
	\draw [ultra thick, ->-=0.5, darkgreen] (-1,0) to [out=90,in=190] (1.5,1.25);
	\end{scope}
\filldraw [darkgreen] (1,1.18) circle (2pt) node [below right] {$z'$};
\draw [thin, ->, darkgrey] (-0.8,0) to [out=90,in=188] node [below right] {$\xi$} (1,1);
\end{scope}

\begin{scope}[yshift=-5cm, xshift=4cm]
\fill [grey] (-2,-2) rectangle (2,2);
\node at (2,2) [below left] {$\CC_z$};
\draw [dashed] (1.5,-1.25) to [out=170,in=270] (-1,0) to [out=90,in=190] (1.5,1.25);
\filldraw [darkgreen] (-1,0) circle (2pt) node [left] {$z$};
\filldraw [darkgreen] (-1,0) circle (2pt) node [left] {$z$};
	\begin{scope}
	\clip (-1.5,0.5) rectangle (1,-1.5);
	\draw [ultra thick, ->-=0.5, darkgreen] (-1,0) to [out=-90,in=170] (1.5,-1.25);
	\end{scope}
\filldraw [darkgreen] (1,-1.18) circle (2pt) node [above right] {$z''$};
\draw [thin, ->, darkgrey] (-0.8,0) to [out=-90,in=172] node [above right] {$\xi$} (1,-1);
\end{scope}

\draw [->] (-2.5,-1.8) -- node [left] {$\Phi^+_0$} (-3,-3.5);
\draw [->] (2.5,-1.8) -- node [right] {$\Phi^-_0$} (3,-3.5);

\end{tikzpicture}
\caption{Picturing the local holomorphic change of coordinates $\Phi_0^\pm$.}
\label{250303130219}
\end{figure}
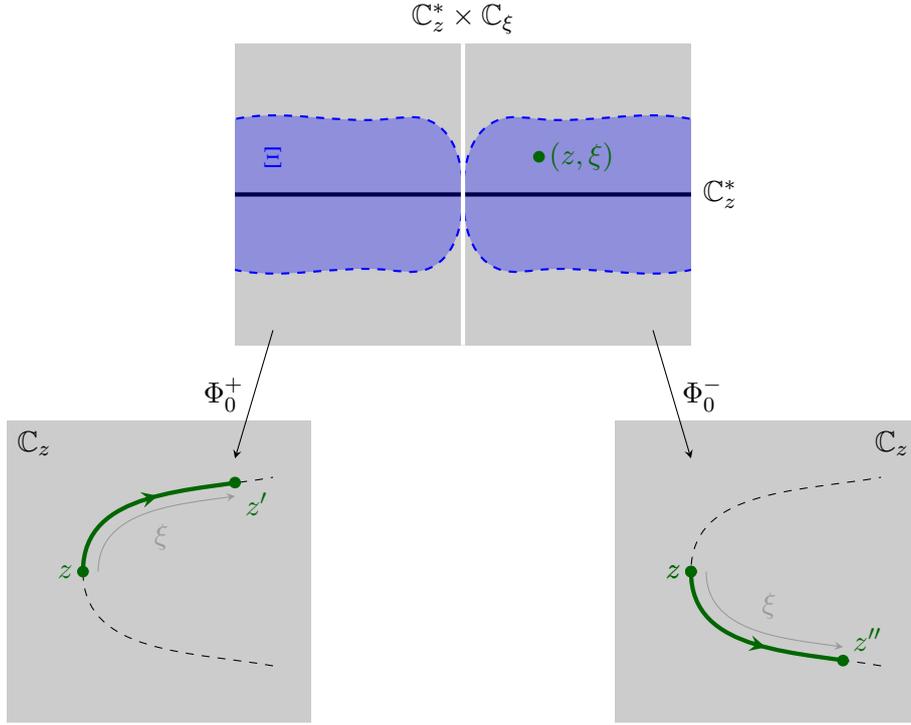

Let us now transform the first equation in \eqref{250303130505} using $\Phi_0^+$ and the second equation using $\Phi_0^-$.
Thus, if we write $\tilde{f}_\pm = f_\pm \circ \Phi_0^\pm$, $\tilde{\F}_\pm = \F_\pm \circ \Phi_0^\pm$, and $\tilde{a}_\pm = a_\pm \circ \Phi_0^\pm$, then the Initial Value Problem \eqref{250303130505} becomes simply
\begin{eqntag}
\label{250303130745}
\begin{cases}
	- \del_u \tilde{f}_+ = \tilde{\F}_+
\\	- \del_u \tilde{f}_- = \tilde{\F}_-
\end{cases}
\qtext{such that}
	\tilde{f}_\pm (w,0) = \tilde{a}_\pm (w)
\fullstop
\end{eqntag}
Consequently, it can be converted into a pair of integral equations $\tilde{f}_\pm = \tilde{a}_\pm - \int_0^u \tilde{\F}_\pm \d{s}$.
However, we must be careful because the righthand side involves not just $\tilde{f}_\pm$ but also the cross terms $\tilde{f}^\dagger_+ = f_+ \circ \Phi_0^-$ and $\tilde{f}^\dagger_- = f_- \circ \Phi_0^+$.
So, more explicitly, the pair of integral equations is as follows:
\begin{eqn}
\begin{cases}
\displaystyle
	\tilde{f}_+ (w,u)
		= \tilde{a}_+ (w) - \int_0^u \tilde{\F}_+ \Big( w,u; \tilde{f}_+ (w,u), \tilde{f}^\dagger_- (w,u) \Big) \d{s}
\fullstop{,}
\\[6pt]
\displaystyle
	\tilde{f}_- (w,u)
		= \tilde{a}_- (w) - \int_0^u \tilde{\F}_- \Big( w,u; \tilde{f}^\dagger_+ (w,u), \tilde{f}_- (w,u) \Big) \d{s}
\fullstop
\end{cases}
\end{eqn}
Undoing the coordinate transformations, we find the following pair of integral equations in the original variables.

\begin{lemma}
\label{250305163610}
There is a neighbourhood $\Xi \subset \CC^\ast_z \times \CC_\xi$ of $\xi = 0$ where the Initial Value Problem \eqref{250303130505} is equivalent to the following integral equation:
\begin{eqn}
\begin{cases}
\displaystyle
	f_+ (z,\xi)
		= a_+ (z) - \int_0^\xi \F_+ \Big( z^+_{\xi - s}, s; f_+ (z^+_{\xi - s}, s), f_- (z^+_{\xi - s}, s) \Big) \d{s}
\fullstop{,}
\\[6pt]
\displaystyle
	f_- (z,\xi)
		= a_- (z) - \int_0^\xi \F_- \Big( z^-_{\xi - s}, s; f_+ (z^-_{\xi - s}, s), f_- (z^-_{\xi - s}, s) \Big) \d{s}
\fullstop{,}
\end{cases}
\end{eqn}
where $z^\pm_{\xi - s}$ is the unique point in $\CC^\ast_z$ satisfying the identity
\begin{eqn}
	\xi - s = \pm \int_{z}^{z^\pm_{\xi-s}} \lambda
\fullstop
\end{eqn}
\end{lemma}
In the same way as before, we may interpret the points $z^+_{\xi}$ and $z^-_{\xi}$ as the target points of a pair of parameterised curves $\gamma^+, \gamma^- : [0,\xi] \to \CC^\ast_z$ with the same source $z$ but opposite central charges $\Z (\gamma^+) = - \Z (\gamma^-)$.
However, the paths $\gamma^+$ and $\gamma^-$ are necessarily distinct so the point $(z,\xi)$ cannot be interpreted simply as an element of a fundamental groupoid such as $\Pi_1 (\CC^\ast_z)$.
In order to make global geometric sense of this coupled system of integral equations, it needs to be reformulated on a more sophisticated geometric space that encodes pairs of parameterised curves on $\CC^\ast_z$ with equal source and opposite central charge.
Such a space is given as the fibre product of fundamental groupoids and the next section is devoted to this geometric construction.

\section{The Borel Space}
\label{250304152932}

In this section, we introduce this paper's main geometric constructions.

\subsection{The Unfolded Borel Space and its Quotient}

In \autoref{250315131538}, we explained how the vector field $\V$ appearing on the lefthand side of the integro-differential system \eqref{250306084618} equips the fundamental groupoid of the $z$-plane with a rich extra structure via the central charge $\Z$ obtained by integrating the holomorphic differential form $\lambda$ dual to $\V$.
As a matter of fact, the lefthand side of \eqref{250306084618} contains two different vector fields involving $+\V$ and $-\V$, and this difference in sign is a fundamental part of the structure of \eqref{250306084618}, made clear by the explicit discussion in \autoref{250315132104}.
In light of this, we must work with both differentials $+ \lambda$ and $-\lambda$ side by side, interpreted as the two eigenvalues of the classical Jacobian $\J_0$ (or, more precisely, $\J_0(z) \d{z}$), and adapt our notation accordingly.

\paragraph{}\removespace
First, we introduce a central charge and a corresponding anchor for each of the eigenvalues of the classical Jacobian.

\begin{definition}[central charge and anchor maps]
\label{250315132917}
For $X = \CC_z$ or $X = \CC^\ast_z$, the holomorphic differential $\pm \lambda$ induces the \dfn{central charge} on $\Pi_1 (X)$ given by
\begin{eqn}
	\pm \Z : \Pi_1 (X) \to \CC
\qtext{given by}
	\pm \Z : \gamma \mapsto \pm \Z (\gamma) = \pm \int_\gamma \lambda
\fullstop
\end{eqn}
Accordingly, the \dfn{anchor maps} on $\Pi_1 (X)$ associated with the differentials $\pm \lambda$ are the holomorphic surjective maps
\begin{eqn}
	\varrho_\pm \coleq (\rm{s}, \pm \Z) : \Pi_1 (X) \to X \times \CC_\xi
\fullstop
\end{eqn}
\end{definition}

Thus, the anchor map $\varrho_+$ is the same as the one discussed in \autoref{250315133506}, whilst the anchor map $\varrho_-$ is only a slight variation involving a negative sign.
This means that all the properties of $\varrho_\pm$ are easily deducible from the discussion in \autoref{250315133506}.

\paragraph{}\removespace
At the same time, the presence of two distinct anchor maps on the same groupoid $\Pi_1 (X)$ is one of the most significant and nontrivial aspects of the geometric content of this paper.
Their presence is the consequence of the fact that the system \eqref{250306084618} consists of two equations, and the fact that these equations are coupled can be encoded in the following geometric structure.

\begin{definition}[unfolded Borel space]
\label{250217111758}
We define the \dfn{unfolded Borel space} $B$ to be the fibre product of the two anchor maps $\varrho_\pm : \Pi_1 (\CC_z) \to \CC_z \times \CC_\xi$:
\begin{eqntag}
	B \coleq \Pi_1 (\CC_z) \!\! \underset{\ \varrho_- \, \varrho_+}{\times} \!\! \Pi_1 (\CC_z)
\fullstop
\end{eqntag}
Similarly, we define the \dfn{unfolded Borel covering space} $\tilde{B}$ to be the fibre product of the two anchor maps $\varrho_\pm : \Pi_1 (\CC^\ast_z) \to \CC^\ast_z \times \CC_\xi$:
\begin{eqntag}
	\tilde{B} \coleq \Pi_1 (\CC^\ast_z) \!\! \underset{\ {\varrho}_- \, {\varrho}_+}{\times} \!\! \Pi_1 (\CC^\ast_z)
\fullstop
\end{eqntag}
\end{definition}

\paragraph{The unfolded Borel space.}
Concretely, the unfolded Borel space is the set of all pairs $\bm{\gamma} \coleq (\gamma^+, \gamma^-)$ of paths on $\CC_z$ with the same source but opposite central charge:
\begin{eqn}
	B	= \set{ \bm{\gamma} = (\gamma^+, \gamma^-) \in \Pi_1 (\CC_z) \times \Pi_1 (\CC_z)
		~\Big|~ \Z (\gamma^+) = -\Z (\gamma^-)
			~\text{and}~ \rm{s} (\gamma^+) = \rm{s} (\gamma^-)}
\fullstop
\end{eqn}
Explicitly in the trivialisation $\Pi_1 (\CC_z) \cong \CC_u \times \CC_z$, the unfolded Borel space $B$ is the subspace of points $(u_1, z_1; u_2, z_2)$ in $\CC^4$ defined by the equations
\begin{eqn}
	z_1 = z_2
\qtext{and}
	\int_{z_1}^{z_1 + u_1} \!\!\!\! z^4 \d{z}
	= - \int_{z_2}^{z_2 + u_2} \!\!\!\! z^4 \d{z}
\qtext{i.e.}
	(z_1 + u_1)^5 - z_1^5 = - (z_2 + u_2)^5 + z_2^5
\fullstop
\end{eqn}
Writing $x,y$ for $u_1, u_2$, and $z$ for $z_1 = z_2$, these equations define a quintic surface in $\CC^3$:
\begin{eqntag}
\label{250217180306}
	B \cong \set{ (x,y,z) \in \CC^3 ~\Big|~ (x + z)^5 + (y + z)^5 = 2z^5 }
\fullstop
\end{eqntag}
A real slice of this quintic surface is depicted in \autoref{250222122854}.
This quintic surface in $\CC^3$ is singular with just one singular point at the origin $(x,y,z) = (0,0,0)$.
Consequently, $B$ is an algebraic surface with only one singular point $\bm{1_0} \coleq (1_0, 1_0) \in B$ representing the pair of constant paths at the origin of the $z$-plane.
The smooth locus of $B$,
\begin{eqn}
	B^\text{sm} 
		\coleq B \smallsetminus \set{ \bm{1_0} }
		\cong \set{ (x,y,z) \in \CC^3 ~\Big|~ (x + z)^5 + (y + z)^5 = 2z^5 } \smallsetminus \set{(0,0,0)}
\fullstop{,}
\end{eqn}
is a smooth algebraic surface, hence a two-dimensional holomorphic manifold.
Note that, using the change of variables $u = x+z, v = y + z, w = -z \sqrt[5]{2}$, the quintic surface \eqref{250217180306} is isomorphic to the Fermat quintic surface in $\CC^3$:
\begin{eqntag}
\label{250315135452}
	B \cong \set{ (u,v,w) \in \CC^3 ~\Big|~ u^5 + v^5 + w^5 = 0 }
\fullstop
\end{eqntag}
This isomorphism was mentioned in the Introduction, and a picture of the Fermat quintic surface \eqref{250315135452} is presented in \autoref{250312110535}.
However, the trivialisation \eqref{250217180306} will be more convenient for our purposes because in the variables $x,y,z$ are more straightforwardly interpretable as parameterising displacements of paths that start at $z$.
We summarise these facts in the following elementary lemma.

\begin{lemma}
\label{250217112814}
The unfolded Borel space $B$ is a singular algebraic surface isomorphic to the Fermat quintic surface in $\CC^3$.
It has only one singular point $\bm{1_0} \coleq (1_0, 1_0) \in B$ representing the pair of constant paths at the origin of the $z$-plane.
\end{lemma}

\paragraph{The unfolded Borel covering space.}
Similarly, the unfolded Borel covering space is concretely the set of all pairs $\bm{\gamma} = (\gamma^+, \gamma^-)$ of paths on $\CC^\ast_z$ with the same source but opposite charge:
\begin{eqn}
	\tilde{B}
		= \set{ \bm{\gamma} = (\gamma^+, \gamma^-) \in \Pi_1 (\CC^\ast_z) \times \Pi_1 (\CC^\ast_z)
		~\Big|~ \Z (\gamma^+) = - \Z (\gamma^-)
			~\text{and}~ \rm{s} (\gamma^+) = \rm{s} (\gamma^-)}
\fullstop
\end{eqn}
Explicitly in the trivialisation $\Pi_1 (\CC^\ast_z) \cong \CC_u \times \CC^\ast_z$, this fibre product is the subspace of points $(u_1, z_1; u_2, z_2)$ in $\CC^4$, with $z_1, z_2$ nonzero, defined by the analytic equations
\begin{eqn}
	z_1 = z_2
\qtext{and}
	\int_{z_1}^{e^{u_1} z_1} \!\!\!\! z^4 \d{z}
	= - \int_{z_2}^{e^{u_2} z_2} \!\!\!\! z^4 \d{z}
\qtext{i.e.}
	( e^{5 u_1} - 1) z_1^5 + ( e^{5 u_2} - 1) z_2^5 = 0
\fullstop
\end{eqn}
Since $z_1 = z_2$ is nonzero, we find that $\tilde{B}$ is isomorphic to a smooth analytic surface in $\CC^3$ by writing $x,y$ for $u_1, u_2$, and $z$ for $z_1 = z_2$, respectively:
\begin{eqntag}
\label{250223102438}
	\tilde{B} \cong \set{ (x,y,z) \in \CC^3 ~\Big|~ e^{5x} + e^{5y} = 2 \qtext{and} z \neq 0}
\fullstop
\end{eqntag}
A real slice of this analytic surface is presented in \autoref{250223102352}.
We summarise these facts in the following elementary lemma.

\begin{lemma}
\label{250217185011}
The unfolded Borel covering space $\tilde{B}$ is a two-dimensional holomorphic manifold isomorphic to an analytic surface in $\CC^3$.
\end{lemma}

The following fact is a simple consequence of \autoref{250210081550}.

\begin{lemma}
\label{250218172045}
The anchor map $\varrho : B \to \CC_z \times \CC_\xi$ is a holomorphic isomorphism near the embedding $\CC^\ast_z \inj B$.
That is, there is an open neighbourhood $\Omega \subset B$ of $\CC^\ast_z \inj B$ and an open neighbourhood $\Xi \subset \CC_z^\ast \times \CC_\xi$ of $\xi = 0$ such that the anchor map restricts to a holomorphic isomorphism $\varrho : \Omega \iso \Xi$.
Namely, let $\Omega_\pm \subset \Pi_1 (\CC_z) |_{\CC^\ast_z}$ and $\Xi \subset \CC_z^\ast \times \CC_\xi$ be open neighbourhoods of $\CC_z^\ast$ such that $\varrho_\pm : \Omega_\pm \iso \Xi$ as in \autoref{250210081550}.
Then we can take $\Omega$ to be the fibre product of $\varrho_\pm : \Omega_\pm \iso \Xi$; i.e.,
\begin{eqn}
	\varrho : \Omega \coleq \Omega_- \underset{\varrho_- \, \varrho_+}{\times} \Omega_+
	\iso \Xi
\fullstop
\end{eqn}
\end{lemma}

\paragraph{Involution.}
The negation map $\sigma : \CC_z \to \CC_z$ sending $z \mapsto -z$ is a holomorphic automorphism which naturally arises on the complex $z$-plane from its structure as the ramified double cover of the $\tau$-plane.
It lifts to a canonical automorphism on the fundamental groupoid
\begin{eqn}
	\sigma : \Pi_1 (X) \too \Pi_1 (X)
\qtext{sending}
	\gamma \mapstoo \sigma \circ \gamma
\fullstop{,}
\end{eqn}
for $X = \CC_z$ or $X = \CC^\ast_z$.
This automorphism is an involution (meaning $\sigma \circ \sigma = \id$) which intertwines the central charges of $\pm \lambda$ in the sense that 
\begin{eqntag}
\label{250315144427}
	\Z = - \sigma^\ast \Z
\qqtext{i.e.}
	\Z (\gamma) = -\Z (\sigma \circ \gamma)
\fullstop
\end{eqntag}
This is because $\sigma^\ast \lambda = - \lambda$ so $\Z (\gamma) = \int_\gamma \lambda = - \int_\gamma \sigma^\ast \lambda = - \int_{\sigma \circ \gamma} \lambda = - \Z (\sigma \circ \gamma)$.
Therefore, $\sigma$ induces an involution automorphism on the unfolded Borel space as well as the unfolded Borel covering space:
\begin{eqntag}
\label{250217133212}
	\sigma : B \to B \qtext{and} \sigma : \tilde{B} \to \tilde{B}  \qtext{sending}  \bm{\gamma} = (\gamma^+, \gamma^-) \mapsto \check{\bm{\gamma}} \coleq (\sigma \circ \gamma^-, \sigma \circ \gamma^+)
\fullstop
\end{eqntag}
To check this, we just need to justify that the $\check{\bm{\gamma}}$ is also an element of $B$ (or $\tilde{B}$).
This is a simple calculation, using the identities \eqref{250315144427} and the fact that $\Z (\gamma^+) = - \Z (\gamma^-)$, which goes like this: $\Z (\sigma \circ \gamma^-) = - \Z (\gamma^-) = \Z (\gamma^+) = - \Z (\sigma \circ \gamma^+)$.
In the trivialisation \eqref{250217180306}, the involution automorphism $\sigma : B \to B$ comes from the global involution
\begin{eqn}
	\sigma : \CC^3 \to \CC^3 \qtext{sending} (x,y,z) \mapsto -(y,x,z)
\fullstop
\end{eqn}

\paragraph{The Borel space.}
Since $B$ is an algebraic variety, the quotient $B/\sigma$ is a well-defined algebraic variety that can be described using the ring of invariant functions.
Let $\CC [B]$ denote the ring of algebraic functions on the unfolded Borel space, so that 
\begin{eqn}
	B = \Spec \big( \CC [B] \big)
\qtext{and}
	\CC [B] \cong \frac{ \CC [x,y,z] }{\inner{(x+z)^5 + (y+z)^5 - 2z^5}}
\fullstop
\end{eqn}
Consider the subalgebra of $\sigma$-invariant algebraic functions on $B$; i.e.,
\begin{eqntag}
	\CC [B]^\sigma \coleq \set{ f \in \CC [B] ~\big|~ \sigma^\ast f = f }
\fullstop
\end{eqntag}
The quotient algebraic surface $M \coleq B/\sigma$ is defined as the spectrum of the ring $\CC [B]^\sigma$.
Let us also observe that the only fixed point of the involution $\sigma : B \to B$ is the singular point $\bm{1_0} = (1_0, 1_0)$ which corresponds to the origin $(0,0,0) \in \CC^3$ in the trivialisation.
This means that the quotient $M^\textup{sm} \coleq B^\textup{sm} / \sigma$ of the smooth locus of $B$ by $\sigma$ is a smooth algebraic surface.
Likewise, $\tilde{B}$ is a holomorphic manifold and the involution $\sigma : \tilde{B} \to \tilde{B}$ has no fixed points, so the quotient $\tilde{M} \coleq \tilde{B} / \sigma$ is also a holomorphic manifold.

\begin{definition}[Borel space]
\label{250217133244}
We define the \dfn{Borel space} for the deformed Painlevé I equation as the quotient algebraic surface
\begin{eqntag}
	M \coleq B / \sigma = \Spec \big( \CC [B]^\sigma \big)
\fullstop
\end{eqntag}
Similarly, we define the \dfn{Borel covering space} for the deformed Painlevé I equation as the two-dimensional holomorphic manifold
\begin{eqntag}
	\tilde{M} \coleq \tilde{B} / \sigma
\fullstop
\end{eqntag}
\end{definition}

\paragraph{The critical elements.}
Recall that a critical path $\gamma \in \Pi_1 (\CC_z)$ in the sense of \autoref{250208170930} is one that terminates at the origin.
These special elements determine special elements in the (unfolded) Borel space defined as follows.

\begin{definition}[critical elements]
\label{250209230849}
An element $\bm{\gamma} = (\gamma^+, \gamma^-) \in B$ is called \dfn{critical} if at least one of $\gamma^+$ or $\gamma^-$ is a critical path.
We denote the closed subset of critical elements and its complement in $B$ respectively by
\begin{eqn}
	\Gamma \coleq \set{ \bm{\gamma} = (\gamma^+, \gamma^-) \in B ~\Big|~ \rm{t} (\gamma^+) = 0 \qtext{or} \rm{t} (\gamma^-) = 0 }
\qqtext{and}
	B^\ast \coleq B \smallsetminus \Gamma
\fullstop
\end{eqn}
\end{definition}
Clearly, $\Gamma$ is the union of two closed subsets corresponding to whether it is the first or the second path that is critical:
\begin{eqn}
	\Gamma = \Gamma^+ \cup \Gamma^-
\qtext{where}
	\Gamma^\pm \coleq \set{ \bm{\gamma} = (\gamma^+, \gamma^-) \in B ~\Big|~ \rm{t} (\gamma^\pm) = 0 }
\fullstop
\end{eqn}

\begin{lemma}
\label{250217122922}
The subset of critical elements $\Gamma \subset B$ is a closed algebraic subset which has two irreducible components $\Gamma^+, \Gamma^-$ whose intersection is the singular locus of $B$:
\begin{eqn}
	\Gamma^+ \cap \Gamma^- = \set{ \bm{1_0} } = B \smallsetminus B^\textup{sm}
\fullstop
\end{eqn}
Furthermore, the involution $\sigma : B \to B$ exchanges these irreducible components:
\begin{eqn}
	\sigma (\Gamma^\pm) = \Gamma^\mp
\fullstop
\end{eqn}
\end{lemma}

\begin{proof}
These properties become evident in the trivialisation \eqref{250217180306}.
Indeed, the two irreducible components of the critical locus are given by
\begin{eqns}
	\Gamma^+ &\cong \set{ (x,y,z) \in \CC^3 ~\big|~ (y+z)^5 = 2z^5 }
\\
	\Gamma^- &\cong \set{ (x,y,z) \in \CC^3 ~\big|~ (x+z)^5 = 2z^5 }
\fullstop
\end{eqns}
The real slices of these subsets are depicted in \autoref{250222132221}.
\end{proof}

The source-fibrewise universal cover $\nu : \Pi_1 (\CC^\ast_z) \to \Pi_1 (\CC_z)$ induces a holomorphic map $\nu : \tilde{B} \to B$ which is again neither injective nor surjective, but if we restrict its codomain away from the locus of critical elements then we get a holomorphic surjective submersion:
\begin{eqn}
	\nu : \tilde{B} \to B^\ast
\fullstop
\end{eqn}

\subsection{The Central Charge}

The spaces $B$ and $\tilde{B}$ as well as $M$ and $\tilde{M}$ inherit several canonical maps which are important ingredients in our constructions.

\paragraph{}\removespace
First, as fibre products, the unfolded Borel space $B$ and the unfolded Borel covering space $\tilde{B}$ fit into the following commutative diagrams:
\begin{eqntag}
\label{250217133208}
\begin{tikzcd}
		B
			\ar[r, "\pr_-"]
			\ar[d, "\pr_+"']
	&	\Pi_1 (\CC_z)
			\ar[d, "\varrho_-"]
\\		\Pi_1 (\CC_z)
			\ar[r, "\varrho_+"']
	&	\CC_z \times \CC_\xi
\end{tikzcd}
\qqtext{and}
\begin{tikzcd}
		\tilde{B}
			\ar[r, "\pr_-"]
			\ar[d, "\pr_+"']
	&	\Pi_1 (\CC^\ast_z)
			\ar[d, "\varrho_-"]
\\		\Pi_1 (\CC^\ast_z)
			\ar[r, "\varrho_+"']
	&	\CC^\ast_z \times \CC_\xi
\fullstop
\end{tikzcd}
\end{eqntag}
Here, the maps $\pr_\pm$ are the projections on the first and the second factor of $\Pi_1 (X) \times \Pi_1 (X)$ for $X = \CC_z$ or $X = \CC^\ast_z$.
By the commutativity of the diagrams, the two equal ways of composing the arrows shows that $B$ and $\tilde{B}$ come equipped with their own \dfn{anchor maps}:
\begin{eqn}
	\varrho \coleq \varrho_\pm \circ \pr_\pm : B \to \CC_z \times \CC_\xi
\qqtext{and}
	\varrho \coleq \varrho_\pm \circ \pr_\pm : \tilde{B} \to \CC^\ast_z \times \CC_\xi
\fullstop
\end{eqn}
By severe abuse of notation, we have denoted them by the same letter which also coincides with the notation for the anchor map encountered in \autoref{250315133506}.
The distinction will also be clear from the context or clarified in words when necessary.
Note that these anchor maps are holomorphic surjections, and $\varrho : B \to \CC_z \times \CC_\xi$ is in fact an algebraic map.

\paragraph{}\removespace
Taking the components of each anchor map $\varrho_\pm = (\rm{s}, \pm \Z)$, we find that $B$ and $\tilde{B}$ also have their own \dfn{source maps} and \dfn{central charges}:
\begin{eqn}
	\rm{s} : B \to \CC_z
\qtext{and}
	\rm{s} : \tilde{B} \to \CC_z
\qqtext{as well as}
	\Z : B \to \CC_z
\qtext{and}
	\Z : \tilde{B} \to \CC_z
\fullstop{,}
\end{eqn}
caveated with the same comment regarding the abuse of notation, defined by
\begin{eqn}
	\rm{s} \coleq \rm{s} \circ \pr_\pm
\qtext{and}
	\Z \coleq \pm \Z \circ \pr_\pm
\qqtext{so that}
	\varrho = (\rm{s}, \Z)
\fullstop
\end{eqn}
Explicitly, the source and the central charge of an element $\bm{\gamma} = (\gamma^+, \gamma^-)$ in $B$ (or $\tilde{B}$) is
\begin{eqn}
	\rm{s} (\bm{\gamma}) = \rm{s} (\gamma^+) = \rm{s} (\gamma^-)
\qtext{and}
	\Z (\bm{\gamma}) = + \Z (\gamma^+) = - \Z (\gamma^-)
\fullstop
\end{eqn}
The fact these equalities hold is the defining feature of the geometric spaces $B$ and $\tilde{B}$.

\paragraph{}\removespace
On the other hand, the fundamental groupoid's target map $\rm{t} : \Pi_1 (X) \to X$ does not fit into the above commutative diagrams, so $\rm{t} \circ \pr_+ \neq \rm{t} \circ \pr_-$.
It follows that $B$ and $\tilde{B}$ are each equipped with two distinct \dfn{target maps} to the $z$-plane, as depicted in \autoref{250223164824}:
\begin{eqn}
	\rm{t}_\pm \coleq \rm{t} \circ \pr_\pm : B \to \CC_z
\qtext{and}
	\rm{t}_\pm \coleq \rm{t} \circ \pr_\pm : \tilde{B} \to \CC^\ast_z
\fullstop
\end{eqn}
Let us also note that there is an obvious closed embedding of the (punctured) complex $z$-plane into the unfolded Borel (covering) space (see \autoref{250222132221}):
\begin{eqn}
	\CC_z \inj B \qtext{and} \CC^\ast_z \inj \tilde{B} 
		\qtext{given by} z \mapsto \bm{1_z} = (1_z, 1_z)
\fullstop
\end{eqn}

\paragraph{}\removespace
Observe next that the involution $\sigma$ on $B$ (or $\tilde{B}$) preserves the central charge $\Z$ on $B$ (or $\tilde{B}$) in the sense that $\sigma^\ast \Z = \Z$.
Indeed: $\Z \big(\sigma (\bm{\gamma}) \big) 	= \pm \Z (\sigma \circ \gamma^\mp) = \pm \Z (\gamma^\mp) = \Z (\bm{\gamma})$.
Therefore, the central charge descends to the quotient by the involution $\sigma$, defining the \dfn{central charge} maps on the Borel space $M$ and the Borel covering space $\tilde{M}$:
\begin{eqn}
	\Z : M \to \CC_\xi
\qqtext{and}
	\Z : \tilde{M} \to \CC_\xi
\fullstop
\end{eqn}
Note that $\Z : \tilde{M} \to \CC_\xi$ is surjective and holomorphic, whilst $\Z : M \to \CC_\xi$ is surjective and algebraic.

On the other hand, the source maps $\rm{s} : B \to \CC_z$ and $\rm{s} : \tilde{B} \to \CC^\ast_z$ do not descend to the quotients by the involution because $\sigma$ comes from the negation map $z \mapsto -z$.
However, if we also take the quotient of the $z$-plane by the negation map, which is naturally the $\tau$-plane (cf. \eqref{250206214139}), then these source maps descend to surjective maps to the $\tau$-plane:
\begin{eqn}
	\rm{s} : M \to \CC_\tau
\qqtext{and}
	\rm{s} : \tilde{M} \to \CC^\ast_\tau
\fullstop
\end{eqn}
We likewise call them the \dfn{source maps} on the Borel space and the Borel covering space.

\begin{figure}[p]
\begin{adjustwidth}{-2cm}{-1.5cm}
\centering
\begin{subfigure}{0.5\textwidth}
\includegraphics[width=\textwidth]{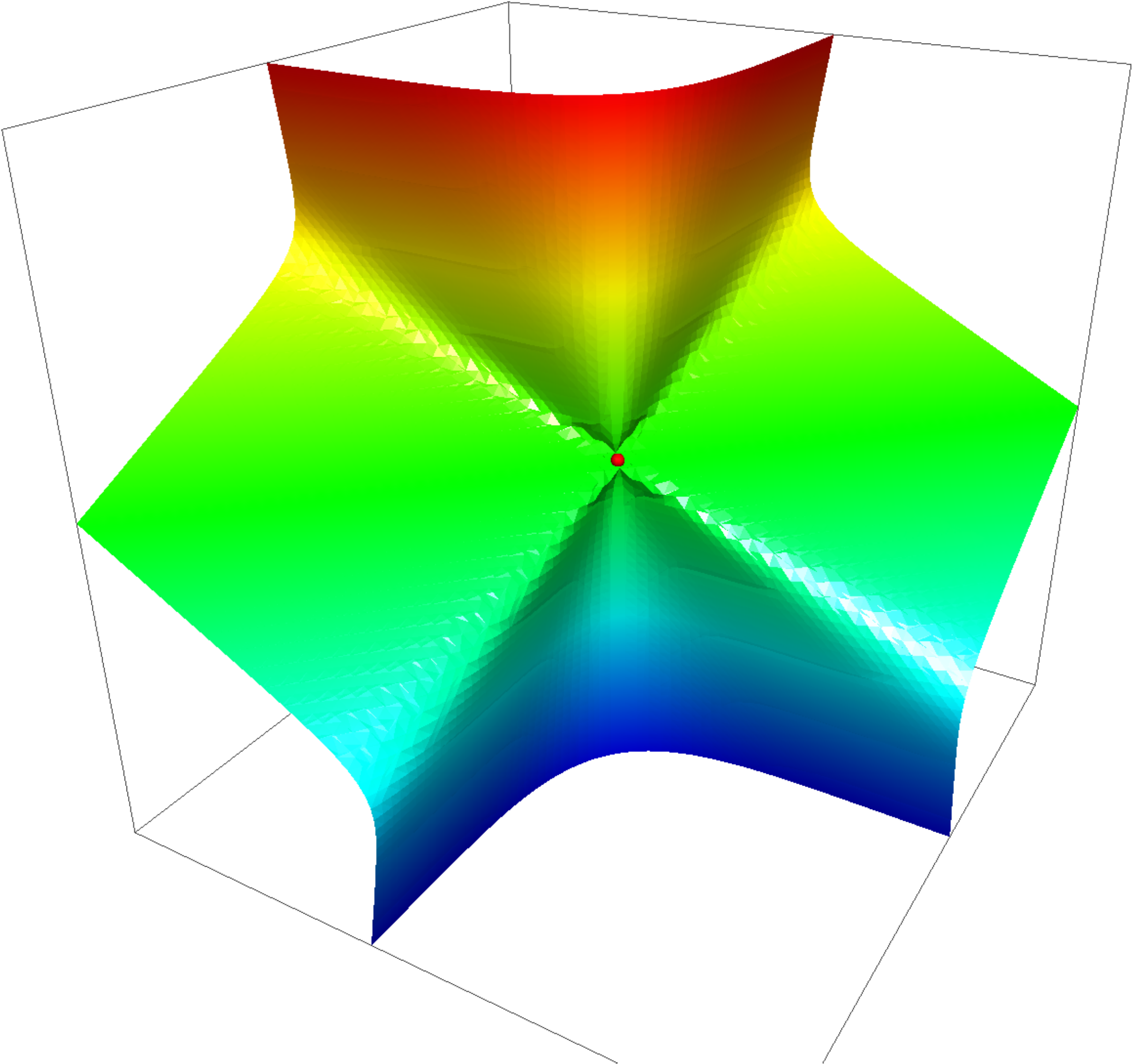}
\caption{}
\label{250222122854}
\end{subfigure}
\begin{subfigure}{0.5\textwidth}
\includegraphics[width=\textwidth]{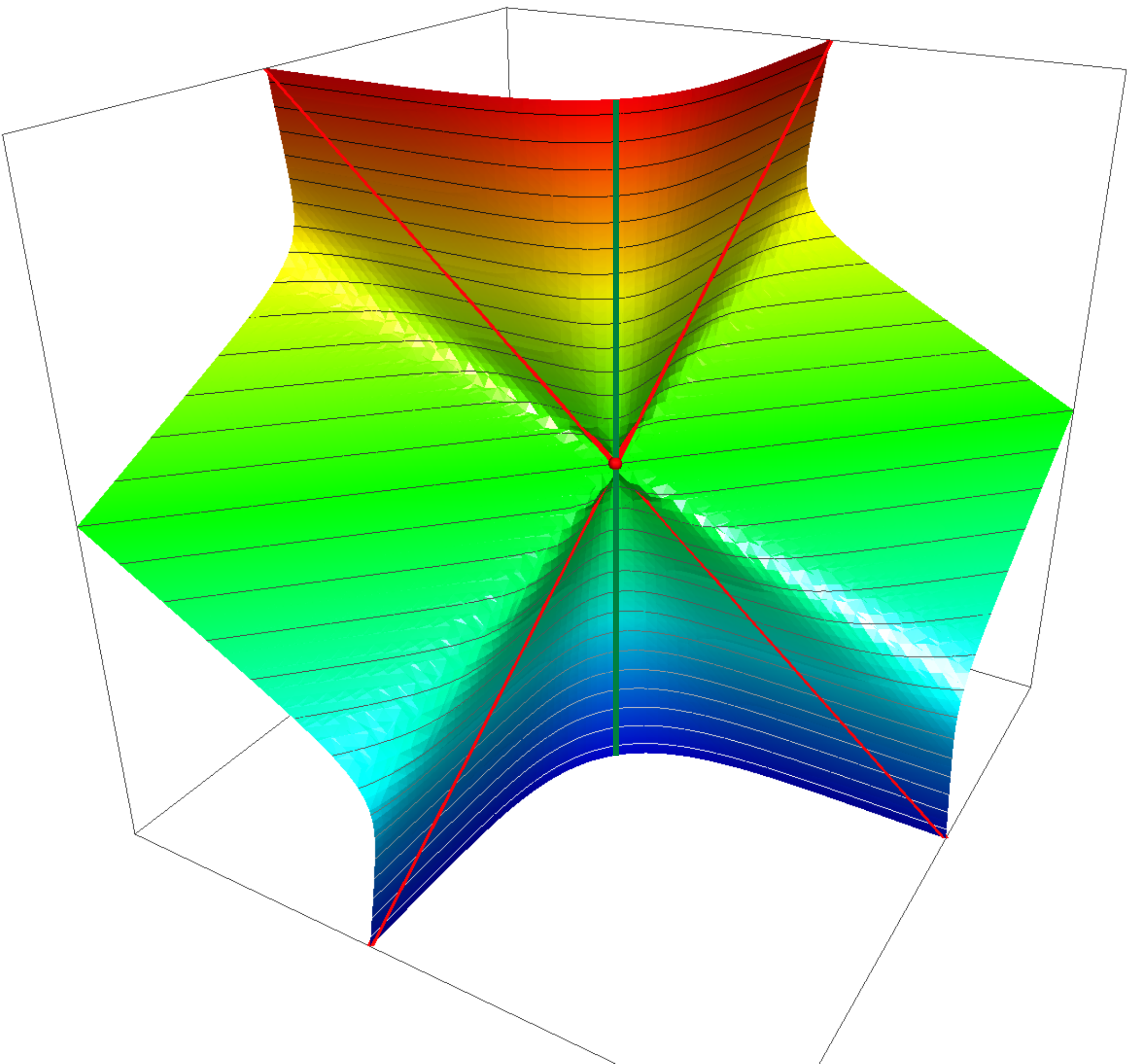}
\caption{}
\label{250222132221}
\end{subfigure}
\end{adjustwidth}
\caption{A real slice of the unfolded Borel space $B$ presented as the quintic $(x+z)^5 + (y+z)^5 = 2z^5$ in $\RR^3$ in the trivialisation \eqref{250217180306}.
For clarity, Figure (a) only indicates the singular point at the origin by a red dot, and Figure (b) includes additional information.
The two straight red lines belong to the surface: they are the two components of the locus of critical elements, $\Gamma = \Gamma^+ \cup \Gamma^-$, given by intersecting the quintic surface with the planes $x+z = 0$ and $y + z = 0$.
The vertical green line is the $z$-axis which also belongs to the quintic surface: it is the copy of (the real axis in the) the $z$-plane embedded into $B$ as the identity bisection $1_{\CC_z} \inj B$.
The plot also shows $z$-level sets: these are (the real slices of) the Borel surfaces $B_z$; a more detailed look at the $z=1$ level set is presented in \autoref{250222133101}.
}
\label{250222135810}

\bigskip\bigskip\bigskip

\begin{adjustwidth}{-2cm}{-1.5cm}
\centering
\begin{subfigure}{0.5\textwidth}
\includegraphics[width=\textwidth]{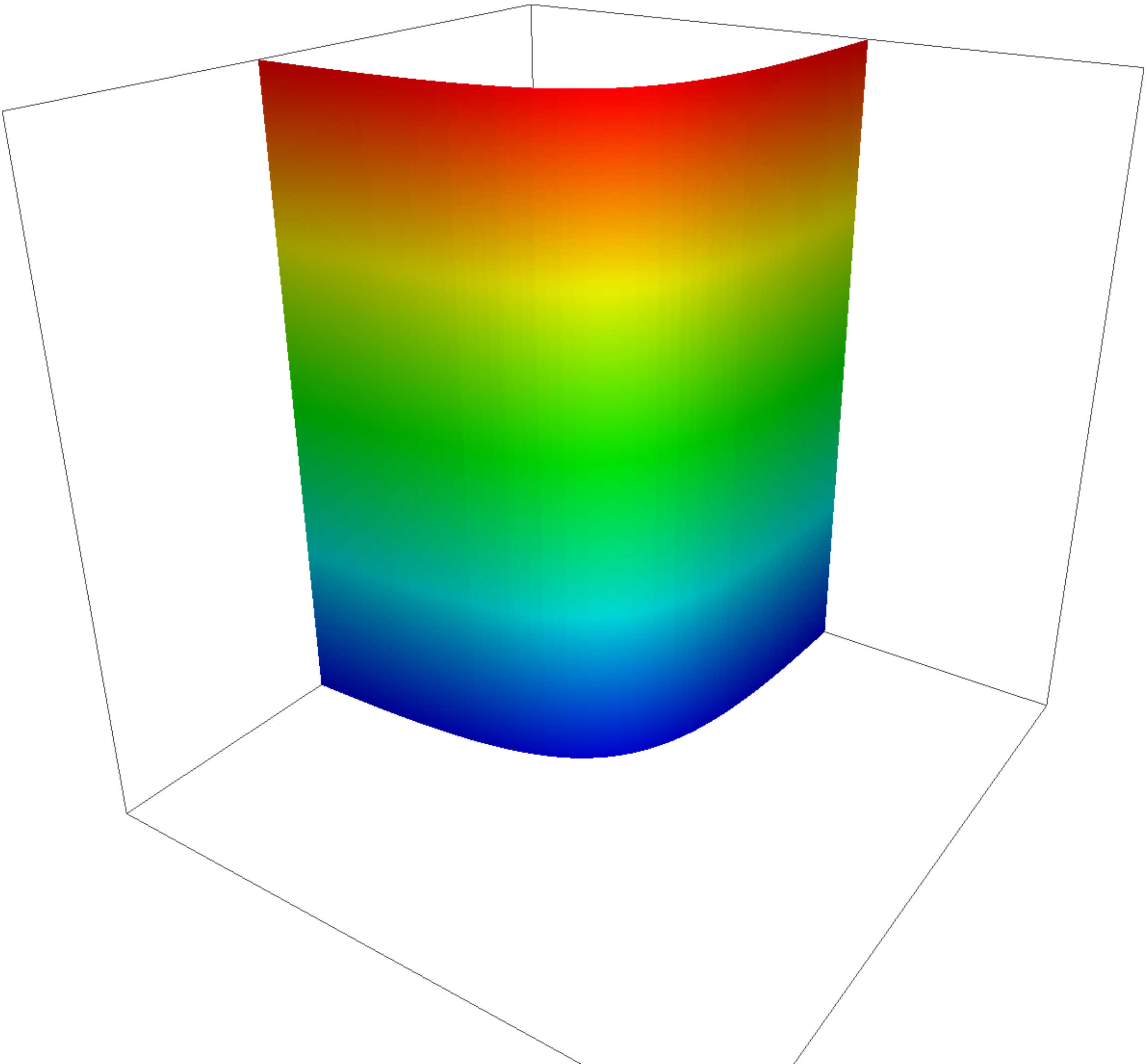}
\caption{}
\label{250223100322}
\end{subfigure}
\begin{subfigure}{0.5\textwidth}
\includegraphics[width=\textwidth]{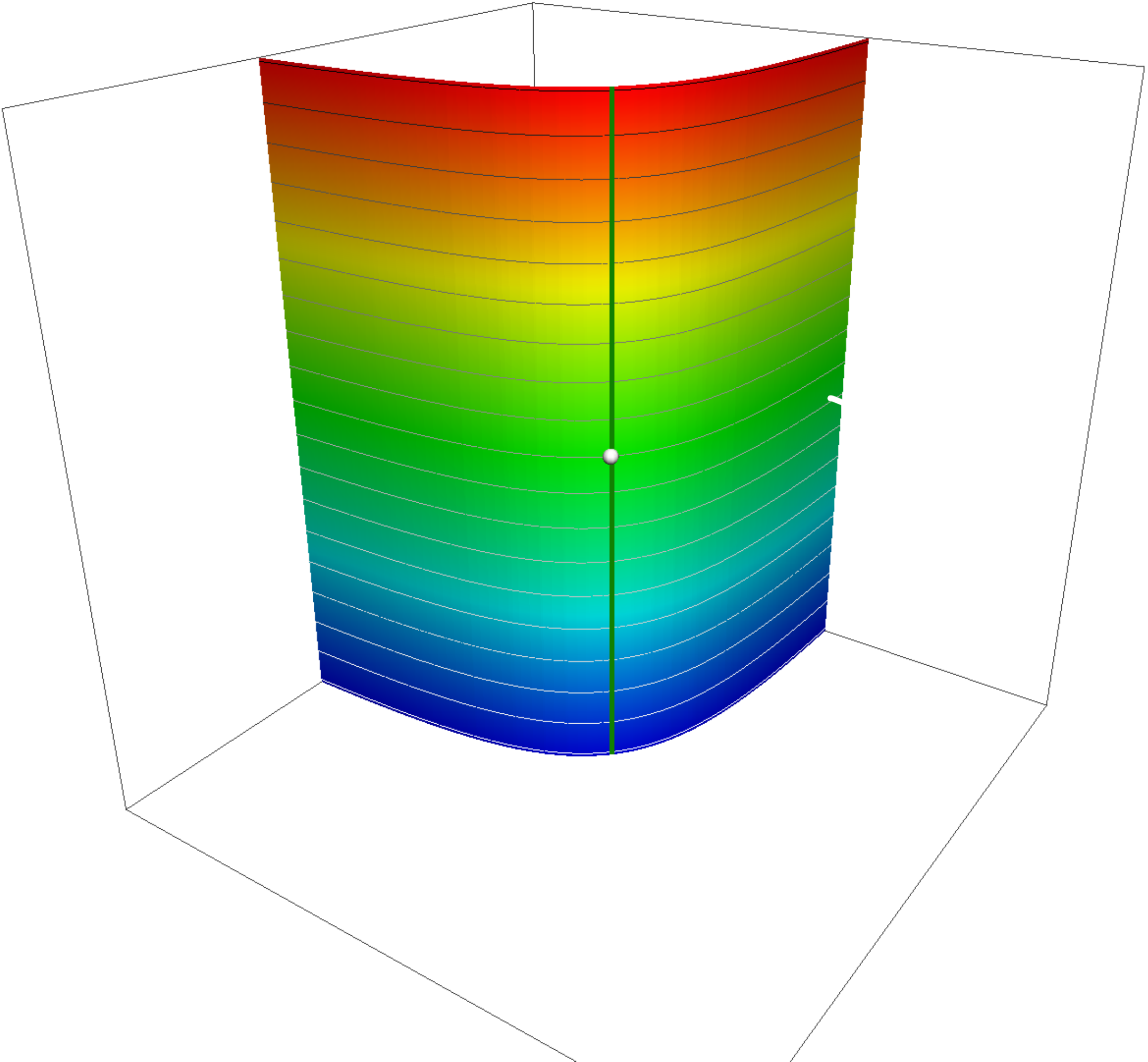}
\caption{}
\label{250223102144}
\end{subfigure}
\end{adjustwidth}
\caption{A real slice of the unfolded Borel covering space $\tilde{B}$ presented as the analytic surface $e^{5x} + e^{5y} = 2$ in $\RR^3$ in the trivialisation \eqref{250223102438}.
In Figure (b), we have plotted the $z$-level sets: these are (the real slices of) the Borel covering surfaces $\tilde{B}_z$.
The vertical green line is the punctured $z$-axis which belongs to the surface: it is the copy of (the real axis in the) the punctured $z$-plane embedded into $\tilde{B}$ as the identity bisection $1_{\CC^\ast_z} \inj B^\ast$.
}
\label{250223102352}
\end{figure}

\begin{figure}[t]
\centering
\includegraphics[width=0.4\linewidth]{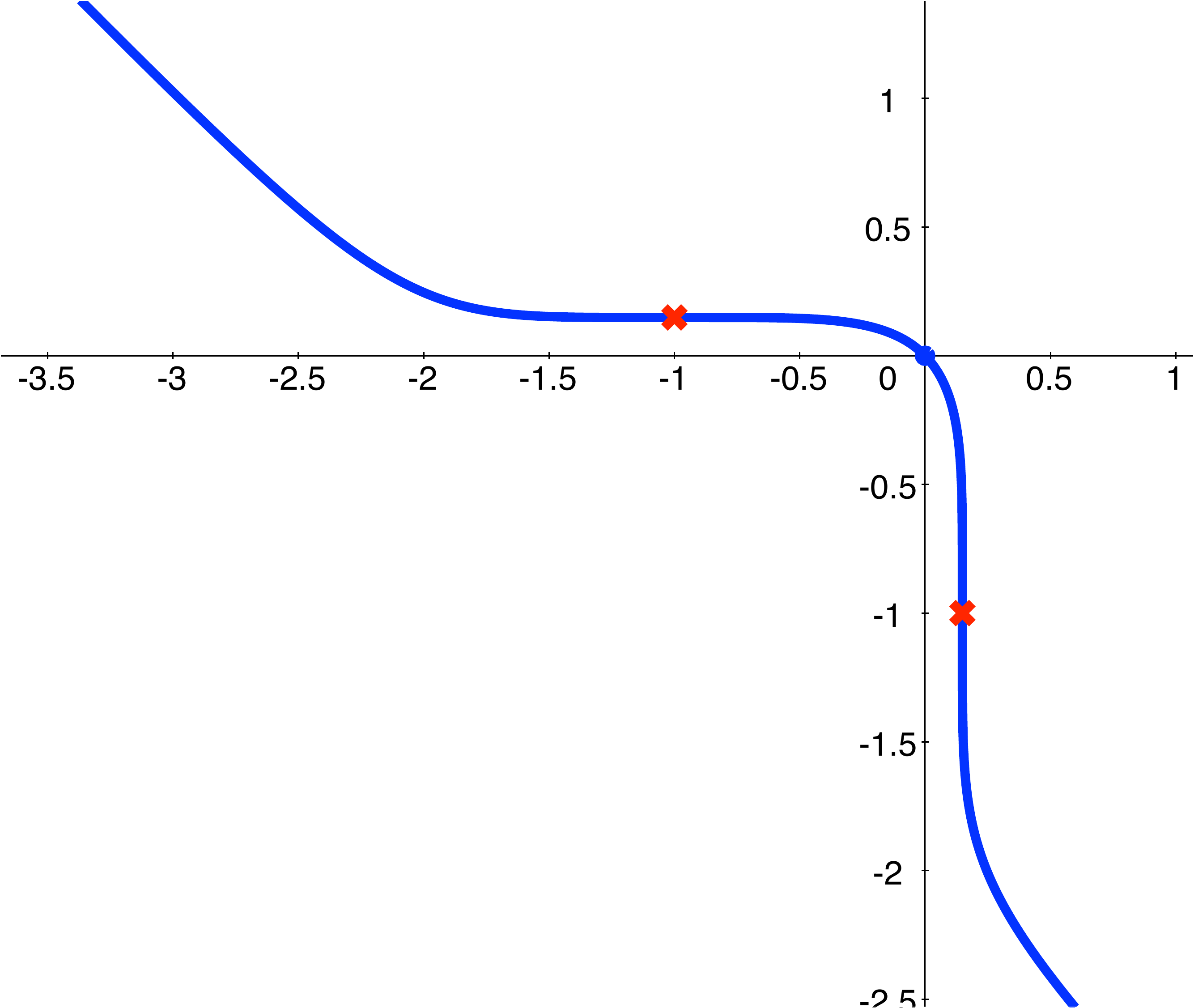}\quad
\includegraphics[width=0.4\linewidth]{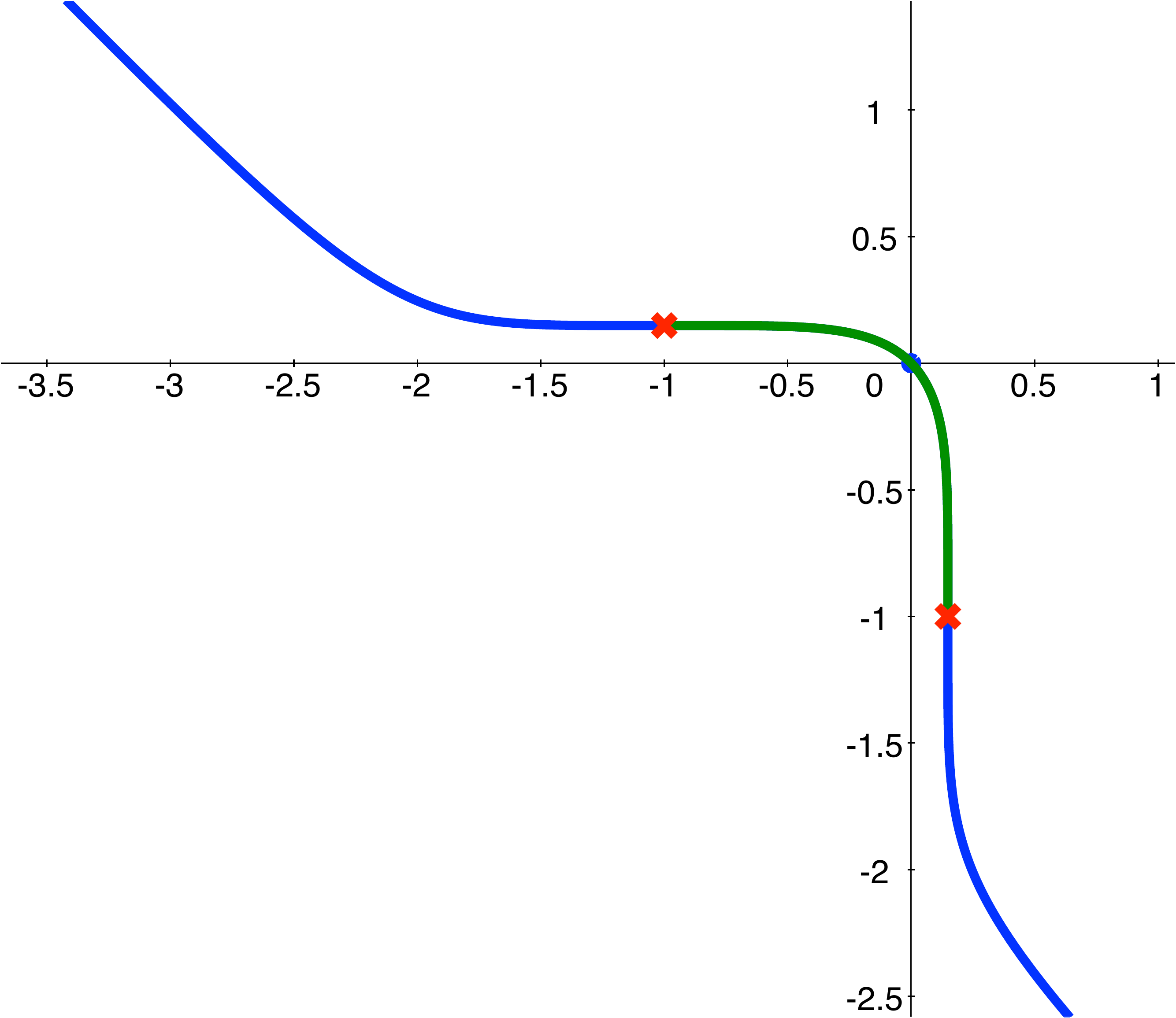}
\caption{A real slice of the Borel surface $B_z$ for $z = 1$ presented as the quintic $(x+1)^5 + (y+1)^5 = 2$ in $\RR^2$ in the trivialisation \eqref{250217180306}, pictured in blue.
The origin $(x,y) = 0$ is the distinguished `origin' point on $B_z$ which represents the constant path pair $\bm{1_z} = (1_z,1_z)$ at $z = 1$.
The two straight red lines are shadows (i.e., projections onto the $z=1$ plane) of the divisor $\Gamma$.
The divisor $\Gamma$ intersects $B_z$ in ten points, two of which are captured by this real slice, depicted with red crosses, and located at $(x,y) = (-1, \sqrt[5]{2} - 1)$ and $(x,y) = (\sqrt[5]{2} - 1, -1)$.
In the image on the right, the two critical geodesics emanating from the `origin' and captured by this real slices are shown in green.}
\label{250222133101}
\end{figure}

\subsection{Borel Surfaces}

The source maps on each of the four spaces introduced induce holomorphic fibrations which are a key part of the resurgent structure.
We give these fibres a special name and explore their properties in this subsection.

\paragraph{Unfolded Borel surfaces.}
The source map $\rm{s} : B \to \CC_z$ on the unfolded Borel space is an algebraic surjective submersion.
This means that the algebraic surface $B$ forms an algebraic fibration over the $z$-plane; i.e., fibres of $\rm{s}$ in $B$ are algebraic curves.

\begin{definition}[unfolded Borel surface]
\label{250217123915}
For any $z \in \CC_z$, we define the \dfn{unfolded Borel surface} $B_z$ to be the source fibre of $z$ under $\rm{s} : B \to \CC_z$:
\begin{eqntag}
	B_z \coleq \rm{s}^{-1} (z) \subset B
\fullstop
\end{eqntag}
\end{definition}

As an elementary consequence of \autoref{250217112814}, we obtain the following explicit description of each unfolded Borel surface.

\begin{lemma}
\label{250217175342}
For every nonzero $z$, the unfolded Borel surface $B_z$ is a smooth algebraic curve isomorphic to a smooth quintic plane curve of genus $6$:
\begin{eqntag}
\label{250217200557}
	B_z \cong \set{ (x,y) \in \CC^2 ~\Big|~ (x + z)^5 + (y + z)^5 = 2z^5 }
\fullstop
\end{eqntag}
In contrast, the unfolded Borel surface for $z = 0$ is a nodal curve isomorphic to the Fermat quintic curve in $\CC^2$; i.e., $B_{z = 0} \cong \set{ (x,y) \in \CC^2 ~\big|~ x^5 + y^5 = 0 }$.
\end{lemma}

A real slice of the unfolded Borel surface $B_z$ for $z = 1$ is pictured in \autoref{250222133101} with respect to the trivialisation \eqref{250217200557}.

\paragraph{Critical elements.}
For any $z \in \CC$, we also denote the subset of critical elements and their complement in $B_z$ respectively by
\begin{eqn}
	\Gamma_z \coleq \Gamma \cap B_z
\qqtext{and}
	B^\ast_z \coleq B_z \smallsetminus \Gamma_z
\fullstop
\end{eqn}

\begin{lemma}
\label{250315185310}
For every nonzero $z$, the intersection $\Gamma^+_z \cap \Gamma^-_z$ is empty and the subset $\Gamma_z = \Gamma^+_z \cup \Gamma^-_z \subset B_z$ consists of exactly ten points.
In the trivialisation \eqref{250217200557},
\begin{eqns}
	\Gamma^+_z &\cong \set{ (-z,y_k,z) ~\big|~ k = 1, \ldots, 5 }
\qtext{with}
	y_k = (\varepsilon^k - 1) z
\fullstop{;}
\\
	\Gamma^-_z &\cong \set{ (x_k,-z,z) ~\big|~ k = 1, \ldots, 5 }
\qtext{with}
	x_k = (\varepsilon^k - 1) z
\fullstop{,}
\end{eqns}
where $\varepsilon$ is a primitive root of $\varepsilon^5 = 2$.
\end{lemma}

\begin{proof}
Take a critical element $\bm{\gamma} = (\gamma^+, \gamma^-) \in \Gamma_z$, and work in the trivialisation of $B_z$ as the quintic curve \eqref{250217200557}.
If $\gamma^+$ is critical then $x + z = 0$, and if $\gamma^-$ is critical then $y + z = 0$.
Since $z$ is assumed to be nonzero, $(x + z)$ and $(y+z)$ cannot both be zero, so $\Gamma^+_z \cap \Gamma^-_z = \emptyset$.
Then it follows that $\Gamma^+_z$ has exactly five points because the condition $x + z = 0$ implies $(y + z)^5 = 2z^5$; i.e., $y = (\varepsilon^k - 1) z$ for $k = 1, \ldots, 5$ where $\varepsilon$ is a primitive root of $\varepsilon^5 = 2$.
Similarly, $\Gamma^-_z$ has exactly five points because $y + z = 0$ implies $(x + z)^5 = 2z^5$; i.e., $x = (\varepsilon^k - 1) z$ for $k = 1, \ldots, 5$.
\end{proof}

\paragraph{}\removespace
Next, we describe the restriction of the central charge to each unfolded Borel surface.

\begin{lemma}
\label{250217200331}
For any nonzero $z \in \CC_z$, the restriction of the central charge $\Z : B \to \CC_\xi$ to the unfolded Borel surface $B_{z}$ determines a fivefold ramified covering map
\begin{eqn}
	\Z_{z} \coleq \Z \big|_{B_{z}} : B_{z} \to \CC_\xi
\fullstop
\end{eqn}
Its ramification locus is the ten-point subset $\Gamma_z \subset B_z$.
Every ramification point has order $5$.
There are exactly two branch points located at $\xi_\pm \coleq \pm \tfrac{1}{30} z^5$.
The preimage of $\xi_\pm$ is the five-point subset $\Gamma^\pm_z \subset B_z$.
Thus, $(B_z, \Z_z)$ is an endless Riemann surface of algebraic type in the sense of \cite[Definition 1.6 and 1.7]{240622121512}.
\end{lemma}

A cartoon of $B_z$ as a fivefold ramified cover of $\CC_\xi$ is presented in \autoref{250304114549}.

\begin{proof}
To see this, we work in the trivialisation of $B_z$ as the quintic curve \eqref{250217200557}.
In these coordinates, the map $\Z_z$ becomes explicit:
\begin{eqn}
	\Z_z : (x, y) \mapstoo \xi 
		= \tfrac{1}{30} \big( z^5 - (x + z)^5 \big)
		= - \tfrac{1}{30} \big( z^5 - (y + z)^5\big)
\fullstop
\end{eqn}
Since $z$ is assumed to be nonzero, $(x + z)$ and $(y+z)$ cannot both be zero.
Therefore, the open sets $\set{ x + z \neq 0 }$ and $\set{ y + z \neq 0}$ form an open cover of $B_z$.

First, on the open subset of $B_z$ where $y + z \neq 0$, it follows from the Implicit Function Theorem that every point has a neighbourhood in which we can take $x$ as a local coordinate.
The differential of $\Z_z$ in this coordinate is then simply $\d \Z_z = - \tfrac{1}{6} (x + z)^4 \d{x}$, so there is ramification point of order $5$ at every point $(x,y) \in B_z$ satisfying $x + z = 0$.
But this locus of points is exactly the subset $\Gamma^+_z$.

Similarly, on the open subset of $B_z$ where $x + z \neq 0$, we can take $y$ as a local coordinate in a sufficiently small neighbourhood of any point.
The differential of $\Z_z$ in this coordinate is $\d \Z_z = + \tfrac{1}{6} (y + z)^4 \d{y}$, so there is a ramification point of order $5$ at every point $(x,y) \in B_z$ satisfying $y + z = 0$.
Again, this subset is nothing but $\Gamma^-_z$.
\end{proof}

\paragraph{Unfolded Borel covering surfaces.}
Similarly, the source map $\rm{s} : \tilde{B} \to \CC^\ast_z$ on the unfolded Borel covering space is a holomorphic surjective submersion (which is not algebraic), so $\tilde{B}$ is a holomorphic fibration over the punctured $z$-plane; i.e., the source fibres of $\tilde{B}$ are Riemann surfaces.

\begin{definition}[unfolded Borel covering surface]
\label{250315190402}
For any nonzero $z \in \CC^\ast_z$, we define the \dfn{unfolded Borel covering surface} $\tilde{B}_z$ to be the source fibre of $z$ under $\rm{s} : \tilde{B} \to \CC^\ast_z$:
\begin{eqntag}
	\tilde{B}_z \coleq \rm{s}^{-1} (z) \subset \tilde{B}
\fullstop
\end{eqntag}
\end{definition}

\begin{lemma}
\label{250315184627}
For every nonzero $z$, the unfolded Borel covering surface $\tilde{B}_z$ is a simply connected Riemann surface isomorphic to the following analytic plane curve:
\begin{eqntag}
\label{250315184705}
	\tilde{B}_z \cong \set{ (x,y) \in \CC^2 ~\Big|~ e^{5x} + e^{5y} = 2 }
\fullstop
\end{eqntag}
Furthermore, the restriction of the holomorphic surjective submersion $\nu : \tilde{B} \to B^\ast$ restricts to each fibre $\tilde{B}_z$ is the universal covering map of the source fibre $B_z$ punctured along the locus of critical elements $\Gamma_z$ with basepoint $\bm{1_z} = (1_z,1_z) \in B_z$: 
\begin{eqn}
	\nu : \tilde{B}_z = \widetilde{ B_z^\ast } \to B^\ast_z = B_z \smallsetminus \Gamma_z
\fullstop
\end{eqn}
\end{lemma}

\begin{lemma}
\label{250315193721}
For any nonzero $z \in \CC_z$, the restriction of the central charge $\Z : \tilde{B} \to \CC_\xi$ to the unfolded Borel covering surface $\tilde{B}_{z}$ is the universal cover of the twice-punctured complex plane $\CC_\xi \smallsetminus \set{ \xi_+, \xi_- }$ based at the origin, where $\xi_\pm$ are the two branch points of the central charge $\Z_z : B_z \to \CC_\xi$:
\begin{eqn}
	\Z_z \coleq \Z \big|_{\tilde{B}_z} : 
		\tilde{B}_z = \widetilde{ \CC_\xi \smallsetminus \set{ \xi_+, \xi_- } }
		\too \CC_\xi \smallsetminus \set{ \xi_+, \xi_- }
\fullstop
\end{eqn}
Thus, the map $\Z_z : \tilde{B}_z \to \CC_\xi$ has two logarithmic branch points located at $\xi_\pm$.
Consequently, $(\tilde{B}_z, \Z_z)$ is an endless Riemann surface of log-algebraic type in the sense of \cite[Definition 1.6 and 1.7]{240622121512}.
\end{lemma}

\begin{proof}
This becomes obvious in the trivialisation \eqref{250315184705} because the map $\Z_z$ becomes explicit:
\begin{eqn}
	\Z_z : (x,y) \mapstoo \xi 
		= \tfrac{1}{30} ( 1 - e^{5x} ) z^5 
		= - \tfrac{1}{30} ( 1 - e^{5y} ) z^5
\fullstop
\tag*{\qedhere}
\end{eqn}
\end{proof}

\paragraph{The Borel surface.}
Similarly, the source maps $\rm{s} : M \to \CC_\tau$ and $\rm{s} : \tilde{M} \to \CC^\ast_\tau$ define respectively algebraic and holomorphic fibrations that play a key role in the resurgence of the deformed Painlevé I equation.

\begin{definition}[Borel surface]
\label{250315183236}
For any $\tau \in \CC_\tau$, we define the \dfn{Borel surface} $M_\tau$ to be the source fibre of $\tau$ under $\rm{s} : M \to \CC_\tau$.
Similarly, for any nonzero $\tau \in \CC^\ast_\tau$, we define the \dfn{Borel covering surface} $\tilde{M}_\tau$ to be the source fibre of $z$ under $\rm{s} : \tilde{M} \to \CC^\ast_\tau$.
Thus:
\begin{eqntag}
	M_\tau \coleq \rm{s}^{-1} (\tau) \subset M
\qtext{and}
	\tilde{M}_\tau \coleq \rm{s}^{-1} (\tau) \subset \tilde{M}
\fullstop
\end{eqntag}
\end{definition}

\begin{figure}[t]
\centering
\includegraphics[width=0.45\textwidth]{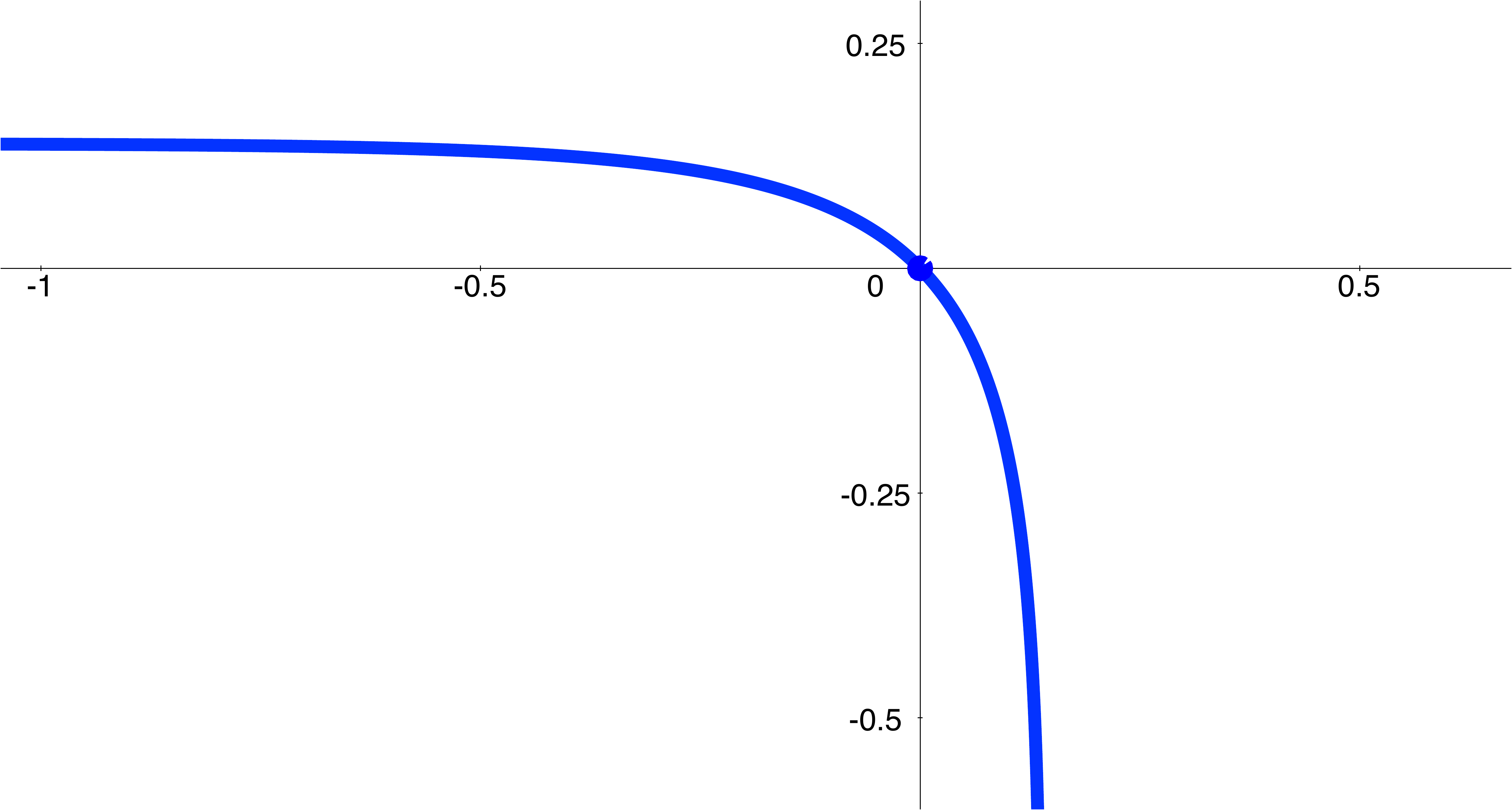}
\caption{A real slice of any unfolded Borel covering surface $\tilde{B}_z$, $z \in \CC^\ast_z$, presented as the analytic curve $e^{5x} + e^{5y} = 2$ in $\RR^2$.
The origin $(x,y) = (0,0)$ is the distinguished point corresponding to $\bm{1_z} = (1_z,1_z) \in \tilde{B}_z$.
}
\label{250223103744}
\end{figure}

\subsection{Convolution and Target-Truncations}
\label{250307192033}

In this subsection, we explain several other constructions which can be made over the unfolded Borel space that will be of use to us later.

\paragraph{The convolution product.}
First, we explain how to lift the convolution product $\ast$ to the unfolded Borel space $B$ or the unfolded Borel covering space $\tilde{B}$ using the anchor map.
For this, we need to introduce a notion of parameterisation of elements in $B$ (or $\tilde{B}$) which is `synchronised' according to the central charge in the following precise sense.

\begin{definition}
\label{250315195528}
By a \dfn{synchronised parameterisation} of an element $\bm{\gamma} = (\gamma^+, \gamma^-)$ in $B$ (or $\tilde{B}$) we mean any pair of parameterisations $\gamma^\pm : I \to \CC_z$ whose source-truncations $\gamma_s^\pm = \gamma^\pm |_{[0,s]} : [0,s] \subset I \to \CC_z$ satisfy $\Z (\gamma_s^+) = - \Z (\gamma_s^-)$ for all $s \in I$.
See \autoref{250306131237}.
\end{definition}

\begin{definition}[source-truncation]
\label{250306123324}
If $\bm{\gamma} = (\gamma^+, \gamma^-)$ in $B$ (or $\tilde{B}$) is synchronously parameterised, then each pair $\bm{\gamma}_s \coleq (\gamma_s^+, \gamma_s^-)$ defines an element of $B$ (or $\tilde{B}$) which we call the \dfn{source-truncation} of (the given synchronised parameterisation of) $\bm{\gamma}$.
See \autoref{250306142700}.
\end{definition}

\begin{definition}[arc-length]
\label{250307163500}
We define the \dfn{arc-length} of a synchronised parameterisation of any $\bm{\gamma} = (\gamma^+, \gamma^-)$ in $B$ (or $\tilde{B}$) to be the positive real number
\begin{eqn}
	|\bm{\gamma}| \coleq |\gamma^+| = |\gamma^-| = \int_{\gamma^\pm} |\lambda|
\fullstop
\end{eqn}
Any such $\bm{\gamma}$ has a natural \dfn{arc-length parameterisation} whereby we take the arc-length parameterisations $\gamma^\pm : [0, \L] \to \CC_z$ with $\L \coleq |\bm{\gamma}|$.
Then for every $s \in [0, \L]$, we can define the \dfn{arc-length source-truncation} to be
\begin{eqn}
	\bm{\gamma}_s \coleq (\gamma_s^+, \gamma_s^-)
\qtext{where}
	\gamma_s^\pm \coleq \gamma^\pm |_{[0,s]} : [0,s] \subset [0, \L] \to \CC_z
\fullstop
\end{eqn}
\end{definition}

\begin{definition}[convolution product]
\label{250307163506}
For any pair of holomorphic functions $f,g$ on $B$ (or $\tilde{B}$), their \dfn{convolution product} is the holomorphic function $f \ast g$, respectively on $B$ (or $\tilde{B}$), given by the following formula: for all $\bm{\gamma}$,
\begin{eqntag}
\label{250307184040}
	(f \ast g) (\bm{\gamma}) \coleq \int_0^{\L} f (\bm{\gamma}_s) g (\bm{\gamma}_{\L - s}) \D{s}
\qtext{where}
	\D{s} \coleq \pm \rm{t}_\pm^\ast \lambda_{\gamma^\pm_s}
\fullstop{,}
\end{eqntag}
where $\bm{\gamma}_s$ is the arc-length source-truncation of any synchronised representative of $\bm{\gamma}$ of length $|\bm{\gamma}| = \L$.
\end{definition}

\begin{figure}[t]
\centering
\begin{adjustwidth}{-1cm}{-1cm}
\begin{subfigure}[b]{0.24\linewidth}
\centering
\begin{tikzpicture}
\fill [darkgreen] (0,0) circle (2pt) node [left] {$z_0$};
\fill [darkgreen] (3,2) circle (2pt) node [right] {$z_1^+$};
\fill [darkgreen] (3,-1) circle (2pt) node [right] {$z_1^-$};
\fill [darkgreen] (2,1.85) circle (1pt) node [above] {$z_s^+$};
\fill [darkgreen] (2,-0.95) circle (1pt) node [below] {$z_s^-$};
\draw [darkgreen, thick] (0,0) to [out=120, in=210] (1,1.4);
\draw [darkgreen, thick, ->-=0.1] (1,1.4) to [out=30, in=195] (2,1.85);
\draw [darkgreen, thick] (2,1.85) to [out=15, in=180] (3,2);
\draw [darkgreen, thick] (0,0) to [out=-60, in=170] (1,-0.8);
\draw [darkgreen, thick, ->-=0.2] (1,-0.8) to [out=-10, in=175] (2,-0.95);
\draw [darkgreen, thick] (2,-0.95) to [out=-5, in=180] (3,-1);
\node [darkgreen] at (1,1.3) [above left] {$\gamma^+$};
\node [darkgreen] at (1,-0.65) [below left] {$\gamma^-$};
\node [darkgreen] at (2.25,0.25) {$\bm{\gamma}$};
\end{tikzpicture}
\caption{}
\label{250306131237}
\end{subfigure}
\begin{subfigure}[b]{0.24\linewidth}
\centering
\begin{tikzpicture}
\fill [darkgreen] (0,0) circle (2pt) node [left] {$z_0$};
\fill [darkgreen] (3,2) circle (2pt) node [right] {$z_1^+$};
\fill [darkgreen] (3,-1) circle (2pt) node [right] {$z_1^-$};
\fill [darkgreen] (1,1.4) circle (2pt) node [above] {$z_s^+$};
\fill [darkgreen] (1,-0.8) circle (2pt) node [below] {$z_s^-$};
\draw [darkgreen, thick, ->-=0.5] (0,0) to [out=120, in=210] (1,1.4);
\draw [darkgreen, thick, dashed] (1,1.4) to [out=30, in=195] (2,1.85);
\draw [darkgreen, thick, dashed] (2,1.85) to [out=15, in=180] (3,2);
\draw [darkgreen, thick, ->-=0.5] (0,0) to [out=-60, in=170] (1,-0.8);
\draw [darkgreen, thick, dashed] (1,-0.8) to [out=-10, in=175] (2,-0.95);
\draw [darkgreen, thick, dashed] (2,-0.95) to [out=-5, in=180] (3,-1);
\node [darkgreen] at (0.5,1) [above left] {$\gamma_s^+$};
\node [darkgreen] at (0.5,-0.5) [below left] {$\gamma_s^-$};
\node [darkgreen] at (0.75,0.25) {$\bm{\gamma_s}$};
\end{tikzpicture}
\caption{}
\label{250306142700}
\end{subfigure}
\begin{subfigure}[b]{0.24\linewidth}
\centering
\begin{tikzpicture}
\fill [darkgreen] (0,0) circle (2pt) node [left] {$z_0$};
\fill [darkgreen] (3,2) circle (2pt) node [right] {$z_1^+$};
\fill [darkgreen] (3,-1) circle (2pt) node [right] {$z_1^-$};
\fill [darkgreen] (2,1.85) circle (2pt) node [below] {$z_{1-s}^+$};
\fill [darkgreen] (2,-0.95) circle (2pt) node [above] {$z_{1-s}^-$};
\draw [darkgreen, thick, dashed] (0,0) to [out=120, in=210] (1,1.4);
\draw [darkgreen, thick, dashed] (1,1.4) to [out=30, in=195] (2,1.85);
\draw [darkgreen, thick, ->-=0.5] (2,1.85) to [out=15, in=180] (3,2);
\draw [darkgreen, thick, dashed] (0,0) to [out=-60, in=170] (1,-0.8);
\draw [darkgreen, thick, dashed] (1,-0.8) to [out=-10, in=175] (2,-0.95);
\draw [darkgreen, thick, ->-=0.5] (2,-0.95) to [out=-5, in=180] (3,-1);
\node [darkgreen] at (2.65,2) [above] {$\bar{\gamma}_s^+$};
\node [darkgreen] at (2.65,-1) [below] {$\bar{\gamma}_s^-$};
\end{tikzpicture}
\caption{}
\label{250306143123}
\end{subfigure}
\begin{subfigure}[b]{0.24\linewidth}
\centering
\begin{tikzpicture}
\fill [darkgreen] (0,0) circle (2pt) node [left] {$z_0$};
\fill [darkgreen] (3,2) circle (2pt) node [right] {$z_1^+$};
\fill [darkgreen] (3,-1) circle (2pt) node [right] {$z_1^-$};
\fill [darkgreen] (2,1.85) circle (2pt) node [below, scale=0.75] {$z_{1-s}^+$};
\fill [darkgreen] (2,-0.95) circle (2pt) node [above, scale=0.75] {$z_{1-s}^-$};
\fill [darkgreen] (1,1.4) circle (2pt);
\fill [darkgreen] (1,-0.8) circle (2pt);
\draw [darkgreen, thick, dashed] (0,0) to [out=120, in=210] (1,1.4);
\draw [darkgreen, thick, -<-=0.5] (1,1.4) to [out=30, in=195] (2,1.85);
\draw [darkgreen, thick, ->-=0.5] (2,1.85) to [out=15, in=180] (3,2);
\draw [darkgreen, thick, dashed] (0,0) to [out=-60, in=170] (1,-0.8);
\draw [darkgreen, thick, -<-=0.5] (1,-0.8) to [out=-10, in=175] (2,-0.95);
\draw [darkgreen, thick, ->-=0.5] (2,-0.95) to [out=-5, in=180] (3,-1);
\node [darkgreen] at (2.65,2) [above] {$\bar{\gamma}_s^+$};
\node [darkgreen] at (2.65,-1) [below] {$\bar{\gamma}_s^-$};
\node [darkgreen] at (1.25,1.7) [above, scale=0.9] {$(\bar{\gamma}_s^+)^\dagger$};
\node [darkgreen] at (1.4,-1) [below, scale=0.9] {$(\bar{\gamma}_s^-)^\dagger$};
\node [darkgreen] at (2.5,1) {$\bar{\bm{\gamma}}^+_s$};
\node [darkgreen] at (2.5,-0.25) {$\bar{\bm{\gamma}}^-_s$};
\end{tikzpicture}
\caption{}
\label{250308095823}
\end{subfigure}
\end{adjustwidth}
\caption{In (a), a synchronised parameterisation of $\bm{\gamma} = (\gamma^+, \gamma^-)$.
In (b), its source-truncation $\bm{\gamma}_s = (\gamma^+_s, \gamma^-_s)$.
In (c), the pair of target-truncations $\bar{\gamma}^\pm_s$ of $\bar{\gamma}^\pm$ do \textit{not} form an element of $B$ (nor $\tilde{B}$) because in general they do not have the same source.
In (d), the two target-truncations $\bar{\bm{\gamma}}^+_s$ and $\bar{\bm{\gamma}}^-_s$ of $\bm{\gamma}$.}
\label{250306131511}
\end{figure}
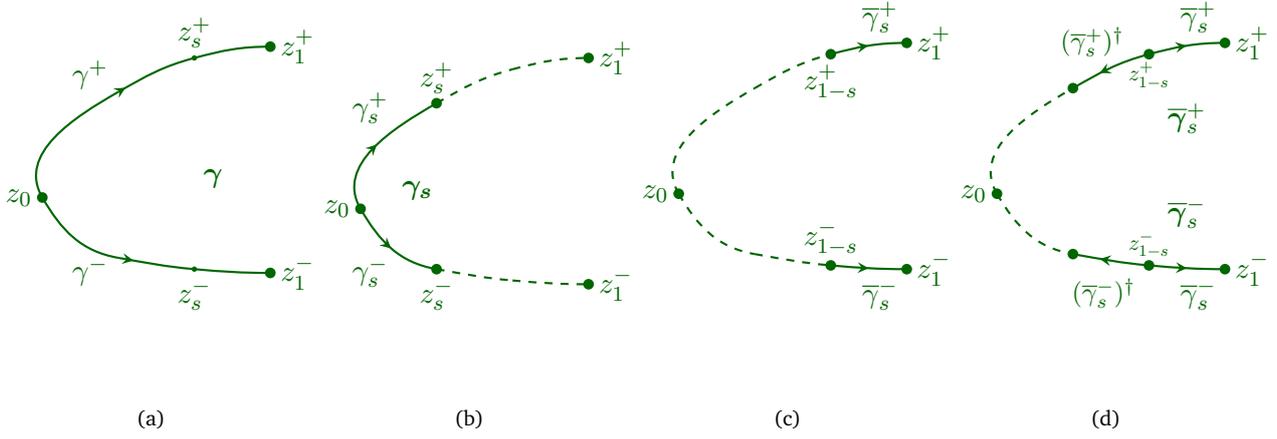

Then we deduce the following analogue of \autoref{250306181554}.

\begin{lemma}
\label{250308085903}
The pullback by the anchor $\varrho$ distributes over the convolution product $\ast$.
That is, for any two holomorphic functions $f = f(z,\xi)$ and $g = g(z,\xi)$, we have the following identity of holomorphic functions on $B$ (or $\tilde{B}$):
\begin{eqntag}
\label{250308090118}
	\varrho^\ast (f \ast g) = (\varrho^\ast f) \ast (\varrho^\ast g)
\fullstop
\end{eqntag}
\end{lemma}

\paragraph{Two target-truncations on $B$.}
Unlike the source-truncation $\bm{\gamma}_s = (\gamma^+_s, \gamma^-_s)$ of a given synchronously parameterised element $\bm{\gamma} = (\gamma^+, \gamma^-)$, the pair $(\bar{\gamma}^+_s, \bar{\gamma}^{\,-}_s)$ formed by the target-truncations of $\gamma^\pm$ does \textit{not} define an element of $B$ (nor $\tilde{B}$) because $\rm{s} (\bar{\gamma}^+_s) \neq \rm{s} (\bar{\gamma}^{\,-}_s)$ in general; see \autoref{250306143123}.
However, there is a canonical way to promote each target-truncation $\bar{\gamma}^+_s$ and $\bar{\gamma}^{\,-}_s$ to an element of $B$ (or $\tilde{B}$).

\begin{definition}[target-truncations]
\label{250308094543}
Given any synchronously parameterised element $\bm{\gamma} = (\gamma^+, \gamma^-)$ of $B$ (or $\tilde{B}$), let $\L = |\bm{\gamma}|$ and consider the arc-length target-truncations $\bar{\gamma}^\pm_s : [\L - s, \L] \to \CC_z$ of $\gamma^\pm : [0,\L] \to \CC_z$.
Take the negatively parameterised inverse paths $(\gamma^\pm)^\inv : [-\L,0] \to \CC_z$ and concatenate them with $\gamma^\mp$ like so:
\begin{eqn}
	\gamma^+ \circ (\gamma^-)^\inv : [-\L,\L] \to \CC_z
\qqtext{and}
	\gamma^- \circ (\gamma^+)^\inv : [-\L,\L] \to \CC_z
\fullstop
\end{eqn}
Then, for each $s \in [0, \L]$, the two \dfn{arc-length target-truncations} of $\bm{\gamma}$ at time $s$ are the elements of $B$ (or $\tilde{B}$) given by the following two pairs of paths:
\begin{eqntag}
	\bar{\bm{\gamma}}^+_s \coleq \big( \bar{\gamma}^+_s, (\bar{\gamma}^+_s)^\dagger \big)
\qqtext{and}
	\bar{\bm{\gamma}}^-_s \coleq \big( (\bar{\gamma}^{\,-}_s)^\dagger, \bar{\gamma}^{\,-}_s \big)
\end{eqntag}
where $(\bar{\gamma}^+_s)^\dagger$ and $(\bar{\gamma}^{\,-}_s)^\dagger$ are the parameterised paths (see \autoref{250308095823})
\begin{eqntag}
	(\bar{\gamma}^+_s)^\dagger \coleq \big( \gamma^- \circ (\gamma^+)^\inv \big) \big|_{[-\L + s, \L + 2s]}
\qtext{and}
	(\bar{\gamma}^{\,-}_s)^\dagger \coleq \big( \gamma^+ \circ (\gamma^{\,-})^\inv \big) \big|_{[-\L + s, \L + 2s]}
\fullstop
\end{eqntag}
\end{definition}

In words, $(\bar{\gamma}^+_s)^\dagger$ is the concrete path that starts at the source of $\bar{\gamma}^+_s$ but travels backwards for time $s$ along $(\gamma^+)^\inv$ and $\gamma^-$.
Similarly, $(\bar{\gamma}^{\,-}_s)^\dagger$ is the concrete path that starts at the source of $\bar{\gamma}^{\,-}_s$ but travels backwards for time $s$ along $(\gamma^-)^\inv$ and $\gamma^+$.

\begin{figure}[!t]
\centering
\begin{subfigure}[t]{0.45\textwidth}
\centering
\begin{tikzpicture}
\begin{scope}
\clip (-3,-2) rectangle (3,2);
\fill [grey] (-3,-2) rectangle (3,2);
\fill [black] (0,0) circle (1pt) node [below] {$0$};
\node at (135:2.45) {$\CC_\xi$};
\node [cross, red] at (0:1) {};
\node [cross, red] at (180:1) {};
\node at (0:1) [red, below] {$\xi_+$};
\node at (180:1) [red, below] {$\xi_-$};
\draw [dashed, red, ultra thick] (0:1) -- (0:4);
\draw [dashed, red, ultra thick] (180:1) -- (180:4);
\end{scope}
\end{tikzpicture}
\caption{The central charge $\Z$ restricts to the unfolded Borel surface $B_z$ for $z=1$ to give a fivefold covering map of the $\xi$-plane which is ramified in ten points, but with only two branch points located at $\xi_\pm = \pm 1/30$.
The `origin' $\bm{1_z}$ of $B_z$ is mapped to the origin $\xi = 0$.
Two of these ramification points are captured by the real slice in \autoref{250222133101}.
By drawing branch cuts (red dashed lines) emanating from the two branch points, this picture shows the sheet of the Riemann surface $B_z$ containing the `origin' $\bm{1_z}$; there are five sheets stacked like pancakes one on top of another and glued along the branch cuts appropriately.}
\label{250222150214}
\end{subfigure}
\quad
\begin{subfigure}[t]{0.45\textwidth}
\centering
\begin{tikzpicture}
\begin{scope}
\clip (-3,-2) rectangle (3,2);
\fill [grey] (-3,-2) rectangle (3,2);
\draw [ultra thick, darkgreen, ->-=0.5] (0,0) -- (0:1);
\draw [ultra thick, darkgreen, ->-=0.5] (0,0) -- (180:1);
\draw [ultra thick, darkgreen, ->-=0.5] (0,0) to [out=90, in=0] (120:1) to [out=180, in=90] (180:2);
\draw [ultra thick, darkgreen, ->-=0.5, dashed] (180:2) to [out=270, in=180] (-120:1) to [out=0, in=-120] (0:1);
\fill [black] (0,0) circle (1pt) node [below] {$0$};
\node at (135:2.45) {$\CC_\xi$};
\node [cross, red] at (0:1) {};
\node [cross, red] at (180:1) {};
\node at (0:1) [red, above] {$\xi_+$};
\node at (180:1) [red, below] {$\xi_-$};
\draw [dashed, red, ultra thick] (0:1) -- (0:4);
\draw [dashed, red, ultra thick] (180:1) -- (180:4);
\end{scope}
\end{tikzpicture}
\caption{The two ramification points which are captured by the real slice in \autoref{250222133101} can be reached from the `origin' $\bm{1_z}$ via geodesics (i.e., paths on $B_z$ whose projection to the $\xi$-plane are straight line segments).
The projections to the $\xi$-plane of these two geodesics are pictured by straight green line segments from $0$ to $\xi_\pm$.
Only these two ramification points can be reached by a geodesic, so they are the only visible singularities of $(B_z, \Z_z)$.
To reach any of the other eight ramification points, we need to go around one of the visible singularities and to another sheet of the Riemann surface $B_z$.
One such path that connects the `origin' $\bm{1_z}$ to a ramification point gets projected to the $\xi$-plane as depicted by the green curvy path from $0$ to $\xi_+$.}
\label{250222151307}
\end{subfigure}
\caption{Projection under the central charge.}
\label{250222160721}
\end{figure}

\paragraph{Geodesics.}
Let us now describe the geodesics on the unfolded Borel space.

\begin{definition}[geodesics]
\label{250223185940}
A \dfn{geodesic} on $B$ is a smooth real curve $\tilde{\bm{\gamma}} : [0,1] \to B$ which lies in a single source fibre of $B$, avoids the subset $\Gamma$ except possibly at the endpoints, and whose projection $\Z \circ \tilde{\bm{\gamma}} : [0,1] \to \CC_\xi$ to the $\xi$-plane is a straight line segment.
A \dfn{critical geodesic} is a geodesic whose target belongs to $\Gamma$.
An \dfn{infinite geodesic} on $B$ is a smooth real curve $\tilde{\bm{\gamma}} : [0,1) \to B$ which lies in a single source-fibre of $B$, avoids the subset $\Gamma$ except possibly at the source, and whose projection $\Z \circ \tilde{\bm{\gamma}}$ to the $\xi$-plane is an infinite straight ray.
\end{definition}

Since points of $B$ represent pairs $\bm{\gamma} = (\gamma^+, \gamma^-)$ of paths on $\CC_z$, a real curve $\tilde{\bm{\gamma}} : [0,1] \to B$ which lies a source fibre $B_{z_0}$ is nothing but a pair $\bm{\gamma}_s = (\gamma^+_s, \gamma^-_s)$ of families of paths on $\CC_z$ parameterised by $s \in [0,1]$ such that $\rm{s} (\bm{\gamma}_s) = \rm{s} (\gamma^\pm_s) = z_0$ for all $s$.
Likewise, the projection $\Z \circ \bm{\gamma}$ is nothing but a path $s \mapsto \xi_s$ in the $\xi$-plane, where $\xi_s = \Z (\bm{\gamma}_s) = \pm \int_{\gamma^\pm_s} \lambda$.
Then $\bm{\gamma}$ is a geodesic if and only if none of the paths $\gamma^\pm_s$ for $s \in (0,1)$ is critical and the phase of $\xi_s$ remains constant in $s$; i.e., $\arg (\xi_s) = \const$; i.e., there is some $\alpha \in \RR / 2 \pi \ZZ$ such that $\Z (\bm{\gamma}_s) = \pm \Z (\gamma^\pm) \in e^{i \alpha} \RR_+$ for all $s \in [0,1]$.
Said differently, the constraint on the phase of $\Z (\bm{\gamma}_s)$ reads
\begin{eqn}
	\Im \left( \pm e^{-i \alpha} \int_{\gamma_s^\pm} \lambda \right) = 0
\qtext{and}
	\Re \left( \pm e^{-i \alpha} \int_{\gamma_s^\pm} \lambda \right) \geq 0
\fullstop
\end{eqn}
The fixed-endpoint homotopy class of any $\gamma_s^\pm$ is determined by its endpoints $z_0 = \rm{s} (\bm{\gamma}_s) = \rm{s} (\gamma^\pm_s)$ and $z^\pm_s \coleq \rm{t}_\pm (\bm{\gamma}_s) = \rm{t} (\gamma^\pm_s)$.
So the above constraint on the phase of $\Z (\bm{\gamma}_s)$ reads more explicitly as follows:
\begin{eqntag}
\label{250224075106}
	\Im \left( \pm e^{-i \alpha} \int_{z_0}^{z^\pm_s} \lambda \right) = 0
\qtext{and}
	\Re \left( \pm e^{-i \alpha} \int_{z_0}^{z^\pm_s} \lambda \right) \geq 0
\fullstop
\end{eqntag}
Thus, a geodesic $\bm{\gamma}$ on $B$ is nothing but a pair of concrete paths $(z_s^+)$ and $(z_s^-)$ on the $z$-plane defined by the equations \eqref{250224075106}, each of which is understood as the target $\rm{t}_\pm$ projection of $\bm{\gamma}$.
This is depicted in \autoref{250223164824} and \autoref{250223191822}.

Concrete paths on $\CC_z$ defined by \eqref{250224075106} are known as \dfn{geodesic trajectories} of the differential $\pm \lambda$.
They form (singular) real foliations of the $z$-plane known as \dfn{geodesic foliations} \textit{with phase $\alpha$} associated with the differential $\pm \lambda$.
These foliations are pictured in \autoref{250223183907}.

\begin{figure}[p]
\begin{adjustwidth}{-2cm}{-2cm}
\centering
\raisebox{1cm}{
\begin{subfigure}{0.35\textwidth}
\centering
\includegraphics[width=\textwidth]{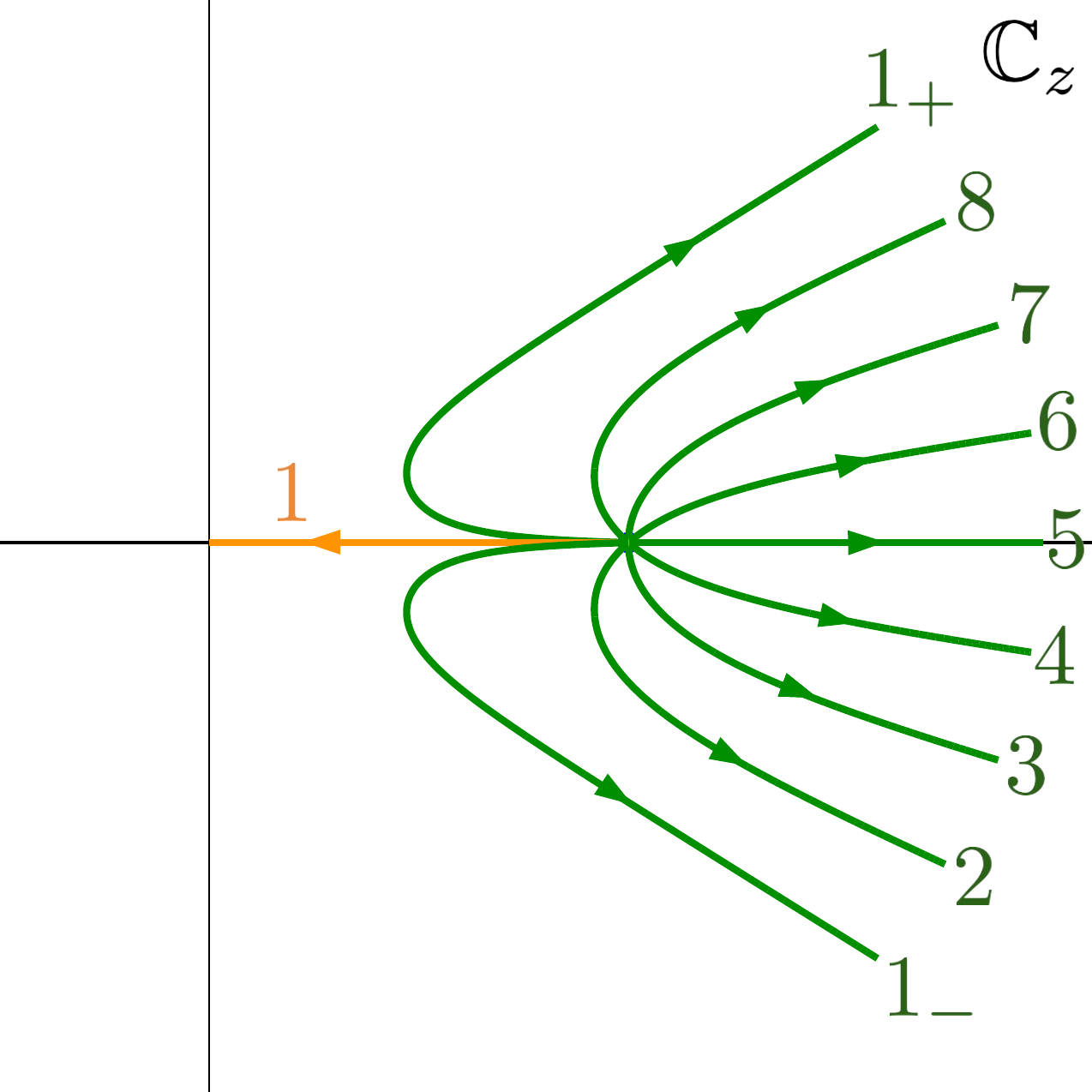}
\caption{}
\label{250223172028}
\end{subfigure}
}
\begin{subfigure}{0.35\textwidth}
\centering
\begin{tikzpicture}
\begin{scope}[yshift=4cm]
\node at (0,0) {$B_z$};
\draw [->] (-90:0.5) -- (-90:1.5) node[midway, right] {$\Z$};
\draw [->] (-20:0.5) -- (-20:2) node[midway, above right] {$\rm{t}_+$};
\draw [->] (200:0.5) -- (200:2) node[midway, above left] {$\rm{t}_-$};
\end{scope}
\begin{scope}[scale=0.7, every node/.style={scale=1}]
\clip (-3,-3) rectangle (3,3);
\fill [grey] (-3,-3) rectangle (3,3);
\draw [ultra thick, darkgreen, ->-=0.65] (0,0) -- (3:2) node [above right] {$1_+$};
\draw [ultra thick, darkgreen, ->-=0.65] (0,0) -- (-3:2) node [below right] {$1_-$};
\draw [ultra thick, darkgreen, ->-=0.65] (0,0) -- (177:2) node [above left] {$5_-$};
\draw [ultra thick, darkgreen, ->-=0.65] (0,0) -- (-177:2) node [below left] {$5_+$};

\draw [ultra thick, darkgreen, ->-=0.5] (0,0) -- (45:2) node [above right] {$2$};
\draw [ultra thick, darkgreen, ->-=0.5] (0,0) -- (90:2) node [above] {$3$};
\draw [ultra thick, darkgreen, ->-=0.5] (0,0) -- (135:2) node [above left] {$4$};
\draw [ultra thick, darkgreen, ->-=0.5] (0,0) -- (-135:2) node [below left] {$6$};
\draw [ultra thick, darkgreen, ->-=0.5] (0,0) -- (-90:2) node [below] {$7$};
\draw [ultra thick, darkgreen, ->-=0.5] (0,0) -- (-45:2) node [below right] {$8$};

\draw [ultra thick, orange, ->-=0.5] (0,0) -- (0:2) node [right] {$1$};
\draw [ultra thick, orange, ->-=0.5] (0,0) -- (180:2) node [left] {$5$};
\fill [black] (0,0) circle (1pt);
\node at (135:3.5) {$\CC_\xi$};
\node [cross, red] at (0:2) {};
\node [cross, red] at (180:2) {};
\end{scope}
\end{tikzpicture}
\caption{}
\label{250223162611}
\end{subfigure}
\raisebox{1cm}{
\begin{subfigure}{0.35\textwidth}
\centering
\includegraphics[width=\textwidth]{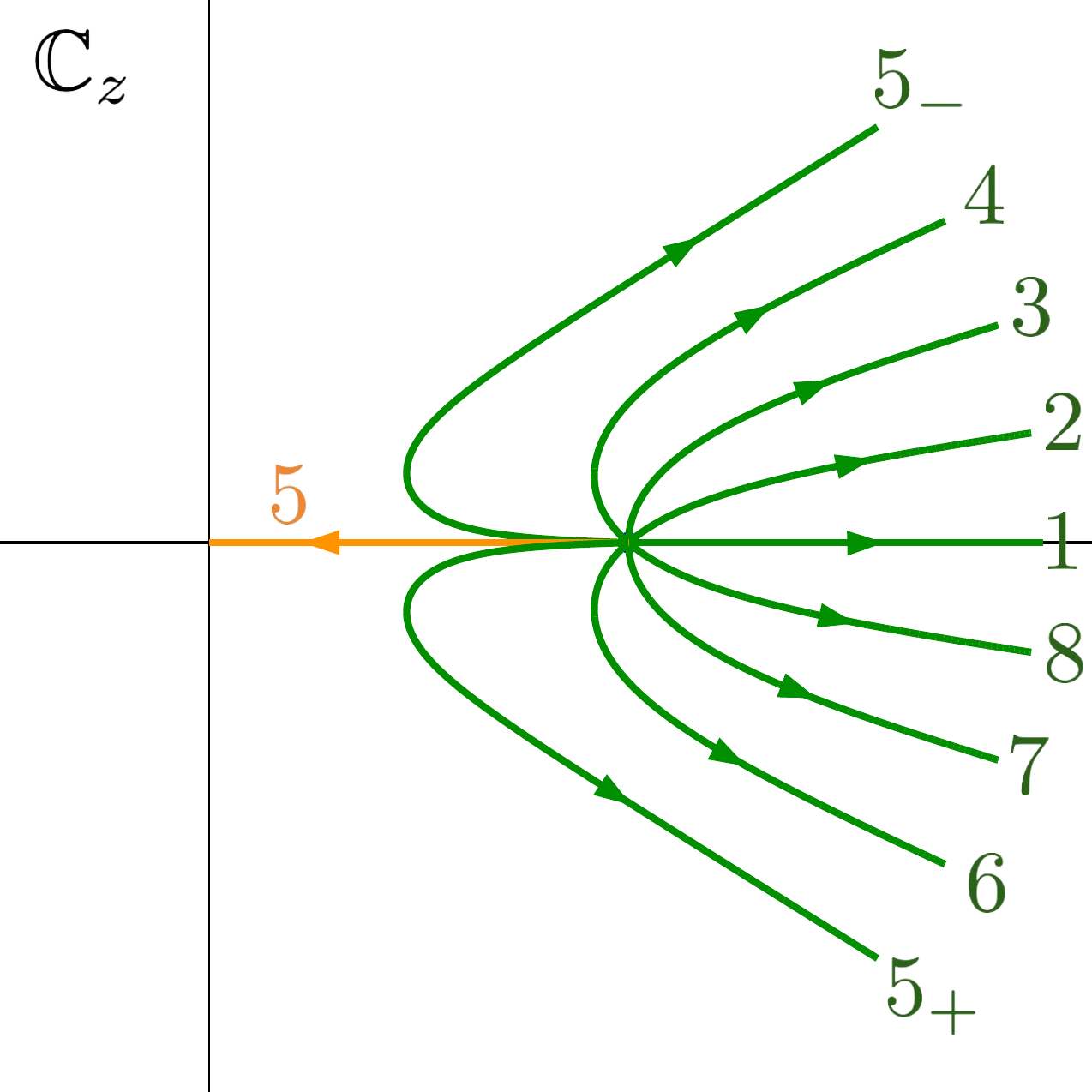}
\caption{}
\label{250223172029}
\end{subfigure}
}
\end{adjustwidth}
\caption{Projections of geodesics in $B_z$ for $z=1$ to the $\xi$-plane via $\Z$ and to the $z$-plane via $\rm{t}_\pm$.
The straight line segments in the $\xi$-plane labelled $k = 1, \ldots, 8$ are the segments $[0, \tfrac{1}{30} e^{i\alpha_k}]$ where $\alpha_k = \pi (k-1)/4$, respectively.
The four extra line segments labelled $1_\pm$ and $5_\pm$ are respectively $[0,\tfrac{1}{30} e^{\pm i \epsilon}]$ and $[0,-\tfrac{1}{30} e^{\pm i \epsilon}]$, where $\epsilon = 0.01\pi$.
The two branch points of $\Z$ in the $\xi$-plane are located at $\pm \tfrac{1}{30}$.
}
\label{250223164824}

\smallskip

\begin{adjustwidth}{-2cm}{-2cm}
\centering
\begin{subfigure}{0.4\textwidth}
\centering
\begin{tikzpicture}
\begin{scope}[yshift=4cm]
\node at (0,0) {$B_z$};
\draw [->] (-90:0.5) -- (-90:1.5) node[midway, right] {$\Z$};
\draw [->,yshift=2pt, xshift=1pt] (-20:0.5) -- (-20:3) node[midway, above right] {$\rm{t}_+$};
\draw [->,yshift=-2pt, xshift=-1pt] (-20:0.5) -- (-20:3) node[midway, below left] {$\rm{t}_-$};
\end{scope}
\begin{scope}[scale=0.7, every node/.style={scale=1}]
\clip (-3,-3) rectangle (3,3);
\fill [grey] (-3,-3) rectangle (3,3);
\draw [ultra thick, darkgreen, ->-=0.5] (0,0) -- (45:5) node [midway, below right] {$1$};
\draw [ultra thick, darkgreen, ->-=0.5] (0,0) -- (135:5) node [midway, below left] {$2$};

%
\fill [black] (0,0) circle (1pt);
\node at (90:2.5) {$\CC_\xi$};
\node [cross, red] at (0:2) {};
\node [cross, red] at (180:2) {};
\end{scope}
\end{tikzpicture}
\caption{}
\label{250223192028}
\end{subfigure}
\quad
\begin{subfigure}{0.6\textwidth}
\includegraphics[width=\textwidth]{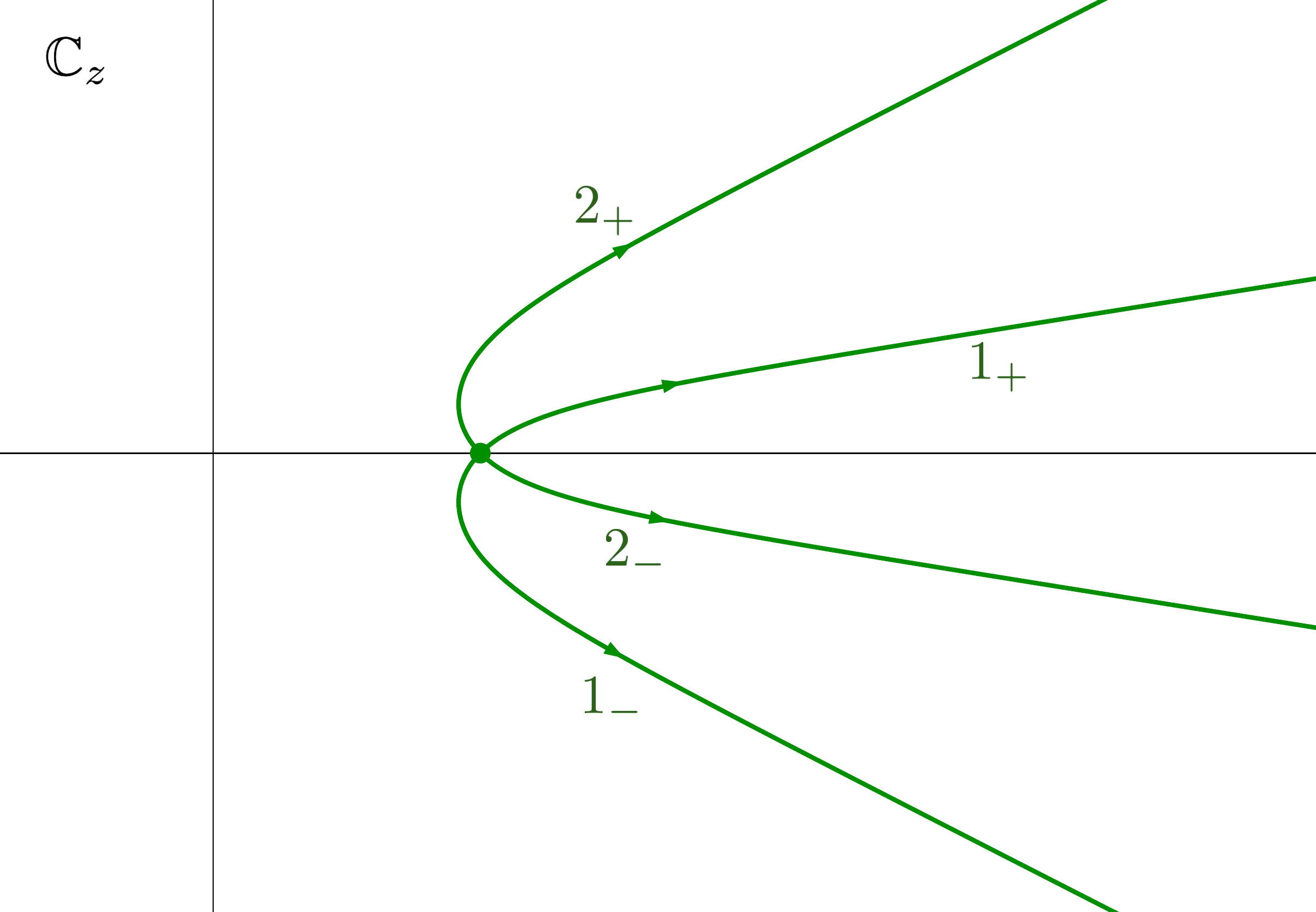}
\caption{}
\label{250223191141}
\end{subfigure}
\end{adjustwidth}
\caption{Projections of infinite geodesics in $B_z$ for $z = 1$ to the $\xi$-plane via $\Z$ and to the $z$-plane via $\rm{t}_\pm$.
The two straight infinite rays in the $\xi$-plane labelled $1$ and $2$ have phases $\pi/4$ and $3\pi/4$, respectively.
The infinite geodesics in $B_z$ whose projections to the $\xi$-plane are the straight infinite rays labelled $1$ and $2$ are projected to the $z$-plane via $\rm{t}_\pm$ with images labelled by $1_\pm$ and $2_\pm$, respectively.
}
\label{250223191822}
\end{figure}

\begin{figure}[p]
\begin{adjustwidth}{-2cm}{-2cm}
\centering
\begin{subfigure}{0.4\textwidth}
\centering
\begin{tikzpicture}
\begin{scope}[yshift=4cm]
\node at (0,0) {$B_z$};
\draw [->] (-90:0.5) -- (-90:1.5) node[midway, right] {$\Z$};
\draw [->,yshift=2pt, xshift=1pt] (-20:0.5) -- (-20:3) node[midway, above right] {$\rm{t}_+$};
\draw [->,yshift=-2pt, xshift=-1pt] (-20:0.5) -- (-20:3) node[midway, below left] {$\rm{t}_-$};
\end{scope}
\begin{scope}[scale=0.7, every node/.style={scale=1}]
\clip (-4,-3) rectangle (4,3);
\fill [grey] (-4,-3) rectangle (4,3);
\begin{scope}[yshift=1.5]
\draw [ultra thick, darkgreen, ->-=0.55] (0,0) -- (0:1.6) arc (180:0:0.4) -- (0:4);
\end{scope}
\begin{scope}[yshift=-1.5]
\draw [ultra thick, darkgreen, ->-=0.55] (0,0) -- (0:1.6) arc (180:360:0.4) -- (0:4);
\end{scope}

\node at (0:1) [above, darkgreen] {$1'$};
\node at (2,0.5) [above, darkgreen] {$1''$};
\node at (0:3) [above, darkgreen] {$1'''$};

\node at (0:1) [below, darkgreen] {$2'$};
\node at (2,-0.5) [below, darkgreen] {$2''$};
\node at (0:3) [below, darkgreen] {$2'''$};


%
\fill [black] (0,0) circle (1pt);
\node at (90:2.5) {$\CC_\xi$};
\node [cross, red] at (0:2) {};
\node [cross, red] at (180:2) {};
\end{scope}
\end{tikzpicture}
\caption{}
\label{250223200010}
\end{subfigure}
\quad
\begin{subfigure}{0.4\textwidth}
\includegraphics[width=\textwidth]{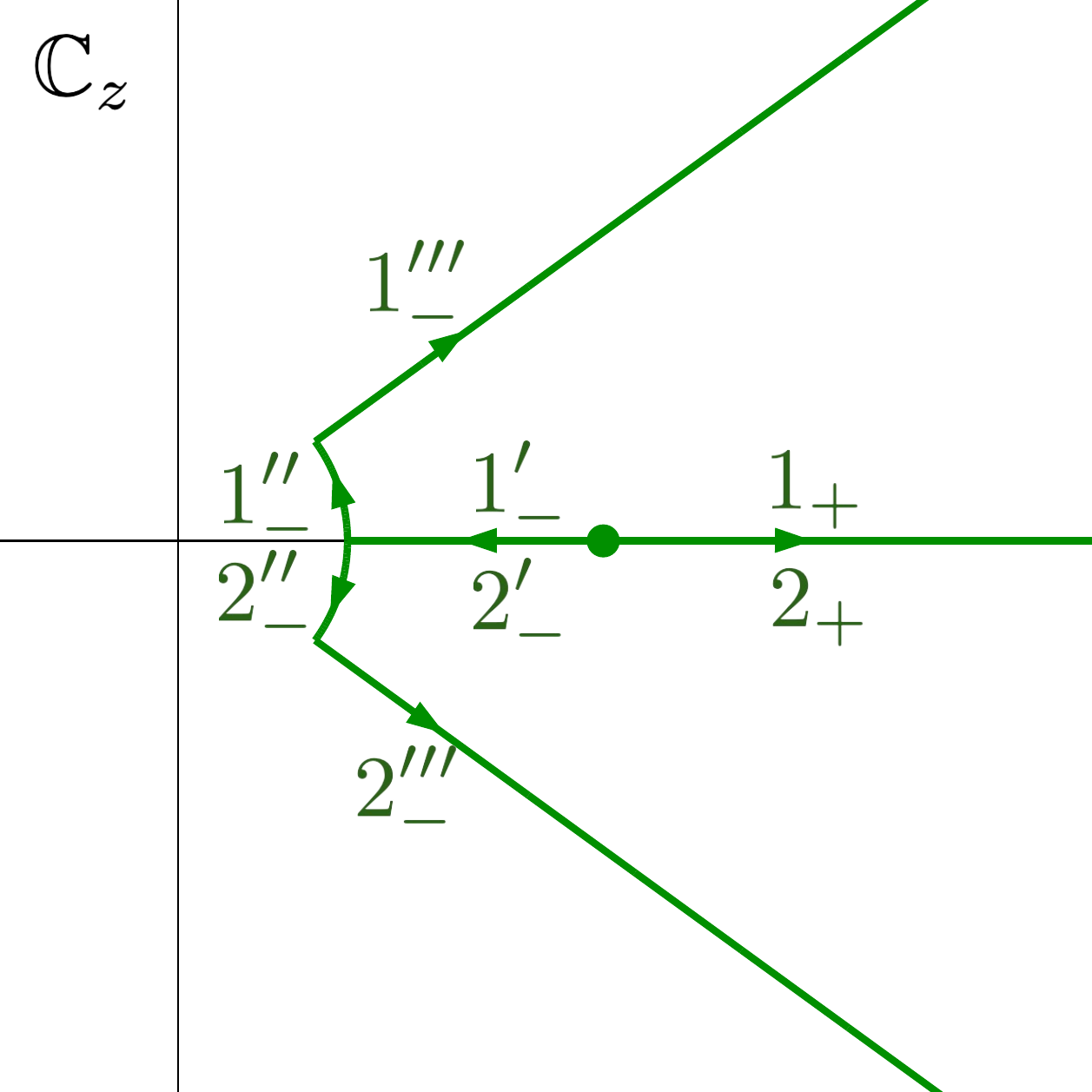}
\caption{}
\label{250223202658}
\end{subfigure}
\end{adjustwidth}
\caption{Projections of two infinite distorted geodesics in $B_z$ for $z = 1$ to the $\xi$-plane via $\Z$ and to the $z$-plane via $\rm{t}_\pm$.
The projection of the distorted geodesic labelled $1$ to the $\xi$-plane consists of the straight line segment $1'$, the semicircular arc $1''$, and the straight infinite ray $1'''$; and similarly for the distorted geodesic labelled $2$.
The projection of the distorted geodesic labelled $1$ to the $z$-plane via $\rm{t}_-$ consists of the straight line segment $1'_-$, the circular arc $1''_-$, and the straight infinite ray $1'''_-$ in the direction $\pi/5$; similarly for the distorted geodesic labelled $2$, for which the ray $2'''_-$ is in the direction $-\pi/5$.
Meanwhile, the projection to the $z$-plane via $\rm{t}_+$ of both distorted geodesics is the straight infinite ray $1_+$ or $2_+$.
(Strictly speaking, the paths $1_+$ and $2_+$ have small semicircular `kinks' which are the projections of the portions of the geodesics that correspond to the semicircular arcs $1''$ and $2''$, but they are too small to be visible in this plot.)
}
\label{250224080358}

\smallskip

\centering
\begin{subfigure}{0.4\textwidth}
\includegraphics[width=\textwidth]{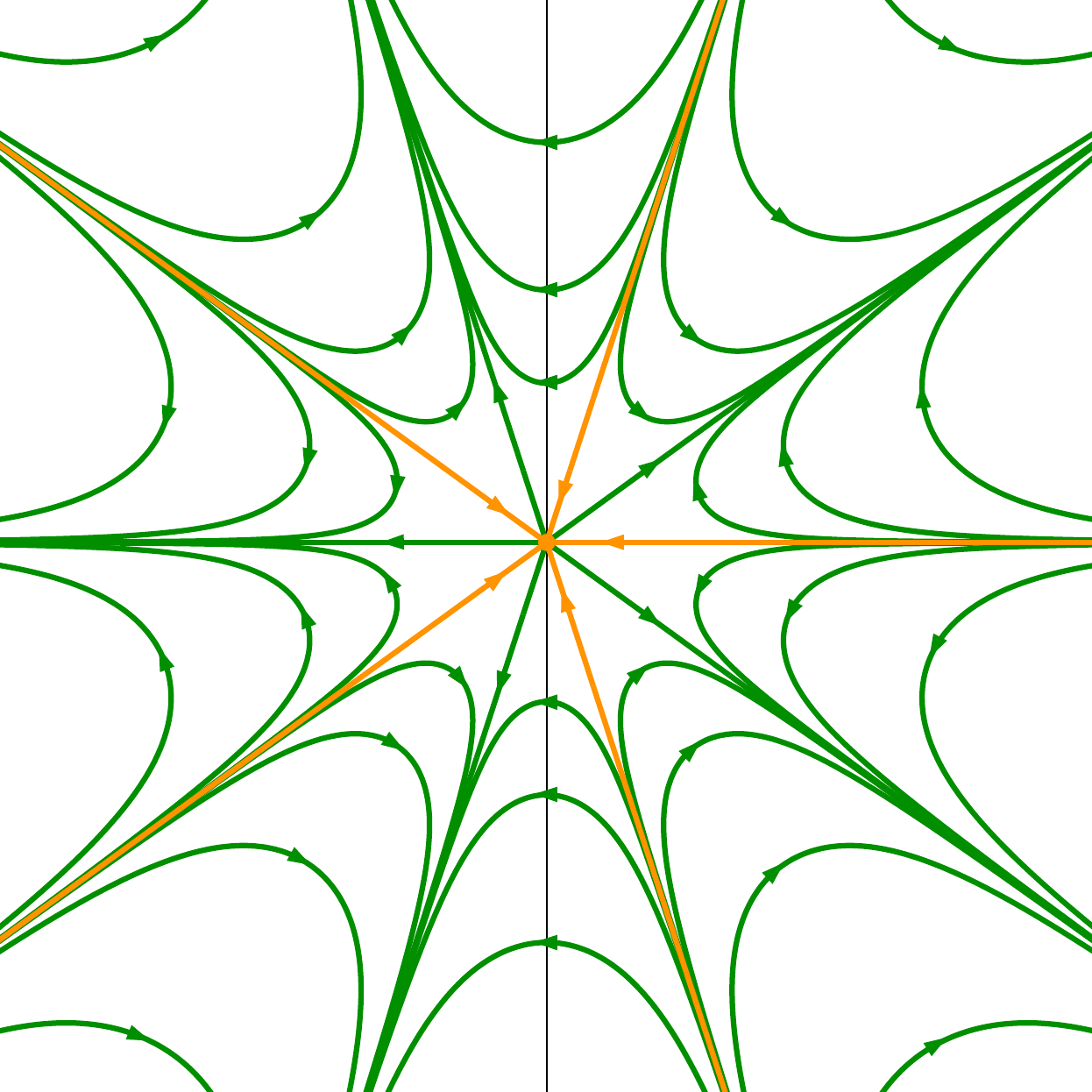}
\caption{$-\lambda$}
\label{250223182215}
\end{subfigure}
\quad
\begin{subfigure}{0.4\textwidth}
\includegraphics[width=\textwidth]{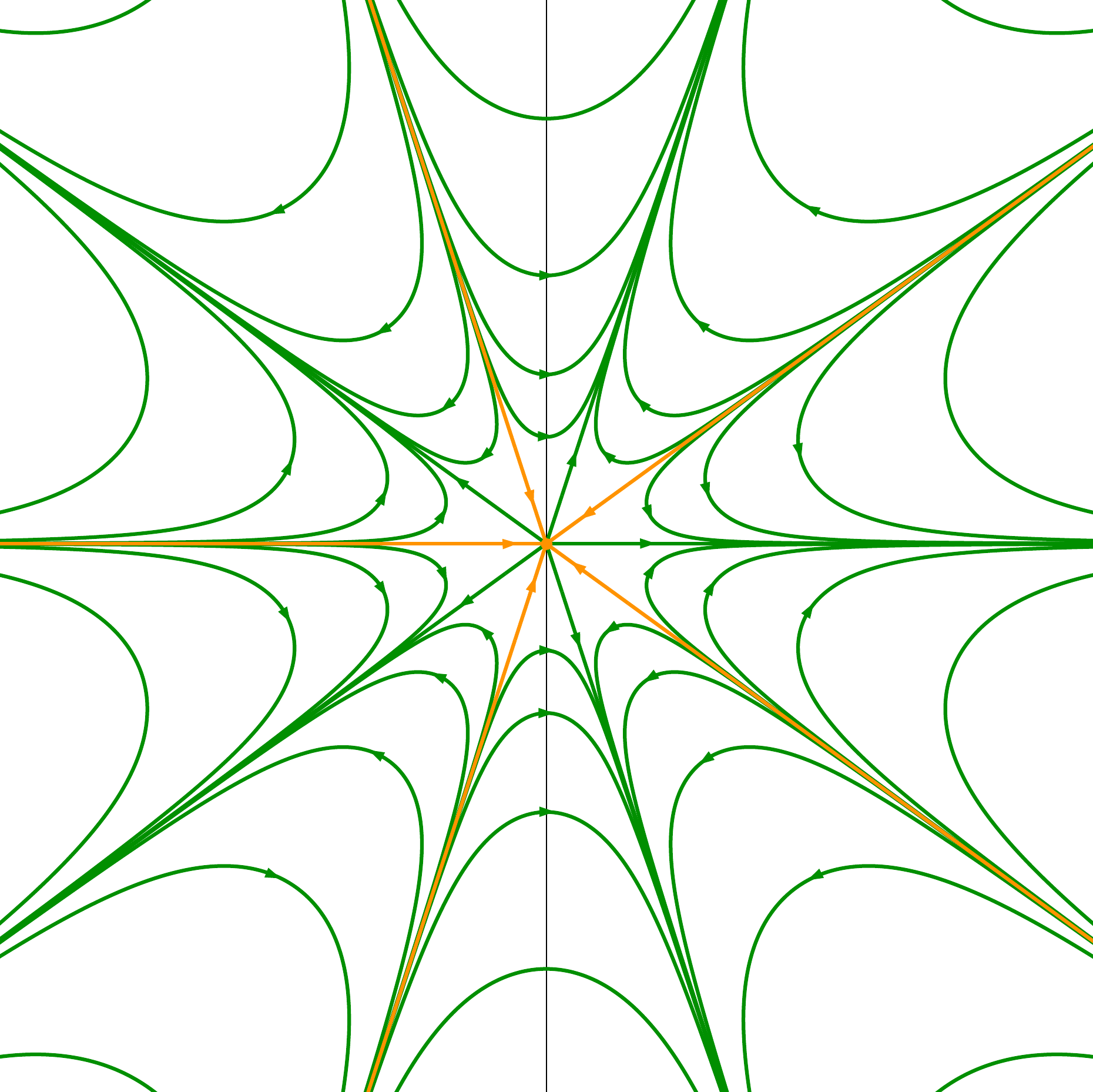}
\caption{$+\lambda$}
\label{250223182216}
\end{subfigure}
\caption{Geodesic foliations of the $z$-plane for the two eigenvalues of the classical Jacobian $\J_0$ for the phase $\alpha = 0$.
Regular geodesics are coloured green, and critical geodesics are coloured orange.}
\label{250223183907}

\end{figure}

\subsection{Initial Value Problem on the Unfolded Borel Space}
\label{250221200950}

We are now ready to reformulate in global geometric terms the Initial Value Problem \eqref{250306084618}, which we quote here for convenience:
\begin{eqntag}
\label{250218203125}
\begin{cases}
	(+\V - \del_\xi) \phi_+ = \Phi
\\	(-\V - \del_\xi) \phi_- = \Phi
\end{cases}
\qqtext{such that}
	\phi_\pm (z,0) = c (z)
\fullstop{,}
\end{eqntag}
where $\V = -6z^{-4} \del_z$, $c(z) = -6z^{-8}$, and $\Phi$ is the following expression in $\phi_+, \phi_-$ with coefficients $a(z) = 3z^{-2}$ and $b(z) = -3z^{-5}$:
\begin{eqn}
	\Phi = \Phi (z, \xi; \phi_+, \phi_-)
	\coleq a (\phi_+ - \phi_-)^{\ast 2} + b (\phi_+ - \phi_-)
\fullstop
\end{eqn}

\paragraph{Lifting the local solution to the unfolded Borel space.}
Recall that in \autoref{250221201337} we constructed a local holomorphic solution $\hat{\phi} = (\hat{\phi}_+, \hat{\phi}_-)$ of this Initial Value Problem valid for all nonzero $z$ and all sufficiently small $\xi$.
More precisely, by \autoref{250221201747}, $\hat{\phi}$ is well-defined on an open neighbourhood $\Xi \subset \CC^\ast_z \times \CC_\xi$ of $\xi = 0$; in symbols, $\hat{\phi} \in \cal{O} (\Xi, \CC^2)$.
By \autoref{250218172045}, there are open neighbourhoods $\Omega_\pm \subset \Pi_1 (\CC_z) |_{\CC^\ast_z}$ of the identity bisection $1_{\CC^\ast_z}$ such that the restriction to $\Omega_\pm$ of the anchor map $\varrho_\pm$ is an isomorphisms $\varrho_\pm : \Omega_\pm \iso \Xi$.
Moreover, the fibre product
\begin{eqn}
	\Omega \coleq \Omega_- \underset{\Xi}{\times} \Omega_+ 
		= \set{ \bm{\gamma} = (\gamma^+, \gamma^-) \in \Omega_+ \times \Omega_- ~\big|~ \varrho_+ (\gamma^+) = \varrho_- (\gamma^-) }
		\subset B
\end{eqn}
is an open neighbourhood in $B$ such that the anchor map $\varrho : B \to \CC_z \times \CC_\xi$ restricts to a biholomorphism $\varrho : \Omega \iso \Xi$.
In addition, there are open subsets $\tilde{\Omega}_\pm \subset \Pi_1 (\CC^\ast_z)$ and $\tilde{\Omega} \subset \tilde{B}$ such that the source-fibrewise universal covering map $\nu : \Pi_1 (\CC^\ast_z) \to \Pi_1 (\CC_z) |_{\CC^\ast_z}$ induces isomorphisms $\tilde{\Omega}_\pm \iso \Omega_\pm$ and $\tilde{\Omega} \iso \Omega$.
Altogether, these isomorphism can be organised into the following commutative diagram where all arrows are biholomorphisms:
\begin{eqn}
\begin{tikzcd}
		\tilde{\Omega}
			\ar[rr, "\pr_+"]
			\ar[dd, "\pr_-"']
			\ar[dr, "\nu" description]
			\ar[ddrr, bend right=20pt, "\varrho"' near start]
	&
	&	\tilde{\Omega}_+
			\ar[d, "\nu"]
			\ar[dd, bend left = 40pt, "\varrho_+"]
\\
	&	\Omega
			\ar[r]
			\ar[d, crossing over]
			\ar[dr, "\varrho" description]
	&	\Omega_+
			\ar[d, "\varrho_+"]
\\
		\tilde{\Omega}_-
			\ar[r, "\nu"']
			\ar[rr, bend right = 25pt, "\varrho_-"']
	&	\Omega_-
			\ar[r, "\varrho_-"']
	&	\Xi
\fullstop
\end{tikzcd}
\end{eqn}
With the help of these isomorphisms, we can naturally and interchangeably treat the local solution $\hat{\phi}$ as a holomorphic vector-valued function on any of the open subsets $\Omega, \Omega_+, \Omega_-$ via pullback by the corresponding anchor map, as well as on any of the open subsets $\tilde{\Omega}, \tilde{\Omega}_+, \tilde{\Omega}_-$ via a further pullback by $\nu$.
Rather than introduce specialised symbols for all these functions, we opt for keeping our expressions uncluttered at the expense of slight abuse of notation.
We denote the pullbacks to both $\Omega_\pm$ and $\tilde{\Omega}_\pm$ by $\hat{\phi}^\tinybrac{\pm} = (\hat{\phi}^\tinybrac{\pm}_+, \hat{\phi}^\tinybrac{\pm}_-)$, and we continue to denote the pullbacks to both $\Omega$ and $\tilde{\Omega}$ simply by $\hat{\phi}$.
In other words, from now on, the symbol ``\,$\hat{\phi}$\,'' denotes both the holomorphic vector-valued function of $(z,\xi) \in \Xi \subset \CC_z \times \CC_\xi$, the corresponding holomorphic vector-valued function on $\Omega \subset B$ obtained via pullback by $\varrho$, as well as the holomorphic vector-valued function on $\tilde{\Omega} \subset B$ obtained via pullback by $\varrho$, with the intended meaning always clear from the context.
Similarly, for the symbols ``\,$\hat{\phi}^\tinybrac{\pm}$\,''.
We summarise these conventions as follows:
\begin{eqns}
	\hat{\phi}^\tinybrac{\pm} 
		= (\hat{\phi}^\tinybrac{\pm}_+, \hat{\phi}^\tinybrac{\pm}_-)
		\coleq \varrho^\ast_\pm \hat{\phi}
\qquad
	&\text{as an element of $\cal{O} (\Omega_\pm, \CC^2)$ or $\cal{O} (\tilde{\Omega}_\pm, \CC^2)$;}
\\
	\hat{\phi}
		= (\hat{\phi}_+, \hat{\phi}_-)
		\coleq \varrho^\ast \hat{\phi}
\qquad
	&\text{as an element of $\cal{O} (\Omega, \CC^2)$ or $\cal{O} (\tilde{\Omega}, \CC^2)$.}
\end{eqns}
More explicitly, for any $\bm{\gamma} = (\gamma^+, \gamma^-)$ in $B$ or in $\tilde{B}$ with sufficiently small central charge $\xi = \Z (\bm{\gamma}) = \pm \Z (\gamma^\pm)$, we write:
\begin{eqntag}
\label{250305174840}
	\hat{\phi} (\bm{\gamma}) 
		= \hat{\phi}^\tinybrac{\pm} (\gamma^\pm)
		\coleq \hat{\phi} \big( \varrho (\bm{\gamma}) \big)
		= \hat{\phi} \big( \varrho_\pm (\gamma^\pm) \big)
		= \hat{\phi} (z, \xi)
\fullstop{,}
\end{eqntag}
where $z = \rm{s} (\bm{\gamma}) = \rm{s} (\gamma^\pm)$ is such that $(z,\xi) \in \Xi$.

\paragraph{Lifting the coefficients to the unfolded Borel space.}
Correspondingly, we can reinterpret the Initial Value Problem \eqref{250218203125} over the open subset $\Omega$ or $\tilde{\Omega}$ as follows.
We pullback the functions $a = a(z), b = b(z), c = c(z)$ using the corresponding anchor maps.
In fact, these functions are globally well-defined meromorphic functions on $\CC_z \times \CC_\xi$.
So, abusing notation, we let
\begin{eqn}
		a_{\pm} \coleq \varrho_\pm^\ast a,
\quad 	b_{\pm} \coleq \varrho_\pm^\ast b,
\quad	c_{\pm} \coleq \varrho_\pm^\ast c
\end{eqn}
denote both their pullbacks to globally well-defined meromorphic functions on $\Pi_1 (\CC_z)$ and their pullbacks to globally well-defined holomorphic functions on $\Pi_1 (\CC^\ast_z)$.
By a similar abuse of notation, we continue to denote their pullbacks by $\varrho$ to global meromorphic functions on $B$ and global holomorphic functions on $\tilde{B}$ simply by $a,b,c$:
\begin{eqn}
	a \coleq \varrho^\ast a
\fullstop{,}
\qqquad
	b \coleq \varrho^\ast b
\fullstop{,}
\qqquad
	c \coleq \varrho^\ast c
\fullstop
\end{eqn}
Explicitly, for any $\bm{\gamma} = (\gamma^+, \gamma^-)$ in $B$ or in $\tilde{B}$ with source $\rm{s} (\bm{\gamma}) = \rm{s} (\gamma^\pm) = z$, we have 
\begin{eqntag}
\label{250305174910}
	a (\bm{\gamma}) = a^\tinybrac{\pm} (\gamma^\pm) \coleq a (z),
\quad
	b (\bm{\gamma}) = b^\tinybrac{\pm} (\gamma^\pm) \coleq b (z),
\quad
	c (\bm{\gamma}) = c^\tinybrac{\pm} (\gamma^\pm) \coleq c (z).
\end{eqntag}

\paragraph{Lifting the vector fields to the unfolded Borel space.}
It remains to interpret the vector fields $\pm \V - \del_\xi$ on the lefthand side of \eqref{250218203125} as vector fields on the Borel space $B$.

\begin{lemma}
\label{250222200044}
The pairs
\begin{eqn}
	\V_+ \coleq \left( \V^\rm{s}, \V^\rm{s} + 2 \V^\rm{t} \right)
\qtext{and}
	\V_- \coleq \left( - \V^\rm{s} - 2 \V^\rm{t}, -\V^\rm{s} \right)
\end{eqn}
define meromorphic vector fields on the algebraic surface $B$.
Similarly, they define holomorphic vector fields on the holomorphic manifold $\tilde{B}$.
They are the unique such vector fields that satisfy $\varrho_\ast \V_\pm = \pm \V - \del_\xi$.
\end{lemma}

\begin{proof}
By definition of the fibre product, a tangent vector to $B$ at $\bm{\gamma} = (\gamma^+, \gamma^-)$ is a pair $(v^+, v^-)$ consisting of a tangent vector to $\Pi_1 (\CC_z)$ at $\gamma^+$ and a tangent vector at $\gamma^-$ such that $(\varrho_+)_\ast v^+ = (\varrho_-)_\ast v^-$ where $(\varrho_\pm)_\ast$ is the derivative of $\varrho_\pm$ at $\gamma^\pm$.
So, for example, to check that $\V_+$ is a well-defined vector field on $B$, we just need to verify the identity $(\varrho_+)_\ast \V^\rm{s} = (\varrho_-)_\ast (\V^\rm{s} + 2 \V^\rm{t})$.
This is a simple calculation using the properties of source- and target-lifts of $\V$ that can be found in \autoref{250307191800}: the lefthand side is $\V - \del_\xi$ and the righthand side is $\V + \del_\xi - 2 \del_\xi$.
\end{proof}

\paragraph{IVP on the unfolded Borel space.}
We can now consider the following Initial Value Problem over the unfolded Borel space $B$:
\begin{eqntag}
\label{250305170843}
\begin{cases}
	\V_+ \phi_+ = \Phi
\\	\V_- \phi_- = \Phi
\end{cases}
\qqtext{such that}
	\phi_\pm \big|_{\CC^\ast_z} = c
\fullstop{,}
\end{eqntag}
where $\Phi$ is the following expression in $\phi_+,\phi_-$ that depends holomorphically on $\bm{\gamma} \in B$:
\begin{eqntag}
\label{250305190114}
	\Phi = \Phi \big( \bm{\gamma}, \phi_+, \phi_- \big)
		\coleq a (\gamma) (\phi_+ - \phi_-)^{\ast 2} (\bm{\gamma}) + b (\gamma) (\phi_+ - \phi_-) (\bm{\gamma})
\fullstop
\end{eqntag}
Recall that the convolution product on $B$ was defined in \autoref{250307192033}.
Then we have the following equivalence.

\begin{lemma}
\label{250222201015}
The Initial Value Problems \eqref{250218203125} and \eqref{250305170843} are equivalent over $\varrho : \Omega \iso \Xi$.
\end{lemma}

Now we rewrite \eqref{250305170843} as an integral equation on the unfolded Borel covering space.
Recall that by \autoref{250305163610}, if the neighbourhood $\Xi \subset \CC^\ast_z \times \CC_\xi$ of $\xi = 0$ is sufficiently small, then the Initial Value Problem \eqref{250218203125} is equivalent for all $(z,\xi) \in \Xi$ to the following integral equation:
\begin{eqntag}
\label{250305165333}
\begin{cases}
\displaystyle
	\phi_+ (z,\xi)
		= c (z) - \int_0^\xi \Phi \Big( z^+_{\xi - s}, s; \phi_+ (z^+_{\xi - s}, s), \phi_- (z^+_{\xi - s}, s) \Big) \d{s}
\fullstop{,}
\\[6pt]
\displaystyle
	\phi_- (z,\xi)
		= c (z) - \int_0^\xi \Phi \Big( z^-_{\xi - s}, s; \phi_+ (z^-_{\xi - s}, s), \phi_- (z^-_{\xi - s}, s) \Big) \d{s}
\fullstop{,}
\end{cases}
\end{eqntag}
where $z^\pm_{\xi - s}$ is the unique point in $\CC^\ast_z$ satisfying the identity
\begin{eqn}
	\xi - s = \pm \int_{z}^{z^\pm_{\xi-s}} \lambda
\fullstop
\end{eqn}
If we interpret $z$ in \eqref{250305165333} as the source of a pair of curves $\gamma^+, \gamma^- : [0,\xi] \to \CC^\ast_z$ with central charge $\Z(\gamma^+) = - \Z (\gamma^-) = \xi$, then the point $z^\pm_{\xi - s}$ in the integrand can be interpreted as the source of the target-truncation $\bar{\gamma}^\pm_s \coleq \gamma^\pm |_{[\xi - s, \xi]}$ with central charge $\pm \Z (\bar{\gamma}^\pm_s) = s$.
Consequently, the two equations in \eqref{250305165333} may be written respectively on $\tilde{\Omega}_+$ and $\tilde{\Omega}_-$ as follows:
\begin{eqntag}
\label{250305173948}
\begin{cases}
\displaystyle
	\phi_+^\tinybrac{+} (\gamma^+)
		= c^\tinybrac{+} (\gamma^+) - \int_0^{+\Z (\gamma^+)} \Phi \Big( \bar{\gamma}^{\,+}_s; \phi^\tinybrac{+}_+ (\bar{\gamma}^{\,+}_s), \phi^\tinybrac{+}_- (\bar{\gamma}^{\,+}_s) \Big) \d{s}
\fullstop{,}
\\[6pt]
\displaystyle
	\phi_-^\tinybrac{-} (\gamma^-)
		= c^\tinybrac{-} (\gamma^-) - \int_0^{-\Z (\gamma^-)} \Phi \Big( \bar{\gamma}^{\,-}_s; \phi^\tinybrac{-}_+ (\bar{\gamma}^{\,-}_s), \phi^\tinybrac{-}_- (\bar{\gamma}^{\,-}_s) \Big) \d{s}
\fullstop
\end{cases}
\end{eqntag}
These equations are valid for all pairs of paths $(\gamma^+, \gamma^-) \in \tilde{\Omega}_+ \times \tilde{\Omega}_-$ with the same source and opposite central charge (i.e., such that $\tilde{\varrho}_+ (\gamma^+) = \tilde{\varrho}_- (\gamma^-)$), and any target-truncations $\bar{\gamma}^\pm_s \coleq \gamma^\pm |_{[\xi - s, \xi]}$ with central charge $\pm \Z (\bar{\gamma}^\pm_s) = s$.
It follows that the pair $\bm{\gamma} = (\gamma^+, \gamma^-)$ defines a point in the open subset $\tilde{\Omega}$ of the unfolded Borel covering space.
Correspondingly, we lift the target-truncations $\bar{\gamma}^\pm_s$ to the two pairs $\bar{\bm{\gamma}}^+_s \coleq \big( \bar{\gamma}^+_s, (\bar{\gamma}^+_s)^\dagger \big)$ and $\bar{\bm{\gamma}}^-_s \coleq \big( (\bar{\gamma}^-_s)^\dagger, \bar{\gamma}^-_s \big)$ that define the target-truncations of $\bm{\gamma}$ as explained in \autoref{250307192033}.
Therefore, taking into account our notational conventions \eqref{250305174840} and \eqref{250305174910} as well as the fact that $\Z (\bm{\gamma}) = \pm \Z (\gamma^\pm)$ and $\Z (\bar{\bm{\gamma}}^\pm_s) = \pm \Z (\bar{\gamma}_s^\pm)$, the system of integral equations \eqref{250305173948} may be written as follows:
\begin{eqntag}
\label{250305173150}
\begin{cases}
\displaystyle
	\phi_+ (\bm{\gamma})
		= c (\bm{\gamma}) - \int_0^{\Z (\bm{\gamma})} \Phi \Big( \bar{\bm{\gamma}}^+_s; \phi_+ (\bar{\bm{\gamma}}^+_s), \phi_- (\bar{\bm{\gamma}}^+_s) \Big) \d{s}
\fullstop{,}
\\[6pt]
\displaystyle
	\phi_- (\bm{\gamma})
		= c (\bm{\gamma}) - \int_0^{\Z (\bm{\gamma})} \Phi \Big( \bar{\bm{\gamma}}^-_s; \phi_+ (\bar{\bm{\gamma}}^-_s), \phi_- (\bar{\bm{\gamma}}^-_s) \Big) \d{s}
\fullstop
\end{cases}
\end{eqntag}
Finally, reparameterising the target-truncations $\bar{\bm{\gamma}}^\pm_s$ using their arc-length, we get the following lemma.

\begin{lemma}
\label{250305181745}
The Initial Value Problem \eqref{250305170843} is equivalent to the following system of integral equations over the unfolded Borel covering space $\tilde{B}$:
\begin{eqntag}
\label{250305181917}
\begin{cases}
\displaystyle
	\phi_+ (\bm{\gamma})
		= c (\bm{\gamma}) - \int_0^{|\bm{\gamma}|} \Phi \Big( \bar{\bm{\gamma}}^+_s; \phi_+ (\bar{\bm{\gamma}}^+_s), \phi_- (\bar{\bm{\gamma}}^+_s) \Big) \D{s}
\fullstop{,}
\\[6pt]
\displaystyle
	\phi_- (\bm{\gamma})
		= c (\bm{\gamma}) - \int_0^{|\bm{\gamma}|} \Phi \Big( \bar{\bm{\gamma}}^-_s; \phi_+ (\bar{\bm{\gamma}}^-_s), \phi_- (\bar{\bm{\gamma}}^-_s) \Big) \D{s}
\fullstop{,}
\end{cases}
\end{eqntag}
where $\D{s} \coleq \pm \rm{t}^\ast \lambda_{\gamma^\pm_s}$ and $\bm{\gamma}^\pm_s$ are the two arc-length target-truncations of any synchronised parameterisation of $\bm{\gamma}$.
\end{lemma}

\section{Endless Analytic Continuation}
\label{250304161635}

In this final section, we construct the global solution $\phi$ of the Initial Value Problem \eqref{250305170843}, or equivalently the integral equation \eqref{250305181917}, and finish the proof of our main \autoref{250304131355} as well as all the remaining results stated in \autoref{250313191046}.
Our construction of $\phi$, presented in \autoref{250304154803}, uses the Contraction Mapping Principle over the Borel covering space which in particular allows for tight control on the behaviour of $\phi$ at infinity.
This allows us to deduce in \autoref{250311190749} the resurgence properties of the formal vector-valued solution $\hat{f}$ of the differential system \eqref{250216174108} and consequently the resurgence properties of the formal solution $\hat{q}$ of the deformed Painlevé I equation.
Finally, in \autoref{250311191044}, we flesh out the Borel summability properties of $\hat{f}$ and associated Stokes phenomenon and explain how they imply the analogous properties of $\hat{q}$ that were stated in \autoref{250304162404}.

\subsection{A Contraction Mapping}
\label{250304154803}

In this subsection, we construct the global solution of the integral equation \eqref{250305181917}.
The main result, which does the majority of this paper's (mild) analytic heavy-lifting, can be formulated as follows.

\begin{proposition}
\label{250311180618}
There is a unique vector-valued global holomorphic function $\phi \in \cal{O} (\tilde{B}, \CC^2)$ on the unfolded Borel covering space $\tilde{B}$ that satisfies the integral equation \eqref{250305181917} or equivalently the Initial Value Problem \eqref{250305170843}.
\end{proposition}

This subsection is devoted to the proof of this Proposition.
We employ the familiar strategy of realising the solution $\phi$ as the fixed point of a contraction mapping $\Upsilon$ on a suitable function space.
This function space is the infinite-dimensional vector space $\cal{O} (\tilde{B}, \CC^2)$ of vector-valued global holomorphic functions over the unfolded Borel covering space $\tilde{B}$.
However, the Contraction Mapping Principle for $\Upsilon$ fails if we allow paths on $\CC^\ast_z$ to come arbitrarily close to the origin.
To address this problem, we will introduce an exhaustion of $\CC^\ast_z$ by open subsets of the punctured $z$-plane obtained by excising closed discs around the origin of smaller and smaller radius.
This will allow us to introduce a sequence of suitable Banach spaces of functions over the unfolded Borel covering space where $\Upsilon$ can be shown to have a unique fixed point using the Contraction Mapping Principle.
We describe this excision process in \autoref{250221134517} and define the relevant Banach spaces in \autoref{250221134607}.
Then in \autoref{250221134622} we prove \autoref{250311180618} through a sequence of lemmas.

\subsubsection{The excised unfolded Borel covering space.}
\label{250221134517}

Let $\bar{\DD} \subset \CC_z$ be any compact and simply connected domain around the origin, such as a closed disc centred at the origin of finite nonzero radius, and consider the \dfn{excised $z$-plane}
\begin{eqn}
	\CC^\circ_z \coleq \CC_z \smallsetminus \bar{\DD}
\fullstop
\end{eqn}
Then the spaces $\CC^\circ_z$ and $\CC^\ast_z$ are homotopy equivalent, which means any concrete path in $\CC^\ast_z$ with source and target outside of $\bar{\DD}$ is homotopy equivalent to a concrete path entirely contained in $\CC_z \smallsetminus \bar{\DD}$.
This defines a canonical isomorphism
\begin{eqn}
	\Pi_1 (\CC^\circ_z) \iso \Pi_1 (\CC^\ast_z) \big|_{\CC^\circ_z}
\fullstop
\end{eqn}
The upshot is that the constructions of \autoref{250304152932} go through without any substantial modifications, resulting in the \dfn{excised unfolded Borel covering space}:
\begin{eqn}
	\tilde{B}^\circ
		\coleq \Pi_1 (\CC^\circ_z) \underset{\varrho_- \, \varrho_+}{\times} \Pi_1 (\CC^\circ_z)
\fullstop
\end{eqn}
Consequently, $\tilde{B}^\circ$ inherits a source map $\rm{s} : \tilde{B}^\circ \to \CC^\circ_z$, a central charge $\Z : \tilde{B}^\circ \to \CC_\xi$, and therefore an anchor map $\varrho : \tilde{B}^\circ \to \CC^\circ_z \times \CC_\xi$.
There is an obvious inclusion $\tilde{B}^\circ \subset \tilde{B}$ which translates into the inclusion $\cal{O} (\tilde{B}, \CC^2) \inj \cal{O} (\tilde{B}^\circ, \CC^2)$ of vector spaces.

\subsubsection{A Banach algebra over the unfolded Borel covering space.}
\label{250221134607}

Let $\CC^\circ_z \subset \CC^\ast_z$ be any excised $z$-plane in the above sense, and let $\tilde{B}^\circ$ be the corresponding excised unfolded Borel covering space.
Then we can define the following norm on the infinite-dimensional complex vector space of global holomorphic functions over $\tilde{B}^\circ$.

\begin{definition}
\label{250217202735}
For any $\K > 0$, we define the norm of any $f \in \cal{O} (\tilde{B}^\circ)$ by
\begin{eqntag}
\label{250131144547}
	\norm{ f }_\K \coleq 
		\sup_{\bm{\gamma} \in \tilde{B}^\circ}
		\inf
		\int_0^{|\bm{\gamma}|} e^{- \K s} \big| f (\bm{\gamma}_s) \D{s} \big|
\fullstop{,}
\end{eqntag}
where the infimum is taken over all source-truncations $\bm{\gamma}_s$ of $\bm{\gamma}$.
We also extend this norm to the vector spaces $\mathcal{O} (\tilde{B}^\circ, \CC^n)$ and $\mathcal{O} (\tilde{B}^\circ, \rm{M}_n)$ of vector-valued and matrix-valued functions by replacing the absolute value $| \cdot |$ with any extension to a vector and matrix norm; for definiteness, we take these to be the Euclidean norm and the Frobenius norm, respectively.
\end{definition}

\begin{definition}
\label{250221083403}
We introduce the vector subspace of $\cal{O} (\tilde{B}^\circ)$ of functions which are bounded with respect to this norm:
\begin{eqntag}
\label{250218131343}
	\mathscr{H}_\K (\tilde{B}^\circ)
		\coleq \set{ f \in \mathcal{O} (\tilde{B}^\circ) ~\Big|~ \norm{f}_\K < \infty }
\fullstop
\end{eqntag}
We will also consider the closed unit ball in this vector space, which we denote by
\begin{eqntag}
\label{250219221924}
	\mathscr{H}_\K^\heartsuit (\tilde{B}^\circ)
		\coleq \set{ f \in \mathcal{O} (\tilde{B}^\circ) ~\Big|~ \norm{f}_\K \leq 1 }
\fullstop
\end{eqntag}
We define the vector subspaces $\mathscr{H}_\K (\tilde{B}^\circ, \CC^n)$ and $\mathscr{H}_\K (\tilde{B}^\circ, \rm{M}_n)$ as well as the unit balls inside them accordingly.
\end{definition}

\begin{lemma}
\label{250131103416}
For any $\K > 0$ and any $n \geq 1$, the space $\mathscr{H}_\K (\tilde{B}^\circ, \CC^n)$ is a Banach space and the space $\mathscr{H}_\K (\tilde{B}^\circ, \rm{M}_n)$ is a Banach algebra with respect to the convolution product $\ast$; i.e., for all $f,g \in \mathscr{H}_\K (\tilde{B}^\circ, \rm{M}_n)$, we have the following inequality:
\begin{eqntag}
\label{250219210228}
	\norm{ f \ast g }_\K \leq \norm{f}_\K \norm{g}_\K
\fullstop
\end{eqntag}
In particular, the norm of the constant function is $\norm{ 1 }_\K = 1/\K$.
\end{lemma}

\begin{proof}
It is enough to consider the space $\mathscr{H}_\K (\tilde{B}^\circ)$.

\emph{Banach space structure.}
First, we show that the norm \eqref{250131144547} is complete, making $\big( \mathscr{H}_\K (\tilde{B}^\circ), \norm{\shortdummy} \big)$ into a Banach space.
The argument is very similar to the standard argument for the completeness of the $\L^1$-norm.

Let $(f_n)_{n=0}^\infty$ be a Cauchy sequence of functions in $\mathscr{H}_\K (\tilde{B}^\circ)$.
Recall that this means that for any $\epsilon > 0$ there is some $\N$ such that $\norm{f_n - f_m}_\K < \epsilon$ whenever $n,m > \N$.
We need to show that it converges in this norm; i.e., there is a function $f \in \mathscr{H}_\K (\tilde{B}^\circ)$ such that $\norm{f_n - f}_\K \to 0$ as $n \to \infty$.

Consider a subsequence $(f_{n_k})_{k=0}^\infty$ of $(f_n)$ such that $\norm{f_{n_{k+1}} - f_{n_k}}_\K \leq 2^{-k}$ for all $k \geq 0$.
We claim that the following telescopic series is the desired limit of $(f_n)$:
\begin{eqn}
	f \coleq f_{n_0} + \sum_{k=1}^\infty \big( f_{n_{k+1}} - f_{n_k} \big)
\fullstop
\end{eqn}
To see this, note that for any positive integer $\N$, its $\N$-th partial sum is bounded in the $\norm{\shortdummy}_\K$-norm by
\begin{eqn}
	\norm{ f_{n_0} }_\K + \sum_{k=1}^\N \norm{ f_{n_{k+1}} - f_{n_k} \big. }_\K
	\leq \norm{ f_{n_0} }_\K + \sum_{k=1}^\N \tfrac{1}{2^k}
\fullstop
\end{eqn}
This means that the series defining $f$ converges in the $\norm{\shortdummy}_\K$-norm and that the subsequence $(f_{n_k})$ converges to $f$.
Moreover, $f$ is a holomorphic function on $\tilde{B}^\circ$ because it is the limit of partial sums of a sequence of holomorphic functions on $\tilde{B}^\circ$.
We conclude that $f \in \mathscr{H}_\K (\tilde{B}^\circ)$, and so it remains to show that $f$ is the limit of $(f_n)$.
This follows from the fact that $(f_n)$ is a Cauchy sequence.
Indeed, for any $\epsilon > 0$, let $\N$ be such that $\norm{f_n - f_m}_\K < \epsilon / 2$ whenever $n,m > \N$.
Let $\M$ be such that $n_{\M} > \N$ and $\norm{ f_{n_{\M}} - f } < \epsilon / 2$.
Then
\begin{eqn}
	\norm{ f_n - f }_\K
	\leq \norm{ f_n - f_{n_\M} }_\K + \norm{ f_{n_\M} - f }_\K
	< \epsilon
\fullstop
\end{eqn}

\emph{Banach algebra structure.}
Next, we show that $\big( \mathscr{H}_\K (\tilde{B}^\circ), \norm{\shortdummy} \big)$ is a Banach algebra with respect to the convolution product defined in \autoref{250307163506}.
Suppose $f, g \in \mathscr{H}_\K (\tilde{B}^\circ)$; we need to demonstrate the inequality \eqref{250219210228}.
For any $\bm{\gamma} \in \tilde{B}^\circ$, choose any synchronously parameterised representative and let $\L \coleq |\bm{\gamma}|$.
Then we compute as follows:
\begin{eqns}
	\int_0^{\L} e^{-\K s} 
		\big| (f \ast g) (\bm{\gamma}_s) \D{s} \big|
	&\leq \int_0^{\L} e^{-\K s} 
		\int_0^{s} \big| f (\bm{\gamma}_{s,u}) \big| 
			\big| g (\bm{\gamma}_{s,s-u}) \big| 
			\big| \D{u} \D{s} \big|
\\	&= \int_0^{\L} 
		\int_0^{s} e^{-\K u} \big| f (\bm{\gamma}_{s,u}) \big| 
			e^{-\K (s-u)} \big| g (\bm{\gamma}_{s,s-u}) \big| 
			\big| \D{u} \D{s}\big|
\\	&\leq \int_0^{\L} 
		\int_0^{\L} e^{-\K u} \big| f (\bm{\gamma}_{u}) \big| 
			e^{-\K r} \big| g (\bm{\gamma}_{r}) \big| 
			\big| \D{u} \D{r}\big|
\\	&= \int_0^{\L} e^{-\K u} \big| f (\bm{\gamma}_{u}) \D{u} \big| 
		\cdot \int_0^{\L} e^{-\K r} \big| g (\bm{\gamma}_{r}) \D{r} \big|
\fullstop
\end{eqns}
Here, $\bm{\gamma}_s$ is the source-truncation of $\bm{\gamma}$ of length $s$, and $\bm{\gamma}_{s,u}$ is the source-truncation of $\bm{\gamma}_s$ of length $u$.
Going from the second line to the third, we first noticed that $f (\bm{\gamma}_{s,u}) = f (\bm{\gamma}_u)$ and $g (\bm{\gamma}_{s,s-u}) = g (\bm{\gamma}_{s-u})$, and then we changed the integration variable $s$ to $r = s - u$ and extended the integration domain in $r$.
The last line is the result of applying Fubini's Theorem.
Taking the infimum and then the supremum of both sides of this inequality yields \eqref{250219210228}.
\end{proof}

\subsubsection{A contraction mapping.}
\label{250221134622}
We now turn to the proof of \autoref{250311180618} by showing that the following integral operator on $\cal{O} (\tilde{B}, \CC^2)$ admits a unique fixed point.

\begin{definition}
\label{250305182841}
We define a mapping
\begin{eqntag}
	\Upsilon = (\Upsilon_+, \Upsilon_-) : \cal{O} (\tilde{B}, \CC^2) \to \cal{O} (\tilde{B}, \CC^2)
\fullstop{,}
\end{eqntag}
whose components $\Upsilon_\pm : \cal{O} (\tilde{B}, \CC^2) \to \cal{O} (\tilde{B})$ are given on all $\phi = (\phi_+, \phi_-) \in \cal{O} (\tilde{B}, \CC^2)$ and all $\bm{\gamma} \in \tilde{B}$ by the following formula:
\begin{eqntag}
\label{250219195751}
	\Upsilon_\pm \big[ \, \phi \, \big] (\bm{\gamma})
	\coleq c(\bm{\gamma}) 
		- \int_0^{|\bm{\gamma}|} \Phi \Big( \bar{\bm{\gamma}}^\pm_s; \phi_+ (\bar{\bm{\gamma}}^\pm_s), \phi_- (\bar{\bm{\gamma}}^\pm_s) \Big) \D{s}
\fullstop{,}
\end{eqntag}
where $\D{s} = \pm \rm{t}^\ast_{\pm} \lambda_{\bm{\gamma}_s}$ and $\bar{\bm{\gamma}}^\pm_s$ are the two arc-length target-truncations of any concrete representative of $\bm{\gamma}$.
For any excised complex plane $\CC^\circ_z \subset \CC^\ast_z$, the mapping $\Upsilon$ readily extends to a mapping $\Upsilon : \cal{O} (\tilde{B}^\circ, \CC^2) \to \cal{O} (\tilde{B}^\circ, \CC^2)$ by the very same formula.
\end{definition}

First, we show that the restriction of $\Upsilon$ to every Banach space $\mathscr{H}_\K (\tilde{B}^\circ, \CC^2)$ is well-defined.

\begin{lemma}
\label{250218210947}
Let $\CC^\circ_z \subset \CC^\ast_z$ be any excised complex $z$-plane, and let $\tilde{B}^\circ$ be the corresponding excised unfolded Borel covering space.
Then, for any $\K > 0$, the integral operator $\Upsilon$ restricts to a mapping
\begin{eqntag}
\label{250219224527}
	\Upsilon : \mathscr{H}_\K (\tilde{B}^\circ, \CC^2) \to \mathscr{H}_\K (\tilde{B}^\circ, \CC^2)
\fullstop
\end{eqntag}
\end{lemma}

\begin{proof}
Recall that $a(z) = 3z^{-2}$, $b(z) = -3z^{-5}$, and $c(z) = -6z^{-8}$.
They are bounded functions on any excised $z$-plane $\CC^\circ_z$, so there is a constant $\C > 0$ that gives the following bounds:
\begin{eqntag}
\label{250221093210}
	\big| a(z) \big|, \big| b(z) \big|, \big| c (z) \big| \leq \C
\qquad
	\forall\, z \in \CC^\circ_z
\fullstop
\end{eqntag}
Since the elements of $\tilde{B}^\circ$ are pairs of paths entirely contained in $\CC^\circ_z$, these bounds automatically translate into the bounds over the space $\tilde{B}^\circ$:
\begin{eqntag}
\label{250219204814}
	\big| a (\bm{\gamma}) \big|, \big| b (\bm{\gamma}) \big|, \big| c (\bm{\gamma}) \big| \leq \C
\qquad
	\forall\, \bm{\gamma} \in \tilde{B}^\circ
\fullstop
\end{eqntag}

Now, take an arbitrary element $\phi = (\phi_+, \phi_-) \in \mathscr{H}_\K (\tilde{B}^\circ, \CC^2)$.
We need to show that $\norm{ \Upsilon [ \phi ] }_\K < \infty$.
To this end, take an arbitrary synchronously parameterised $\bm{\gamma} \in \tilde{B}^\circ$, let $\L \coleq |\bm{\gamma}|$, and estimate the individual terms making up the expression
\begin{eqn}
	\int_0^{\L} e^{- \K s} \big| \Upsilon_\pm [\phi] (\bm{\gamma}_s) \D{s} \big|
\fullstop
\end{eqn}
Explicitly using \eqref{250219195751} and \eqref{250305190114}, this quantity is bounded by
\begin{eqns}
	\int_0^{\L} e^{- \K s} | c (\bm{\gamma}_s) \D{s} \big|
		&+ \int_0^{\L} e^{- \K s} \int_0^{s} \Big| \Phi_\pm \Big( \bar{\bm{\gamma}}^\pm_{s,u}; \phi_+ (\bar{\bm{\gamma}}^\pm_{s,u}), \phi_- (\bar{\bm{\gamma}}^\pm_{s,u}) \Big) \D{u}\D{s} \Big|
\\
	\leq	\int_0^{\L} e^{- \K s} | c (\bm{\gamma}_s) \D{s} \big|
		&+ \int_0^{\L} e^{- \K s} 
			\int_0^{s} \Big| a (\bar{\bm{\gamma}}^\pm_{s,u}) (\phi_+ - \phi_-)^{\ast 2} (\bar{\bm{\gamma}}^\pm_{s,u}) 
				\D{u} \D{s} \Big|
\\		&+ \int_0^{\L} e^{- \K s} 
			\int_0^{s} \Big| b (\bar{\bm{\gamma}}^\pm_{s,u}) (\phi_+ - \phi_-) (\bar{\bm{\gamma}}^\pm_{s,u}) 
				\D{u} \D{s} \Big|
\fullstop
\end{eqns}
Here, $\bm{\gamma}_s$ is the source-truncation of $\bm{\gamma}$ of length $s$, whilst $\bar{\bm{\gamma}}^\pm_{s,u}$ are the two target-truncations of $\bm{\gamma}_s$ of length $u$.
For the first term, we find:
\begin{eqn}
	\int_0^{\L} e^{-\K s} \big| c (\bm{\gamma}_s) \D{s} \big|
	\leq \C \int_0^{\L} e^{-\K s} |\D{s}|
\fullstop
\end{eqn}
The infimum of this quantity is bounded by $\C \norm{ 1 }_\K = \C / \K$, which is independent of $\bm{\gamma}$.
Therefore, we conclude that the supremum is finite and
\begin{eqntag}
\label{250221093419}
	\norm{ c }_\K \leq \C / \K
\fullstop
\end{eqntag}

Next, we estimate the terms involving $b \phi_\pm$ using \eqref{250219204814} and a similar line of reasoning as the one used in the proof of \autoref{250131103416}:
\begin{eqns}
	\int_0^{\L} e^{- \K s} 
		\int_0^{s} \Big| b (\bar{\bm{\gamma}}^\pm_{s,u}) \phi_\pm (\bar{\bm{\gamma}}^\pm_{s,u})
			\D{u} \D{s} \Big|
	&\leq \C \int_0^{\L}
		\int_0^{s} e^{- \K (s-u)} e^{- \K u} \big| \phi_\pm (\bar{\bm{\gamma}}^\pm_{s,u}) 
			\D{u} \D{s} \big|
\\
	&\leq \C \int_0^{\L}
		\int_0^{\L} e^{- \K r} e^{- \K u} \big| \phi_\pm (\bar{\bm{\gamma}}^\pm_{u}) 
			\D{u} \D{r} \big|
\fullstop
\end{eqns}
The infimum of this quantity is bounded by $\C \norm{1}_\K \norm{ \phi_\pm }_\K = (\C / \K) \norm{ \phi_\pm }_\K$ which is independent of $\bm{\gamma}$.
Therefore, we conclude that the supremum is finite and
\begin{eqntag}
\label{250221104144}
	\norm{ b \phi_\pm }_\K \leq (\C / \K) \norm{ \phi_\pm }_\K
\fullstop
\end{eqntag}
We can similarly estimate the terms involving $a \phi_+ \ast \phi_+,~ a \phi_+ \ast \phi_-,~ a \phi_- \ast \phi_-$.
For any one of these terms $a \phi_\bullet \ast \phi_\bullet$, we find:
\begin{eqn}
\nonumber
	\inf \int_0^{\L} e^{- \K s} \int_0^{s}
		\big| a (\bar{\bm{\gamma}}^\pm_{s,u}) \big( \phi_\bullet \ast \phi_\bullet \big) (\bar{\bm{\gamma}}^\pm_{s,u}) 
			\D{u} \D{s} \big|
	\leq \C \norm{1}_\K \norm{ \phi_\bullet \ast \phi_\bullet }_\K
\fullstop
\end{eqn}
where the infimum is taken over all synchronised parameterisations of $\bm{\gamma}$.
Therefore, we conclude that the supremum of the lefthand side is finite and
\begin{eqntag}
\label{250221104408}
	\norm{ a \phi_\bullet \ast \phi_\bullet }_\K \leq (\C / \K ) \norm{ \phi_\bullet }_\K \norm{ \phi_\bullet }_\K
\fullstop
\end{eqntag}
We conclude that $\norm{ \Upsilon [ \phi ] }_\K < \infty$.
\end{proof}

Next, we prove the key contractive property of $\Upsilon$ by collecting the estimates obtained in the proof just finished.

\begin{lemma}
\label{250221092010}
Let $\CC^\circ_z \subset \CC^\ast_z$ be any excised complex $z$-plane, and let $\tilde{B}^\circ$ be the corresponding excised unfolded Borel covering space.
Then there is a sufficiently large $\K > 0$ such that the integral operator $\Upsilon$ restricts to a contraction mapping
\begin{eqntag}
\label{250219211219}
	\Upsilon : \mathscr{H}^\heartsuit_\K (\tilde{B}^\circ, \CC^2) \to \mathscr{H}^\heartsuit_\K (\tilde{B}^\circ, \CC^2)
\fullstop
\end{eqntag}
That is, there is a strictly positive real number $\chi < 1$ such that, for all $\phi,\psi \in \mathscr{H}^\heartsuit_\K (\tilde{B}^\circ, \CC^2)$,
\begin{eqntag}
\label{250221113514}
	\norm{ \Upsilon [0] \big. }_\K < 1 - \chi
\qqtext{and}
	\norm{ \Upsilon [\phi] - \Upsilon [\psi] \big.}_\K 
		\leq \chi \norm{ \phi - \psi }_\K
\fullstop
\end{eqntag}
\end{lemma}

\begin{proof}
First, let us show that $\Upsilon$ restricts to a mapping \eqref{250219211219} for a sufficiently large $\K > 0$.
Let $\C > 0$ be the same constant as in the proof of \autoref{250218210947}, so that the bounds \eqref{250221093210} are true.
Pick an arbitrary $\bm{\gamma} \in \tilde{B}^\circ$.
Then, for any $\K > 0$ and any $\phi \in \mathscr{H}_\K (\tilde{B}^\circ, \CC^2)$, upon combining the estimates \eqref{250221093419}, \eqref{250221104144}, and \eqref{250221104408}, we find:
\begin{eqn}
	\inf
	\int_0^{|\bm{\gamma}|} \!\! e^{- \K s} \big| \Upsilon_\pm [\phi] (\bm{\gamma}_s) 
		\D{s} \big|
	\leq (\C / \K) \Big( 1 + 2 \norm{ \phi_+ }_\K \norm{ \phi_- }_\K + \norm{ \phi_+ }_\K^2 + \norm{ \phi_- }_\K^2
		+ \norm{ \phi_+ }_\K + \norm{ \phi_- }_\K \Big)
\fullstop{,}
\end{eqn}
where, as ever, the infimum is taken over all synchronised parameterisations of $\bm{\gamma}$.
Note that the righthand side is independent of $\bm{\gamma}$, hence the supremum over all $\bm{\gamma}$ of the lefthand side is also bounded by this quantity.
Therefore, if $\phi \in \mathscr{H}^\heartsuit_\K (\tilde{B}^\circ, \CC^2)$ so that $\norm{ \phi_\pm }_\K \leq 1$, then
\begin{eqn}
	\norm{ \Upsilon_\pm [ \phi ] }_\K \leq 7 \C / \K
\qqtext{and so}
	\norm{ \Upsilon [ \phi ] }_\K \leq 7 \C \sqrt{2} / \K
\fullstop
\end{eqn}
We conclude that $\norm{ \Upsilon [ \phi ] }_\K \leq 1$ provided $\K \geq 7 \C \sqrt{2}$.

Next, we show that $\Upsilon$ is a contraction mapping on the unit ball $\mathscr{H}^\heartsuit_\K (\tilde{B}^\circ, \CC^2)$ for all sufficiently large $\K > 0$.
First, let us note that $\Upsilon_\pm [0] (\bm{\gamma}) = c (\bm{\gamma})$ for all $\bm{\gamma} \in \tilde{B}^\circ$.
From \eqref{250221093419}, we know that $\norm{ c }_\K \leq \C / \K$ for all $\K > 0$, and so
\begin{eqn}
	\norm{ \Upsilon [0] }_\K \leq \C \sqrt{2} / \K
\fullstop
\end{eqn}
Next, for any $\K > 0$, any $\phi, \psi \in \mathscr{H}^\heartsuit_\K (\tilde{B}^\circ, \CC^2)$, and any synchronously parameterised $\bm{\gamma} \in \tilde{B}^\circ$, the expression
\begin{eqntag}
\label{250221112240}
	\int_0^{|\bm{\gamma}|} e^{- \K s} \Big| \Upsilon_\pm [\phi] (\bm{\gamma}_s) - \Upsilon_\pm [\psi] (\bm{\gamma}_s) \Big| \big| \D{s} \big|
\end{eqntag}
is bounded by
\begin{eqns}
		&\int_0^{|\bm{\gamma}|} e^{- \K s} 
			\int_0^{|\bm{\gamma}_s|} 
				\Big| a (\bar{\bm{\gamma}}^\pm_{s,u}) 
					\Big( (\phi_+ - \phi_-)^{\ast 2} - (\psi_+ - \psi_-)^{\ast 2} \Big) 
					(\bar{\bm{\gamma}}^\pm_{s,u}) \D{u} \D{s} \Big|
\\		&+\int_0^{|\bm{\gamma}|} e^{- \K s} 
			\int_0^{|\bm{\gamma}_s|} 
				\Big| b (\bar{\bm{\gamma}}^\pm_{s,u}) 
					\Big( (\phi_+ - \phi_-) - (\psi_+ - \psi_-) \Big) 
					(\bar{\bm{\gamma}}^\pm_{s,u}) \D{u} \D{s} \Big|
\fullstop
\end{eqns}
The goal is to estimate this expression in terms of $\norm{ \phi - \psi }_\K$.
The term involving $b$ is easy to deal with because $\big| b \big( (\phi_+ - \phi_-) - (\psi_+ - \psi_-) \big) \big| \leq |b| \big( \big| \phi_+ - \psi_+ \big| + \big| \phi_- - \psi_- \big| \big)$.
Consequently, we can argue just as we did in \eqref{250221104144} that the infimum of the term involving $b$ is bounded by $2(\C/\K) \norm{ \phi - \psi }_\K$.
For the term involving $a$, we rewrite $(\phi_+ - \phi_-)^{\ast 2} - (\psi_+ - \psi_-)^{\ast 2}$ as follows:
\begin{eqns}
		& \big( \phi_+^{\ast 2} - \psi_+^{\ast 2} \big) + \big( \phi_-^{\ast 2} - \psi_-^{\ast 2} \big)
			- 2 \phi_+ \ast \phi_- + 2 \psi_+ \ast \psi_-
\\		=& \big( \phi_+ - \psi_+ \big) \ast (\phi_+ + \psi_+)
			+ \big( \phi_- - \psi_- \big) \ast (\phi_- + \psi_-)
			- 2\phi_+ \ast (\phi_- - \psi_-) - 2 (\phi_+ - \psi_+) \ast \psi_-
\\		=& \big( \phi_+ - \psi_+ \big) \ast (\phi_+ + \psi_+ - 2 \psi_-)
			+ \big( \phi_- - \psi_- \big) \ast (\phi_- + \psi_- - 2 \phi_+)
\fullstop
\end{eqns}
Therefore, using the fact that $\norm{\phi_\pm}_\K \leq \norm{\phi}_\K \leq 1$, etc., we find that the infimum of the term involving $a$ is bounded by $8 (\C / \K) \norm{ \phi - \psi }_\K$.
Altogether, we find that \eqref{250221112240} is bounded by $10 (\C / \K) \norm{ \phi - \psi }_\K$, and so 
\begin{eqn}
	\norm{ \Upsilon [\phi] - \Upsilon [\psi] \big.}_\K \leq 10 \sqrt{2} (\C / \K) \norm{ \phi - \psi }_\K
\fullstop
\end{eqn}
If we put $\chi \coleq 10 \sqrt{2} (\C / \K)$, then in order for \eqref{250221113514} to be true, we need to choose $\K$ such that $ 10 \sqrt{2} (\C / \K) < 1$ and $\C \sqrt{2} / \K < 1 - 10 \sqrt{2} (\C / \K)$.
Both of these inequalities are satisfied if $\K > 11 \C \sqrt{2}$.
For example, we can take $\K = 16 \C$, in which case $\chi = 5\sqrt{2}/8$.
\end{proof}

Finally, we can be more specific about the choice of constants $\C$ and $\K$ if we are more specific about the excision $\CC_z^\circ = \CC_z \smallsetminus \bar{\DD}$.

\begin{lemma}
\label{250221114438}
Fix $\R > 0$ and let $\bar{\DD} = \bar{\DD (0, \R)}$ be the closed disc around the origin of radius $\R > 0$.
Consider the excised complex $z$-plane $\CC_z^\circ = \CC_z \smallsetminus \bar{\DD}$, and let $\tilde{B}^\circ$ be the corresponding excised unfolded Borel covering space.
Then we can be more specific about the choice of constant $\K$ in \autoref{250221092010}.
Namely, we can take $\K = 96 / \R^8$ if $\R \leq \sqrt[6]{2}$ and $\K = 48 / \R^2$ if $\R \geq \sqrt[6]{2}$.
In particular, we can take $\K$ arbitrarily small provided that $\R$ is sufficiently large.
\end{lemma}

\begin{proof}
From the proof of \autoref{250221092010}, we know that we can take $\K = 16 \C$ where $\C$ is any constant that ensures the bounds \eqref{250221093210}.
Recall that $a(z) = 3z^{-2}$, $b(z) = -3z^{-5}$, and $c(z) = -6z^{-8}$, so
\begin{eqntag}
	\big| a(z) \big| \leq 3/\R^2~,~~ \big| b(z) \big| \leq 3 / \R^5~,~~ \big| c (z) \big| \leq 6/\R^8
\qquad
	\forall\, z \in \CC^\circ_z
\fullstop
\end{eqntag}
Thus, we can take $\C = 6 / \R^8$ if $\R \leq \sqrt[6]{2}$ and $\C = 3 / \R^2$ if $\R \geq \sqrt[6]{2}$.
\end{proof}

\begin{proof}[Proof of \autoref{250311180618}.]
Choose any decreasing sequence of strictly positive real numbers $(\R_n)$ that converges to $0$; i.e., $\R_n > \R_{n+1}$ and $\R_n \to 0$ as $n \to \infty$.
For each $n$, consider the excised $z$-plane $\CC^\circ_{z,n} \coleq \CC_z \smallsetminus \bar{\DD (0, \R_n)}$.
We obtain a nested increasing sequence of excised $z$-planes which exhausts the punctured $z$-plane $\CC^\ast_z$:
\begin{eqn}
	\CC^\circ_{z,n} \subset \CC^\circ_{z, n+1} \subset \cdots \subset \CC^\ast_z
\fullstop
\end{eqn}
In turn, each $\CC^\circ_{z,n}$ induces a corresponding excised unfolded Borel covering space that we denote by $\tilde{B}^\circ_n$, and we also obtain a nested increasing sequence of excised unfolded Borel covering spaces which exhausts the Borel covering space $\tilde{B}$:
\begin{eqn}
	\tilde{B}^\circ_n \subset \tilde{B}^\circ_{n+1} \subset \cdots \subset \tilde{B}
\fullstop
\end{eqn}

By \autoref{250221092010}, for each $n$, there is a constant $\K_n > 0$ such that the integral operator $\Upsilon$ restricts to a contraction mapping on the unit ball $\mathscr{H}^\heartsuit_{\K_n} (\tilde{B}^\circ_n, \CC^2)$.
By the Contraction Theorem on a Ball (for example, see \cite[Theorem 5.1-4]{MR992618}), $\Upsilon$ has a unique fixed point $\phi_n \in \mathscr{H}^\heartsuit_{\K_n} (\tilde{B}^\circ_n, \CC^2)$.
Thus, we have produced a sequence of vector-valued holomorphic functions $\phi_n \in \cal{O} (\tilde{B}_n^\circ, \CC^2)$ each of which is a fixed point of the integral operator $\Upsilon$.
We will now argue that this sequence converges to an element $\phi \in \cal{O} (\tilde{B}, \CC^2)$ as $n \to \infty$.

By \autoref{250221114438}, we can arrange the constants $\K_n$ such that $\K_n \leq \K_{n+1}$, yielding inclusions
\begin{eqn}
	\mathscr{H}^\heartsuit_{\K_n} (\tilde{B}^\circ_n, \CC^2)
	\subset \mathscr{H}^\heartsuit_{\K_{n+1}} (\tilde{B}^\circ_n, \CC^2)
\fullstop
\end{eqn}
At the same time, each inclusion $\tilde{B}^\circ_n \inj \tilde{B}^\circ_{n+1}$ induces the reverse inclusion
\begin{eqn}
	\mathscr{H}^\heartsuit_{\K_{n+1}} (\tilde{B}^\circ_{n+1}, \CC^2)
	\subset \mathscr{H}^\heartsuit_{\K_{n+1}} (\tilde{B}^\circ_n, \CC^2)
\fullstop
\end{eqn}
This means that for every $n$ we can view both $\phi_n$ and the restriction of $\phi_{n+1}$ to the open subset $\tilde{B}^\circ_n \subset \tilde{B}^\circ_{n+1}$ as elements of the same function space $\mathscr{H}^\heartsuit_{\K_{n+1}} (\tilde{B}^\circ_n, \CC^2)$.
It follows that they must be equal because they are fixed points of $\Upsilon$ but, by \autoref{250221092010}, the integral operator $\Upsilon$ restricts to a contraction mapping on $\mathscr{H}^\heartsuit_{\K_{n+1}} (\tilde{B}^\circ_n, \CC^2)$ and so must have a unique fixed point.
We deduce that the holomorphic vector-valued function $\phi_{n+1} \in \cal{O} (\tilde{B}^\circ_{n+1}, \CC^2)$ is the analytic continuation to $\tilde{B}^\circ_{n+1}$ of the holomorphic vector-valued function $\phi_n \in \cal{O} (\tilde{B}^\circ_{n}, \CC^2)$ from the open subset $\tilde{B}^\circ_n \subset \tilde{B}^\circ_{n+1}$.
Consequently, the sequence $(\phi_n)$ converges to a holomorphic function $\phi \in \cal{O} (\tilde{B}, \CC^2)$ that satisfies $\Upsilon [\phi] = \phi$.

The uniqueness of $\phi \in \cal{O} (\tilde{B}, \CC^2)$ is the result of holomorphy, uniqueness of fixed points $\Upsilon$, and uniqueness of analytic continuation.
Indeed, any other such $\phi' \in \cal{O} (\tilde{B}, \CC^2)$ would necessarily coincide with $\phi$ on the open subset $\tilde{B}^\circ_1$ by the uniqueness of fixed points of $\Upsilon$, so its analytic continuation to all of $\tilde{B}$ is nothing but $\phi$.
\end{proof}

As an immediate corollary of the proof of \autoref{250311180618} (and in particular of \autoref{250221092010}), we obtain the following strong bound on the solution $\phi$.

\begin{corollary}
\label{250313194333}
For any excised Borel covering space $\tilde{B}^\circ$, there is a constant $\K > 0$ such that the unique solution $\phi \in \cal{O} (\tilde{B}, \CC^2)$ of the integral equation \eqref{250305181917} satisfies the bound
\begin{eqntag}
\label{250315115800}
	\norm{ \phi \big|_{\tilde{B}^\circ} }_\K \leq 1
\fullstop
\end{eqntag}
Moreover, $\K$ can be taken arbitrarily small provided that the excision $\bar{\DD} \subset \CC_z$ is sufficiently large.
In particular, for any $\bm{\gamma} \in \tilde{B}^\circ$,
\begin{eqntag}
\label{250315115805}
	\inf \int_0^{|\bm{\gamma}|} e^{-\K s} \big| \phi (\bm{\gamma}_s) \D{s} \big|
	\leq 1
\fullstop{,}
\end{eqntag}
where $\D{s} = \pm \rm{t}^\ast_\pm \lambda$ and the infimum is taken over all arc-length source-truncations $\bm{\gamma}_s$ of $\bm{\gamma}$.
\end{corollary}

\subsection{Resurgence of the Associated System}
\label{250311190749}

Recall that that in \autoref{250221201337} we constructed a \textit{local} holomorphic solution $\hat{\phi}$ of the Initial Value Problem \eqref{250306084618} as the Borel transform of the formal vector-valued solution $\hat{f}$ of the differential system \eqref{250216174108}.
By \autoref{250221201747}, that local solution $\hat{\phi} = \hat{\phi} (z, \xi)$ is valid for all nonzero $z$ and all sufficiently small $\xi$; i.e., in an open neighbourhood $\Xi \subset \CC^\ast_z \times \CC_\xi$ of $\xi = 0$.
Meanwhile, in \autoref{250221200950} we explained that this open neighbourhood $\Xi$ is canonically isomorphic to an open neighbourhood $\tilde{\Omega} \subset \tilde{B}$ of the embedding $\CC^\ast_z \inj \tilde{B}$ using the appropriate anchor map $\varrho$.
Furthermore, by \autoref{250222201015}, the Initial Value Problem \eqref{250306084618} is equivalent over the isomorphism $\varrho : \tilde{\Omega} \to \Xi$ to the Initial Value Problem \eqref{250305170843}.
Therefore, by constructing the global solution $\phi$ of \eqref{250305170843} in \autoref{250304154803}, we have constructed the analytic continuation of the local solution $\hat{\phi}$ to the entire unfolded Borel covering space $\tilde{B}$.
We can therefore deduce from \autoref{250311180618} the following analogue of our main \autoref{250304131355}.

\begin{proposition}
\label{250311180705}
The formal solution $\hat{f} (z,\hbar)$ of the differential system \eqref{250216174108} is resurgent.
\end{proposition}

Much like our main \autoref*{250304131355}, this Proposition packs a lot of information, so we break it down into three parts corresponding to the three defining aspects of resurgence.
First, the convergence of the Borel transform of $\hat{f}$ is equivalent to the existence of the local solution $\hat{\phi}$ which was covered by \autoref{250221201747}; for completeness, we state this explicitly in \autoref{250314131034}, namely, \autoref{250314131125}.
Second, the property of endless analytic continuation of the Borel transform is described in \autoref{250314125038}; namely, \autoref{250314125007}.
Third, the property of exponential type of the Borel transform is explained in \autoref{250314125118}; namely, \autoref{250314125137}.
Throughout, we indicate how these assertions imply the corresponding results in \autoref{250313153618} for the deformed Painlevé I equation.

\subsubsection{Convergence of the Borel transform.}
\label{250314131034}
First, we make explicit the convergence of the Borel transform of $\hat{f}$.
This follows from \autoref{250221201747} and the discussion in \autoref{250221200950}.

\begin{proposition}[convergence of the Borel transform]
\label{250314131125}
The Borel transform
\begin{eqntag}
\label{250314132334}
	\hat{\phi} (z,\xi) = \sum_{n=0}^\infty \phi_n (z) \xi^n
	\coleq \Borel \big[ \, \hat{f} \, \big] (z, \xi)
	= \sum_{n=0}^\infty \tfrac{1}{n!} f_{n+1} (z) \xi^n
\end{eqntag}
of the formal solution $\hat{f} (z,\hbar)$ of the differential system \eqref{250216174108} is a convergent power series in $\xi$, locally uniformly for all nonzero $z$.
Moreover, $\hat{\phi}$ can be canonically regarded via the anchor map $\varrho : B \to \CC_z \times \CC_\xi$ as a holomorphic vector-valued germ $\hat{\phi}$ along the embedding $\CC^\ast_z \inj B$.
As such, there is an open neighbourhood $\Omega \subset B$ of the embedding $\CC^\ast_z \inj B$ such that $\hat{\phi}$ defines a holomorphic vector-valued function on $\Omega$.
Equivalently, $\hat{\phi}$ is a holomorphic vector-valued germ along the embedding $\CC^\ast_z \inj \tilde{B}$ via the anchor map $\varrho : \tilde{B} \to \CC_z^\ast \times \CC_\xi$.
As such, there is an open neighbourhood $\tilde{\Omega} \subset \tilde{B}$ of the embedding $\CC^\ast_z \inj B$ such that $\hat{\phi}$ defines a holomorphic vector-valued function on $\tilde{\Omega}$.
In symbols, $\hat{\phi} \in \cal{O} (\Omega, \CC^2) \cong \cal{O} (\tilde{\Omega}, \CC^2)$.
\end{proposition}

\subsubsection{Endless Analytic Continuation.}
\label{250314125038}
Second, we describe the property of endless analytic continuation of the Borel transform $\hat{\phi}$
In essence, it says that for every fixed nonzero $z$, the convergent $\xi$-power series $\hat{\phi} (z,\xi)$ admits maximal analytic continuation to the universal cover of a Riemann surface punctured in ten points which is a fivefold covering of the Borel $\xi$-plane.
This Riemann surface is nothing but the Borel surface $B_z = \rm{s}^\inv (z)$, the ten points are nothing but the critical elements $\Gamma_z \subset B_z$, and the fivefold cover is nothing but the central charge $\Z_z : B_z \to \CC_\xi$ with branch points at $\xi_\pm = \pm \tfrac{1}{30} z^5$.
Furthermore, as $z$ varies but stays nonzero, these different analytic continuations of $\hat{\phi}$ patch together into a global holomorphic vector-valued function $\phi$ on a two-dimensional manifold, which turns out to be the unfolded Borel covering space $\tilde{B}$.
Therefore, the following assertion is an immediate consequence of \autoref{250311180618}.

\begin{proposition}[endless analytic continuation]
\label{250314125007}
The Borel transform $\hat{\phi}$ admits endless analytic continuation $\phi$ to the unfolded Borel space $B$ away from the critical locus $\Gamma \subset B$.
Namely, $\phi$ defines a global vector-valued holomorphic function on the unfolded Borel covering space $\tilde{B}$; in symbols, $\phi \in \cal{O} (\tilde{B}, \CC^2)$.
\end{proposition}

We may sometimes refer to the holomorphic vector-valued function $\phi \in \cal{O} (\tilde{B}, \CC^2)$ as the \dfn{global Borel transform} of $\hat{f}$ in order to distinguish it from the holomorphic germ $\hat{\phi}$.
The property of endless analytic continuation of $\hat{\phi}$ may be stated in the following rather more elementary terms by sacrificing the more refined geometric information.

\begin{corollary}
\label{250311170342}
Fix any nonzero $z_0 \in \CC_z$ and let $\xi_\pm \coleq \pm \tfrac{1}{30} z_0^5$.
Then the Borel transform $\hat{\phi} (z_0, \xi) \in \CC^2 \otimes \CC \cbrac{\xi}$ extends to a holomorphic vector-valued function $\phi (z_0, \xi)$ on the universal cover of the twice-punctured $\xi$-plane $\CC_\xi \smallsetminus \set{ \xi_+, \xi_- }$.
In symbols, $\hat{\phi} (z_0, \xi) \in \cal{O} (\widetilde{\CC_\xi \!\smallsetminus\! \set{ \xi_\pm} }, \CC^2)$.
Moreover, since the universal cover of a twice-punctured complex plane is (noncanonically) isomorphic to the upper halfplane $\HH$, we can (noncanonically) view $\phi$ as a holomorphic vector-valued function on $\HH$; i.e., $\phi \in \cal{O} (\HH, \CC^2)$.
\end{corollary}

We can now infer the analogous endless analytic continuation property for the deformed Painlevé I equation and therefore prove \autoref{250210171316}.

\begin{proof}[Proof of \autoref{250210171316}.]
By \autoref{250224110855}, the Borel transform $\hat{\omega} (t,\xi)$ of the formal power series solution $\hat{q} (t,\hbar)$ is related to the Borel transform $\hat{\phi}$ by the identity \eqref{250223140730}, which we quote here for convenience:
\begin{eqntag}
\label{250311193209}
	\hat{\omega} (t,\xi) = f_0^+ (z) + f_0^- (z) + \int_0^\xi \big( \hat{\phi}_+ (z,s) + \hat{\phi}_- (z,s) \big) \d{s}
\qtext{where}
	z^4 = -24 t
\fullstop
\end{eqntag}
So, for any $\bm{\gamma} \in \tilde{B}$ and any synchronised parameterisation, we define
\begin{eqntag}
\label{250310180010}
	\omega (\bm{\gamma}) 
		\coleq \rm{s}^\ast f^+_0 (\bm{\gamma}) + \rm{s}^\ast f^-_0 (\bm{\gamma})
			+ \int_0^{|\bm{\gamma}|} \big( \phi_+ (\bm{\gamma}_s) + \phi_- (\bm{\gamma}_s) \big) \D{s}
\end{eqntag}
where $\D{s} = \pm \rm{t}_\pm^\ast \lambda_{\bm{\gamma}_s}$.
This expression does not depend on the choice of representative of $\bm{\gamma}$ because $\phi_+, \phi_-$, and $\rm{t}^\ast_\pm \lambda$ are holomorphic on $\tilde{B}$.
Therefore, it defines a global holomorphic function $\omega$ on the unfolded Borel covering space $\tilde{B}$.
By \autoref{250224110855}, we also have under the negation map $z \mapsto -z$ the following symmetries: $\hat{\phi}_\pm (-z,\xi) = \hat{\phi}_\mp (z, \xi)$,~ $f_0^\pm (-z) = f_0^\mp (z)$,~ and $\hat{\omega} (-z,\xi) = \hat{\omega} (z,\xi)$.
But the lift to $\tilde{B}$ of this negation map is the involution $\sigma$, so $\omega$ is a $\sigma$-invariant function on $\tilde{B}$; i.e., $\sigma^\ast \omega = \omega$.
Thus, $\omega$ defines a holomorphic function on the quotient $\tilde{B} / \sigma$, which is the Borel covering space $\tilde{M}$ by definition.
In symbols, $\omega \in \cal{O}(\tilde{B})^\sigma = \cal{O} (\tilde{B}/\sigma) = \cal{O} (\tilde{M})$.
Now, all the other assertions of \autoref{250210171316} follow from the detailed geometric descriptions in \autoref{250304152932}.
\end{proof}

\subsubsection{Exponential Type.}
\label{250314125118}
Finally, we describe the exponential bounds at infinity in $\xi$ on the global Borel transform $\phi$.

\begin{proposition}[exponential type]
\label{250314125137}
The Borel transform $\phi \in \cal{O} (\tilde{B}, \CC^2)$ has exponential type at infinity on $\tilde{B}$, locally uniformly for all nonzero $z$.
That is to say, for every nonzero $z$ and any direction $\alpha$ at infinity in $\tilde{B}_z$, there are real constants $\C, \K > 0$ and a sectorial neighbourhood $\sfSigma_z \subset \tilde{B}_z$ whose opening contains $\alpha$ such that
\begin{eqntag}
\label{250312170037}
	\big| \phi (\bm{\gamma}) \big| \leq \C e^{\K |\Z (\bm{\gamma})|}
\qqquad
	\forall \bm{\gamma} \in \sfSigma_{\tau} \subset \tilde{B}_z
\fullstop
\end{eqntag}
Furthermore, every nonzero $z_0 \in \CC_z$ has a neighbourhood $U_0 \subset \CC_z$ with the property that $\C, \K$ can be chosen uniformly for all $z \in U_0$, provided that $\alpha$ and $\sfSigma_z$ are chosen continuously for all $z \in U_0$.
Moreover, $\K$ can be taken arbitrarily small provided that $z_0$ is sufficiently large.
\end{proposition}

\begin{proof}
This follows from \autoref{250313194333}, and more specifically from \eqref{250315115805}, by recognising that the central charge of $\bm{\gamma}$ is bounded by the arc-length of any synchronised parameterisation of $\bm{\gamma}$; i.e., $|\Z(\bm{\gamma})| \leq |\bm{\gamma}|$.
\end{proof}

The property of exponential type of $\phi$ may likewise be stated in more elementary terms using the point of view on $\phi$ as a multivalued function over the twice-punctured $\xi$-plane as follows.

\begin{corollary}
\label{250311182904}
Fix any nonzero $z_0 \in \CC_z$ and put $\xi_\pm \coleq \pm \tfrac{1}{30} z_0^5$.
Then the Borel transform $\phi (z_0, \xi)$ has exponential type at infinity in all directions of the $\xi$-plane.
That is to say, for any choice of branch of the multivalued function $\phi (z_0, \xi)$ on $\CC_\xi \smallsetminus \set{\pm \xi}$, there are constants $\C, \K > 0$ such that, for all $\xi$ sufficiently large,
\begin{eqntag}
\label{250315121642}
	\big| \phi (z_0, \xi) \big| \leq \C e^{\K |\xi|}
\fullstop
\end{eqntag}
Furthermore, every nonzero $z_0 \in \CC_z$ has a neighbourhood $U_0 \subset \CC_z$ with the property that $\C, \K$ can be chosen uniformly for all $z \in U_0$, provided that the branch of $\phi (z, \xi)$ is chosen continuously for all $z \in U_0$.
Moreover, $\K$ can be taken arbitrarily small provided that $z_0$ is sufficiently large.
\end{corollary}

We may now complete the proof of the final ingredient in the proof of our main \autoref{250304131355} concerning the resurgence of the deformed Painlevé I equation, which is \autoref{250313152819} about the exponential type of the Borel transform $\hat{\omega}$ of $\hat{q}$.
This relies on the following analogue of \autoref{250313194333} which follows immediately from \autoref*{250313152819} and from the defining formula \eqref{250310180010} for $\omega$.

\begin{corollary}
\label{250314205102}
For any excised Borel covering space $\tilde{B}^\circ$, there is a constant $\K > 0$ such that the global Borel transform $\omega \in \cal{O} (\tilde{B})^\sigma$ satisfies the bound
\begin{eqntag}
\label{250314205140}
	\norm{ \omega \big|_{\tilde{B}^\circ} }_\K \leq 1
\fullstop
\end{eqntag}
Moreover, $\K$ can be taken arbitrarily small provided that the excision $\bar{\DD} \subset \CC_z$ is sufficiently large.
In particular, for any $\bm{\gamma} \in \tilde{B}^\circ$,
\begin{eqntag}
\label{250315120535}
	\inf \int_0^{|\bm{\gamma}|} e^{-\K s} \big| \omega (\bm{\gamma}_s) \D{s} \big|
	\leq 1
\fullstop{,}
\end{eqntag}
where $\D{s} = \pm \rm{t}^\ast_\pm \lambda$ and the infimum is taken over all arc-length source-truncations $\bm{\gamma}_s$ of $\bm{\gamma}$.
\end{corollary}

\subsection{Borel Summability and the Stokes Phenomenon}
\label{250311191044}

In this final subsection, we describe the Borel summability properties of the formal solution $\hat{f}$ of the differential system \eqref{250216174108}.
In addition, the Borel summability properties of the formal solution $\hat{q}$ of the deformed Painlevé I system, described in \autoref{250304162404}, are direct consequences of their analogues for $\hat{f}$, which we make explicit in this subsection.
Just like in \autoref{250304162404}, we prioritise clarity and precision over conciseness.

First, we introduce two special directions at every nonzero point in the $z$-plane in which the Borel resummation of $\hat{f}$ is not well-defined.

\begin{definition}[Stokes directions]
\label{250305124829}
For any $z_0 \in \CC^\ast_z$, let $\vartheta_0 \coleq \arg (z_0)$.
We define the \dfn{Stokes directions} at $z_0$ to be the directions
\begin{eqn}
	\alpha_+ \coleq 5\vartheta_0
\qtext{and}
	\alpha_- \coleq 5\vartheta_0 + \pi
\fullstop
\end{eqn}
See \autoref{250221154526}.
All other directions are called \dfn{regular directions}, and we put
\begin{eqn}
	A_1 \coleq (\alpha_+, \alpha_-)
\qtext{and}
	A_2 \coleq (\alpha_-, \alpha_+)
\fullstop
\end{eqn}
\end{definition}

\subsubsection{Pointwise Borel summability.}
First, we spell out the case of Borel summability of $\hat{f}$ for a fixed value of $z$.
We start by considering a single regular direction.

\begin{proposition}[Pointwise Borel Summability in a Single Direction]
\label{250224130127}
For any nonzero $z_0 \in \CC_z$, the formal vector-valued solution $\hat{f} (z_0, \hbar) \in \CC^2 \otimes \CC \bbrac{\hbar}$ is stably Borel summable in every regular direction $\alpha$ at $z_0$.
Thus, the Borel resummation 
\begin{eqntag}
\label{250310154303}
	f_\alpha (z_0, \hbar) 
		\coleq s_\alpha [\, \hat{f} \,] (z_0, \hbar)
		= f_0 + \Laplace_\alpha [\, \hat{\phi} \,] (z_0, \hbar)
\end{eqntag}
in the direction $\alpha$ defines a holomorphic vector-valued function on a sectorial neighbourhood $S \subset \CC_\hbar$ of the origin with opening $\sfop{Arc}_\pi (\alpha)$ which is asymptotic to $\hat{f} (z_0, \hbar)$ of uniform factorial type:
\begin{eqntag}
\label{250224153942}
	f_\alpha (z_0, \hbar) \simeq \hat{f} (z_0, \hbar)
\qquad
	\text{as $\hbar \to 0$ unif. along $\sfop{Arc}_\pi (\alpha)$}
\fullstop
\end{eqntag}
In fact, $f_\alpha (z_0, \hbar)$ is unique with respect to this property.
Furthermore, the sectorial neighbourhood $S \subset \CC_\hbar$ can be chosen to be the straight sector $S = \sfop{Sect}_\pi (\alpha; r)$ of some radius $r > 0$ that can be taken arbitrarily large provided that $z_0$ is sufficiently large.
In particular, \autoref{250223133010} is true.
\end{proposition}

\begin{proof}
This is a direct consequence of the exponential type property of $\phi$, \autoref{250314125137}.
In fact, in this case we may appeal to the more elementary formulation in \autoref{250311182904}.
Indeed, if $\alpha$ is a regular ray at $z_0$, then the infinite straight ray $e^{i\alpha} \RR_+ \subset \CC_\xi$ does not pass through any of the two Borel singular values $\xi_\pm$.
Consequently, the Laplace transform of $\phi$ in the direction $\alpha$,
\begin{eqn}
	\int_{e^{i\alpha} \RR_+} e^{-\xi/\hbar} \phi (z_0, \xi) \d{\xi}
\end{eqn}
is well-defined because there is some $\K > 0$ such that $\phi$ satisfies the exponential bound \eqref{250315121642} for sufficiently large $\xi$.
Furthermore, $\alpha$ necessarily belongs to an open arc $(\alpha - \epsilon, \alpha + \epsilon)$ of regular directions, which means the Laplace transform of $\phi$, and hence the Borel resummation of $\hat{f}$, is well-defined in a sectorial neighbourhood of the origin in the $\hbar$-plane with opening angle $\pi + 2\epsilon$.
Such a sectorial domain necessarily contains a straight sector of opening angle $\pi$.
\end{proof}

As we vary the ray $\alpha$ through an arc $A$ consisting of regular directions only, the different Borel resummations $f_\alpha (z_0, \hbar)$ can be patched together into a single holomorphic function $f_A (z_0, \hbar)$ on a larger sector in the $\hbar$-plane.
This leads to the following description of Borel summability of $\hat{f}$ along an arc of directions.

\begin{corollary}[Pointwise Borel Summability in an Arc of Directions]
\label{250224154010}
For any nonzero $z_0 \in \CC_z$ and any arc $A \subset \SS^1$ of regular directions at $z_0$, the Borel resummations $f_\alpha (z_0, \hbar)$ of $\hat{f} (z_0, \hbar)$ for $\alpha \in A$ assemble into a single holomorphic vector-valued function $f_A (z_0, \hbar)$ defined on a sectorial neighbourhood $S \subset \CC_\hbar$ of the origin with opening $\sfop{Arc}_\pi (A)$ which is asymptotic to $\hat{f} (z_0, \hbar)$ of factorial type:
\begin{eqntag}
\label{250224154126}
	f_A (z_0, \hbar) \simeq \hat{f} (z_0, \hbar)
\qquad
	\text{as $\hbar \to 0$ along $\sfop{Arc}_\pi (A)$}
\fullstop
\end{eqntag}
In fact, $f_A (z_0, \hbar)$ is unique with respect to this property.
Furthermore, if $A$ is not bounded by a Stokes ray at $z_0$, then the sectorial neighbourhood $S \subset \CC_\hbar$ can be chosen to be the straight sector $S = \sfop{Sect}_\pi (A; r)$ of some radius $r > 0$ that can be taken arbitrarily large provided that $z_0$ is sufficiently large.
In particular, \autoref{250224153234} is true.
\end{corollary}

\subsubsection{Locally uniform Borel summability.}
Next, we state the Borel summability property of $\hat{f}$ when $z$ is allowed to vary in a small local neighbourhood of a fixed nonzero point in the $z$-plane.
Again, we start by considering summability in a single regular direction.

\begin{proposition}[Locally Uniform Borel Summability in a Single Direction]
\label{250224161315}
For any nonzero $z_0 \in \CC_z$ and any regular direction $\alpha$ at $z_0$, there is a neighbourhood $U \subset \CC_z^\ast$ around $z_0$ such that $\hat{f} (z,\hbar)$ is stably Borel summable in the direction $\alpha$ uniformly for all $z \in U$.
Thus, there is a sectorial neighbourhood $S \subset \CC_\hbar$ of the origin with opening $\sfop{Arc}_\pi (\alpha)$ such that the Borel resummation $f_\alpha (z,\hbar)$ of $\hat{f} (z,\hbar)$ in the direction $\alpha$ defines a holomorphic vector-valued function on the domain $U \times S$ which is uniformly asymptotic to $\hat{f} (z, \hbar)$ of uniform factorial type:
\begin{eqntag}
\label{250224161547}
	f_\alpha (z, \hbar) \simeq \hat{f} (z, \hbar)
\qquad
	\text{as $\hbar \to 0$ unif. along $\sfop{Arc}_\pi (\alpha)$,}
\end{eqntag}
uniformly for all $z \in U$.
In fact, $f_\alpha$ is unique with respect to this property.
Furthermore, the sectorial neighbourhood $S \subset \CC_\hbar$ can be chosen to be the straight sector $S = \sfop{Sect}_\pi (\alpha; r)$ of some radius $r > 0$ that can be taken arbitrarily large provided that $z_0$ is sufficiently large and $U$ is sufficiently small.
In particular, \autoref{250211094021} is true.
\end{proposition}

As we vary both the ray $\alpha$ in an arc $A$ and $z$ in a neighbourhood $U_0$ of $z_0$, we can patch the Borel resummations provided that $A$ and $U_0$ are sufficiently small in a coherent way as follows.

\begin{corollary}[Locally Uniform Borel Summability in an Arc of Directions]
\label{250224160903}
For any nonzero $z_0 \in \CC_z$ and any arc $A \subset \SS^1$ of regular directions at $z_0$ whose boundary does not contain any Stokes directions at $z_0$, there is a neighbourhood $U \subset \CC_z^\ast$ around $z_0$ such that $\hat{f} (z,\hbar)$ is Borel summable in every direction $\alpha \in A$ uniformly for all $z \in U$.
Thus, there is a sectorial neighbourhood $S \subset \CC_\hbar$ of the origin with opening $\sfop{Arc}_\pi (A)$ such that the Borel resummations $f_\alpha (z,\hbar)$ of $\hat{f} (z,\hbar)$ for $\alpha \in A$ assemble into a single holomorphic vector-valued function $f_A (z,\hbar)$ defined on the domain $U \times S$ which is uniformly asymptotic to $\hat{f} (z, \hbar)$ of factorial type:
\begin{eqntag}
\label{250224161227}
	f_A (z, \hbar) \simeq \hat{f} (z, \hbar)
\qquad
	\text{as $\hbar \to 0$ along $\sfop{Arc}_\pi (A)$,}
\end{eqntag}
uniformly for all $z \in U$.
In fact, $f_A$ is unique with respect to this property.
Furthermore, the sectorial neighbourhood $S \subset \CC_\hbar$ can be chosen to be the straight sector $S = \sfop{Sect}_\pi (A; r)$ of some radius $r > 0$ that can be taken arbitrarily large provided that $z_0$ is sufficiently large and $U$ is sufficiently small.
In particular, \autoref{250224155618} is true.
\end{corollary}

\subsubsection{Borel summability in Stokes sectors.}

Now we describe the Borel summability properties of $\hat{f}$ in the large.
For this purpose, we fix a direction for the resummation and describe the distinguished regions in the $z$-plane where the Borel resummation in this direction is well-defined.

\begin{definition}[Stokes lines and sectors]
\label{250310192447}
Fix any phase $\alpha \in \SS^1$.
A \dfn{Stokes line} in the $z$-plane is any infinite straight ray $e^{i \theta} \RR_+ \subset \CC_z$ with phase $\theta$ satisfying $10 \theta + 2 \alpha \equiv 0$.
The union of all Stokes lines is called the \dfn{Stokes graph}.
Any connected component $V \subset \CC_z$ of the complement of the Stokes graph is called a \dfn{Stokes sector}.
\end{definition}

Notice that the equation $10 \theta + 2 \alpha \equiv 0$ has exactly ten distinct solutions in $\SS^1$ distributed evenly around the circle.
So for the differential system \eqref{250216174108}, there are a total of ten Stokes lines and therefore ten Stokes sectors each of which is an infinite sector in the $z$-plane with opening angle $\pi/5$.

\begin{proposition}[Borel Summability in Stokes Sectors]
\label{250305151549}
Fix any phase $\alpha \in \SS^1$ and select a Stokes sector $V \subset \CC_z$.
Then $\hat{f} (z,\hbar)$ is stably Borel summable in the direction $\alpha$, locally uniformly for all $z \in V$.
Thus, the Borel resummation $\hat{f}_\alpha (z,\hbar) = s_\alpha [\, \hat{f} \,] (z,\hbar)$ defines a holomorphic vector-valued function on a domain $\VV \subset V \times \CC_\hbar$ with the following property: every $z_0 \in V$ has a neighbourhood $V_0 \subset V$ such that there is a sectorial neighbourhood $S_0 \subset \CC_\hbar$ of the origin with opening $\sfop{Arc}_\pi (\alpha)$ satisfying $V_0 \times S_0 \subset \VV$.
Furthermore, $f_\alpha$ is locally uniformly asymptotic to $\hat{f}$ of uniform factorial type:
\begin{eqntag}
	f_\alpha (z, \hbar) \simeq \hat{f} (z,\hbar)
\qquad
	\text{as $\hbar \to 0$ unif. along $\sfop{Arc}_\pi (\alpha)$,}
\end{eqntag}
locally uniformly for all $z \in V$.
In fact, $f_\alpha$ is the unique holomorphic vector-valued function on $\VV$ with this property.
Moreover, $V$ is a maximal domain of Borel summability of $\hat{f}$ in the direction of $\alpha$; i.e., there does not exist another domain $V'$ properly containing $V$ such that $\hat{f} (z,\hbar)$ is stably Borel summable in the direction $\alpha$, locally uniformly for all $z \in V$.
In particular, \autoref{250210133432} is true.
\end{proposition}

\subsubsection{The Stokes phenomenon.}
Finally, we describe the failure of Borel summability in a Stokes direction and the associated Stokes phenomenon.

\begin{corollary}[Lateral Borel Summability]
\label{250225105212}
Fix any nonzero $z_0 \in \CC_z$ and suppose $\alpha$ is a Stokes ray at $z_0$.
Then $\hat{f} (z_0, \hbar)$ is laterally Borel summable in the direction $\alpha$.
Thus, the left and right lateral Borel resummations 
\begin{eqntag}
\label{250315122645}
	f_\alpha^\textup{L/R} (z_0, \hbar) 
		\coleq s_\alpha^\textup{L/R} \big[ \, \hat{f} \, \big] (z_0, \hbar)
		= f_0 + \Laplace_\alpha^\textup{L/R} \big[ \, \phi \, \big] (z_0, \hbar)
\end{eqntag}
in the direction $\alpha$ define two holomorphic vector-valued functions on a sectorial neighbourhood $S \subset \CC_\hbar$ of the origin with opening $\sfop{Arc}_\pi (\alpha)$, each of which is asymptotic to $\hat{f} (z_0, \hbar)$ of factorial type:
\begin{eqntag}
\label{250225105520}
	f_\alpha^\textup{L/R} (z_0, \hbar) \simeq \hat{f} (z_0, \hbar)
\qquad
	\text{as $\hbar \to 0$ along $\sfop{Arc}_\pi (\alpha)$}
\fullstop
\end{eqntag}
\end{corollary}

\begin{definition}[Stokes jump]
\label{250312132123}
For any nonzero $z_0 \in \CC_z$, suppose $\alpha$ is a Stokes ray at $z_0$, and let $f_\alpha^\textup{L/R} (z_0, \hbar) \in \cal{O} (S, \CC^2)$ be the two lateral Borel resummations of $\hat{f} (z_0, \hbar)$ in the direction $\alpha$ from \autoref{250225105212}.
The \dfn{Stokes jump} across $\alpha$ from right to left is the difference between the left and right lateral Borel resummations of $\hat{f}$ in the direction $\alpha$:
\begin{eqntag}
\label{250314212108}
	\Delta_{\alpha} \hat{f} (z_0, \hbar) \coleq f_\alpha^\textup{L} (z_0, \hbar) - f_\alpha^\textup{R} (z_0, \hbar)
\fullstop
\end{eqntag}
\end{definition}

The Stokes jump is clearly a holomorphic vector-valued function on $S$; i.e., $\Delta_{\alpha} \hat{f} (z_0, \hbar) \in \cal{O} (S, \CC^2)$.
Furthermore, since both $f_\alpha^\textup{L}$ and $f_\alpha^\textup{R}$ admit the same asymptotic expansion as $\hbar \to 0$ in $S$, the Stokes jump is asymptotic to $0$ of factorial type:
\begin{eqntag}
\label{250314212112}
	\Delta_{\alpha} \hat{f} (z_0, \hbar) \simeq 0
\qquad
	\text{as $\hbar \to 0$ along $\sfop{Arc}_\pi (\alpha)$}
\fullstop
\end{eqntag}
Again, let us point out the absence of the qualifier ``uniformly'' in \eqref{250314212112}, for otherwise it would force the Stokes jump to be identically zero.

\begin{definition}[variation]
\label{250312143920}
For any nonzero $z_0 \in \CC_z$, suppose $\alpha$ is a Stokes ray at $z_0$, and let $\xi_0 \in \CC_\xi$ be the Borel singular value corresponding to the visible singularity in the direction $\alpha$ at $z_0$.
The \dfn{variation} of the multivalued vector function $\phi (z_0, \xi)$ at $\xi_0$ is the difference between its values on two consecutive sheets:
\begin{eqntag}
\label{250314213357}
	\Delta_{\xi_0} \phi (z_0, \xi) 
		\coleq \phi (z_0, \xi_0 + \xi^\textup{L}) - \phi (z_0, \xi_0 + \xi^\textup{R})
\fullstop
\end{eqntag}
where $\xi^\textup{L}$ and $\xi^\textup{R}$ are two preimages of $\xi$ on the universal cover of the punctured neighbourhood of $\xi_0$ related to each other by the anti-clockwise primitive generator of the deck transformations; i.e., $\xi^\textup{R} = e^{2 \pi i } \xi^\textup{L}$.
\end{definition}

The variation $\Delta_{\xi_0} \phi (z_0, \xi)$ may be regarded as a sectorial germ at the origin in the Borel $\xi$-plane.
But since both $\phi (z_0, \xi_0 + \xi^\textup{L})$ and $\phi (z_0, \xi_0 + \xi^\textup{R})$ admit analytic continuation in the direction $\alpha$ of exponential type at infinity, we obtain the following characterisation of the Stokes jump in terms of the Laplace transform of the Borel transform's variation at the visible singularity.

\begin{proposition}[Stokes Phenomenon]
\label{250312131625}
Fix any nonzero $z_0 \in \CC_z$, suppose $\alpha$ is a Stokes direction at $z_0$, and let $f_\alpha^\textup{L/R} (z_0, \hbar) \in \cal{O} (S)$ be the two lateral Borel resummations of $\hat{f} (z_0, \hbar)$ in the direction $\alpha$.
Let $\xi_0 \in \CC_\xi$ be the Borel singular value corresponding to the visible singularity in the direction $\alpha$ at $z_0$.
Then for all $\hbar \in S$, the Stokes jump across $\alpha$ is
\begin{eqntag}
\label{250315124501}
	\Delta_{\alpha} \hat{f} (z_0, \hbar)
	= e^{-\xi_0/\hbar} \Laplace_\alpha \big[ \Delta_{\xi_0} \phi \big] (z_0, \hbar)
\fullstop{,}
\end{eqntag}
where $\Delta_{\xi_0} \phi (z_0, \xi)$ is the variation of $\phi$ at $\xi_0$ given by \eqref{250314213357}.
\end{proposition}

\begin{proof}
This is a matter of a simple calculation.
Expanding \eqref{250314212108} according to its definition \eqref{250315122645}, we have:
\begin{eqn}
	\Delta_{\alpha} \hat{f} (z_0, \hbar) 
		= f_\alpha^\textup{L} (z_0, \hbar) - f_\alpha^\textup{R} (z_0, \hbar)
		= \Laplace_\alpha^\textup{L} \big[ \, \phi \, \big] (z_0, \hbar) 
			- \Laplace_\alpha^\textup{R} \big[ \, \phi \, \big] (z_0, \hbar)
		= \int_{\mathscr{C}} e^{-\xi/\hbar} \phi (z_0, \xi) \d{\xi}
\end{eqn}
where the integration contour $\mathscr{C}$ is homotopic to the keyhole contour around $\xi_0$ obtained by concatenating the inverse of the contour $e^{i \alpha} \RR_+^\textup{R}$ with $e^{i \alpha} \RR_+^\textup{L}$.
The two arms of this keyhole contour lift via the central charge $\Z_{z_0} : B_{z_0} \to \CC_\xi$ to infinite geodesics on two consecutive sheets of the Borel surface emanating from the visible Borel singularity above $\xi_0$.
Therefore, the above contour integral may be written like so:
\begin{eqn}
	\int_{\xi_0}^{e^{i \alpha} \cdot \infty} e^{-\xi/\hbar} \Big( \phi (z_0, \xi^\textup{L}) - \phi (z_0, \xi^\textup{R}) \Big) \d{\xi}
\fullstop{,}
\end{eqn}
where $\phi (z_0, \xi^\textup{L})$ and $\phi (z_0, \xi^\textup{R})$ denote the values at $(z_0, \xi)$ of the two relevant branches of the multivalued function $\phi (z_0, \xi)$.
Shifting the integration variable $\xi \to \xi_0 + \xi$, the factor $e^{-\xi_0/\hbar}$ appears and the integral becomes the Laplace transform of the variation \eqref{250314213357}.
So we get \eqref{250315124501}.
\end{proof}

\begin{adjustwidth}{-2cm}{-1.5cm}
{\footnotesize
\bibliographystyle{nikolaev}
\bibliography{References}
}
\end{adjustwidth}


\bigskip

\makeatletter
\newcommand\footnoteref[1]{\protected@xdef\@thefnmark{\ref{#1}}\@footnotemark}
\makeatother

\textsc{Mohamad Alameddine}\\
\small Universit\'{e} Jean Monnet, Institut Camille Jordan%
\footnote{\label{ICJ}Institut Camille Jordan, CNRS UMR 5208,
Les Forges 2, 20 Rue du Dr Annino, 42000 Saint-\'{E}tienne, France.}

\smallskip

\textsc{Olivier Marchal}\\
\small Universit\'{e} Jean Monnet, Institut Camille Jordan%
\footnoteref{ICJ}\\
\small Institut Universitaire de France

\smallskip

\textsc{Nikita Nikolaev}\\
\small School of Mathematics, University of Birmingham%
\footnote{Watson Building, Edgbaston, Birmingham, B15 2TT, United Kingdom.}

\smallskip

\textsc{Nicolas Orantin}\\
\small Section de Math\'{e}matiques, Universit\'{e} de Gen\`{e}ve%
\footnote{24 rue du G\'{e}n\'{e}ral Dufour, 1211 Gen\`{e}ve 4, Suisse.}

\end{document}